\documentclass[11pt,reqno]{book}

\usepackage[english]{babel}

\usepackage{amsmath}
\usepackage{amsfonts}
\usepackage{amssymb}
\usepackage{multicol}
\usepackage{amsthm}
\usepackage{graphicx}
\usepackage{epstopdf}
\usepackage{enumitem}
\usepackage{geometry}
\usepackage{tikz}
\usepackage[all]{xy}
\usepackage{appendix}

\usepackage{makeidx}

\usepackage{amscd}
\usepackage{stmaryrd}
\usepackage{mathrsfs} 
\usepackage{tocvsec2}
\usepackage[nottoc]{tocbibind}
\usepackage{textcomp}

\usepackage{titlesec}
\usepackage[hyperindex]{hyperref}

\usepackage{bookmark}

\newcommand{\commentout}[1]{}

\usepackage{suffix}

\newcommand\chapterauthor[1]{\authortoc{#1}\printchapterauthor{#1}}
\WithSuffix\newcommand\chapterauthor*[1]{\printchapterauthor{#1}}

\makeatletter
\newcommand{\printchapterauthor}[1]{%
  {\parindent0pt\vspace*{-25pt}%
  \linespread{1.1}\large\scshape#1%
  \par\nobreak\vspace*{35pt}}
  \@afterheading%
}
\newcommand{\authortoc}[1]{%
  \addtocontents{toc}{\vskip-10pt}%
  \addtocontents{toc}{%
    \protect\contentsline{chapter}%
    {\hskip1.3em\mdseries\scshape\protect\scriptsize#1}{}{}}
  \addtocontents{toc}{\vskip5pt}%
}
\makeatother
\newtheorem{theorem}{Theorem}[section]
\newtheorem{apptheorem}{Theorem}[chapter]
\newtheorem{proposition}[theorem]{Proposition}
\newtheorem*{proposition*}{Proposition}
\newtheorem{appproposition}[apptheorem]{Proposition}
\newtheorem{corollary}[theorem]{Corollary}
\newtheorem*{corollary*}{Corollary}

\newtheorem{lemma}[theorem]{Lemma}

\newtheorem{problem}[theorem]{Problem}
\newtheorem{note}[theorem]{Note}
\newtheorem*{theorem*}{Theorem}
\newenvironment{hint}[1][Hint]{\noindent\textit{#1.} }{\ \rule{0.5em}{0.5em}}

\theoremstyle{definition}

\newtheorem{example}[theorem]{Example}
\newtheorem{appexample}[apptheorem]{Example}
\newtheorem{remark}[theorem]{Remark}
\newtheorem{appremark}[apptheorem]{Remark}
\newtheorem{conjecture}[theorem]{Conjecture}
\newtheorem{openbroblem}[theorem]{Open Problem}
\newtheorem{exercise}[theorem]{Exercise}
\newtheorem{appexercise}[apptheorem]{Exercise}
\newtheorem{definition}[theorem]{Definition}
\newtheorem{appdefinition}[apptheorem]{Definition}

\newcommand{\Hom}{\operatorname{\mathbb{H}om}}

\newcommand{\abs}[1]{\lvert#1\rvert}
\newcommand{\norm}[1]{\lVert#1\rVert}
\newcommand{\R}{{\mathbb R}}

\newcommand{\N}{{\mathbb N}}
\newcommand{\ve}{{\varepsilon}}

\makeindex

\begin{document}

\pagenumbering{roman}
\setcounter{page}{1}
%\author{Valerio Capraro}
%\address{University of Neuchatel, Switzerland}
%\thanks{Partially supported by Swiss SNF Sinergia project CRSI22-130435}
%\email{valerio.capraro@unine.ch, valerio.capraro@virgilio.it}

%\keywords{Connes' embedding conjecture, Kirchberg theorem, hyperlinear groups, Brown's invariant}

%\subjclass[2000]{Primary 52A01; Secondary 46L36}

\date{}

\begin{titlepage}
\thispagestyle{empty}
\title{Introduction to Sofic and Hyperlinear groups\\ and Connes' Embedding Conjecture}
\author{VALERIO CAPRARO\\ Centrum voor Wiskunde en Informatica\\ Amsterdam, The Netherlands
        \and MARTINO LUPINI\\ Department of Mathematics\\ York University\\ Toronto, Canada\\\\\\
\and With an appendix by\\\\ VLADIMIR PESTOV\\ Departamento de Matem\'atica\\ Universidade Federal de Santa Catarina\\ Florian\'opolis-SC, Brasil}

\end{titlepage}

\maketitle

\setcounter{page}{2}
\thispagestyle{empty}
\frontmatter

\setcounter{page}{3}
\chapter*{Preface\label{Section: preface}}

Analogy is one of the most effective techniques of human reasoning: When we
face new problems we compare them with simpler and already known ones, in
the attempt to use what we know about the latter ones to solve the former
ones. This strategy is particularly common in Mathematics, which offers
several examples of abstract and seemingly intractable objects: Subsets of
the plane can be enormously complicated but, as soon as they can be
approximated by rectangles, then they can be measured; Uniformly finite
metric spaces can be difficult to describe and understand but, as soon as
they can be approximated by Hilbert spaces, then they can be proved to
satisfy the coarse Novikov's and Baum-Connes's conjectures.

These notes deal with two particular instances of such a strategy: Sofic and
hyperlinear groups are in fact the countable discrete groups that can be
approximated in a suitable sense by finite symmetric groups and groups of
unitary matrices. These notions, introduced by Gromov and R\u{a}dulescu, respectively, at
the end of the 1990s, turned out to be very deep and fruitful, and
stimulated in the last 15 years an impressive amount of research touching
several seemingly distant areas of mathematics including geometric group
theory, operator algebras, dynamical systems, graph theory, and more
recently even quantum information theory. Several long-standing conjectures
that are still open for arbitrary groups were settled in the case of sofic
or hyperlinear groups. These achievements aroused the interest of an
increasing number of researchers into some fundamental questions about the
nature of these approximation properties. Many of such problems are to this
day still open such as, outstandingly: Is there any countable discrete group
that is not sofic or hyperlinear? A similar pattern can be found in the
study of \textup{II}$_{1}$ factors. In this case the famous conjecture due to Connes
(commonly known as Connes' embedding conjecture) that any \textup{II}$_{1}$ factor
can be approximated in a suitable sense by matrix algebras inspired several
breakthroughs in the understanding of \textup{II}$_{1}$ factors, and stands out today
as one of the major open problems in the field.

The aim of this monograph is to present in a uniform and accessible way some
cornerstone results in the study of sofic and hyperlinear groups and Connes'
embedding conjecture. These notions, as well as the proofs of many results,
are here presented in the framework of model theory for metric structures.
We believe that this point of view, even though rarely explicitly adopted
in the literature, can contribute to a better understanding of the ideas
therein, as well as provide additional tools to attack many remaining open
problems. The presentation is nonetheless self-contained and accessible to
any student or researcher with a graduate-level mathematical background. In
particular no specific knowledge of logic or model theory is required.

Chapter 1 presents the conjectures and open problems that will serve as
common thread and motivation for the rest of the survey: Connes' embedding
conjecture, Gottschalk's conjecture, and Kaplansky's conjecture. Chapter 2
introduces sofic and hyperlinear groups, as well as the general notion of
metric approximation property; outlines the proofs of Kaplansky's direct finiteness
conjecture and the algebraic eigenvalues conjecture for sofic groups; and
develops the theory of entropy for sofic group actions, yielding a proof of
Gottschalk's surjunctivity conjecture in the sofic case. Chapter 3 discusses
the relationship between hyperlinear groups and the Connes' embedding
conjecture; establishes several equivalent reformulations of the Connes'
embedding conjecture due to Haagerup-Winsl\o w and Kirchberg; describes the purely algebraic approach initiated by R\u adulescu and carried over by Klep-Schweighofer and Juschenko-Popovich; and finally
outlines the theory of Brown's invariants for \textup{II}$_{1}$ factors satisfying
the Connes' embedding conjecture. An appendix by V. Pestov provides a pedagogically new introduction to
the concepts of ultrafilters, ultralimits, and ultraproducts for those
mathematicians who are not familiar with them, and aiming to make
these concepts appear very natural.

The choice of topics is unavoidably not exhaustive. A more detailed
introduction to the basic results about sofic and hyperlinear groups can be
found in \cite{pestov} and \cite{Kwiatkowska-Pestov}. The surveys \cite%
{Oz,Oz3,Oz4} contain several other equivalent reformulations of the Connes'
embedding conjecture in purely algebraic or C*-algebraic terms.

This survey originated from a short intensive course that the authors gave
at the Universidade Federal de Santa Catarina in 2013 in occasion of the
``Workshop on sofic and hyperlinear groups and the Connes' embedding
conjecture'' supported by CAPES (Brazil) through the program ``Science
without borders", PVE project 085/2012. We would like to gratefully thank
CAPES for its support, as well as the organizers of the workshop Daniel Gon\c{c}alves and Vladimir Pestov for their kind hospitality, and for their
constant and passionate encouragement.

Moreover we are grateful to Hiroshi Ando, Goulnara Arzhantseva, Samuel Coskey, Ilijas Farah, Tobias Fritz, Benjamin Hayes, Liviu P\u aunescu, Vladimir Pestov, David Sherman, Alain Valette, and five anonymous referees for several useful comments and suggestions.

\begin{multicols}{2}
\begin{center}
Valerio Capraro\\ Department of Mathematics\\ University of Southampton\\ United Kingdom
\end{center}
\columnbreak
\begin{center}
Martino Lupini\\ Department of Mathematics\\ York University\\ Toronto, Canada
\end{center}
\end{multicols}

%the course given by the authors at the
%University of Florianopolis, Brazil, in June 2013, aim to collect a
%reasonable amount of different topics around sofic and hyperlinear groups
%and the related Connes' embedding conjecture, in order to serve to both
%researcher and students from different areas for a self-contained and
%reasonably complete introduction to these problems. In particular, to help
%the student with a gradual and active learning, we developed a system of
%exercises that might help the newer to familiarize with the different
%techniques that are used in this field. Intersections with previously
%published survey papers, most notably by Ozawa \cite{Oz}, \cite{Oz4} and
%Pestov \cite{pestov}, are intentionally minimized. We refer to Ozawa's
%papers for an introduction to Connes' embedding conjecture from a more $C^*$%
%-algebraic point of view of from a more purely algebraic point of view,
%respectively, and to Pestov's survey for a more detailed introduction to
%basic results about sofic and hyperlinear groups.

\tableofcontents

\mainmatter

\setcounter{page}{1}
\pagenumbering{arabic}

\chapter{Introduction}
\chapterauthor{Valerio Capraro and Martino Lupini}

\section{Von Neumann algebras and \texorpdfstring{\textup{II}$_{1}$}{II1}
factors\label{Section: vN algebras and II1 factors}}

Denote by $B(H)$ the algebra of bounded linear operators on the Hilbert
space $H$. Recall that $B(H)$ is naturally endowed with an involution $%
x\mapsto x^{\ast }$ associating with an operator $x$ its \textit{adjoint }$%
x^{\ast }$. The \textit{operator norm}%
\index{norm!operator} $\left\Vert x\right\Vert $ of an element of $B(H)$ is
defined by%
\begin{equation*}
\left\Vert x\right\Vert =\sup \left\{ \left\Vert x\xi \right\Vert :\xi \in H%
\text{, }\left\Vert \xi \right\Vert \leq 1\right\} \text{.}
\end{equation*}%
Endowed with this norm, $B(H)$ is a Banach algebra with involution satisfying
the identity%
\begin{equation}
\left\Vert x^{\ast }x\right\Vert =\left\Vert x\right\Vert ^{2} 
\tag{C*-identity}
\end{equation}%
i.e.\ a \emph{C*-algebra}%
\index{C*-algebra}.

The \textit{weak operator topology}%
\index{topology!weak operator} on $B(H)$ is the weakest topology making the
map%
\begin{equation*}
x\mapsto \left\langle x\xi ,\eta \right\rangle
\end{equation*}%
continuous for every $\xi ,\eta \in H$, where $\left\langle \cdot ,\cdot
\right\rangle $ denotes the scalar product of $H$. The \textit{strong
operator topology}%
\index{topology!strong operator} on $B(H)$ is instead the weakest topology
making the maps%
\begin{equation*}
x\mapsto \left\Vert x\xi \right\Vert
\end{equation*}%
continuous for every $\xi \in H$. As the names suggest the strong operator
topology is stronger than the weak operator topology.\ It is a consequence
of the Hahn-Banach theorem that, conversely, a convex subset of $B(H)$
closed in the strong operator topology is also closed in the weak operator
topology (see Theorem 5.1.2 of \cite{Ka-Ri1}).

A (concrete) \textit{von Neumann algebra}%
\index{von Neumann algebra}\index{von Neumann algebra} is a unital *-subalgebra (i.e.\ closed with
respect to taking adjoints) of $B(H)$ that is closed in the weak (or,
equivalently, strong) operator topology. It is easy to see that if $X$ is a
subset of $B(H)$, then the intersection of all von Neumann algebras $%
M\subset B(H)$ containing $X$ is again a von Neumann algebra, called the von
Neumann algebra generated by $X$. Theorem \ref{Theorem: double commutant} is
a cornerstone result of von Neumann, known as von Neumann double commutant
theorem, asserting that the von Neumann algebra generated by a subset $X$ of 
$B(H)$ can be characterized in a purely algebraic way. The \textit{commutant}%
\index{commutant} $X^{\prime }$\index{commutant} of a subset $X$ of $B(H)$ is the set of $%
y\in B(H)$ commuting with every element of $X$. The \textit{double commutant}%
\index{commutant!duble} $X^{\prime \prime }$ \index{commutant!double} of $X$ is just the commutant
of $X^{\prime }$.

\begin{theorem}
\label{Theorem: double commutant}
The von Neumann algebra generated by a subset $X$ of $B(H)$ containing the unit and closed with respect to taking
adjoints coincides with the double commutant $X^{\prime \prime }$ of $X$.
\end{theorem}

A \textit{faithful normal} \emph{trace}%
\index{trace} on a von Neumann algebra $M$ is a linear functional $\tau $ on 
$M$ such that:

\begin{itemize}
\item $\tau (x^{\ast }x)\geq 0$ for every $x\in M$ ($\tau $ is \textit{%
positive});

\item $\tau (x^{\ast }x)=0$ implies $x=0$ ($\tau $ is \textit{faithful});

\item $\tau (xy)=\tau (yx)$ for every $x,y\in M$ ($\tau $ is \textit{tracial}%
);

\item $\tau (1)=1$ ($\tau $ is \textit{unital})

\item $\tau $ is continuous on the unit ball of $M$ with respect to the weak
operator topology ($\tau $ is \textit{normal}).
\end{itemize}

A (finite) von Neumann algebra endowed with a distinguished trace will be
called a \textit{tracial von Neumann algebra}%
\index{von Neumann algebra!tracial}. A tracial von Neumann algebra is always 
\textit{finite}\index{von Neumann algebra!finite} as in \cite[Definition 6.3.1]{Kadison-RingroseII}.
Conversely any finite von Neumann algebra faithfully represented on a
separable Hilbert space has a faithful normal trace by \cite[Theorem 8.2.8]%
{Kadison-RingroseII}.

The \textit{center}%
\index{center} $Z(M)$ of a von Neumann algebra $M\subset B(H)$ is the
subalgebra of $M$ consisting of the operators in $M$ commuting with any
other element of $M$. A von Neumann algebra $M$ is called a \textit{factor}%
\index{factor} if its center is as small as possible, i.e.\ it contains only
the scalar multiples of the identity. A finite factor has a unique faithful
normal trace (see \cite[Theorem 8.2.8]{Kadison-RingroseII}). Moreover any
faithful unital tracial positive linear functional on a finite factor is
automatically normal, and hence coincides with its unique faithful normal
trace.

\begin{example}
\label{Exercise: fd finite factors}If $H_{n}$ is a Hilbert space of finite
dimension $n$, then $B(H_{n})$ is a finite factor denoted by $M_{n}(\mathbb{C%
})$ isomorphic to the algebra of $n\times n$ matrices with complex
coefficients. The unique trace\index{trace} on $M_{n}(\mathbb{C})$ is the usual
normalized trace of matrices.
\end{example}

It is a consequence of the type classification of finite factors (see \cite[%
Section 6.5]{Kadison-RingroseII}) that the ones described in Example \ref%
{Exercise: fd finite factors} are the unique examples of finite factors that
are finite dimensional as vector spaces.

\begin{definition}
\label{defin:II1factors}A \textup{II}$_{1}$\emph{\ factor}%
\index{factor!type II$_1$} is an infinite-dimensional finite factor.
\end{definition}

Theorem \ref{Theorem: projections II1 factor} is a cornerstone result of
Murray and von Neumann (see \cite[Theorem XIII]{Mu-vN4}), offering a
characterization of \textup{II}$_{1}$ factors within the class of finite factors.

\begin{theorem}
\label{Theorem: projections II1 factor} If $M$ is a \textup{II}$_{1}$ factor, then $M$
contains, for every natural number $n$, a unital copy of $M_{n}(\mathbb{C})$%
, i.e.\ there is a trace preserving *-homomorphism from $M_{n}(\mathbb{C})$
to $M$.
\end{theorem}

It follows from weak continuity of the trace and the type classification of
finite factors that if $M$ is a finite factor, then the following statements
are equivalent:

\begin{enumerate}
\item $M$ is a \textup{II}$_{1}$ factor;

\item the trace\index{trace} $\tau $ of $M$ attains on projections all the real values
between $0$ and $1$.
\end{enumerate}

%\textup{II}$_{1}$ factors can be regarded as a sort of noncommutative version of atomless probability spaces. 
The unique trace $\tau $ on a finite factor $M$ allows to define the
Hilbert-Schmidt norm%
\index{norm!Hilbert-Schmidt} $\left\Vert \cdot \right\Vert _{2}$ on $M$, by $%
||x||_{2}=\tau (x^{\ast }x)^{%
\frac{1}{2}}$. The Hilbert-Schmidt norm is continuous with respect to the
operator norm $\left\Vert \cdot \right\Vert $ inherited from $B(H)$. A
finite factor is called \textit{separable}%
\index{factor!separable} if it is separable with respect to the topology
induced by the Hilbert-Schmidt norm.

Let us now describe one of the most important constructions of \textup{II}$_{1}$
factors. If $\Gamma $ is a countable discrete group then the \textit{complex
group algebra}
\index{group algebra!complex} $\mathbb{C}\Gamma $ is the complex algebra of
formal finite linear combinations%
\begin{equation*}
\lambda _{1}\gamma _{1}+\cdots +\lambda _{k}\gamma _{k}
\end{equation*}%
of elements of $\gamma $ with coefficients from $\mathbb{C}$. Any element of 
$\mathbb{C}\Gamma $ can be written as%
\begin{equation*}
\sum_{\gamma }a_{\gamma }\gamma ,
\end{equation*}%
where $(a)_{\gamma \in \Gamma }$ is a family of complex numbers all but
finitely many of which are zero. Sum and multiplication of elements of $%
\mathbb{C}\Gamma $ are defined by%
\begin{equation*}
\left( \sum_{\gamma }a_{\gamma }\gamma \right) +\left( \sum_{\gamma
}b_{\gamma }\gamma \right) =\sum_{\gamma }\left( a_{\gamma }+b_{\gamma
}\right) \gamma 
\end{equation*}%
and%
\begin{equation*}
\left( \sum_{\gamma }a_{\gamma }\gamma \right) \left( \sum_{\gamma
}b_{\gamma }\gamma \right) =\sum_{\gamma }\left( \sum_{\rho \rho ^{\prime
}=\gamma }a_{\rho }b_{\rho ^{\prime }}\right) \gamma 
\text{.}
\end{equation*}%
Consider now the Hilbert space $\ell ^{2}(\Gamma )$ of square-summable
complex-valued functions on $\Gamma $. Each $\gamma \in \Gamma $ defines a
unitary operator $\lambda _{\gamma }$ on $\ell ^{2}(\Gamma )$ by: 
\begin{equation*}
\lambda _{\gamma }(f)(x)=f(\gamma ^{-1}x)\text{.}
\end{equation*}%
The function $\gamma \rightarrow \lambda _{\gamma }$ extends by linearity to
an embedding of $\mathbb{C}\Gamma $ into the algebra $B\left( \ell
^{2}(\Gamma )\right) $ of bounded linear operators on $\ell ^{2}(\Gamma )$.

\begin{definition}
\label{Definition: group vN algebra}The weak closure of $\mathbb{C}\Gamma $
(identified with a subalgebra of $B\left( \ell ^{2}(\Gamma )\right) $) is a
von Neumann algebra denoted by $L\Gamma $ and called the \emph{group von
Neumann algebra}%
\index{group von Neumann algebra} of\textit{\ }$\Gamma $.
\end{definition}

The group von Neumann algebra\index{group von Neumann algebra} $L\Gamma $ is canonically endowed with the
trace $\tau $ obtained by extending by continuity the condition:%
\begin{equation*}
\tau \left( \sum_{\gamma }a_{\gamma }\gamma \right) =a_{1_{\Gamma }},
\end{equation*}%
for every element $\sum_{\gamma }a_{\gamma }\gamma $ of $\mathbb{C}\Gamma $.

\begin{exercise}
\label{exer:groupfactor} Assume that $\Gamma $ is an ICC%
\index{group!ICC} group, i.e.\ every nontrivial conjugacy class of $\Gamma $
is infinite. Show that $L\Gamma $ is a \textup{II}$_{1}$ factor.
\end{exercise}

\begin{exercise}
\label{ex:ICCgroups} Show that the free group $\mathbb{F}_n$ on $n\geq2$
generators and the group $S_{\infty }^{%
\text{fin}}$ of finitely supported permutations of a countable set are ICC.
\end{exercise}

It is a milestone result of Murray and von Neumann from \cite{Mu-vN4} (see
also Theorem 6.7.8 of \cite{Kadison-RingroseII}) that the group factors
associated with the free group $\mathbb{F}_{2}$ 
\index{group!free}and, respectively, the group $S_{\infty }^{%
\text{fin}}$ are nonisomorphic. It is currently a major open problem in the
theory of \textup{II}$_{1}$ factors to determine whether free groups over different
number of generators have isomorphic associated factors.

Connes' embedding conjecture%
\index{conjecture!Connes' embedding} asserts that any separable \textup{II}$_{1}$
factor can be approximated by finite dimensional factors, i.e.\ matrix
algebras. More precisely:

\begin{conjecture}[Connes, 1976]
\label{Conjecture: Connes}If $M$ is a separable \textup{II}$_{1}$ factor with trace $%
\tau _{M}$ then for every $\varepsilon >0$ and every finite subset $F$ of $M$
there is a function $\Phi $ from $M$ to a matrix algebra $M_{n}(\mathbb{C})$
that on $F$ approximately preserves the operations and the trace. This means that $\Phi (1)=1$ and for every $x,y\in F$:
\end{conjecture}

\begin{itemize}
\item $\left\Vert \Phi (x+y)-\left( \Phi (x)+\Phi (y)\right) \right\Vert
_{2}<\varepsilon $;

\item $\left\Vert \Phi (xy)-\Phi (x)\Phi (y)\right\Vert _{2}<\varepsilon $;

\item $\left\vert \tau _{M}(x)-\tau _{M_{n}(\mathbb{C})}\left( \Phi
(x)\right) \right\vert <\varepsilon $.
\end{itemize}

\section{Voiculescu's free entropy}

Connes' embedding\index{conjecture!Connes' embedding} conjecture is related to many other problems in pure and
applied mathematics and can be stated in several different equivalent ways.
The simplest and most practical way is probably through Voiculescu's free
entropy.

Consider a random variable which outcomes the set $\{1,\ldots ,n\}$ with
probabilities $p_{1},\ldots ,p_{n}$. Observe that its Shannon's entropy, $%
-\sum p_{i}\log (p_{i})$, can also be constructed through the following
procedure:

\begin{enumerate}
\item Call \emph{microstate}\index{microstate} any function $f$ from $\left\{ 1,2,\ldots ,N\right\} $ to $\left\{
1,2,\ldots ,n\right\}$. A microstate $f$ \emph{$\varepsilon $-approximates} the
discrete distribution $p_{1},\ldots p_{n}$ if for every $i\in \left\{
1,2,\ldots ,n\right\} $%
\begin{equation*}
\left\vert 
\frac{\left\vert f^{-1}(i)\right\vert }{N}-p_{i}\right\vert <\varepsilon 
\text{.}
\end{equation*}%
Denote the number of such microstates by $\Gamma (p_{1},\ldots
,p_{n},N,\varepsilon )$.

\item Take the limit of 
\begin{equation*}
N^{-1}\log |\Gamma (p_{1},\ldots ,p_{n},N,\varepsilon )|,
\end{equation*}%
as $N\rightarrow \infty $

\item Finally take the limit as $\varepsilon$ goes to zero.
\end{enumerate}

One can show that the result is just the opposite of the Shannon entropy,
that is, $\sum_{i\in n}p_{i}\log p_{i}$. Over the early 1990s, Voiculescu, in
part motivated by the isomorphism problem of whether the group von Neumann
algebras associated to different free groups are isomorphic or not, realized
that this construction could be imitated in the noncommutative world of II$%
_{1}$ factors.

\begin{enumerate}
\item Microstates are self-adjoint matrices, instead of functions between
finite sets. Formally, let $\varepsilon ,R>0$, $m,k\in \mathbb{N}$, and $%
X_{1},\ldots ,X_{n}$ be free random variables\footnote{Freeness is the analogue of independence in the noncommutative framework of $II_1$-factors. For a formal definition we refer the reader to \cite{Vo85}.} on a \textup{II}$_{1}$-factor $M$. We
denote $\Gamma _{R}(X_{1},\ldots ,X_{n};m,k,\varepsilon )$ the set of $%
(A_{1},\ldots ,A_{n})\in (M_{k}(\mathbb{C})_{sa})^{n}$ such that $\left\Vert
A_{m}\right\Vert \leq R$ and 
\begin{equation*}
|tr(A_{i_{1}}\cdots A_{i_{m}})-\tau (X_{i_{1}}\cdots X_{i_{m}})|<\varepsilon
\end{equation*}%
for every $1\leq m\leq p$ and $\left( i_{1},\ldots ,i_{m}\right) \in \left\{
1,2,\ldots ,n\right\} ^{m}$.

\item The discrete measure is replaced by the Lebesgue measure $\lambda $ on 
$(M_{k}(\mathbb{C})_{sa})^{n}$.

%Formally, let $k$ be a positive integer and $(M_{k}(\mathbb{C})_{sa})^{n}$
%be the set of $n$-tuples of self-adjoint $k\times k$ complex matrices. Let $%
%\lambda $ be the Lebesgue measure on $(M_{k}(\mathbb{C})_{sa})^{n}$
%corresponding to the Euclidean norm 
%\begin{equation*}
%||(A_{1},\ldots ,A_{n})||_{HS}^{2}=Tr(A_{1}^{2}+\cdots +A_{n}^{2})
%\end{equation*}%
%where $Tr$ is the non-normalized trace on $M_{k}(\mathbb{C})$.

\item The limits are replaced by suitable limsups, sups, and infs.
\end{enumerate}

Specifically, set 
\begin{equation*}
\chi _{R}(X_{1},\ldots ,X_{n};m,k,\varepsilon ):=\log \lambda (\Gamma
_{R}(X_{1},\ldots X_{n};m,k,\varepsilon )),
\end{equation*}%
\begin{equation*}
\chi _{R}(X_{1},\ldots X_{n};m,\varepsilon ):=\limsup_{k\rightarrow \infty
}(k^{-2}\chi _{R}(X_{1},\ldots X_{n};m,k,\varepsilon )+2^{-1}n\log (k)),
\end{equation*}%
\begin{equation*}
\chi _{R}(X_{1},\ldots X_{n}):=\inf \{\chi _{R}(X_{1},\ldots
X_{n};m,\varepsilon ):m\in \mathbb{N},\varepsilon >0\},
\end{equation*}%
Finally, define \textbf{entropy }of the variables $X_{1},\ldots ,X_{n}$ the
quantity 
\begin{equation*}
\chi (X_{1},\ldots X_{n}):=\sup \{\chi _{R}(X_{1},\ldots X_{n}):R>0\}\text{.}
\end{equation*}%
The factor $k^{-2}$ instead of $k^{-1}$ comes from the normalization. The
addend $2^{-1}n\log (k)$ is necessary, since otherwise $\chi
_{R}(X_{1},\ldots X_{n};m,\varepsilon )$ would always be equal to $-\infty $.

It is not clear why this construction should give entropy different from $%
-\infty $. Indeed Voiculescu himself proved that it is $-\infty $ when $%
X_{1},\ldots ,X_{n}$ are linearly dependent (see \cite[Proposition 3.6]{Vo2}%
). A necessary condition to have $\chi (X_{1},\ldots ,X_{n})>-\infty $ is
that $\Gamma _{R}(X_{1},\ldots ,X_{n},m,k,\varepsilon )$ is not empty for
some $k$, i.e. the finite subset $X=\{X_{1},\ldots ,X_{n}\}$ of $M_{sa}$ has
microstates. This requirement turns to be equivalent to the fact that $M$
satisfies Connes' embedding conjecture\index{conjecture!Connes' embedding}.

\begin{theorem*}
\label{microstates} Let $M$ be a II$_1$ factor. The following conditions are
equivalent

\begin{enumerate}
\item Every finite subsets $X\subseteq M_{sa}$ has microstates.

\item $M$ verifies Connes' embedding conjecture.
\end{enumerate}
\end{theorem*}

\section{History of Connes' embedding conjecture}

Connes' embedding conjecture\index{conjecture!Connes' embedding} had its origin in Connes' sentence:
\textquotedblleft \emph{We now construct an approximate imbedding of $N$ in $%
R$. Apparently such an imbedding ought to exist for all \textup{II}$_{1}$ factors
because it does for the regular representation of free groups. However the
construction below relies on condition 6}\textquotedblright\ (see \cite[page
105]{Co76}). This seemingly innocent observation received in the first
fifteen years after having been formulated relatively little attention. The
situation changed drastically in the early 1990s, when two fundamental
papers appeared on the scene: the aforementioned \cite{Vo2} where the Connes
embedding conjecture is used to define free entropy, and \cite{Ki} where
Eberhard Kirchberg obtained several unexpected reformulations of the Connes
embedding conjecture very far from its original statement, such as the fact
that the maximal and minimal tensor products of the group C*-algebra of the
free group on infinitely many generators coincide:%
\begin{equation}
C^{\ast }(F)\otimes _{\mathrm{min}}C^{\ast }(F)=C^{\ast }(F)\otimes _{%
\mathrm{max}}C^{\ast }(F)\text{\label{eq:kirchberg}}.
\end{equation}%
What is most striking in this equivalence is that Connes' embedding\index{conjecture!Connes' embedding}
conjecture concerns the class of all separable \textup{II}$_{1}$ factors, while (%
\eqref{eq:kirchberg}) is a statement about a single C*-algebra. Even more
surprising is the fact that the equivalence of these statements can be
proved a topological way as shown in \cite{Ha-Wi2}.

Following the aforementioned papers by Voiculescu and Kirchberg, a series of papers from
different authors proving various equivalent reformulations of Connes'
embedding conjecture appeared, such as \cite{Br},\cite{Oz}, \cite{Co-Dy}, \cite{Ha}, \cite{FKPT} contributing to arouse the interest around this conjecture. In
particular Florin R\u{a}dulescu showed in \cite{Ra2} that Connes' embedding
conjecture is equivalent to some noncommutative analogue of Hilbert's 17th
problem. This in turn inspired work of Klep and Schweighofer, who proved in 
\cite{Kl-Sc} a purely algebraic reformulation of Connes' embedding
conjecture (see also the more recent work \cite{Ju-Po}). In \cite{Ca-Ra} it
is observed that Connes' embedding conjecture could theoretically be
checked by an algorithm if a certain problem of embedding Hilbert spaces
with some additional structure into the Hilbert space $L^{2}(M,\tau )$
associated to a \textup{II}$_{1}$ factor has a positive solution. Another
computability-theoretical reformulation of the Connes' embedding conjecture
has more recently been proved in \cite{Goldbring-Hart}. Other very recent
discoveries include the fact that Connes' embedding conjecture is related to
Tsirelson's problem, a major open problem in Quantum Information Theory (see 
\cite{Fr}, \cite{Ju-Na-Pa-PeGa-Sc-We}, and \cite{Oz3}) and that it is connected to the ``minimal'' and ``commuting'' tensor products \cite{KPTT} of some group operator systems \cite{Ka11}, \cite{Fa-Pa}, \cite{FKP}.

\section{Hyperlinear and sofic groups\label{Subsection: hyperlinear and
sofic groups}}

In \cite{Radulescu} R\u{a}dulescu considered the particular case of Connes'
embedding conjecture for \textup{II}$_{1}$ factors arising as group factors $%
L\Gamma $ of countable discrete ICC groups. A countable discrete ICC group $%
\Gamma $ is called \textit{hyperlinear}%
\index{group!hyperlinear} if $L\Gamma $ verifies Connes' embedding%
\index{conjecture!Connes' embedding} conjecture. As we will see in Section %
\ref{Section: definition hyperlinear groups} and Section \ref{Section: logic
invariant length groups} the notion of hyperlinear group admits several
equivalent characterizations, as well as a natural generalization to the
class of all countable groups. The notion of hyperlinear group turned out to
be tightly connected with the notion of \emph{sofic groups}%
\index{group!sofic}. Sofic groups are a class of countable discrete groups
introduced by Gromov in \cite{Gr} (even though the name \textquotedblleft
sofic\textquotedblright\ was coined by Weiss in \cite{Weiss-sofic}). A group
is sofic if, loosely speaking, it can be locally approximated by finite
permutation groups $S_{n}$ up to an error measured in terms of the Hamming
metric $S_{n}$ (see Section \ref{Section: definition sofic groups}). Elek
and Szab\'{o} in \cite{Elek-Szabo-hyperlinearity} showed that every sofic%
\index{group!sofic} group is hyperlinear%
\index{group!hyperlinear}. Gromov's motivation to introduce the notion of
sofic groups came from an open problem in symbolic dynamics known as
Gottschalk's surjunctivity conjecture: Suppose that $\Gamma $ is a countable
group and $A$ is a finite set. Denote by $A^{\Gamma }$ the set of $\Gamma $%
-sequences of elements of $A$. The product topology on $A^{\Gamma }$ with
respect to the discrete topology on $A$ is compact and metrizable. The 
\textit{Bernoulli shift}%
\index{Bernoulli shift} of $\Gamma $ with alphabet $A$ is the left action of 
$\Gamma $ on $A^{\Gamma }$ defined by%
\begin{equation*}
\rho \cdot \left( a_{\gamma }\right) _{\gamma \in \Gamma }=\left( a_{\rho
^{-1}\gamma }\right) _{\gamma \in \Gamma }.
\end{equation*}%
A continuous function $f:A^{\Gamma }\rightarrow A^{\Gamma }$ is \textit{%
equivariant}%
\index{equivariant function} if it preserves the Bernoulli action, i.e.\ $%
f\left( \rho \cdot x\right) =\rho \cdot f(x)$, for every $x\in A^{\Gamma }$
and $\rho \in \Gamma $. Conjecture \ref{Conjecture: Gottschalk} was proposed
by Gottschalk in \cite{Gott-notions} and it is usually referred to as
Gottschalk's surjunctivity conjecture%
\index{conjecture!Gottschalk's surjunctivity}.

\begin{conjecture}
\label{Conjecture: Gottschalk}Suppose that $\Gamma $ is a discrete group and 
$A$ is a finite set. If $f:A^{\Gamma }\rightarrow A^{\Gamma }$ is a
continuous injective equivariant function, then $f$ is surjective.
\end{conjecture}

It is currently an open problem to determine whether Gottschalk's
surjunctivity conjecture holds for all countable discrete groups. Gromov
proved in \cite{Gr} using graph-theoretical methods that sofic%
\index{group!sofic} groups satisfy Gottschalk's surjunctivity conjecture.
Another proof was obtained Kerr and Li in \cite{Kerr-Li-Variational} as an
application of the theory of entropy for actions of sofic groups developed
by Bowen, Kerr, and Li (see \cite{Bowen-measure-entropy}, \cite%
{Kerr-Li-Variational}, and \cite{KerrLiDynamical}). The countable discrete
groups satisfying Gottschalk's surjunctivity conjecture are sometimes called 
\textit{surjunctive}%
\index{group!surjunctive}. An example of a \emph{monoid }that does not
satisfy the natural generalization of surjunctivity for monoids has been
provided in \cite{ceccherini-silberstein_surjunctive_2014}. More information
about Gottschalk's surjunctivity conjecture and surjunctive groups can be
found in the monograph \cite{Ceccherini-Coornaert}.

Since Gromov's proof of Gottschalk's conjecture for sofic%
\index{group!sofic} groups, many other open problems have been settled for
sofic or hyperlinear groups such as the Kervaire-Laudenbach conjecture (see
Section \ref{Section: Kervaire-Laudenbach for hyperlinear}), Kaplansky's
direct finiteness conjecture (see Section \ref{Section: direct finiteness}
and Section \ref{Section: rank rings and finiteness conjecture}), and the
algebraic eigenvalues conjecture (see Section \ref{Section: algebraic
eigenvalues}). This showed how deep and fruitful the notion of hyperlinear%
\index{group!hyperlinear} and sofic%
\index{group!sofic} groups are, and contributed to bring considerable
attention to the following question which is strikingly still open:

\begin{openbroblem}
Is there any countable discrete group $\Gamma $ that is not sofic or
hyperlinear?
\end{openbroblem}

\section{Kaplansky's direct finiteness conjecture\label{Section: direct
finiteness}}

Suppose that $\Gamma $ is a countable discrete group, and consider the
complex group algebra $\mathbb{C}\Gamma $ defined in Section \ref{Section:
vN algebras and II1 factors}.\index{group algebra!complex} Kaplanksi showed in \cite%
{Kaplansky-fields-rings} that $\mathbb{C}\Gamma $ is a directly finite ring.
This means that if $a,b\in \mathbb{C}\Gamma $ are such that $ab=1$ then $%
ba=1 $. In \cite{Burger-Valette} Burger and Valette gave a short proof of
Kaplansky's result by means of the group von Neumann algebra\index{group von Neumann algebra} construction
introduced in Section \ref{Section: vN algebras and II1 factors}. Indeed,
one can regard $\mathbb{C}\Gamma $ as a subalgebra of the group von Neumann
algebra $L\Gamma $ and prove that $L\Gamma $ is directly finite using
analytic methods. This is the content of Theorem \ref{Theorem: direct finite
vN algebra}.%

\begin{theorem}
\label{Theorem: direct finite vN algebra}If $M$ is a von Neumann algebra
endowed with a faithful finite trace $\tau $, then $M$ is a directly finite
algebra.
\end{theorem}

We now prove this theorem. Let $M$ be a von Neumann algebra and $\tau $ be a faithful
normalized trace on $M$. If $x,y\in M$ are such that $xy=1$ then $yx\in M$
is an idempotent element such that%
\begin{equation*}
\tau (yx)=\tau (xy)=\tau (1)=1.
\end{equation*}%
It is thus enough to prove that if $e\in M$ is an idempotent element such
that $\tau (e)=1$ then $e=1$. This is equivalent to the assertion that if $%
e\in M$ is an idempotent element such that $\tau (e)=0$ then $e=0$. The
latter statement is proved in Lemma \ref{Lemma: trace idempotent vN} (see 
\cite[Lemma 2.1]{Burger-Valette}).

\begin{lemma}
\label{Lemma: trace idempotent vN}If $M$ is a von Neumann algebra endowed
with a faithful finite trace $\tau $ and $e\in M$ is an idempotent such that 
$\tau (e)=0$, then $e=0$.
\end{lemma}

\begin{proof}
The conclusion is obvious if $e$ is a self-adjoint idempotent element (i.e.
a projection). In fact in this case%
\begin{equation*}
\tau (e)=\tau \left( e^{\ast }e\right) =0
\end{equation*}%
implies $e=0$ by faithfulness of $\tau $. In order to establish the general
case it is enough to show that if $e\in M$ is idempotent, then there is a
self-adjoint invertible element $z$ of $M$ such that $f=ee^{\ast }z^{-1}$ is
a projection and $\tau (e)=\tau (f)$. Define%
\begin{equation*}
z=1+\left( e^{\ast }-e\right) ^{\ast }\left( e^{\ast }-e\right) .
\end{equation*}%
Observe that $z$ is an invertible element (see \cite[II.3.1.4]{Blackadar})
commuting with $e$. It is not difficult to check that $f=ee^{\ast }z^{-1}$
has the required properties.
\end{proof}

The construction of the complex group algebra of $\Gamma $ introduced in
Section \ref{Section: vN algebras and II1 factors} can be generalized
replacing $\mathbb{C}$ with an arbitrary field $K$. One thus obtains the
group $K$-algebra \index{group algebra} $K\Gamma $ of $\Gamma $ consisting of formal finite linear
combinations%
\begin{equation*}
k_{1}\gamma _{1}+\ldots +k_{n}\gamma _{n}
\end{equation*}%
of elements of $\Gamma $ with coefficients in $K$. As before a typical
element $a$ of $K\Gamma $ can be denoted by%
\begin{equation*}
\sum_{\gamma }a_{\gamma }\gamma
\end{equation*}%
where the coefficients $a_{\gamma }\in K$ are zero for all but finitely many 
$\gamma \in \Gamma $. The operations on $K\Gamma $ are defined by%
\begin{equation*}
\left( \sum_{\gamma }a_{\gamma }\gamma \right) +\left( \sum_{\gamma
}b_{\gamma }\gamma \right) =\sum_{\gamma }\left( a_{\gamma }+b_{\gamma
}\right) \gamma
\end{equation*}%
and%
\begin{equation*}
\left( \sum_{\gamma }a_{\gamma }\gamma \right) \left( \sum_{\gamma
}b_{\gamma }\gamma \right) =\sum_{\gamma }\left( \sum_{\rho \rho ^{\prime
}=\gamma }a_{\rho }b_{\rho ^{\prime }}\right) \gamma 
\text{.}
\end{equation*}%
Conjecture \ref{Conjecture: direct finiteness} is due to Kaplansky and
usually referred to as direct finiteness conjecture%
\index{conjecture!direct finiteness}.

\begin{conjecture}
\label{Conjecture: direct finiteness}The algebra $K\Gamma $ is directly
finite, i.e.\ any one-side invertible element of $K\Gamma $ is invertible.
\end{conjecture}

Kaplansky's direct finiteness conjecture was established for residually amenable groups
in \cite{AOMP} and then for sofic\index{group!sofic} groups in \cite{Elek-Szabo-finiteness}.
Section \ref{Section: rank rings and finiteness conjecture} contains a proof
of Conjecture \ref{Conjecture: direct finiteness} for sofic groups involving
the notion of rank ring and rank ultraproduct of rank rings. It is a well
know fact (see Observation 2.1 in \cite{Kwiatkowska-Pestov}) that any field
embeds as a subfield of an ultraproduct of finite fields (see Section 2 of
for the definition of ultraproduct of fields). It is not difficult to deduce
from this observation that a group $\Gamma $ satisfies Kaplansky's direct
finiteness conjecture as soon as $K\Gamma $ is directly finite for every
finite field $K$ (this fact was pointed out to us by Vladimir Pestov). This
allows one to deduce that Gottschalk's conjecture\index{conjecture!Gottschalk's surjunctivity} is stronger than Kaplansky's
direct finiteness conjecture\index{conjecture!direct finiteness}. In fact suppose that $\Gamma $ is a group
satisfying Gottschalk's conjecture and $K$ is a field that, without loss of
generality, we can assume to be finite. Consider the Bernoulli action of $%
\Gamma $ with alphabet $K$. We will denote an element $(a)_{\gamma \in
\Gamma }$ of $K^{\Gamma }$ by $\sum_{\gamma }a_{\gamma }\gamma $, and regard
the group algebra $K\Gamma $ as a subset of $K^{\Gamma }$. Defining%
\begin{equation*}
\left( \sum_{\gamma }a_{\gamma }\gamma \right) \cdot \left( \sum_{\gamma
}b_{\gamma }\gamma \right) =\sum_{\gamma }\left( \sum_{\rho \rho ^{\prime
}=\gamma }a_{\rho }b_{\rho ^{\prime }}\right) \gamma 
\end{equation*}%
for $\sum_{\gamma }a_{\gamma }\gamma \in K^{\Gamma }$ and $\sum_{\gamma
}b_{\gamma }\gamma \in K\Gamma $ one obtains a right action of $K\Gamma $ on 
$K^{\Gamma }$ that extends the multiplication operation in $K\Gamma $ and
commutes with the left action of $\Gamma $ on $K\Gamma $. Suppose that $a$, $%
b\in K\Gamma $ are such that $ab=1_{\Gamma }$. Define the continuous
equivariant map $f:K^{\Gamma }\rightarrow K^{\Gamma }$ by $f(x)=x\cdot a$.
It follows from the fact that $b$ is a right inverse of $a$ that $f$ is
injective. Since $\Gamma $ is assumed to satisfy Gottschalk's conjecture, $f$
must be surjective. In particular there is $x_{0}\in K^{\Gamma }$ such that 
\begin{equation*}
x_{0}\cdot a=f(x_{0})=1_{\Gamma }%
\text{.}
\end{equation*}%
It follows that $a$ has a left inverse. A standard calculation allows one to
conclude that $a$ is invertible with inverse $b$.

\section{Kaplansky's group ring conjectures}

Conjecture \ref{Conjecture: direct finiteness} is only one of several
conjectures concerning the group algebra\index{group algebra} $K\Gamma $ for a countable discrete 
\emph{torsion-free} group $\Gamma $ attributed to Kaplansky. Here is the complete list:

\begin{itemize}
\item Zero divisors conjecture%
\index{conjecture!zero divisors}: $K\Gamma $ has no zero divisors;

\item Nilpotent elements conjecture%
\index{conjecture!nilpotent elements}: $K\Gamma 
$ has no nilpotent elements;

\item Idempotent elements conjecture%
\index{conjecture!idempotent elements}: the only idempotent elements of $%
K\Gamma $ are $0$ and $1$;

\item Units conjecture%
\index{conjecture!units}: the only units of $K\Gamma $ are $k\gamma $ for $%
k\in K\left\backslash \left\{ 0\right\} \right. $ and $\gamma \in \Gamma $;

\item Trace of idempotents conjecture%
\index{conjecture!trace of idempotents}: if $b$ is an idempotent element of $%
K\Gamma $ then the coefficient $b_{1_{\Gamma }}$ corresponding to the
identity $1_{\Gamma }$ of $\Gamma $ belongs to the prime field $K_{0}$ of $K$
(which is the minimum subfield of $K$).
\end{itemize}

All these conjectures are to this day still open, apart from the last one
which has been established by Zalesskii in \cite{Zalesskii}. We will present now a proof of the particular case
of Zalesskii's result when $K$ is a finite field of characteristic $p$. Recall that a \textit{trace }on $K\Gamma 
$ is a $K$-linear map $\tau :K\Gamma \rightarrow K$ such that $\tau
(ab)=\tau (ba)$ for every $a,b\in K\Gamma $. If $n$ is a nonnegative integer
and $a\in K\Gamma $, define $\tau _{n}(a)$ to be the sum of the coefficients
of $a$ corresponding to elements of $\Gamma $ of order $p^{n}$. In
particular $\tau _{0}(a)$ is the coefficient of $a$ corresponding t the
identity element $1_{\Gamma }$ of $\Gamma $.

\begin{exercise}
\label{Exercise: trace}Verify that $\tau _{n}$ is a trace on $K\Gamma $.
\end{exercise}

\begin{exercise}
\label{Exercise: Frobenius}Show that if $K$ is a finite field of
characteristic $p$ and $\tau $ is any trace on $K\Gamma $, then%
\begin{equation}
\tau \left( (a+b)^{p}\right) =\tau (a^{p})+\tau (b^{p})\label{Equation: Frobenius},
\end{equation}%
for every $a,b\in K\Gamma $. Thus by induction%
\begin{equation*}
\tau (a^{p})=\sum_{\gamma }a_{\gamma }^{p}\tau (\gamma ^{p})%
\text{.}
\end{equation*}
\end{exercise}

The identity \eqref{Equation: Frobenius} can be referred to as
\textquotedblleft Frobenius under trace"%
\index{Frobenius under trace} in analogy with the corresponding identity for
elements of a field of characteristic $p$.\textit{\ }Suppose now that $e\in
K\Gamma $ is an idempotent element. We want to show that $\tau (e)$ belongs
to the prime field $K_{0}$ of $K$. To this purpose it is enough to show that 
$\tau (e)^{p}=\tau (e)$. For $n\geq 1$ we have, by Exercise \ref{Exercise:
trace} and Exercise \ref{Exercise: Frobenius}, denoting by $\left\vert
\gamma \right\vert $ the order of an element $\gamma $ of $\Gamma $:%
\begin{eqnarray*}
\tau _{n}(e) &=&\tau _{n}(e^{p}) \\
&=&\sum_{\gamma }e_{\gamma }^{p}\tau _{n}(\gamma ^{p}) \\
&=&\sum_{\left\vert \gamma \right\vert =p^{n+1}}e_{\gamma }^{p} \\
&=&\left( \sum_{\left\vert \gamma \right\vert =p^{n+1}}e_{\gamma }\right)
^{p} \\
&=&\tau _{n+1}(e)^{p}.
\end{eqnarray*}%
On the other hand:%
\begin{eqnarray*}
\tau _{0}(e) &=&\tau _{0}(e^{p}) \\
&=&\sum_{\gamma }e_{\gamma }^{p}\tau _{0}(\gamma ^{p}) \\
&=&\sum_{\left\vert \gamma \right\vert =1}e_{\gamma }^{p}+\sum_{\left\vert
\gamma \right\vert =p}e_{\gamma }^{p} \\
&=&\tau _{0}(e)^{p}+\tau _{1}(e)^{p}.
\end{eqnarray*}%
From these identities it is easy to prove by induction that%
\begin{equation*}
\tau _{0}(e)=\tau _{0}(e)^{p}+\tau _{n}(e)^{p^{n}},
\end{equation*}%
for every $n\in \mathbb{N}$. Since $e$ has finite support, there is $n\in 
\mathbb{N}$ such that $\tau _{n}(e)=0$. This implies that $\tau _{0}(e)=\tau
_{0}(e)^{p}$ and hence $\tau _{0}(e)\in K_{0}$.

The proof of the general case of Zalesskii's theorem can be inferred from
the particular case presented here. The details can be found in \cite%
{Burger-Valette}.

\chapter{Sofic and hyperlinear groups}

\chapterauthor{Martino Lupini}

\section{Definition of sofic groups\label{Section: definition sofic groups}}

A\textit{\ length function%
\index{length function}} $\ell $ on a group $G$ is a function $\ell
:G\rightarrow \left[ 0,1\right] $ such that for every $x,y\in G$:

\begin{itemize}
\item $\ell (xy)\leq \ell (x)+\ell (y)$;

\item $\ell (x^{-1})=\ell (x)$;

\item $\ell (x)=0$ if and only if $x$ is the identity $1_{G}$ of $G$.
\end{itemize}

A length function is called \textit{invariant%
\index{length function!invariant} }if it is moreover invariant by
conjugation. This means that for every $x,y\in G$%
\begin{equation*}
\ell (xyx^{-1})=\ell (x)
\end{equation*}%
or equivalently%
\begin{equation*}
\ell (xy)=\ell (yx)%
\text{.}
\end{equation*}%
A group endowed with an invariant length function is called an\textit{\
invariant length group%
\index{group!invariant length}}. If $G$ is an invariant length group with
invariant length function $\ell $, then the function 
\begin{equation*}
d:G\times G\rightarrow \left[ 0,1\right]
\end{equation*}%
defined by $d(x,y)=\ell (xy^{-1})$ is a bi-invariant metric on $G$. This
means that $d$ is a metric on $G$ such that left and right translations in $%
G $ are isometries with respect to $d$. Conversely any bi-invariant metric $%
d $ on $G$ gives rise to an invariant length function $\ell $ on $G$ by%
\begin{equation*}
\ell (x)=d(x,1_{G})%
\text{.}
\end{equation*}%
This shows that there is a bijective correspondence between invariant length
functions and bi-invariant metrics on a group $G$. If $\Gamma $ is any
group, then the function

\begin{equation*}
\ell _{0}(x)=%
\begin{cases}
0 & \text{if }x=1_{G}\text{,} \\ 
1 & \text{otherwise;}%
\end{cases}%
\end{equation*}%
is an invariant length function on $\Gamma $, called the trivial invariant
length function. In the following any discrete group will be regarded as an
invariant length group endowed with the trivial invariant length function.

Consider for $n\in 
%TCIMACRO{\U{2115} }%
%BeginExpansion
\mathbb{N}
%EndExpansion
$ the group $S_{n}$ of permutations over the set $n=\left\{ 0,1,\ldots
,n-1\right\} $. The \textit{Hamming invariant length function}%
\index{length function!Hamming} $\ell $ on $S_{n}$ is defined by%
\begin{equation*}
\ell _{S_{n}}(\sigma )=%
\frac{1}{n}\left\vert \left\{ i\in n:\sigma (i)\neq i\right. \right\} \text{.%
}
\end{equation*}

\begin{exercise}
Verify that $\ell _{S_{n}}$ is in fact an invariant length function on $%
S_{n} $.
\end{exercise}

The bi-invariant metric on $S_{n}$ associated with the invariant length
function $\ell _{S_{n}}$ will be denoted by $d_{S_{n}}$.

\begin{definition}
\label{Definition: sofic} A countable discrete group $\Gamma $ is \emph{sofic%
}%
\index{group!sofic} if for every $\varepsilon >0$ and every finite subset $F$
of $\Gamma \left\backslash \left\{ 1_{\Gamma }\right\} \right. $ there is a
natural number $n$ and a function $\Phi :\Gamma \rightarrow S_{n}$ such that 
$\Phi \left( 1_{\Gamma }\right) =1_{S_{n}}$ and for every $g,h\in
F\left\backslash \left\{ 1_{\Gamma }\right\} \right. $:

\begin{itemize}
\item $d_{S_{n}}\left( \Phi (gh),\Phi (g)\Phi (h)\right) <\varepsilon $;

\item $\ell _{S_{n}}\left( \Phi (g)\right) >r(g)$ where $r(g)$ is a positive
constant depending only on $g$.
\end{itemize}
\end{definition}

This \textit{local approximation property }can be reformulated in terms of
embedding into a (length) ultraproduct of the permutation groups%
\index{ultraproduct!of permutations groups}. The product%
\begin{equation*}
\prod_{n\in \mathbb{N}}S_{n}
\end{equation*}%
is a group with respect to the coordinatewise multiplication. Fix a free
ultrafilter $\mathcal{U}$ over $\mathbb{N}$ (see \ref{app:ultrafilters} for
an introduction to ultrafilters). Define the function 
\begin{equation*}
\ell _{\mathcal{U}}:\prod_{n\in 
%TCIMACRO{\U{2115} }%
%BeginExpansion
\mathbb{N}
%EndExpansion
}S_{n}\rightarrow \left[ 0,1\right]
\end{equation*}%
by%
\begin{equation*}
\ell _{\mathcal{U}}\left( (\sigma _{n})_{n\in \mathbb{N}}\right)
=\lim_{n\rightarrow \mathcal{U}}\ell _{S_{n}}(\sigma _{n})%
\text{.}
\end{equation*}%
It is not hard to check%
\begin{equation*}
N_{\mathcal{U}}=\left\{ x\in \prod_{n\in 
%TCIMACRO{\U{2115} }%
%BeginExpansion
\mathbb{N}
%EndExpansion
}S_{n}:\ell _{\mathcal{U}}(x)=0\right\}
\end{equation*}%
is a normal subgroup of $\prod_{n}S_{n}$. The quotient of $\prod_{n}S_{n}$
by $N_{\mathcal{U}}$ is denoted by 
\begin{equation*}
\prod\nolimits_{\mathcal{U}}S_{n}
\end{equation*}%
and called the \emph{ultraproduct} relative to the free ultrafilter $%
\mathcal{U}$ of the sequence $\left( S_{n}\right) _{n\in 
%TCIMACRO{\U{2115} }%
%BeginExpansion
\mathbb{N}
%EndExpansion
}$ of permutation groups regarded as invariant length groups.

\begin{exercise}
\label{Exercise: ultralimit length}Show that if $(\sigma _{n})_{n\in 
%TCIMACRO{\U{2115} }%
%BeginExpansion
\mathbb{N}
%EndExpansion
},\left( \tau _{n}\right) _{n\in 
%TCIMACRO{\U{2115} }%
%BeginExpansion
\mathbb{N}
%EndExpansion
}\in \prod_{n}S_{n}$ belong to the same coset of $N_{\mathcal{U}}$ then $%
\ell _{\mathcal{U}}\left( (\sigma _{n})_{n\in 
%TCIMACRO{\U{2115} }%
%BeginExpansion
\mathbb{N}
%EndExpansion
}\right) =\ell _{\mathcal{U}}\left( \left( \tau _{n}\right) _{n\in 
%TCIMACRO{\U{2115} }%
%BeginExpansion
\mathbb{N}
%EndExpansion
}\right) $.
\end{exercise}

Exercise \ref{Exercise: ultralimit length} shows that the function $\ell _{%
\mathcal{U}}$ passes to the quotient inducing a canonical invariant length
function on $\prod\nolimits_{\mathcal{U}}S_{n}$ still denoted by $\ell _{%
\mathcal{U}}$. If $x\in \prod\nolimits_{\mathcal{U}}S_{n}$ belongs to the
coset of $N_{\mathcal{U}}$ associated with $(\sigma _{n})_{n\in 
%TCIMACRO{\U{2115} }%
%BeginExpansion
\mathbb{N}
%EndExpansion
}\in \prod_{n}S_{n}$, then $(\sigma _{n})_{n\in 
%TCIMACRO{\U{2115} }%
%BeginExpansion
\mathbb{N}
%EndExpansion
}$ is called a \textit{representative sequence%
\index{sequence!representative}} for the element $x$. It is not difficult to
reformulate the notion of sofic%
\index{group!sofic} group in term of existence of an embedding into $%
\prod\nolimits_{\mathcal{U}}S_{n}$.

\begin{exercise}
\label{Exercise: sofic equivalence 1}Suppose that $\Gamma $ is a countable
discrete group. Show that the following statements are equivalent:

\begin{enumerate}
\item $\Gamma $ is sofic%
\index{group!sofic};

\item there is an injective group *-homomorphism $\Phi :\Gamma \rightarrow
\prod\nolimits_{\mathcal{U}}S_{n}$ for every free ultrafilter $\mathcal{U}$
over $N$;

\item there is an injective group *-homomorphism $\Phi :\Gamma \rightarrow
\prod\nolimits_{\mathcal{U}}S_{n}$ for some free ultrafilter $\mathcal{U}$
over $N$.
\end{enumerate}
\end{exercise}

\begin{hint}
For $1.\Rightarrow 2.$ observe that the hypothesis implies that there is a
sequence $\left( \Phi _{n}\right) _{n\in 
%TCIMACRO{\U{2115} }%
%BeginExpansion
\mathbb{N}
%EndExpansion
}$ of maps from $\Gamma $ to $S_{n}$ such that $\Phi _{n}\left( 1_{\Gamma
}\right) =1_{S_{n}}$ and for every $g,h\in \Gamma \left\backslash \left\{
1_{\Gamma }\right\} \right. $%
\begin{equation*}
\lim_{n\rightarrow +\infty }d_{S_{n}}\left( \Phi _{n}\left( gh\right) ,\Phi
_{n}(h)\Phi _{n}\left( g\right) \right) =0
\end{equation*}%
and%
\begin{equation*}
\liminf_{n\rightarrow +\infty }\ell _{S_{n}}\left( \Phi _{n}\left( g\right)
\right) \geq r\left( g\right) >0%
\text{.}
\end{equation*}%
Define $\Phi :\Gamma \rightarrow \prod\nolimits_{\mathcal{U}}S_{n}$ sending $%
g$ to the element of $\prod\nolimits_{\mathcal{U}}S_{n}$ having $\left( \Phi
_{n}\left( g\right) \right) _{n\in 
%TCIMACRO{\U{2115} }%
%BeginExpansion
\mathbb{N}
%EndExpansion
}$ as representative sequence. For $3.\Rightarrow 1.$ observe that if $\Phi
:\Gamma \rightarrow \prod\nolimits_{\mathcal{U}}S_{n}$ is an injective
*-homomorphism and for every $g\in G$%
\begin{equation*}
\left( \Phi _{n}\left( g\right) \right) _{n\in 
%TCIMACRO{\U{2115} }%
%BeginExpansion
\mathbb{N}
%EndExpansion
}
\end{equation*}%
is a representative sequence of $\Phi \left( g\right) $ then the maps $\Psi
_{n}=\Phi _{n}\left( 1_{\Gamma }\right) ^{-1}\Phi _{n}\left( g\right) $
satisfy the following properties: $\Psi _{n}\left( 1_{\Gamma }\right)
=1_{S_{n}}$ and for every $g,h\in \Gamma \left\backslash \left\{ 1_{\Gamma
}\right\} \right. $%
\begin{equation*}
\lim_{n\rightarrow \mathcal{U}}d_{S_{n}}\left( \Psi _{n}\left( gh\right)
,\Psi _{n}\left( g\right) \Psi _{n}(h)\right) =0
\end{equation*}%
and 
\begin{equation*}
\lim_{n\rightarrow \mathcal{U}}\ell _{S_{n}}\left( \Psi _{n}\left( g\right)
\right) \geq \ell _{\mathcal{U}}\left( \Phi \left( g\right) \right) \text{.}
\end{equation*}%
If $F\subset \Gamma $ is finite and $\varepsilon >0$ then the maps $\Psi
_{n} $ for $n$ large enough witness the condition of soficity%
\index{group!sofic} of $\Gamma $ relative to $F$ and $\varepsilon >0$.
\end{hint}

In view of the fact that a (countable, discrete) group is sofic%
\index{group!sofic} if and only embeds in $\prod\nolimits_{\mathcal{U}}S_{n}$
for some free ultrafilter $\mathcal{U}$, the (length) ultraproducts of the
sequence of permutations are called \emph{universal sofic groups}%
\index{group!universal sofic} (see \cite{Elek-Szabo-hyperlinearity}, \cite%
{Kwiatkowska-Pestov}, \cite{Thomas}, \cite{Lupini}, \cite{Paunescu}).

Universal sofic%
\index{group!universal sofic} groups are not separable. This fact follows
from a well known argument. Two functions $f,g:%
%TCIMACRO{\U{2115} }%
%BeginExpansion
\mathbb{N}
%EndExpansion
\rightarrow 
%TCIMACRO{\U{2115} }%
%BeginExpansion
\mathbb{N}
%EndExpansion
$ are \textit{eventually distinct} if they coincide only on finitely many $n$%
's.

\begin{exercise}
\label{Exercise: almost disjoint family}Suppose that $h:%
%TCIMACRO{\U{2115} }%
%BeginExpansion
\mathbb{N}
%EndExpansion
\rightarrow 
%TCIMACRO{\U{2115} }%
%BeginExpansion
\mathbb{N}
%EndExpansion
$ is an unbounded function. Show that there is a family $\mathcal{F}$ of
size continuum of pairwise eventually distinct functions from $%
%TCIMACRO{\U{2115} }%
%BeginExpansion
\mathbb{N}
%EndExpansion
$ to $%
%TCIMACRO{\U{2115} }%
%BeginExpansion
\mathbb{N}
%EndExpansion
$ such that $f\left( n\right) \leq h\left( n\right) $ for every $f\in 
\mathcal{F}$.
\end{exercise}

\begin{hint}
For every subset $A$ of $%
%TCIMACRO{\U{2115} }%
%BeginExpansion
\mathbb{N}
%EndExpansion
$ define $f_{A}:%
%TCIMACRO{\U{2115} }%
%BeginExpansion
\mathbb{N}
%EndExpansion
\rightarrow 
%TCIMACRO{\U{2115} }%
%BeginExpansion
\mathbb{N}
%EndExpansion
$ by%
\begin{equation*}
f_{A}\left( n\right) =\sum_{k<n}\chi _{A}\left( n\right) 2^{k}
\end{equation*}%
where $\chi _{A}$ denotes the characteristic function of $A$. Observe that%
\begin{equation*}
\mathcal{F}_{0}=\left\{ f_{A}:A\subset 
%TCIMACRO{\U{2115} }%
%BeginExpansion
\mathbb{N}
%EndExpansion
\right\}
\end{equation*}%
is a family of size continuum of pairwise eventually distinct functions such
that $f\left( n\right) \leq 2^{n}$ for every $n\in 
%TCIMACRO{\U{2115} }%
%BeginExpansion
\mathbb{N}
%EndExpansion
$.
\end{hint}

\begin{exercise}
\label{Exercise: nonseparable universal sofic}Show that for any free
ultrafilter $\mathcal{U}$ there is a subset $X$ of $\prod\nolimits_{\mathcal{%
U}}S_{n}$ of size continuum such that every two distinct elements of $X$
have distance one. Conclude that any dense subset of $\prod\nolimits_{%
\mathcal{U}}S_{n}$ has the cardinality of the continuum.
\end{exercise}

\begin{hint}
By Exercise \ref{Exercise: almost disjoint family} there is a family $%
\mathcal{F}$ of size continuum of pairwise eventually distinct functions
such that $f\left( n\right) \leq 
\sqrt{n}$ for every $n\in 
%TCIMACRO{\U{2115} }%
%BeginExpansion
\mathbb{N}
%EndExpansion
$ and $f\in \mathcal{F}$. For every $f\in \mathcal{F}$ consider the element $%
x_{f}$ of $\prod\nolimits_{\mathcal{U}}S_{n}$ having 
\begin{equation*}
\left( \sigma _{n}^{f\left( n\right) }\right) _{n\in 
%TCIMACRO{\U{2115} }%
%BeginExpansion
\mathbb{N}
%EndExpansion
}
\end{equation*}%
as representative sequence, where $\sigma _{n}$ is any cyclic permutation of 
$n$. Show that 
\begin{equation*}
X:=\left\{ x_{f}:f\in \mathcal{F}\right\}
\end{equation*}
has the required property.
\end{hint}

An amplification argument of Elek and Szab\'{o} (see \cite%
{Elek-Szabo-hyperlinearity}) shows that the condition of soficity%
\index{group!sofic} is equivalent to the an apparently stronger
approximation property, which is discussed in the following Exercise \ref%
{Exercise: amplification sofic}.

\begin{definition}
\label{Definition: approximate morphism}Suppose that $G,H$ are invariant
length groups, $F$ is a subset of $G$, and $\varepsilon $ is a positive real
number. A function $\Phi :G\rightarrow H$ is called an $\left( F,\varepsilon
\right) $-\emph{approximate morphism}%
\index{approximate morphism!of invariant length groups} if $\Phi \left(
1_{G}\right) =1_{H}$ and for every $g,h\in F$:

\begin{itemize}
\item $\left\vert \ell _{H}\left( \Phi (gh)\Phi (h)^{-1}\Phi (g)\right)
\right\vert <\varepsilon $;

\item $\left\vert \ell _{H}\left( \Phi (g)\right) -\ell _{G}(g)\right\vert
<\varepsilon $.
\end{itemize}
\end{definition}

\begin{exercise}
\label{Exercise: amplification sofic}Prove that a countable discrete group $%
\Gamma $ is sofic%
\index{group!sofic} if and only if for every positive real number $%
\varepsilon $ and every finite subset $F$ of $\Gamma $ there is a natural
number $n$ and an $\left( F,\varepsilon \right) $-approximate morphism%
\index{approximate morphism!of invariant length groups} from $\Gamma $ to $%
S_{n}$, where $\Gamma $ is regarded as an invariant length group with
respect to the trivial invariant length function.
\end{exercise}

\begin{hint}
If $n,k\in \mathbb{N}$ and $\sigma \in S_{n}$ consider the permutation $%
\sigma ^{\otimes k}$ of the set $\left[ n\right] ^{k}$ of $k$-sequences of
elements of $n$ defined by%
\begin{equation*}
\sigma ^{\otimes k}\left( i_{1},\ldots ,i_{k}\right) =\left( \sigma
(i_{1}),\ldots ,\sigma (i_{k})\right) .
\end{equation*}%
Identifying the group of permutations of $\overset{k%
\text{ times}}{\overbrace{n\times \cdots \times n}}$ with $S_{n^{k}}$, the
function%
\begin{equation*}
\sigma \mapsto \sigma ^{\otimes k}
\end{equation*}%
defines a group *-homomorphism from $S_{n}$ to $S_{n^{k}}$ such that%
\begin{equation*}
1-\ell _{S_{n^{k}}}\left( \sigma ^{\otimes k}\right) =\left( 1-\ell
_{S_{n}}\left( \sigma \right) \right) ^{k}\text{.}
\end{equation*}
\end{hint}

Using Exercise \ref{Exercise: amplification sofic} one can express the
notion of soficity in terms of length-preserving embedding into
ultraproducts of permutations groups.

\begin{exercise}
\label{Exercise: sofic equivalence 2}Suppose that $\Gamma $ is a countable
discrete group%
\index{group!sofic} regarded as an invariant length group endowed with the
trivial invariant length. Show that the following statements are equivalent:

\begin{itemize}
\item $\Gamma $ is sofic;

\item there is a length-preserving *-homomorphism $\Phi :\Gamma \rightarrow
\prod\nolimits_{\mathcal{U}}S_{n}$ for every free ultrafilter $\mathcal{U}$
over $N$;

\item there is an length-preserving *-homomorphism $\Phi :\Gamma \rightarrow
\prod\nolimits_{\mathcal{U}}S_{n}$ for some free ultrafilter $\mathcal{U}$
over $N$.
\end{itemize}
\end{exercise}

\begin{hint}
Follow the same steps as in the proof of Exercise \ref{Exercise: sofic
equivalence 2}, replacing the condition given in the definition of sofic
group with the equivalent condition expressed in Exercise \ref{Exercise:
amplification sofic}.
\end{hint}

In a number of cases it is useful to consider approximate morphisms
satisfying convenient additional properties. An example of such morphisms
with additional properties is considered in the following exercise,
originally proved in \cite[Lemma 2.1]{Elek-Szabo-onSofic}.

\begin{exercise}
\label{Exercise: nice sofic approximation}Suppose that $\Gamma $ is a
countable discrete sofic group, $F$ is a finite symmetric subset of $\Gamma $
containing the identity element, and $\varepsilon $ is a positive real
number. Show that one can find an $\left( F,\varepsilon \right) $%
-approximate morphism $g\mapsto \sigma _{g}$ from $\Gamma $ to $S_{n}$ for
some $n\in \mathbb{N}$ such that, for every $g\in F$, $\sigma _{g}$ has no
fixed points and $\sigma _{g^{-1}}=\sigma _{g}^{-1}$.
\end{exercise}

\begin{hint}
Fix $\eta >0$ small enough and set $%
\widehat{F}=F\cdot F$. By Exercise \ref{Exercise: amplification sofic} one
can consider an $(\widehat{F},\eta )$-approximate morphism $g\mapsto \tau
_{g}$ from $\Gamma $ to $S_{n}$ for some $n\in \mathbb{N}$. For any
nonidentity element $g$ of $\widehat{F}$ consider the largest subset $A_{g}$
of $n$ such that $\tau _{g^{-1}\tau g}$ is the identity on $A_{g}$ and $\tau
_{g}$ has no fixed points on $A_{g}$. Verify that $\tau _{g}$ defines a
bijection between $A_{g}$ and $A_{g^{-1}}$. Define an approximate morphism $%
g\mapsto \sigma _{g}$ from $\Gamma $ to $S_{2n}$ in the following way.
Identify $S_{2n}$ with the group of premutations of $n\times 2$. For any
nonidentity element $g$ of $\widehat{F}$ fix a bijection $\phi _{g}$ from $%
A_{g}\backslash A_{g^{-1}}\times 2$ to $A_{g^{-1}}\backslash A_{g}\times 2$
and a fixed-point free involution $\psi _{g}$ of $\left( n\backslash
(A_{g}\cup A_{g^{-1}})\right) \times 2$ such that $\phi _{g^{-1}}=\phi
_{g}^{-1}$ and $\psi _{g^{-1}}=\psi _{g}^{-1}$. Define then%
\begin{equation*}
\tau _{g}\left( i,j\right) =\left\{ 
\begin{array}{ll}
\left( \sigma _{g}\left( i\right) ,j\right) & \text{if }i\in A_{g}\text{,}
\\ 
\phi _{g}\left( i,j\right) & \text{if }i\in A_{g}\backslash A_{g^{-1}}\text{,%
} \\ 
\phi _{g^{-1}}\left( i,j\right) & \text{if }i\in A_{g^{-1}}\backslash A_{g}%
\text{,} \\ 
\psi _{g}\left( i,j\right) & \text{otherwise.}%
\end{array}%
\right.
\end{equation*}%
Verify that for $\eta $ small enough the assignment $g\mapsto \tau _{g}$ is
an $\left( F,\varepsilon \right) $-approximate morphism with the required
extra properties.
\end{hint}

\section{Definition of hyperlinear groups\label{Section: definition
hyperlinear groups}}

Recall that $M_{n}(\mathbb{C})$ denotes the tracial von Neumann algebra of $%
n\times n$ matrices over the complex numbers. The \textit{normalized} trace $%
\tau $ of $M_{n}(\mathbb{C})$ is defined by%
\begin{equation*}
\tau \left( (a_{ij})\right) =\frac{1}{n}\sum_{i=1}^{n}a_{ii}\text{.}
\end{equation*}%
The \textit{Hilbert-Schmidt norm} $\left\Vert x\right\Vert _{2}$ on $M_{n}(%
\mathbb{C})$ is defined by%
\begin{equation*}
\left\Vert x\right\Vert _{2}=\tau (x^{\ast }x)^{\frac{1}{2}}\text{.}
\end{equation*}%
An element $x$ of $M_{n}(\mathbb{C})$ is unitary if $x^{\ast }x=xx^{\ast }=1$%
. The set $U_{n}$ of unitary elements of $M_{n}(\mathbb{C})$ is a group with
respect to multiplication. The \textit{Hilbert-Schmidt invariant length
function%
\index{length function!Hilbert-Schmidt}} on $U_{n}$ is defined by%
\begin{equation*}
\ell _{U_{n}}(u)=%
\frac{1}{2}\left\Vert u-1\right\Vert _{2}\text{.}
\end{equation*}

\begin{exercise}
Show that $\ell _{U_{n}}$ is an invariant length function on $U_{n}$ such
that $\ell _{U_{n}}(u)^{2}=\frac{1}{2}\left( 1-\mathrm{Re}\left( \tau
(u)\right) \right) $.
\end{exercise}

Hyperlinear%
\index{group!hyperlinear} groups are defined exactly as sofic groups, where
the permutation groups with the Hamming invariant length function are
replaced by the unitary groups endowed with the Hilbert-Schmidt invariant
length function.

\begin{definition}
\label{Definition: hyperlinear} A countable discrete group $\Gamma $ is 
\emph{hyperlinear}%
\index{group!hyperlinear} if for every $\varepsilon >0$ and every finite
subset $F$ of $\Gamma \left\backslash \left\{ 1_{\Gamma }\right\} \right. $
there is a natural number $n$ and a function $\Phi :\Gamma \rightarrow U_{n}$
such that $\Phi \left( 1_{\Gamma }\right) =1_{U_{n}}$ and for every $g,h\in
F $:

\begin{itemize}
\item $d_{U_{n}}\left( \Phi (gh),\Phi (g)\Phi (h)\right) <\varepsilon $;

\item $\ell _{U_{n}}\left( \Phi (g)\right) >r(g)$ where $r(g)$ is some
positive constant depending only on $g$.
\end{itemize}
\end{definition}

As before this notion can be equivalently reformulated in terms of embedding
into ultraproducts. If $\mathcal{U}$ is a free ultrafilter over $\mathbb{N}$
the ultraproduct $\prod\nolimits_{\mathcal{U}}U_{n}$ of the unitary groups
regarded as invariant length groups%
\index{ultraproduct!of unitary groups} is the quotient of $\prod_{n}U_{n}$
with respect to the normal subgroup%
\begin{equation*}
N_{\mathcal{U}}=\left\{ \left( u_{n}\right) _{n\in 
%TCIMACRO{\U{2115} }%
%BeginExpansion
\mathbb{N}
%EndExpansion
}:\lim_{n\rightarrow \mathcal{U}}\ell _{U_{n}}\left( u_{n}\right) =0\right\}
\end{equation*}%
endowed with the invariant length function%
\begin{equation*}
\ell _{\mathcal{U}}\left( \left( u_{n}\right) N_{\mathcal{U}}\right)
=\lim_{n\rightarrow \mathcal{U}}\ell (u_{n}) 
\text{.}
\end{equation*}

\begin{exercise}
\label{Exercise: hyperlinear equivalence 1}Suppose that $\Gamma $ is a
countable discrete group%
\index{group!hyperlinear}. Show that the following statements are equivalent:

\begin{itemize}
\item $\Gamma $ is hyperlinear%
\index{group!hyperlinear};

\item there is an injective *-homomorphism $\Phi :\Gamma \rightarrow
\prod\nolimits_{\mathcal{U}}U_{n}$ for every free ultrafilter $\mathcal{U}$
over $N$;

\item there is an injective *-homomorphism $\Phi :\Gamma \rightarrow
\prod\nolimits_{\mathcal{U}}U_{n}$ for some free ultrafilter $\mathcal{U}$
over $N$.
\end{itemize}
\end{exercise}

As in the case of sofic groups, (length) ultraproducts of sequences of
unitary groups can be referred to as \emph{universal hyperlinear groups}%
\index{group!universal hyperlinear} in view of Exercise \ref{Exercise:
hyperlinear equivalence 1}. Exercise \ref{Exercise: hyperlinear equivalence
3} in Section \ref{Section: hyperlinear and radulescu} shows that the same
conclusion of Exercise \ref{Exercise: hyperlinear equivalence 1} holds for
(length) ultrapowers of the unitary group of the hyperfinite \textup{II}$%
_{1} $ factor (see Definition \ref{Definition: hyperfinite}).

If $\sigma $ is a permutation over $n$ denote by $P_{\sigma }$ the
permutation matrix associated with $\sigma $, acting as $\sigma $ on the
canonical basis of $\mathbb{C}^{n}$. Observe that $P_{\sigma }$ is a unitary
matrix and the function%
\begin{equation*}
\sigma \mapsto P_{\sigma }
\end{equation*}%
is a *-homomorphism from $S_{n}$ to $U_{n}$. Moreover 
\begin{equation*}
\tau \left( P_{\sigma }\right) =1-\ell _{S_{n}}(\sigma )
\end{equation*}%
and hence%
\begin{equation}
\ell _{U_{n}}\left( P_{\sigma }\right) ^{2}=%
\frac{1}{2}\ell _{S_{n}}(\sigma )\text{\label{Equation: relation Hamming and
Hilbert length functions}.}
\end{equation}%
It is not difficult to deduce from this that any sofic group is hyperlinear%
\index{group!hyperlinear}. This is the content of Exercise \ref{Exercise:
sofic implies hyperlinear}.%
\index{group!sofic}

\begin{exercise}
\label{Exercise: sofic implies hyperlinear}Fix a free ultrafilter $\mathcal{U%
}$ over $N $. Show that the function%
\begin{equation*}
(\sigma _{n})_{n\in \mathbb{N}}\mapsto \left( P_{\sigma _{n}}\right) _{n\in 
\mathbb{N}}
\end{equation*}%
from $\prod_{n}S_{n}$ to $\prod_{n}U_{n}$ induces an algebraic embedding of $%
\prod\nolimits_{\mathcal{U}}S_{n}$ into $\prod\nolimits_{\mathcal{U}}U_{n}$.
\end{exercise}

Proposition is an immediate consequence of Exercise \ref{Exercise: sofic
implies hyperlinear}, together with \ref{Exercise: sofic equivalence 1}, and %
\ref{Exercise: hyperlinear equivalence 1}.

\begin{proposition}
\label{Proposition: sofic implies hyperlinear}Every countable discrete sofic
group is hyperlinear.
\end{proposition}

It is currently and open problem to determine whether the notions of sofic%
\index{group!sofic} and hyperlinear%
\index{group!hyperlinear} group are actually distinct, since no example of a
nonsofic groups is known.

\begin{exercise}
\label{Exercise: nonseparable universal hyperlinear}Show that for any free
ultrafilter $\mathcal{U}$ there is a subset $X$ of $\prod\nolimits_{\mathcal{%
U}}U_{n}$ of size continuum such that every pair of distinct elements of $X$
has distance $%
\frac{1}{\sqrt{2}}$.
\end{exercise}

\begin{hint}
Proceed as in Exercise \ref{Exercise: nonseparable universal sofic}, where
for every $n\in 
%TCIMACRO{\U{2115} }%
%BeginExpansion
\mathbb{N}
%EndExpansion
$ the cyclic permutation $\sigma _{n}$ is replaced with the unitary
permutation matrix $P_{\sigma _{n}}$ associated with $\sigma _{n}$. Recall
the relation between the Hamming length function and the Hilbert-Schmidt
length functions given by Equation \ref{Equation: relation Hamming and
Hilbert length functions}.
\end{hint}

It follows from \ref{Exercise: nonseparable universal hyperlinear} that a
dense subset of $\prod\nolimits_{\mathcal{U}}U_{n}$ has the cardinality of
the continuum. In particular the universal hyperlinear%
\index{group!universal hyperlinear} groups are not separable.

An amplification argument from \cite{Radulescu} due to R\u{a}dulescu
(predating the analogous argument for permutation groups of Elek and Szab%
\'{o}) shows that hyperlinearity%
\index{group!hyperlinear} is equivalent to an apparently stronger
approximation property. In order to present this argument we need to recall
some facts about tensor product of matrix algebras.

If $M_{n}(\mathbb{C})$ and $M_{m}(\mathbb{C})$ are the algebras of,
respectively, $n\times n$ and $m\times m$ matrices with complex
coefficients, then their tensor product (see Appendix \ref{suse:tensor
product}) can be canonically identified with the algebra $M_{nm}(\mathbb{C})$
of $nm\times nm$ matrices. Concretely if $A=(a_{ij})\in M_{n}(\mathbb{C})$
and $B=\left( b_{ij}\right) \in M_{m}(\mathbb{C})$ then $A\otimes B$ is
identified with the block matrix%
\begin{equation*}
\begin{pmatrix}
a_{11}B & a_{12}B & \ldots & a_{1n}B \\ 
a_{21}B & a_{22}B & \ldots & a_{2n}B \\ 
\vdots & \vdots & \ddots & \vdots \\ 
a_{n1}B & \ldots & \ldots & a_{nn}B%
\end{pmatrix}%
\text{.}
\end{equation*}%
It is easily verified that%
\begin{equation*}
\tau \left( A\otimes B\right) =\tau (A)\tau \left( B\right)
\end{equation*}%
and $A\otimes B$ is unitary if both $A$ and $B$ are.

\begin{exercise}
\label{Exercise: trace reduction}Suppose that $u\in U_{n}$. If $u$ is
different from the identity, then the absolute value of the trace of%
\begin{equation*}
\begin{pmatrix}
u & 0 \\ 
0 & 1_{U_{n}}%
\end{pmatrix}%
\in U_{2n}
\end{equation*}%
is strictly smaller than $1$.
\end{exercise}

\begin{exercise}
\label{Exercise: amplification hyperlinear}If $u\in U_{n}$ then define
recursively $u^{\otimes 1}=u\in U_{n}$ and $u^{\otimes \left( k+1\right)
}=u^{\otimes k}\otimes u\in U_{n^{k+1}}$. Show that the function%
\begin{equation*}
u\mapsto u^{\otimes k}
\end{equation*}%
is a group *-homomorphism from $U_{n}$ to $U_{n^{k}}$ such that%
\begin{equation*}
\tau (u)=\tau (u)^{k}\text{.}
\end{equation*}
\end{exercise}

We can now state and prove the promised equivalent characterization of
hyperlinear%
\index{group!hyperlinear} groups.

\begin{proposition}
\label{Proposition: amplification hyperlinear}A countable discrete group $%
\Gamma $ is hyperlinear%
\index{group!hyperlinear} if and only for every finite subset $F$ of $\Gamma 
$ and every positive real number $\varepsilon $ there is a natural number $n$
and an $\left( F,\varepsilon \right) $-approximate morphism%
\index{approximate morphism!of invariant length groups} from $\Gamma $ to $%
U_{n}$, where $\Gamma $ is regarded as invariant length groups with respect
to the trivial length function.
\end{proposition}

\begin{proof}
Suppose that $\Gamma $ is hyperlinear. If $F$ is a finite subset of $\Gamma $
and $\varepsilon $ is a positive real number, consider the map $\Phi :\Gamma
\rightarrow U_{n}$ obtained from $F$ and $\varepsilon $ as in the definition
of hyperlinear group. By Exercise \ref{Exercise: trace reduction} after
replacing $\Phi $ with the map%
\begin{eqnarray*}
\Gamma  &\rightarrow &U_{2n} \\
\gamma  &\mapsto &%
\begin{pmatrix}
\Phi (\gamma ) & 0 \\ 
0 & 1_{U_{n}}%
\end{pmatrix}%
\end{eqnarray*}%
we can assume that%
\begin{equation*}
\left\vert \tau \left( \Phi (\gamma )\right) \right\vert <1
\end{equation*}%
for every $\gamma \in F$. It is now easy to show using Exercise \ref%
{Exercise: amplification hyperlinear} that the map%
\begin{eqnarray*}
\Gamma  &\rightarrow &U_{n^{k}} \\
\gamma  &\mapsto &\Phi (\gamma )^{\otimes k}
\end{eqnarray*}%
for $k\in 
%TCIMACRO{\U{2115} }%
%BeginExpansion
\mathbb{N}
%EndExpansion
$ large enough satisfies the requirements of the statement. The converse
implication is obvious.
\end{proof}

As before we can deduce a characterization of hyperlinear%
\index{group!hyperlinear} groups in terms of length-preserving embeddings
into ultraproducts of unitary groups.

\begin{exercise}
\label{Exercise: hyperlinear equivalence 2}Suppose that $\Gamma $ is a
countable discrete group%
\index{group!hyperlinear} regarded as an invariant length group endowed with
the trivial invariant length function. Show that the following statements
are equivalent:

\begin{itemize}
\item $\Gamma $ is hyperlinear%
\index{group!hyperlinear};

\item there is a length preserving *-homomorphism $\Phi :\Gamma \rightarrow
\prod\nolimits_{\mathcal{U}}U_{n}$ for every free ultrafilter $\mathcal{U}$
over $N$;

\item there is a length preserving *-homomorphism $\Phi :\Gamma \rightarrow
\prod\nolimits_{\mathcal{U}}U_{n}$ for some free ultrafilter $\mathcal{U}$
over $N$.
\end{itemize}
\end{exercise}

We conclude this section with yet another reformulation of the definition of
hyperlinear group in terms of what are usually called \emph{microstates}%
\index{microstate}. Suppose that $\Gamma $ is a finitely generated group
with finite generating set $S=\left\{ g_{1},\ldots ,g_{m}\right\} $. If $%
n\in \mathbb{N}$ and $\varepsilon >0$ then an $\left( n,\varepsilon \right) $%
-microstate for $\Gamma $ is a map $\Phi :\Gamma \rightarrow M_{k}\left( 
\mathbb{C}\right) $ for some $k\in \mathbb{N}$ such that for every word $w$
in $x_{i}$ and $x_{i}^{-1}$ for $i\leq m$ of length at most $n$ one has that%
\begin{equation*}
\left\vert \tau \left( w\left( \Phi \left( g_{1}\right) ,\ldots ,\Phi \left(
g_{m}\right) \right) \right) \right\vert <\varepsilon \quad 
\text{whenever }w\left( g_{1},\ldots ,g_{m}\right) \neq 1_{\Gamma }\text{,}
\end{equation*}
and%
\begin{equation*}
\left\vert \tau \left( w\left( \Phi \left( g_{1}\right) ,\ldots ,\Phi \left(
g_{m}\right) \right) \right) -1\right\vert <\varepsilon \text{ whenever }%
w\left( g_{1},\ldots ,g_{n}\right) =1_{\Gamma }\text{.}
\end{equation*}%
A countable discrete groups is hyperlinear if and only if it admits $\left(
n,\varepsilon \right) $-microstates for every $n\in \mathbb{N}$ and $%
\varepsilon >0$. In order to verify that such a definition is equivalent to
the previous ones---and, in particular, does not depend on the choice of the
generating set $S$---let us consider the \emph{relation }$R\left( x\right) $
defined by%
\begin{equation*}
\max \left\{ \left\Vert xx^{\ast }-1\right\Vert _{2},\left\Vert x^{\ast
}x-1\right\Vert _{2}\right\} \text{.}
\end{equation*}%
It is obvious that an element $a$ in $M_{n}\left( \mathbb{C}\right) $ is a
unitary if and only if $R(a)=0$. Moreover such a relation has the following
property, usually called \emph{stability}%
\index{stability}: for every $\varepsilon >0$ there is $\delta >0$ such that
for every $n\in \mathbb{N}$ and every element $a$ of $M_{n}\left( \mathbb{C}%
\right) $ if $R(a)<\delta $ then there is a unitary $u\in M_{n}\left( 
\mathbb{C}\right) $ such that $\left\Vert a-u\right\Vert _{2}<\varepsilon $.
In other words any approximate solution to the equation $R\left( x\right) =0$
is close to an exact solution. Such a property, which holds even if one
replaces $M_{n}\left( \mathbb{C}\right) $ with an arbitrary tracial von
Neumann algebra, can be established by means of the \emph{polar decomposition%
}%
\index{polar decomposition}\emph{\ }of an element inside a von Neumann
algebra; see \cite[Section III.1.1.2]{Blackadar}.

\begin{exercise}
\label{Exercise:microstates}Using stability of the relation defining unitary
elements, verify that the microstates formulation of hyperlinearity is
equivalent to the original definition.
\end{exercise}

\section{Classes of sofic and hyperlinear groups\label{Section: classes of
sofic and hyperlinear groups}}

A classical theorem of Cayley (see \cite{Cayley}) asserts that any finite
group is isomorphic to a group of permutations (with no fixed points) on a
finite set. To see this just let the group act on itself by left
translation. This observation implies in particular that finite groups%
\index{group!finite} are sofic%
\index{group!sofic}. This argument can be generalized to prove that amenable
groups are sofic. Recall that a countable discrete group $\Gamma $ is \emph{%
amenable}%
\index{group!amenable} if for every finite subset $F$ of $\Gamma $ and for
every $\varepsilon >0$ there is an $\left( H,\varepsilon \right) $-invariant
finite subset $K$ of $\Gamma $, i.e.\ such that%
\begin{equation*}
\left\vert hK\triangle K\right\vert <\varepsilon \left\vert K\right\vert
\end{equation*}%
for every $h\in H$. Amenable groups were introduced in 1929 in \cite%
{vonNeumann} by von Neumann in relation with his investigations upon the
Banach-Tarski paradox. Since then they have been the subject of an intensive
study from many different perspectives with recent applications even in game
theory \cite{Ca-Mo}, \cite{Ca-Sc}. More information about this topic can be
found in the monographs \cite{Greenleaf} and \cite{Paterson}.

\begin{proposition}
\label{prop:amenable are sofic} Amenable%
\index{group!amenable} groups are sofic.
\end{proposition}

\begin{proof}
Suppose that $\Gamma $ is an amenable countable discrete group, $F$ is a
finite subset of $\Gamma $, and $\varepsilon $ is a positive real number.
Fix a finite $\left( F\cup F^{-1},\varepsilon \right) $-invariant subset $K$
of $\Gamma $. If $\gamma \in F$ then define%
\begin{equation*}
\sigma _{\gamma }(x)=\gamma x
\end{equation*}%
for $x\in \gamma ^{-1}K\cap K$ and extended $\sigma _{\gamma }$ arbitrarily
to a permutation of $K$. Observe that this defines an $\left( F,2\varepsilon
\right) $-approximate morphism\index{approximate morphism!of invariant length groups} from $G$ to the group $S_{K}$ of permutations
of $K$ endowed with the Hamming length (which is isomorphic as invariant
length group to $S_{n}$ where $n$ is the cardinality of $K$). This concludes
the proof that $\Gamma $ is sofic\index{group!sofic}.
\end{proof}

If $\mathcal{C}$ is a class of (countable, discrete) groups, then a group $%
\Gamma $ is \textit{locally embeddable%
\index{group!locally embeddable}} into elements of $\mathcal{C}$ if for
every finite subset $F$ of $\Gamma $ there is a function $\Phi $ from $%
\Gamma $ to a group $T\in \mathcal{C}$ such that $\Phi $ is nontrivial and
preserves the operation on $F$, i.e.\ for every $g,h\in F$:%
\begin{equation*}
\Phi (gh)=\Phi (g)\Phi (h)%
\text{,}
\end{equation*}%
and%
\begin{equation*}
\Phi (g)\neq 1,\qquad \text{for all $g\neq 1$.}
\end{equation*}%
The group $\Gamma $ is \textit{residually} in $\mathcal{C}$ if moreover $%
\Phi $ is required to be a surjective *-homomorphism. Recall also that $%
\Gamma $ is \textit{locally} in $\mathcal{C}$ if every finitely generated
subgroup of $\Gamma $ belongs to $\mathcal{C}$. It is clear that if $\Gamma $
is either locally in $\mathcal{C}$ or residually in $\mathcal{C}$, then in
particular $\Gamma $ is locally embeddable into elements of $\mathcal{C}$.

It is clear from the very definition that soficity%
\index{group!sofic} and hyperlinearity%
\index{group!hyperlinear} are property concerning only \textit{finite subsets%
} of the group. This is made precise in Exercise \ref{Exercise: locally
sofic}.

\begin{exercise}
\label{Exercise: locally sofic}Show that a group that is locally embeddable
into sofic groups is sofic. The same is true replacing sofic with hyperlinear%
\index{group!hyperlinear}.
\end{exercise}

In particular Exercise \ref{Exercise: locally sofic} shows that locally
sofic and residually sofic groups are sofic.

Groups that are locally embeddable into finite groups were introduced and
studied in \cite{Gordon-Vershik} under the name of LEF groups%
\index{group!LEF}. Since finite groups are sofic%
\index{group!sofic}, Exercise \ref{Exercise: locally sofic} implies that LEF
groups are sofic. In particular residually finite groups are sofic%
\index{group!residually finite}. More generally groups that are locally
embeddable into amenable groups%
\index{group!LEA} (LEA)---also called \emph{initially subamenable}%
\index{group!initially subamenable}---are sofic. It is a standard result in
group theory that free groups are residually finite (see \cite[Example 1.3]%
{Kwiatkowska-Pestov}). Therefore the previous discussion implies that free
groups%
\index{group!free} are sofic 
\index{group!free}. Examples of hyperlinear and sofic%
\index{group!hyperlinear} groups that are not LEA have been recently
constructed by Andreas Thom \cite{Th} and Yves de Cornulier \cite{Co11},
respectively. We will present these examples in Section \ref{Section:
examples}.

It should be now mentioned that it is to this day not know if there is any
group which is not sofic, nor if there is any group which is not hyperlinear.

\begin{openbroblem}
Are the classes of sofic and hyperlinear%
\index{group!hyperlinear} groups proper subclasses of the class of all
countable discrete groups?
\end{openbroblem}

It is not even known if every \emph{one-relator group}%
\index{group!one-relator}, i.e.\ a finitely presented group with only one
relation,%
\index{group!one-relator} is necessarily sofic or hyperlinear. This is an
open problem suggested by Nate Brown. Similarly it is not known whether all 
\emph{Gromov hyperbolic}%
\index{group!hyperbolic} groups are sofic. In fact it is not known whether
there is a Gromov hyperbolic groups that is not residually finite%
\index{group!residually finite}. An example of a \emph{monoid }that does not
satisfy the natural generalization of soficity for monoids is provided in 
\cite{ceccherini-silberstein_sofic_2014}.

\section{Closure properties of the classes of sofic and hyperlinear groups 
\label{Section: closure properties}}

The classes of sofic%
\index{group!sofic} and hyperlinear%
\index{group!hyperlinear} groups have nice closure properties.

\begin{proposition}
\label{prop:closure properties}The class of sofic groups is closed with
respect to the following operations:

\begin{enumerate}
\item Subgroups;

\item Direct limits%
\index{limi!direct};

\item Direct products%
\index{product!direct};

\item Inverse limits%
\index{limit!inverse};

\item Extensions by amenable groups%
\index{extension};

\item Free products%
\index{product!free};

\item Free products amalgamated over an amenable group%
\index{product!free (amalgamated)};

\item HNN extension over an amenable group%
\index{extension!HNN};

\item Graph product%
\index{product!fraph}.
\end{enumerate}
\end{proposition}

It is clear from the definition that soficity is a local property. Therefore
a group is sofic if and only if all its (finitely generated) subgroups are
sofic. In particular subgroups of sofic groups are sofic. It immediately
follows that the direct limit 
\index{limit!direct}of sofic groups is sofic.

To see that the direct product of two sofic groups is sofic, suppose that $%
\Gamma _{0},\Gamma _{1}$ are sofic groups and $F_{0},F_{1}$ are finite
subsets of $\Gamma _{0}$ and $\Gamma _{1}$ respectively. If $\sigma _{0}$
and $\sigma _{1}$ are elements of $S_{n}$ and $S_{m}$ respectively, then
define $\sigma _{0}\otimes \sigma _{1}\in S_{nm}$ by%
\begin{equation*}
\left( \sigma _{0}\otimes \sigma _{1}\right) \left( im+j\right) =\sigma
_{0}(i)m+\sigma _{1}\left( j\right)
\end{equation*}%
for $i\in n$ and $j\in m$. If $\Phi _{0}:\Gamma _{0}\rightarrow S_{n}$ is an 
$\left( F_{0},\varepsilon \right) $-approximate morphism%
\index{approximate morphism!of invariant length groups} and $\Phi
_{1}:\Gamma _{1}\rightarrow S_{m}$ is an $\left( F_{1},\varepsilon \right) $%
-approximate morphism, then the map 
\begin{equation*}
\Phi _{0}\otimes \Phi _{1}:\Gamma _{0}\times \Gamma _{1}\rightarrow S_{nm}
\end{equation*}%
defined by%
\begin{equation*}
\left( \gamma _{0},\gamma _{1}\right) \mapsto \Phi _{0}\left( \gamma
_{0}\right) \otimes \Phi _{1}\left( \gamma _{1}\right)
\end{equation*}%
is an $\left( F_{0}\times F_{1},2\varepsilon \right) $-approximate morphism%
\index{approximate morphism!of invariant length groups}. This observation is
sufficient to conclude that a direct product%
\index{product!direct} of two sofic groups is sofic. The result easily
generalizes to arbitrary direct products in view of the local nature of
soficity. Since the inverse limits%
\index{limit!inverse} is a subgroup of the direct product, it follows from
what we have observed so far that the inverse limit of sofic groups is sofic.

We now prove the result---due to Elek and Szab\'{o} \cite{Elek-Szabo-onSofic}%
---that the extension 
\index{extension}of a sofic group by an amenable group is sofic.\ Suppose
that $\Gamma $ is a group, and $N$ is a normal sofic subgroup of $\Gamma $
such that the quotient $\Gamma /N$ is amenable. We want to show that $\Gamma 
$ is sofic. Fix $\varepsilon >0$ and a finite subset $F$ of $\Gamma $.
Denote by $g\mapsto 
\overline{g}$ the canonical quotient map from $\Gamma $ to $\Gamma /N$ and
let $r$ be a (set-theoretic) inverse for the quotient map. Observe that $%
r\left( \overline{g}\right) ^{-1}g\in N$ for every $g\in G$. F\o lner's
reformulation of amenability yields a finite subset $\overline{A}$ of $%
\Gamma \backslash N$ such that%
\begin{equation*}
\left\vert \overline{A}\overline{g}\backslash \overline{A}\right\vert \leq
\varepsilon \left\vert \overline{A}\right\vert
\end{equation*}%
for every $g\in F$. Let $A$ be the image of $\overline{A}$ under $r$ and $%
\widehat{F}$ be $N\cap (A\cdot F\cdot A^{-1})$. Since $N$ is sofic there is
a $(\widehat{F},\varepsilon )$-approximate morphism $g\mapsto \sigma _{g}$
from $N$ to $S_{n}$ for some $n\in \mathbb{N}$. Let $k$ be $n|A|=n\left\vert 
\overline{A}\right\vert $ and identify $S_{k}$ with the group of
permutations of $n\times A$. Define the map $g\mapsto \tau _{g}$ from $%
\Gamma $ to $S_{k}$ by setting%
\begin{equation*}
\tau _{g}\left( i,h\right) =\left\{ 
\begin{array}{cc}
\left( \sigma _{\overline{r\left( gh\right) }^{-1}gh}\left( i\right) ,r(%
\overline{gh})\right) & \text{if }\overline{gh}\in \overline{A} \\ 
\left( i,h\right) & \text{otherwise.}%
\end{array}%
\right. \text{.}
\end{equation*}%
This is well defined since, as we observed above, $r\left( \overline{gh}%
\right) ^{-1}gh\in N$.

\begin{exercise}
Show that the map $g\mapsto \tau _{g}$ is an $\left( F,3\varepsilon \right) $%
-approximate morphism from $\Gamma $ to $S_{k}$.
\end{exercise}

\begin{hint}
Suppose that $g\in F\backslash \left\{ 1_{\Gamma }\right\} $. To show that $%
\tau _{g}$ has at most $\varepsilon n\left\vert A\right\vert $ fixed points,
distinguish the cases when $g\notin N$ and $g\in N$. In the first case
observe that $\left( i,h\right) $ is a fixed point only if $\overline{gh}%
\notin \overline{A}$, and use the fact that $\left\vert \overline{A}%
\overline{g}\backslash \overline{A}\right\vert \leq \varepsilon \left\vert 
\overline{A}\right\vert $. In the second case observe that $\left(
i,h\right) $ is a fixed point only if $i$ is a fixed point for $\sigma
_{h^{-1}gh}$, and use the fact that $h^{-1}gh\in \widehat{F}$ and $\gamma
\mapsto \sigma _{\gamma }$ is a $(\widehat{F},\varepsilon )$-approximate
morphism. To conclude fix $g,g^{\prime }\in F$ and observe that for all but $%
\varepsilon \left\vert A\right\vert $ values of $h\in A$, $\overline{gh}$
and $\overline{g^{\prime }gh}$ both belong to $\overline{A}$. For such
values of $h$ show that for all but $\varepsilon n$ values of $i\in n$ one
has that $\tau _{g^{\prime }}\tau _{g}\left( i,h\right) =\tau _{g^{\prime
}g}\left( i,h\right) $.
\end{hint}

It is a little more involved to observe that the free product%
\index{product!free} of sofic groups is sofic, which is a result of Elek and
Szab\'{o} from \cite{Elek-Szabo-onSofic}. Suppose that $\Gamma _{0},\Gamma
_{1}$ are sofic groups, $F_{0},F_{1}$ are finite subsets of $\Gamma _{0}$
and $\Gamma _{1}$, and $\varepsilon >0$. We identify canonically $\Gamma _{0}
$ and $\Gamma _{1}$ with subgroups of the free product $\Gamma =\Gamma
_{0}\ast \Gamma _{1}$. Every element $\gamma $ of $\Gamma $ has a unique
shortest decomposition%
\begin{equation*}
\gamma =g_{1}h_{1}\cdots g_{n}h_{n}
\end{equation*}%
with $g_{i}\in \Gamma _{0}$ and $h_{i}\in \Gamma _{1}$. Fix $N\in \mathbb{N}$
and set $F$ to be the set of elements of $\gamma $ of $\Gamma $ such that
the unique shortest decomposition of $\gamma $ has length at most $N$ and
its elements belong to $F_{0}\cup F_{1}$. By Exercise \ref{Exercise: nice
sofic approximation} one can find $n\in \mathbb{N}$, an $\left(
F_{0},\varepsilon \right) $-approximate morphism $g\mapsto \sigma _{g}^{(0)}$
from $\Gamma _{0}$ to $S_{n}$ and an $\left( F_{1},\varepsilon \right) $%
-approximate morphism $g\mapsto \sigma _{g}^{(1)}$ from $\Gamma _{1}$ to $%
S_{n}$ such that $\left( \sigma _{g}^{(i)}\right) ^{-1}=\sigma
_{g^{-1}}^{\left( i\right) }$ and $\sigma _{g}^{\left( i\right) }$ has no
fixed points for $g\in F^{\left( i\right) }\backslash \left\{ 1_{\Gamma
_{i}}\right\} $ and $i=0,1$. It follows from the already recalled fact that
free groups are residually finite that there is a finite group $V$ generated
by elements $v_{ij}$ for $i,j\in \left[ n\right] $ such that the generators
satisfy no nontrivial relation expressed by words of length at most $N$.
Equivalently the Cayley graph of $V$ with respect to the generators $v_{ij}$
has no cycle of length smaller than or equal to $N$. Consider then $%
k=n^{2}\left\vert V\right\vert $ and identify $S_{k}$ with the permutation
group of $n\times n\times V$. Define for $g\in \Gamma _{0}$ the permutation $%
\tau _{g}$ of $n\times n\times V$ by%
\begin{equation*}
\tau _{g}\left( i,j,w\right) =\left( \sigma _{g}^{\left( 0\right) }\left(
i\right) ,j,w\right) 
\text{.}
\end{equation*}%
Similarly define for $h\in \Gamma _{1}$ the permutation $\tau _{h}$ of $%
\left[ n\right] \times \left[ n\right] \times V$ by%
\begin{equation*}
\tau _{h}\left( i,j,w\right) =\left( i,\sigma _{h}^{\left( 1\right) }\left(
j\right) ,w\cdot v_{ij}^{-1}\cdot v_{i,\sigma _{h}^{\left( 1\right) }\left(
j\right) }\right) \text{.}
\end{equation*}%
Finally is $\gamma $ is an element of $F$ with shortest decomposition $%
g_{1}h_{1}\cdots g_{n}h_{n}$ for $n\leq N$ set%
\begin{equation*}
\tau _{\gamma }=\tau _{g_{1}}\circ \tau _{h_{1}}\circ \cdots \circ \tau
_{g_{n}}\circ \tau _{h_{n}}\text{.}
\end{equation*}

\begin{exercise}
Show that the map $\gamma \mapsto \tau _{\gamma }$ defied above is an $%
\left( F,\varepsilon \right) $-approximate morphism from $\Gamma $ to $S_{k}$%
.
\end{exercise}

\begin{hint}
The fact that the generators $v_{ij}$ of $V$ satisfy no nontrivial relations
expressed by words of length at most $N$ implies that $\tau _{\gamma }$ has
no fixed points whenever $\gamma $ is a nonidentity element of $F$. Suppose
now that $\gamma ,\gamma ^{\prime }\in F\backslash \left\{ 1\right\} $ and
consider the corresponding unique shortest decompositions $\gamma
=g_{1}h_{1}\cdots g_{n}h_{n}$ and $\gamma ^{\prime }=g_{1}^{\prime
}h_{1}^{\prime }\cdots g_{n}^{\prime }h_{n}^{\prime }$. In order to show
that $\tau _{\gamma }\circ \tau _{\gamma ^{\prime }}$ and $\tau _{\gamma
\gamma ^{\prime }}$ differ on at most $\varepsilon k$ points of $\left[ n%
\right] \times \left[ n\right] \times V$ consider different cases depending
whether both, none, or exactly one between $h_{n}$ and $g_{1}^{\prime }$ is
equal to the identity.
\end{hint}

The result about the free product of sofic groups has been later generalized
by Collins and Dykema, who showed that the free product of sofic groups
amalgamated 
\index{product!free (amalgamated)}over a monotileable amenable group is
sofic \cite{Collins-Dykema}. The monotileability assumption has been later
removed by Elek and Szab\'{o} \cite[Theorem 1]{Elek-Szabo-representations}
and, independently, by Paunescu \cite{Pa2}.

Suppose that $\Gamma $ is a group with presentation $\left\langle
S|R\right\rangle $, $H$ is a subgroup of $\Gamma $, and $\alpha
:H\rightarrow \Gamma $ is an injective homomorphism. Let $t$ be a symbol not
in $\Gamma $. The corresponding \emph{HNN extension}%
\index{extension!HNN} $%
\widehat{\Gamma }$ is the group generated by $S\cup \left\{ t\right\} $
subject to the relations $R$ and $t^{-1}ht=\alpha (h)$ for $h\in H$. We
claim that $\widehat{\Gamma }$ is sofic whenever $\Gamma $ is sofic an $H$
is amenable. Define $\Gamma _{i}=t^{-i}\Gamma t^{i}$ for $i\in \mathbb{Z}$,
and let $S$ be the subgroup of $\widehat{\Gamma }$ generated by $\Gamma _{i}$
for $i\in \mathbb{Z}$. Then $\widehat{\Gamma }$ is an extension of $S$ by $%
\mathbb{Z}$. In particular to show that $\widehat{\Gamma }$ is sofic it
suffices to show that $S$ is sofic. To show that $S$ is sofic it is enough
to show that for every $n\in \mathbb{N}$ the group $S_{n}=\left\langle
\Gamma _{i}:i\in \left[ -n,n\right] \cap \mathbb{Z}\right\rangle $ is sofic.
This can be shown by induction on $n$ observing that $S_{0}=\Gamma $, and $%
S_{n+1}$ can be obtained as a free product of $S_{n}$ and isomorphic copies
of $\Gamma $ amalgamated over amenable subgroups isomorphic to $H$.

A different generalization of the result about free products consists in
considering the graph products%
\index{product!graph} of sofic groups. Suppose that $\left( V,E\right) $ is
a simple undirected graph and, for every $v\in V$, $\Gamma _{v}$ is a group.
The corresponding graph product is the quotient of the free products of the $%
\Gamma _{v}$'s by the normal subgroup generated by the relators $\left[
g_{v},g_{w}\right] $ for $g_{v}\in \Gamma _{v}$ and $g_{w}\in \Gamma _{w}$
such that $v$ and $w$ are connected by an edge. Clearly free products
correspond to graph products where there are no edges. Theorem 1.1 of \cite%
{Ciobanu-Holt-Rees} refines the argument for free products to show that in
fact an arbitrary graph product of sofic groups is sofic.

\section{Further examples of sofic group\label{Section: examples}}

In this section we want to list some interesting examples of sofic and
hyperlinear groups. It follows from the discussion in Section \ref{Section:
classes of sofic and hyperlinear groups} that all residually finite groups
are sofic. In particular free groups are sofic. To obtain an example of a
sofic group that is not residually finite one can consider, for $n,m\geq 2$
distinct, the Baumslag-Solitar group%
\index{group!Baumslag-Solitar} $BS(m,n)$ generated by two elements $a,b$
satisfying the relation $a^{-1}b^{m}a=b^{n}$. These groups were introduced
in \cite{baumslag_two-generator_1962} to provide examples of simple finitely
presented groups that are not Hopfian%
\index{group!Hopfian}, i.e.\ admit a surjective but not injective
endomorphism. It is known that one hand such groups are not residually finite%
\index{group!residually finite} \cite[Theorem C]{meskin_nonresidually_1972}.
On the other hand they are residually solvable \cite[Corollary 2]%
{kropholler_baumslag-solitar_1990} and, in particular, sofic.

An example of a sofic 
\index{group!sofic}(and in fact LEF%
\index{group!LEF}) finitely generated group that is not residually amenable%
\index{group!residually amenable} is provided in \cite[Theorem 3]%
{Elek-Szabo-onSofic}, adapting a construction from \cite{Gordon-Vershik}.
Recall that a function $\phi :\Gamma \rightarrow \mathbb{C}$ has positive
type%
\index{function!positive type} if whenever $\lambda _{i}\in \mathbb{C}$ and $%
g_{i}\in \Gamma $ are for $i\leq n$ then%
\begin{equation*}
\sum_{i,j\leq n}\lambda _{i}%
\overline{\lambda }_{j}\phi \left( g_{i}^{-1}g_{j}\right) \geq 0\text{.}
\end{equation*}%
A group $\Gamma $ has Kazhdan's property%
\index{group!property (T)} (T) if any sequence of positive type functions on 
$G$ that converges pointwise to the function $\mathbf{1}$ constantly equal
to $1$ in fact converges to $\mathbf{1}$ uniformly. Such a property,
originally introduced by Kazhdan in \cite{kazdan_connection_1967}, has
played a key role in the latest developments of geometric group theory; see 
\cite{Be-Ha-Va}. Property (T) 
\index{group!property (T)}can be regarded as a strong antithesis to
amenability, since the only amenable groups with property (T) are the finite
groups. Since property (T) is preserved by quotients, it follows that an
amenable quotient of a property (T) group must be finite. We also recall
here that a group with property (T) is finitely generated%
\index{group!finitely generated} \cite[Theorem 1.3.1]{Be-Ha-Va}.

Consider an infinite hyperbolic%
\index{group!hyperbolic} residually finite%
\index{group!residually finite} property (T)%
\index{group!property (T)} group $K$. (Examples of such groups are provided
in \cite{gromov_hyperbolic_1987}.) Let $P$ be the group of finitely
supported permutations of $K$, and $Q$ be the group generated by $P$ and the
left translations by elements of $K$. Then $P$ is a normal subgroup of $Q$
and $Q$ is a semidirect product of $P$ and $K$. Consistently we identify $K$
with a subgroup of $Q$. Let $S$ be a finite set of generators for $K$ and $%
T_{S}$ be the set of transpositions of the form $\left( 1,s\right) $ for $%
s\in S$.

\begin{exercise}
Show that $S\cup T_{S}$ is a set of generators for $Q$.
\end{exercise}

\begin{hint}
Observe that if $s\in S$ and $g\in K$ then the transposition $\left(
g,gs\right) $ can be written as $g^{-1}\left( 1,s\right) g$. Recall that the
group of permutations on $n$ symbols $x_{1},\ldots ,x_{n}$ is generated by
the transpositions $\left( x_{i},x_{i+1}\right) $ for $i\leq n-1$.
\end{hint}

We now observe that $Q$ is not residually amenable. Recall that every
hyperbolic group contains an nontorsion element. Let $t\in K$ such an
element, and let $a\in P$ be the transposition $\left( t,1\right) $. Observe
that $t^{n}at^{-n}=\left( t^{n+1},t\right) \neq a$ for every $n\in \mathbb{N}
$. Suppose by contradiction that there is a homomorphism $\phi :Q\rightarrow
M$ into an amenable group $M$ such that $\phi (a)\neq 1_{M}$ and $\phi
\left( t\right) \neq 1_{M}$. Let $A$ be the (simple) subgroup of $P $ of
even permutations. Since $A$ is simple, $a\in A$ , and $\phi (a)\neq 1_{M}$, 
$\phi $ must be injective on $A$. As observed before the fact that $K $ has
property (T) implies that the image of $K$ under $\phi $ is finite. Let $n$
be the rank of $\phi \left[ K\right] $ and observe that $\phi \left(
t^{n}\right) =1$. Therefore $\phi \left( t^{-n}at^{-n}\right) =\phi (a)$
while $t^{-n}at^{-n}\neq a$. This contradicts the injectivity of $\phi $ on $%
A$.

We now show that $Q$ is sofic and, in fact LEF%
\index{group!LEF}. Denote by $B_{n}(K)$ the $n$-ball around $1_{K}$ in the
Cayley graph of $K$ associated with the generating set $S$. Define $F_{n}$
to be the set of elements of $Q$ of the form $k\sigma $ where $k\in B_{n}(K)$
and $\sigma \in P$ is supported on $B_{n}(K)$. We want to define an
injective map from $F_{2n}$ to a finite group that preserves the operation
on $F_{n}$. Since $K$ is residually finite it has a finite-index normal
subgroup $N_{n}$ such that $N_{n}\cap B_{2n}(K)=\left\{ 1_{K}\right\} $.
Define $H_{n}$ to be the (finite) group of permutations on $K/N_{n}$. Denote
by $\pi _{n}$ the quotient map from $K$ onto $K/H_{n}$ and define $\psi
_{n}:F_{2n}\rightarrow H_{n}$ by setting%
\begin{equation*}
\psi _{n}\left( k\sigma \right) =\psi _{n}\left( k\right) \psi _{n}\left(
\sigma \right) 
\text{.}
\end{equation*}%
Here $\psi _{n}\left( k\right) $ is the left translation by $\pi \left(
k\right) $, while $\psi _{n}\left( \sigma \right) $ is defined by $\psi
_{n}\left( \sigma \right) \left( \pi _{n}(h)\right) =\pi \left( \sigma
(h)\right) $ if $h\in B_{2n}(K)$ and acts as the identity otherwise. Clearly 
$\psi _{n}$ is injective on $F_{2n}$.

\begin{exercise}
Show that $\psi _{n}\left( xy\right) =\psi _{n}\left( x\right) \psi
_{n}\left( y\right) $ for $x,y\in F_{n}$.
\end{exercise}

\begin{hint}
Want to show that $\psi _{n}\left( xy\right) \left( \pi _{n}(h)\right) =\psi
_{n}\left( x\right) \psi _{n}\left( y\right) \left( \pi _{n}(h)\right) $ for
every $h\in K$. Distinguish the cases when $h\in B_{2n}(K)$ and $h\notin
B_{2n}(K)$.
\end{hint}

Another example of a not residually finite sofic group which is moreover
finitely presented (and in fact one-relator%
\index{group!one-relator}) was provided by Jon Bannon in \cite%
{bannon_non-residually_2011}; see also \cite{bannon_note_2010} for other
similar examples. Consider the group $\Gamma $ generated by two elements $%
a,b $ subject to the relation $a=\left[ a,a^{b}\right] $ where $%
a^{b}=bab^{-1}$ and $\left[ a,a^{b}\right] $ is the commutator of $a$ and $%
a^{b}$. Such a group was introduced by Baumslag in \cite%
{baumslag_non-cyclic_1969} as an example of a non-cyclic one-relator group
all of whose finite quotients are cyclic. In particular $\Gamma $ is not
residually finite and, in fact, not residually solvable since $a$ belongs to
all the derived subgroups of $\Gamma $. We now show that it is sofic.

Set $x=a^{-1}$ and $y=bab^{-1}$. Then the relator $a^{-1}\left[ a,a^{b}%
\right] $ becomes $x^{-2}y^{-1}xy$. Observe that the group $H=\left\langle
x,y|y^{-1}xy=x^{2}\right\rangle $ is the Baumslag-Solitar%
\index{group!Baumslag-Solitar} group $B\left( 1,2\right) $ and, in
particular, sofic. Moreover $\Gamma $ is the HNN extension%
\index{extension!HNN} of $H$ with respect to the isomorphism $x^{n}\mapsto
y^{n}$ between the subgroups $\left\langle x\right\rangle $ and $%
\left\langle y\right\rangle $ of $H$. It follows that $\Gamma $ is sofic
since HNN extensions 
\index{extension!HNN}of sofic groups over amenable groups are sofic.

We would now like to present examples of hyperlinear and sofic groups that
are not LEA. Recall that a group is locally embeddable into amenable groups
(LEA)%
\index{group!LEA} if, briefly, every finite portion of its multiplication
table can be realized as a portion of the multiplication table of an
amenable group. The first example of a hyperlinear 
\index{group!hyperlinear}not LEA group was constructed by Thom in \cite{Th},
adapting a construction due to de Cornulier \cite{de_cornulier_finitely_2007}
and Abels \cite{abels_example_1979}.

If $n,m\in \mathbb{N}$ denote by $M_{n,m}\left( \mathbb{Z}\left[ 
\frac{1}{p}\right] \right) $ the space of $n\times m$ matrices over $\mathbb{%
Z}\left[ \frac{1}{p}\right] $ and by $SL_{n}\left( \mathbb{Z}\left[ \frac{1}{%
p}\right] \right) $ 
\index{group!special linear}the group of $n\times n$ matrices of determinant 
$1$. Let $\Gamma $ be the group of matrices in $SL_{8}\left( \mathbb{Z}\left[
\frac{1}{p}\right] \right) $ of the form%
\begin{equation*}
\begin{bmatrix}
1 & a_{12} & a_{13} & a_{14} \\ 
0 & a_{22} & a_{23} & a_{24} \\ 
0 & 0 & a_{33} & a_{34} \\ 
0 & 0 & 0 & 1%
\end{bmatrix}%
\end{equation*}%
where the diagonal blocks are square matrices of rank $1,3,3,1$ and $%
a_{22},a_{33}\in SL_{3}\left( \mathbb{Z}\left[ \frac{1}{p}\right] \right) $.
The center $C$ of $\Gamma $ consists of the matrices of the form%
\begin{equation*}
\begin{bmatrix}
1 & 0 & 0 & a \\ 
0 & I & 0 & 0 \\ 
0 & 0 & I & 0 \\ 
0 & 0 & 0 & 1%
\end{bmatrix}%
\end{equation*}%
where $I$ is the $3\times 3$ identity matrix and $a\in K$. Thus $C$ is
isomorphic to the additive group of $\mathbb{Z}\left[ \frac{1}{p}\right] $.
In particular $C$ contains a subgroup $Z$ isomorphic to $\mathbb{Z}$.

It is shown in \cite[Proposition 2.7]{de_cornulier_finitely_2007} that $%
\Gamma $ is a%
\index{lattice} lattice in a locally compact group with%
\index{group!property (T)} property (T). Since property (T) passes to
lattices \cite[Theorem 1.7.1]{Be-Ha-Va}, it follows that $\Gamma $ has
property (T) and, in particular, is finitely generated. It is moreover shown
in \cite[Section 3]{de_cornulier_finitely_2007} that $\Gamma $ and (hence) $%
\Gamma /Z$ are finitely presented. We present here the argument to show that 
$\Gamma /Z$ is not%
\index{group!Hopfian} Hopfian, i.e.\ it has a surjective endomorphism with
nontrivial kernel.

Denote by $\mathbb{Z}\left[ 
\frac{1}{p}\right] ^{\times }$ the multiplicative group of invertible
elements of $\mathbb{Z}\left[ \frac{1}{p}\right] $. Then $\mathbb{Z}\left[ 
\frac{1}{p}\right] ^{\times }$ identified with the group of matrices of the
form%
\begin{equation*}
\begin{bmatrix}
a & 0 & 0 & 0 \\ 
0 & 1 & 0 & 0 \\ 
0 & 0 & 1 & 0 \\ 
0 & 0 & 0 & 1%
\end{bmatrix}%
\end{equation*}%
naturally acts on $\Gamma $ by conjugation. Considering $p\in \mathbb{Z}%
\left[ \frac{1}{p}\right] ^{\times }$ gives an automorphism $\beta $ of $%
\Gamma $ that maps the center $Z$ of $\Gamma $ to its property subgroup $%
Z^{p} $. Thus $\beta $ induces a surjective endomorphism $\overline{\beta }$
of $\Gamma /Z$ whose kernel is the (nontrivial) subgroup $C/Z$ of $\Gamma /Z$%
; see \cite[Lemma 2.3]{de_cornulier_finitely_2007}. As a consequence $\Gamma
/Z$ is not Hopfian.

As observed in \cite[3.1]{Th} a finitely presented LEA group that has
property (T)\ is necessarily Hopfian. In fact a finitely presented LEA group
is residually amenable by \cite[Proposition 7.3.8]{Ceccherini-Coornaert}. As
recalled before an amenable quotient of a property (T) group is necessarily
finite. Therefore a property (T)\ residually amenable group is residually
finite. Finally a residually finite finitely generated group is easily seen
to be Hopfian. It follows from this argument that the group $\Gamma /Z$,
being finitely presented, property (T), and not Hopfian, is \emph{not }LEA.

To conclude one needs to observe that $\Gamma /Z$ is hyperlinear.
Considering the reduction modulo $q$ where $q$ is an arbitrary positive
integer prime with $p$ shows that the group $\Gamma $ is residually finite
and, in particular, hyperlinear. The argument is concluded by showing that
hyperlinearity is preserved by taking quotients by central subgroups; see 
\cite[Remark 3.4]{Th}.

It is not known whether the group $\Gamma /Z$ above is sofic. We conclude
this sequence of examples with an example due to de Cornulier of a finitely
presented sofic 
\index{group!sofic}group that is not LEA \cite{cornulier_sofic_2011}. Let $%
\Gamma $ be the group of matrices%
\begin{equation*}
\begin{bmatrix}
a & b & u_{02} & u_{03} & u_{04} \\ 
c & d & u_{12} & u_{13} & u_{14} \\ 
0 & 0 & p^{n_{2}} & u_{23} & u_{24} \\ 
0 & 0 & 0 & p^{n_{3}} & u_{34} \\ 
0 & 0 & 0 & 0 & 1%
\end{bmatrix}%
\end{equation*}%
with

\begin{itemize}
\item $%
\begin{bmatrix}
a & b \\ 
c & d%
\end{bmatrix}%
\in SL_{2}\left( \mathbb{Z}\right) $,

\item $u_{ij}\in \mathbb{Z}\left[ 
\frac{1}{p}\right] $, and

\item $n_{2},n_{3}\in \mathbb{Z}$.
\end{itemize}

Let $M$ be the normal subgroup of $\Gamma $ consisting of matrices of the
form%
\begin{equation*}
\begin{bmatrix}
1 & 0 & 0 & 0 & m_{1} \\ 
0 & 1 & 0 & 0 & m_{2} \\ 
0 & 0 & 1 & 0 & 0 \\ 
0 & 0 & 0 & 1 & 0 \\ 
0 & 0 & 0 & 0 & 1%
\end{bmatrix}%
\end{equation*}%
with $m_{1},m_{2}\in \mathbb{Z}\left[ \frac{1}{p}\right] $. The normal
subgroup $M_{\mathbb{Z}}$ is defined similarly with $m_{1},m_{2}\in \mathbb{Z%
}$. Then one can consider the quotient $\Gamma /M_{\mathbb{Z}}$. It is shown
in \cite{cornulier_sofic_2011} that such a group is sofic and finitely
presented but not LEA.

In order to prove soficity, it is enough to show that $\Gamma $ can be
obtained as an extension 
\index{extension}of a sofic group by an amenable group. Consider the normal
subgroup $\Upsilon $ of elements of $\Gamma $ for which $n_{2}=n_{3}=0$.
Observe that $\Gamma /\Upsilon $ is isomorphic to $\mathbb{Z}^{2}$.
Therefore it remains to show that $\Upsilon /M_{\mathbb{Z}}$ is sofic.

Fix $m\in \mathbb{N}$ and consider the subgroup $\Upsilon _{m}$ of elements
of $\Upsilon $ for which

\begin{itemize}
\item $u_{02},u_{12},u_{23},u_{34}\in p^{-m}\mathbb{Z}$,

\item $u_{03},u_{13},u_{24}\in p^{-2m}\mathbb{Z}$, and

\item $u_{04},u_{14}\in p^{-3m}\mathbb{Z}$.
\end{itemize}

It is clear that $\bigcup_{m}\Upsilon _{m}=\Upsilon $. Therefore by the
local nature of soficity we are left with the problem of showing that $%
\Upsilon _{m}/M_{\mathbb{Z}}$ is sofic for every $m\in \mathbb{N}$. Such a
group is in fact residually finite and hence sofic. Consider the subgroup $%
\Lambda $ of elements of $\Gamma $ for which the block%
\begin{equation*}
\begin{bmatrix}
a & 0 \\ 
0 & b%
\end{bmatrix}%
\in SL_{2}\left( \mathbb{Z}\right)
\end{equation*}%
is the identity matrix. Observe that $\left( \Upsilon _{m}\cap \Lambda
\right) /M_{\mathbb{Z}}$ is a normal subgroup of $\Upsilon _{m}/M_{\mathbb{Z}%
}$, and moreover $\Upsilon _{m}/M_{\mathbb{Z}}$ is isomorphic to the
semidirect product $\left( \Upsilon _{m}\cap \Lambda \right) /M_{\mathbb{Z}%
}\rtimes SL_{2}\left( \mathbb{Z}\right) $. Now $\Upsilon _{m}\cap \Lambda $
is 
\index{group!solvable}solvable and finitely generated, and hence residually
finite 
\index{group!residually finite}\cite[2.1]{gruenberg_residual_1957}. The
proof is concluded by observing that a semidirect product of a finitely
generated residually finite group by a residually finite group is residually
finite.

Observe that the quotient map $\Gamma \mapsto \Gamma /M_{\mathbb{Z}}$
restricted to the subgroup (isomorphic to $SL_{2}\left( \mathbb{Z}\right) $)
of elements%
\begin{equation*}
\begin{bmatrix}
a & b & 0 & 0 & 0 \\ 
c & d & 0 & 0 & 0 \\ 
0 & 0 & 1 & 0 & 0 \\ 
0 & 0 & 0 & 1 & 0 \\ 
0 & 0 & 0 & 0 & 1%
\end{bmatrix}%
\end{equation*}%
is injective. Thus $\Gamma /M_{\mathbb{Z}}$ contains a copy of $SL_{2}\left( 
\mathbb{Z}\right) $ and, hence, of the free group on $2$ generators $F_{2}$.
(Recall that the matrices%
\begin{equation*}
A=%
\begin{bmatrix}
1 & 2 \\ 
0 & 1%
\end{bmatrix}%
\quad 
\text{and}\quad B=%
\begin{bmatrix}
1 & 0 \\ 
2 & 1%
\end{bmatrix}%
\end{equation*}%
generate a copy of $F_{2}$ inside $SL_{2}\left( \mathbb{Z}\right) $; see 
\cite[page 3]{Kwiatkowska-Pestov}.)

By \cite[Corollary 7.1.20]{Ceccherini-Coornaert} in order to show that $%
\Gamma /M_{\mathbb{Z}}$ is not LEA it is enough to show that it is an
isolated point in the space of marked groups. Fix $n\in \mathbb{N}$. An $n$%
-marked group%
\index{group!marked} is a group $\Gamma $ endowed with a distinguished
generating $n$-tuple $\left( \gamma _{1},\ldots ,\gamma _{n}\right) $. Let $%
\mathcal{G}_{n}$ be the space of $n$-marked groups. Given an $n$-marked
group one can consider the kernel of the epimorphism $F_{n}\rightarrow
\Gamma $ mapping the canonical free generators of $F_{n}$ to the $\gamma
_{i} $'s. Conversely any normal subgroup $N$ of $F_{n}$ gives rise to the $n$%
-marked group $F_{n}/N$ where the distinguished $n$-tuple of generators is
the image of the free generators of $F_{n}$. This argument shows that one
can identify the space $\mathcal{G}_{n}$ of $n$-marked group with the space
of the normal subgroups of $F_{n}$. Identifying in turn a normal subgroup of 
$F_{n}$ with its characteristic function yields an inclusion of $\mathcal{G}%
_{n}$ into $2^{F_{n}}$ as a closed subspace. This defines a compact
metrizable zero-dimensional topology on $\mathcal{G}_{n}$. Corollary 7.1.20
in \cite{Ceccherini-Coornaert} shows that a group is LEA if and only if it
is a limit of amenable groups in the space of marked groups. Therefore since 
$\Gamma /M_{\mathbb{Z}}$ is not LEA, it is enough to show that $\Gamma /M_{%
\mathbb{Z}}$ is \emph{isolated%
\index{group!isolated}}, i.e.\ it is an isolated point in the space of
marked groups.

It is shown in \cite[Lemma 1]{de_cornulier_isolated_2007} that being
isolated is indeed a well defined property of a group, independent of the
marking. Proposition 2 in \cite{de_cornulier_isolated_2007} provides the
following characterization of being isolated: a group is isolated if and
only if it is finitely presented%
\index{group!finitely presented} and moreover finitely discriminable%
\index{group!finitely discriminable}, i.e.\ it contains a finite subsets
that meets every nontrivial normal subgroup in a nonidentity element. The
proof that $\Gamma /M_{\mathbb{Z}}$ satisfies these conditions is presented
in \cite[Section 3]{cornulier_sofic_2011} and \cite[5.4]%
{de_cornulier_isolated_2007}.

Other examples of finitely presented sofic groups that are not LEA are provided in \cite{kar_non-LEA_2014}.

\section{Logic for invariant length groups\label{Section: logic invariant
length groups}}

The logic for metric structures is a generalization of the usual first order
logic. It is a natural framework to study algebraic structures endowed with
a nontrivial metric and their elementary properties (i.e.\ properties
preserved by ultrapowers or equivalently expressible by formulae). In the
sequel we introduce a particular instance of logic for metric structures to
describe and study groups endowed with an invariant length functions.

A \textit{term%
\index{term!for invariant length groups}} $t(x_{1},\ldots ,x_{n})$ in the
language of invariant length groups in the variables $x_{1},\ldots ,x_{n}$
is a \textit{word}\textbf{\ }in the indeterminates $x_{1},\ldots ,x_{n}$,
i.e.\ an expression of the form%
\begin{equation*}
x_{i_{1}}^{n_{1}}\ldots x_{i_{l}}^{n_{l}}
\end{equation*}%
for $l\in \mathbb{N}$ and $n_{i}\in \mathbb{Z}$ for $i=1,2,\ldots ,l$. For
example%
\begin{equation*}
xyx^{-1}y^{-1}
\end{equation*}%
is a term in the variables $x,y$. The empty word will be denoted by $1$. If $%
G$ is an invariant length group, $g_{1},\ldots ,g_{m}$ are elements of $G$,
and $t( x_{1},\ldots ,x_{n},y_{1},\ldots ,y_{m})$ is a term in the variables 
$x_{1},\ldots ,x_{n},y_{1},\ldots ,y_{m}$, then one can consider the term $%
t(x_{1},\ldots ,x_{n},g_{1},\ldots ,g_{m})$ with \textit{parameters} from $G$%
, which is obtained from $t\left( x_{1},\ldots ,x_{n},y_{1},\ldots
,y_{m}\right) $ replacing formally $y_{i}$ with $g_{i}$ for $i=1,2,\ldots ,m$%
. The evaluation $t^{G}(x_{1},\ldots ,x_{n})$ in a given invariant length
group $G$ of a term $t\left( x_{1},\ldots ,x_{n}\right) $ in the variables $%
x_{1},\ldots ,x_{n}$ (possibly with parameters from $G$) is the function
from $G^{n}$ to $G$ defined by%
\begin{equation*}
\left( g_{1},\ldots ,g_{n}\right) \mapsto t\left( g_{1},\ldots ,g_{n}\right)
\end{equation*}%
where $t\left( g_{1},\ldots ,g_{n}\right) $ is the element of $G$ obtained
replacing in $t(x_{1},\ldots ,x_{n})$ every occurrence of $x_{i} $ with $%
g_{i}$ for $i=1,2,\ldots ,n$. For example the evaluation in a invariant
length group $G$ of the term $xyx^{-1}y^{-1}$ is the function from $G^{2}$
to $G$ that assigns to every pair $\left( g,h\right) $ of elements of $G$
their commutator $ghg^{-1}h^{-1}$. The evaluation of the empty word is the
function on $G$ constantly equal to $1_{G}$.

A \textit{basic formula 
\index{basic formula!for invariant length groups}} $\varphi $ in the
variables $x_{1},\ldots ,x_{n}$ is an expression of the form%
\begin{equation*}
\ell \left( t(x_{1},\ldots ,x_{n})\right)
\end{equation*}%
where $t(x_{1},\ldots ,x_{n})$ is a term in the variables $x_{1},\ldots
,x_{n}$. The evaluation $\varphi ^{G}(x_{1},\ldots ,x_{n})$ of $\varphi
(x_{1},\ldots ,x_{n})$ in an invariant length group $G$ is the function from 
$G^{n}$ to $\left[ 0,1\right] $ defined by%
\begin{equation*}
\left( g_{1},\ldots ,g_{n}\right) \mapsto \ell _{G}\left( t^{G}\left(
g_{1},\ldots ,g_{n}\right) \right)
\end{equation*}%
where $\ell _{G}$ is the invariant length of $G$. For example%
\begin{equation*}
\ell (xyx^{-1}y^{-1})
\end{equation*}%
is a basic formula%
\index{basic formula!for invariant length groups} whose interpretation in an
invariant length group $G$ is the function assigning to a pair of elements
of $G$ the length of their commutator. This basic formula can be thought as
measuring how much $x$ and $y$ commute. The evaluation at $\left( g,h\right) 
$ of its interpretation in an invariant length group $G$ will be $0$ if and
only if $g$ and $h$ commute.

Finally a \textit{formula%
\index{formula!for invariant length groups}} $\varphi $ is any expression
that can be obtained starting from basic formulae, composing with continuous
functions from $\left[ 0,1\right] ^{n}$ to $\left[ 0,1\right] $, and taking
infima and suprema over some variables. In this framework continuous
functions have the role of \textit{logical connectives}, while infima and
suprema should be regarded as \textit{quantifiers}. With these conventions
the terminology from the usual first order logic is used in this setting. A
formula is \textit{quantifier-free%
\index{formula!quantifier-free}} if it does not contain any quantifier. A
variable $x$ in a formula $\varphi $ is \textit{bound} if it is within the
scope of a quantifier over $x$, and \textit{free} otherwise. a formula
without free variables is called a \textit{sentence%
\index{sentence!for invariant length groups}}. The interpretation%
\index{formula!interpretation} of a formula in a length group $G$ is defined
in the obvious way by recursion on its complexity. For example%
\begin{equation*}
\sup_{x}\sup_{y}\ell (xyx^{-1}y^{-1})
\end{equation*}%
is a sentence%
\index{sentence!for invariant length groups}, with bound variables $x$ and $%
y $. Its evaluation in an invariant length group $G$ is the real number%
\begin{equation*}
\sup_{x\in G}\sup_{y\in G}\ell _{G}(xyx^{-1}y^{-1}).
\end{equation*}%
This sentence%
\index{sentence!for invariant length groups} can be thought as measuring how
much the group $G$ is abelian. Its interpretation in $G$ is zero if and only
if $G$ is abelian. This example enlightens the fact that the possible truth
values of a sentence%
\index{sentence!for invariant length groups} (i.e.\ values of its
evaluations in an invariant length group) are all real numbers between $0$
and $1$. Moreover $0$ can be thought as \textquotedblleft
true\textquotedblright\ while strictly positive real numbers of different
degrees as \textquotedblleft false\textquotedblright\ . In this spirit we
say that a sentence%
\index{sentence!for invariant length groups} $\varphi $ holds in $G$ if and
only if its interpretation in $G\ $is zero. Using this terminology we can
assert for example that an invariant length group $G$ is abelian if and only
if the formula%
\begin{equation*}
\sup_{x}\sup_{y}\ell (xyx^{-1}y^{-1})
\end{equation*}%
holds in $G$. Observe that if $\varphi $ is a sentence%
\index{sentence!for invariant length groups}, then $1-\varphi $ is a
sentence such that $\varphi $ holds in $G$ if and only if the interpretation
of $1-\varphi $ in $G$ is $1$. Thus $1-\varphi $ can be though as a sort of 
\textit{negation}%
\index{negation} of the sentence $\varphi $. Another example of sentence%
\index{sentence!for invariant length groups} is%
\begin{equation*}
\sup_{x}\min \left\{ \left\vert \ell (x)-1\right\vert ,\left\vert \ell
(x)\right\vert \right\} .
\end{equation*}%
Such sentence holds in an invariant length group $G$ if and only if the
invariant length function in $G$ attains values in $\left\{ 0,1\right\} $,
i.e.\ it is the trivial invariant length function on $G$. It is worth noting
at this point that for any sentence $\varphi $ as defined in the logic for
invariant length groups there is a corresponding formula $\varphi _{0}$ in
the usual (discrete) first order logic in the language of groups such that
the evaluation of $\varphi _{0}$ in a discrete group $G$ coincides with the
evaluation of $\varphi $ in $G$ regarded as an invariant metric group
endowed with the trivial invariant length function. For example the (metric)
formula expressing that a group is abelian corresponds to the (discrete)
formula%
\begin{equation*}
\forall x\forall y\,(xy=yx) 
\text{.}
\end{equation*}

Sentences%
\index{sentence!for invariant length groups} in the language of invariant
length groups allow one to determine which properties of an invariant length
group are elementary. A property concerning length groups is \textit{%
elementary%
\index{elementary property}} if there is a set $\Phi $ of sentences such
that an invariant length group has the given property if and only if $G$
satisfies all the sentences in $\Phi $. For example the property of being
abelian is elementary, since an invariant length group is abelian if and
only if it satisfies the sentence $\sup_{x,y}\ell (xyx^{-1}y^{-1})$. Two
invariant length groups $G$ and $G^{\prime }$ are \textit{elementarily
equivalent%
\index{elementary equivalence} }if they have the same elementary properties.
This amounts to say that any sentence has the same evaluation in $G$ and $%
G^{\prime }$. A class $\mathcal{C}$ of invariant length groups will is 
\textit{axiomatizable 
\index{axiomatizable class!of invariant length groups}} if the property of
belonging to $\mathcal{C}$ is elementary. The previous example of sentence
shows that the class of abelian length groups is axiomatizable by a single
sentence. Elementary properties and classes are closely related with the
notion of ultraproduct of invariant length groups.

Suppose that $(G_{n})_{n\in \mathbb{N}}$ is a sequence of invariant length
groups and $\mathcal{U}$ is a free ultrafilter over $\mathbb{N}$. The 
\textit{ultraproduct%
\index{ultraproduct!of invariant length groups}} $\prod\nolimits_{\mathcal{U}%
}G_{n}$ of the sequence $(G_{n})_{n\in \mathbb{N}}$ with respect to the
ultrafilter $\mathcal{U}$ is by definition quotient of the product $%
\prod_{n}G_{n}$ by the normal subgroup%
\begin{equation*}
N_{\mathcal{U}}=\left\{ \left( g_{n}\right) _{n\in 
%TCIMACRO{\U{2115} }%
%BeginExpansion
\mathbb{N}
%EndExpansion
}:\lim_{n\rightarrow \mathcal{U}}\ell _{G_{n}}\left( g_{n}\right) =0\right\}
\end{equation*}%
endowed with the invariant length function%
\begin{equation*}
\ell _{\mathcal{U}}\left( \left( g_{n}\right) N_{\mathcal{U}}\right)
=\lim_{n\rightarrow \mathcal{U}}\ell _{G_{n}}\left( g_{n}\right) 
\text{.}
\end{equation*}%
As before a sequence $\left( g_{n}\right) _{n\in 
%TCIMACRO{\U{2115} }%
%BeginExpansion
\mathbb{N}
%EndExpansion
}$ in $\prod_{n}G_{n}$ is called \textit{representative sequence} 
\index{sequence!representative} for the corresponding element in $%
\prod\nolimits_{\mathcal{U}}G_{n}$. Observe that ultraproducts of the
sequence $\left( S_{n}\right) _{n\in \mathbb{N}}$ of permutation groups
endowed with the Hamming invariant length function as defined in Section \ref%
{Section: definition sofic groups} and ultraproducts of the sequence $\left(
U_{n}\right) _{n\in \mathbb{N}}$ of unitary groups endowed with the
Hilbert-Schmidt invariant length function as defined in Section \ref%
{Section: definition hyperlinear groups} are particular cases of this
definition. When the sequence $(G_{n})_{n\in \mathbb{N}}$ is constantly
equal to a fixed invariant length group $G$ the ultraproduct $%
\prod\nolimits_{\mathcal{U}}G_{n}$ is called \textit{ultrapower%
\index{ultrapower!of invariant length groups}} of $G$ and denoted by $G^{%
\mathcal{U}}$. Observe that \textit{diagonal embedding} of $G$ into $G^{%
\mathcal{U}}$ assigning to $g\in G$ the element of $G^{\mathcal{U}}$ having
the sequence constantly equal to $g$ as representative sequence is a length
preserving group homomorphism. This allows one to identify $G$ with a
subgroup of $G^{\mathcal{U}}$.

The ultraproduct construction behaves well with respect to interpretation of
formulae. This is the content of a theorem proven in the setting of the
usual first order logic by \L o\'s in \cite{Los}. Its generalization to the
logic for metric structures can be found in \cite{BBHU} (Theorem 5.4). We
state here the particular instance of \L o\'s' theorem in the context of
invariant length groups.

\begin{theorem}[\L o\'{s}]
\label{Theorem: Los for invariant length groups}%
\index{\L o\'{s}' theorem!for invariant length groups}Suppose that $\varphi
(x_{1},\ldots ,x_{k})$ is a formula with free variables $x_{1},\ldots ,x_{k}$%
, $(G_{n})_{n\in \mathbb{N}}$ is a sequence of invariant length groups, and $%
\mathcal{U}$ is a free ultrafilter over $%
%TCIMACRO{\U{2115} }%
%BeginExpansion
\mathbb{N}
%EndExpansion
$. If $g^{(1)},\ldots ,g^{(k)}$ are elements of $\prod_{n}G_{n}$ then%
\begin{equation*}
\varphi ^{\prod\nolimits_{\mathcal{U}}G_{n}}\left( g^{(1)},\ldots
,g^{(k)}\right) =\lim_{n\rightarrow \mathcal{U}}\varphi ^{G_{n}}\left(
g_{n}^{(1)},\ldots ,g_{n}^{(k)}\right)
\end{equation*}%
where $g^{(i)}$ is any representative of the sequence $\left(
g_{n}^{(i)}\right) $, for $i=1,2,\ldots ,k$. In particular if $\varphi $ is
a sentence%
\index{sentence!for invariant length groups} then%
\begin{equation*}
\varphi ^{\prod\nolimits_{\mathcal{U}}G_{n}}=\lim_{n\rightarrow \mathcal{U}%
}\varphi ^{G_{n}}%
\text{.}
\end{equation*}
\end{theorem}

Theorem \ref{Theorem: Los for invariant length groups} can be proved by
induction on the complexity of the formula $\varphi $. The particular
instance of Theorem \ref{Theorem: Los for invariant length groups} when the
sequence $(G_{n})_{n\in 
%TCIMACRO{\U{2115} }%
%BeginExpansion
\mathbb{N}
%EndExpansion
}$ is constantly equal to an invariant metric group $G$ shows that $G$ and
any ultrapower $G^{\mathcal{U}}$ of $G$ are elementarily equivalent.
Moreover the diagonal embedding of $G$ into $G^{\mathcal{U}}$ is an \textit{%
elementary embedding}, i.e.\ it preserves the value of formulae possibly
with parameters from $G$.

A particularly useful property of ultraproducts is usually referred to as 
\textit{countable saturation}. Roughly speaking countable saturation of a
group $H$ asserts that whenever one can find elements of $H$ approximately
satisfying any finite subset of a given countably infinite set of
conditions, then in fact one can find an element of $H$ exactly satisfying
simultaneously all the conditions. In order to precisely define this
property, and prove it for ultraproducts, we need to introduce some
model-theoretic terminology, in the particular case of invariant length
groups.

\begin{definition}
\label{Definition: consistent and realized} Suppose that $H$ is a group
endowed with an invariant length function. A countable set of formulae $%
\mathcal{X}$ in the free variables $x_{1},\ldots ,x_{n}$ and possibly with
parameters from $H$ is:

\begin{itemize}
\item \emph{approximately finitely satisfiable}%
\index{set of formulae!approximately finitely satisfiable} in $H$ or \emph{%
consistent} with $H$%
\index{set of formulae!consistent} if for every $\varepsilon >0$ and every
finite collection of formulae $\varphi _{1}\left( x_{1},\ldots ,x_{n}\right)
,\ldots ,\varphi _{m}(x_{1},\ldots ,x_{n})$ from $\mathcal{X}$ there are $%
a_{1},\ldots ,a_{n}\in H$ such that%
\begin{equation*}
\varphi _{i}^{H}( a_{1},\ldots ,a_{n})<\varepsilon
\end{equation*}%
for every $i=1,2,\ldots ,n$;

\item \emph{realized}%
\index{set of formulae!realized} in $H$ if there are $a_{1},\ldots ,a_{n}\in
H$ such that 
\begin{equation*}
\varphi ^{H}( a_{1},\ldots ,a_{n})=0
\end{equation*}%
for every formula $\varphi (x_{1},\ldots ,x_{n})$ in $\mathcal{X}$. In this
case the $n$-tuple $(a)$ is called a realization of $\mathcal{X}$ in $H$.
\end{itemize}
\end{definition}

In model-theoretic jargon, a set of formulae $\mathcal{X}$ as above is
called a \textit{type} over $A$%
\index{type}, where $A$ is the (countable) set of parameters of the formulae
in $\mathcal{X}$. Clearly any set of formulae which is realized in $H$ is
also approximately finitely satisfiable $H$. The converse to this assertion
is in general far from being true. For example suppose that $H$ is the
direct sum $G^{\oplus 
%TCIMACRO{\U{2115} }%
%BeginExpansion
\mathbb{N}
%EndExpansion
}$ of countably many copy of a countable group $G$ with trivial center
endowed with the trivial length function. Consider the set $\mathcal{X}$ of
formulae 
\begin{equation*}
\max \left\{ \left\vert 1-\ell (x)\right\vert ,\ell \left(
xax^{-1}a^{-1}\right) \right\}
\end{equation*}%
in the free variable $x$ where $a$ ranges over the elements of $G^{\oplus 
%TCIMACRO{\U{2115} }%
%BeginExpansion
\mathbb{N}
%EndExpansion
}$. Being the center of $G^{\oplus 
%TCIMACRO{\U{2115} }%
%BeginExpansion
\mathbb{N}
%EndExpansion
}$ trivial, the set of formulae $\mathcal{X}$ is not realized in $G^{\oplus 
%TCIMACRO{\U{2115} }%
%BeginExpansion
\mathbb{N}
%EndExpansion
}$. Nonetheless the fact that one can find in $H$ a nonidentity element
commuting with any given \textit{finite }subset of $G^{\oplus 
%TCIMACRO{\U{2115} }%
%BeginExpansion
\mathbb{N}
%EndExpansion
}$ shows that $\mathcal{X}$ is approximately finitely satisfiable in $H$.

\begin{definition}
\label{Definition: countably saturated} An invariant length group is \emph{%
countably saturated}%
\index{countably saturated!invariant length group} if any countable set of
formulae $\mathcal{X}$ with parameters from $H$ which is approximately
finitely satisfied in $H$ is realized in $H$.
\end{definition}

Thus in a countably saturated invariant length group a countable set of
formulae is approximately finitely satisfiable if and only if it is
realized. Moreover it can be easily shown by recursion on the complexity
that in the evaluation of a formula in a countably saturated structure
infima and suprema can be replaced by minima and maxima respectively. More
generally one can define $\kappa $-saturation for an arbitrary cardinal $%
\kappa $ replacing countable types with types with less than $\kappa $
parameters (it is not difficult to check that countable saturation is the
same as $\aleph _{1}$-saturation). An invariant length group $G$ is \textit{%
saturated}\textbf{\ }if it is $\kappa $-saturated where $\kappa $ is the 
\textit{density character} of $G$, i.e.\ the minimum cardinality of a dense
subset of $G$. The notions of saturation and countable saturation here
introduced are the particular instances in the case of invariant length
groups of the general model-theoretic notions of saturation and countable
saturation. Being countably saturated is one of the fundamental features of
ultraproducts. The proof of this fact in this context can be easily deduced
from \L o\'s' theorem and is therefore left as an exercise.

\begin{exercise}
\label{Exercise: ultraproducts are countably saturated}Suppose that $%
(G_{n})_{n\in 
%TCIMACRO{\U{2115} }%
%BeginExpansion
\mathbb{N}
%EndExpansion
}$ is a sequence of invariant length groups, and $\mathcal{U}$ is a free
ultrafilter over $%
%TCIMACRO{\U{2115} }%
%BeginExpansion
\mathbb{N}
%EndExpansion
$. Show that the ultraproduct $\prod\nolimits_{\mathcal{U}}G_{n}$ is
countably saturated.
\end{exercise}

Saturated structures have been intensively studied in model theory and have
some remarkable properties. In the case of usual first order logic it is a
consequence of the so called Chang-Makkai's theorem (see \cite[Theorem 5.3.6]%
{Chang-Keisler}) that the automorphism group of a saturated structure of
cardinality $\kappa $ has $2^{\kappa }$ elements. The generalization to this
result to the framework of logic for metric structures is an unpublished
result of Ilijas Farah, Bradd Hart, and David Sherman. Proposition is the
particular instance of such result in the case of invariant length groups of
density character $\aleph _{1}$.

\begin{proposition}
\label{Proposition: saturated automorphism}Suppose that $G$ is an invariant
length group. If $G$ is countably saturated and has a dense subset of
cardinality $\aleph _{1}$ then the group of automorphisms of $G$ that
preserve the length has cardinality $2^{\aleph _{1}}$.
\end{proposition}

Proposition \ref{Proposition: external automorphisms} is an immediate
consequence of Proposition \ref{Proposition: saturated automorphism} and
Exercise \ref{Exercise: ultraproducts are countably saturated}. Recall that
the Continuum Hypothesis asserts that the cardinality of the power set of $%
%TCIMACRO{\U{2115} }%
%BeginExpansion
\mathbb{N}
%EndExpansion
$ coincides with the first uncountable cardinal $\aleph _{1}$. A cornerstone
result in set theory asserts that the Continuum Hypothesis is independent
from the usual axioms of set theory, i.e.\ can not be neither proved nor
disproved (see \cite{Godel}, \cite{Cohen1}, \cite{Cohen2}).

\begin{proposition}
\label{Proposition: external automorphisms}Suppose that $(G_{n})_{n\in 
%TCIMACRO{\U{2115} }%
%BeginExpansion
\mathbb{N}
%EndExpansion
}$ is a sequence of separable invariant length groups. If the Continuum
Hypothesis holds then $\prod\nolimits_{\mathcal{U}}G_{n}$ has $2^{\aleph
_{1}}$ length preserving automorphisms for any free ultrafilter $\mathcal{U}$
over $%
%TCIMACRO{\U{2115} }%
%BeginExpansion
\mathbb{N}
%EndExpansion
$. In particular $\prod\nolimits_{\mathcal{U}}G_{n}$ has outer
length-preserving automorphisms.
\end{proposition}

\begin{proof}
Assuming the Continuum Hypothesis $\prod\nolimits_{\mathcal{U}}G_{n}$ has a
dense subset of cardinality $\aleph _{1}$. Moreover by Exercise \ref%
{Exercise: ultraproducts are countably saturated} $\prod\nolimits_{\mathcal{U%
}}G_{n}$ is countably saturated. It follows from Proposition \ref%
{Proposition: saturated automorphism} that $\prod\nolimits_{\mathcal{U}%
}G_{n} $ has $2^{\aleph _{1}}$ length-preserving automorphisms.
\end{proof}

Corollary \ref{Corollary: external automorphism universal} is the particular
instance of Proposition \ref{Proposition: external automorphisms} when the
sequence $(G_{n})_{n\in 
%TCIMACRO{\U{2115} }%
%BeginExpansion
\mathbb{N}
%EndExpansion
}$ is the sequence of permutation groups or the sequence of unitary groups,
one can obtain

\begin{corollary}
\label{Corollary: external automorphism universal}Assume that the Continuum
Hypothesis holds. For any free ultrafilter $\mathcal{U}$ over $%
%TCIMACRO{\U{2115} }%
%BeginExpansion
\mathbb{N}
%EndExpansion
$ the universal sofic group%
\index{group!universal sofic} $\prod\nolimits_{\mathcal{U}}S_{n}$ has $%
2^{\aleph _{1}}$ automorphisms. As a consequence $\prod\nolimits_{\mathcal{U}%
}S_{n}$ has outer automorphisms. The same is true for the universal
hyperlinear%
\index{group!universal hyperlinear} group $\prod\nolimits_{\mathcal{U}}U_{n}$%
.
\end{corollary}

We do not know if there exist models of set theory where some universal
sofic or hyperlinear groups%
\index{group!universal sofic}%
\index{group!universal hyperlinear} have only inner automorphisms. It is
conceivable that this could be proved using ideas and methods from \cite%
{Shelah-differenceIII}. For example in \cite{Lucke-Thomas} Lucke and Thomas
show using a result from \cite{Shelah-differenceIII} that there is a model
of set theory where some ultraproduct of the permutation groups \textit{%
regarded as discrete groups} has only inner automorphisms.

Corollary \ref{Corollary: external automorphism universal} and the
discussion that follows address a question of P\u aunescu from \cite%
{Paunescu}: Theorem 4.1 and Theorem 4.2 of \cite{Paunescu} show that all
automorphisms of the universal sofic groups preserve the length function and
the conjugacy classes. P\u aunescu then asks if it is possible that all
automorphisms of the universal sofic%
\index{group!universal sofic} groups are in fact inner. Corollary \ref%
{Corollary: external automorphism universal} in particular implies that such
assertion does not follow from the usual axioms of set theory.

An application of Theorem 2 from \cite{Dye} (known as Dye's theorem on
automorphisms of unitary groups of factors) allows one to prove the analogue
of Theorem 4.1 and Theorem 4.2 from \cite{Paunescu} in the case of universal
hyperlinear groups. More precisely all automorphisms of the universal
hyperlinear groups preserve the length function, while the normal subgroup
of automorphisms that preserves the conjugacy classes has index $2$ inside
the group of all automorphisms. The not difficult details can be found in 
\cite{Lupini-class}.

Finally let us point out another consequence of countable saturation of
ultraproducts. Suppose that $\Gamma $ is a discrete group of size $\aleph
_{1}$ such that every countable (or equivalently finitely generated)
subgroup of $\Gamma $ is sofic. It is not difficult to infer from Exercise %
\ref{Exercise: sofic equivalence 2} and countable saturation of $%
\prod\nolimits_{\mathcal{U}}S_{n}$ that for any ultrafilter $\mathcal{U}$
over $\mathbb{N}$ there is a length-preserving homomorphism from $\Gamma $
(endowed with the trivial length function) to $\prod\nolimits_{\mathcal{U}%
}S_{n}$. For example this implies that $\prod\nolimits_{\mathcal{U}}S_{n}$
contains a free group on uncountably many generators. Analogue facts hold
for hyperlinear groups and ultraproducts $\prod\nolimits_{\mathcal{U}}U_{n}$.

\section{Model theoretic characterization of sofic and hyperlinear groups 
\label{Secton: model theoretic characterization sofic hyperlinear}}

In this section we shall show that the classes of sofic%
\index{group!sofic} and hyperlinear%
\index{group!hyperlinear} groups are axiomatizable in the logic for
invariant metric groups. Equivalently the properties of being sofic or
hyperlinear are elementary. Recall that the quantifiers in the logic for
invariant metric groups are $\inf $ and $\sup $. More precisely $\sup $ can
be regarded as the \textit{universal quantifier}, analogue to $\forall $ in
usual first order logic, while $\inf$ can be seen as the \textit{existential
quantifier}, which is denoted by $\exists $ in the usual first order logic.
A formula is therefore called \textit{universal%
\index{formula!universal}} if it only contains universal quantifiers, and no
existential quantifiers. The notion of \textit{existential sentence%
\index{formula!existential} }is defined in the same way. A \textit{%
quantifier-free formula%
\index{formula!quantifier-free} }is just a formula without any quantifier.
Say that two formulae $\varphi $ and $\varphi ^{\prime }$ are \textit{%
equivalent }if they have the same interpretation in any invariant length
group. It can be easily proved by induction on the complexity that any
universal sentence is equivalent to a formula of the form%
\begin{equation*}
\sup_{x_{1}}\ldots \sup_{x_{n}}\psi (x_{1},\ldots ,x_{n})
\end{equation*}%
where $\psi (x_{1},\ldots ,x_{n})$ is quantifier-free. An analogous fact
holds for existential sentences. It is easy to infer from this that if $%
\varphi $ is a universal sentence, then $1-\varphi $ is equivalent to an
existential sentence, and vice versa. Exercise \ref{Exercise: convergent
formulae} together with \L o\'s' theorem on ultraproducts shows that
universal and existential formulae have the same values in any ultraproduct
of the symmetric groups regarded as invariant length groups.

\begin{exercise}
\label{Exercise: convergent formulae}If $\varphi $ is a universal sentence
then the sequence $\left( \varphi ^{S_{n}}\right) _{n\in 
%TCIMACRO{\U{2115} }%
%BeginExpansion
\mathbb{N}
%EndExpansion
}$ of its evaluation in the symmetric groups converges. Infer that the same
is true for existential formulae. Deduce that the same holds for existential
formulae.
\end{exercise}

\begin{proof}
Fix $n\in 
%TCIMACRO{\U{2115} }%
%BeginExpansion
\mathbb{N}
%EndExpansion
$ and $\varepsilon >0$. If $N>n$ write%
\begin{equation*}
N=kn+r
\end{equation*}%
for $k\in 
%TCIMACRO{\U{2115} }%
%BeginExpansion
\mathbb{N}
%EndExpansion
$ and $r\in n$. Define the map $\iota :S_{n}\rightarrow S_{N}$ by%
\begin{equation*}
\iota (\sigma )\left( ik+j\right) =%
\begin{cases}
in+\sigma \left( j\right)  & 
\text{if }i\in k\text{ and }j\in n \\ 
in+j & \text{otherwise.}%
\end{cases}%
\end{equation*}%
Observe that $\iota $ is a group homomorphism that preserves the Hamming
length function up to $\frac{1}{k}$. This means that%
\begin{equation*}
\left\vert \ell _{S_{n}}(\sigma )-\ell _{S_{N}}\left( \iota (\sigma )\right)
\right\vert \leq \frac{1}{k}\text{.}
\end{equation*}%
Deduce that if $\varphi $ is a universal sentence\index{sentence!for invariant length groups} and $\varepsilon >0$ then
\begin{equation*}
\left\vert \varphi ^{S_{N}}-\varphi ^{S_{n}}\right\vert <\varepsilon \text{.}
\end{equation*}
for $N\in \mathbb{N}$ large enough.
\end{proof}

Theorem 5.6 of \cite{Lupini} asserts that the same conclusion of Exercise %
\ref{Exercise: convergent formulae} holds for formulae with alternation of
at most two quantifiers. It is currently an open problem if the same holds
for all formulae. A positive answer to this question would imply that the
universal sofic%
\index{group!universal sofic} groups, i.e.\ the length ultraproducts of the
permutation groups, are pairwise isomorphic as invariant length groups if
the Continuum Hypothesis holds. Theorem 1.1 of \cite{Thomas} asserts that if
instead the Continuum Hypothesis fails there are $2^{2^{\aleph _{0}}}$
ultraproducts of the permutation groups that are pairwise nonisomorphic as
discrete groups.

We can now show that the property of being sofic%
\index{group!sofic} is axiomatizable%
\index{axiomatizable class}. Suppose that $\Gamma $ is a sofic group. By
Exercise \ref{Exercise: sofic equivalence 2} there is a length preserving
embedding of $\Gamma $ into $\prod\nolimits_{\mathcal{U}}S_{n}$ for any free
ultrafilter $\mathcal{U}$, where $\Gamma $ is endowed with the trivial
length function. It is easily inferred from this that 
\begin{equation*}
\varphi ^{\Gamma }\leq \varphi ^{\prod\nolimits_{\mathcal{U}%
}S_{n}}=\lim_{n\rightarrow +\infty }\varphi ^{S_{n}}
\end{equation*}%
for any universal sentence $\varphi $ in the language of invariant length
groups. If $\varphi $ is instead an existential sentence then%
\begin{equation*}
\varphi ^{\Gamma }\geq \lim_{n\rightarrow +\infty }\varphi ^{S_{n}}%
\text{.}
\end{equation*}%
In particular if $\varphi $ is an existential sentence that holds in $\Gamma 
$, then $\lim_{n\rightarrow +\infty }\varphi ^{S_{n}}$. Exercise \ref%
{Exercise: sofic elementary} shows that this condition is sufficient for a
group to be sofic.%
\index{group!sofic}

\begin{exercise}
\label{Exercise: sofic elementary}Suppose that $\Gamma $ is a (countable
discrete) group with the property that for any existential sentence $\varphi 
$ that holds in $\Gamma $ the sequence $\left( \varphi ^{S_{n}}\right)
_{n\in 
%TCIMACRO{\U{2115} }%
%BeginExpansion
\mathbb{N}
%EndExpansion
}$ of the evaluation of $\varphi $ in the symmetric groups is vanishing.
Show that $\Gamma $ is sofic.
\end{exercise}

\begin{hint}
Fix $\varepsilon >0$ and a finite subset $F=\left\{ g_{1},\ldots
,g_{n}\right\} $ of $\Gamma $. Write an existential sentence $\varphi $
witnessing the existence of elements $g_{1},\ldots ,g_{n}$ with the
multiplication rules given by $\Gamma $. Infer from the fact that $\varphi $
holds in $\Gamma $ that $\varphi $ approximately holds in $S_{n}$ for $n$
large enough. Use this to construct an $\left( F,\varepsilon \right) $%
-approximate morphism%
\index{approximate morphism!of invariant length groups} $\Phi :\Gamma
\rightarrow S_{n}$.
\end{hint}

It is easy to deduce from Exercise \ref{Exercise: sofic elementary} the
following characterization of sofic%
\index{group!sofic} groups, showing in particular that soficity is an
elementary property.

\begin{proposition}
If $\Gamma $ is a group, the following statements are equivalent

\begin{enumerate}
\item $\Gamma $ is sofic;

\item if $\varphi $ is an existential sentence that holds in $\Gamma $, then 
$\lim_{n\rightarrow +\infty }\varphi ^{S_{n}}=0$;

\item if $\varphi $ is a universal sentence such that $\lim_{n\rightarrow
+\infty }\varphi ^{S_{n}}=0$, then $\varphi ^{\Gamma }=0$.
\end{enumerate}
\end{proposition}

All the results of this section carry over to the case of hyperlinear%
\index{group!hyperlinear} group, when one replaces the permutation groups
with the unitary groups (see \cite{Lupini}). The analogue of Exercise \ref%
{Exercise: convergent formulae} for the unitary groups can be proved in a
similar way, considering for $u\in U_{n}$ and $N=kn+r$ the unitary matrix%
\begin{equation*}
\begin{pmatrix}
u\otimes I_{k} & 0 \\ 
0 & 1_{U_{r}}%
\end{pmatrix}%
\in U_{N}%
\text{.}
\end{equation*}

\section{The Kervaire-Laudenbach conjecture\label{Section:
Kervaire-Laudenbach for hyperlinear}}

Let $\Gamma$ be a countable discrete group and $\gamma _{1},\ldots ,\gamma
_{l}\in \Gamma $. Denote $w(x)$ the monomial%
\begin{equation*}
w(x)=x^{n_{1}}\gamma _{1}\ldots x^{n_{l}}\gamma _{l},
\end{equation*}%
where $n_{i}\in \mathbb{Z}$ for $i=1,2,\ldots ,l$. Consider the following
problem: Determine if the equation%
\begin{equation*}
w(x)=1
\end{equation*}%
has a solution in some group extending $\Gamma $. The answer in general is
\textquotedblleft no". Consider for example the equation%
\begin{equation*}
xax^{-1}b^{-1}=1
\end{equation*}%
If $a$ and $b$ have different orders then clearly this equation has no
solution in any group extending $\Gamma $. Assuming that the sum $%
\sum_{i=1}^{l}n_{i}$ of the exponents of $x$ in $w(x)$ is nonzero is a way
to rule out this obstruction. A conjecture attributed to Kervaire and
Laudenbach%
\index{conjecture!Kervaire-Laudenbach} asserts that this is enough to
guarantee the existence of a solution of the equation $w(x)=1$ in some group
extending $\Gamma $.

We will show in this section that the Kervaire-Laudenbach conjecture holds
for hyperlinear%
\index{group!hyperlinear} groups (see Definition \ref{Definition:
hyperlinear}). This result, first observed by Pestov in \cite{pestov}, will
be a direct consequence of the following theorem by Gerstenhaber and Rothaus
(see \cite{Gerstenhaber-Rothaus}).%
\index{Gerstenhaber-Rothaus theorem}

\begin{theorem}
\label{Theoren: Gerstenhaber, Rothaus} Let $U_{n}$ denote the group of
unitary matrices of rank $n$. If $a_{1},\ldots ,,a_{k} $ are elements of $%
U_{n}$ then the equation%
\begin{equation*}
x^{s_{1}}a_{1}\cdots x^{s_{k}}a_{k} = 1
\end{equation*}%
has a solution in $U_{n}$ as long as $\sum_{i=1}^{k}s_{i}\neq 0$.
\end{theorem}

\begin{proof}
Let $w(x,a_{1},\ldots ,a_{k}):=x^{s_{1}}\cdots x^{s_{k}}a_{k}$ and consider the map $f:U_{n}\rightarrow U_{n}$ defined by
\begin{equation*}
b\mapsto w(b,a_{1},\ldots ,a_{n}). 
\end{equation*}%
We just need to prove that $f$ is onto. Recall that $U_{n}$ is a compact
manifold of dimension $n^{2}$. Thus the homology group $H_{n^{2}}\left(
U_{n}\right) $ is an infinite cyclic group. (A standard reference for
homology theory is \cite[Chapter 23]{Fulton}.) Being continuous (and in fact
smooth) $f$ induces a homomorphism
\begin{equation*}
f_{\ast }:H_{n^{2}}\left( U_{n}\right) \rightarrow H_{n^{2}}\left(
U_{n}\right) \text{.}
\end{equation*}%
If $e$ is a generator of $H_{n^{2}}\left( U_{n}\right) $ then%
\begin{equation*}
f_{\ast }(e)=de
\end{equation*}%
for some $d\in \mathbb{Z}$ called the \textit{degree} of $f$. In order to
show that $f$ is onto, it is enough to show that its degree is nonzero. We
claim that $d=s^{n}$, where $s=\sum_{i=1}^{n}s_{i}$. Being $U_{n}$ connected,
the map $f$ is homotopy equivalent to the map%
\begin{equation*}
f_{s}:U_{n}\rightarrow U_{n}
\end{equation*}%
defined by%
\begin{equation*}
b\mapsto b^{s}\text{.}
\end{equation*}%
By homotopy invariance of the degree of a map, $f$ and $f_{s}$ have the same
degree. Therefore we just have to show that $f_{s}$ has degree $s^{n}$. This
follows from the facts that the generic element of $U_{n}$ has $s^{n}$ $s$%
-roots of unity, and that the degree of a map can be computed locally.
\end{proof}

Theorem \ref{Theoren: Gerstenhaber, Rothaus} is in fact a particular case of 
\cite[Theorem 2]{Gerstenhaber-Rothaus}, where arbitrary compact Lie groups
and system of equations with possibly several variables are considered.

Let us now discuss how one can infer from Theorem \ref{Theoren:
Gerstenhaber, Rothaus} that the Kervaire-Laudenbach conjecture%
\index{conjecture!Kervaire-Laudenbach} holds for hyperlinear%
\index{group!hyperlinear} group. Consider a word%
\begin{equation*}
w(x,y_{1},\ldots ,y_{k})\equiv x^{s_{1}}y_{1}x^{s_{2}}y_{2}\cdots
x^{s_{k}}y_{k},
\end{equation*}%
where $\sum_{i=1}^{k}s_{i}\neq 0$. By Theorem \ref{Theoren: Gerstenhaber,
Rothaus}, formula%
\begin{equation*}
\sup_{y_{1},\ldots ,y_{k}}\inf_{x}\ell \left( w(x,y_{1},\ldots ,y_{k})\right)
\end{equation*}%
evaluates to $0$ in any unitary group $U_{n}$. By \L o\'s' theorem on
ultraproducts, the same formula evaluates to $0$ in any ultraproduct $%
\prod\nolimits_{\mathcal{U}}U_{n}$ of the unitary groups. Thus if $%
a_{1},\ldots ,a_{k}$ are elements of $\prod\nolimits_{\mathcal{U}}U_{k}$%
\begin{equation*}
\inf_{x\in \prod\nolimits_{\mathcal{U}}U_{n}}\ell \left( w(x,a_{1},\ldots
,a_{k})\right) =0%
\text{.}
\end{equation*}%
By countable saturation of $\prod\nolimits_{\mathcal{U}}U_{n}$ one can find $%
b\in \prod\nolimits_{\mathcal{U}}U_{n}$ where the infimum is achieved. This
means that%
\begin{equation*}
\ell \left( w(b,a_{1},\ldots ,a_{k})\right) =0
\end{equation*}%
and hence $b$ is a solution of the equation%
\begin{equation*}
x^{s_{1}}a_{1}\cdots x^{s_{k}}a_{k}\text{.}
\end{equation*}%
This shows that any ultraproduct $\prod\nolimits_{\mathcal{U}}U_{n}$
satisfies the Kervaire-Laudenbach conjecture. Since hyperlinear%
\index{group!hyperlinear} groups are subgroups of $\prod\nolimits_{\mathcal{U%
}}U_{n}$, they will satisfy the Kervaire-Laudenbach conjecture as well.

\section{Other metric approximations and Higman's group\label{Other metric
approximations}}

The notions of sofic%
\index{group!sofic} and hyperlinear%
\index{group!hyperlinear} groups as defined in Section \ref{Section:
definition sofic groups} and Section \ref{Section: definition hyperlinear
groups} respectively admit natural generalizations where one considers
approximate morphisms%
\index{approximate morphism!of invariant length groups} into different
classes of invariant length groups. Suppose that $\mathcal{C}$ is a class of
groups endowed with an invariant length function.

\begin{definition}
\label{Definition: C-approximation}

A countable discrete group $\Gamma $ has the $\mathcal{C}$\textit{%
-approximation property%
\index{approximation property}} if for every finite subset $F$ of $\Gamma
\left\backslash \left\{ 1_{\Gamma }\right\} \right. $ and every $\varepsilon
>0$ there is an $\left( F,\varepsilon \right) $-approximate morphism (as
defined in Definition \ref{Definition: approximate morphism}) from $\Gamma $
endowed with the trivial length function to a group $T$ in $\mathcal{C}$ ,
i.e.\ a function $\Phi :\Gamma \rightarrow T$ such that $\Phi \left(
1_{\Gamma }\right) =1_{T}$ and for every $g,h\in F$:

\begin{itemize}
\item $\ell _{T}\left( \Phi (gh)\Phi (h)^{-1}\Phi (g)^{-1}\right)
<\varepsilon $;

\item $\ell _{T}\left( \Phi (g)\right) >1-\varepsilon $.
\end{itemize}
\end{definition}

The following characterization of groups locally embeddable into some class
of invariant length groups can be proved with the same arguments as Section %
\ref{Section: logic invariant length groups}, where the notions of universal
and existential sentence are introduced.

\begin{proposition}
\label{Proposition: metric approximations} The following statements about a
countable discrete group $\Gamma $ (endowed with the trivial length
function) are equivalent:

\begin{enumerate}
\item $\Gamma $ has the $\mathcal{C}$-approximation property;

\item There is a length-preserving group homomorphism from $\Gamma $ to an
ultraproduct $\prod\nolimits_{\mathcal{U}}G_{n}$ of a sequence $%
(G_{n})_{n\in 
%TCIMACRO{\U{2115} }%
%BeginExpansion
\mathbb{N}
%EndExpansion
}$ of invariant metric groups from the class $\mathcal{C}$;

\item For every existential sentence $\varphi $ in the language of invariant
length groups 
\begin{equation*}
\varphi ^{\Gamma }\geq \inf \left\{ \varphi ^{G}:G\in \mathcal{C}\right\}
\end{equation*}%
where $\varphi ^{G}$ denotes the evaluation of $\varphi $ in the invariant
length group $G$;

\item For every universal sentence $\varphi $ in the language of invariant
length groups 
\begin{equation*}
\varphi ^{\Gamma }\leq \sup \left\{ \varphi ^{G}:G\in \mathcal{C}\right\} 
\text{.}
\end{equation*}
\end{enumerate}
\end{proposition}

This characterization in particular shows that the $\mathcal{C}$%
-approximation property is elementary in the sense of Section \ref{Section:
logic invariant length groups}

Observe that when $\mathcal{C}$ is a class of groups endowed with the
trivial length function, the $\mathcal{C}$-approximation property coincides
with the notion of local embeddability into elements of $\mathcal{C}$ as
defined in Section \ref{Section: classes of sofic and hyperlinear groups}.
It is immediate from the definition that a group is sofic%
\index{group!sofic} if and only if it has the $\mathcal{C}$-approximation
property where $\mathcal{C}$ is the class of symmetric groups endowed with
the Hamming invariant length function. Analogously hyperlinearity%
\index{group!hyperlinear} can be seen as the $\mathcal{C}$-approximation
property where $\mathcal{C}$ is the class of unitary groups endowed with the
Hilbert-Schmidt invariant length functions. \textit{Weakly sofic groups%
\index{group!weakly sofic}} as defined in \cite{Glebski-Rivera} are exactly $%
\mathcal{C}$-approximable groups where $\mathcal{C}$ is the class of all
finite groups endowed with an invariant length function. More recently
Arzhantseva and P\u{a}unescu introduced in \cite{Arzhantseva-Paunescu} the
class of\textit{\ linear sofic groups%
\index{group!linear sofic}}, which can be regarded as $\mathcal{C}$%
-approximable groups where $\mathcal{C}$ is the class of general linear
groups endowed with the invariant length function%
\begin{equation*}
\ell (x)=N(x-1)
\end{equation*}%
where $N$ is the usual normalized rank of matrices.

Giving a complete account of the notion of local metric approximation in
group theory (not to mention other areas of mathematics) would be too long
and beyond the scope of this survey. It should be nonetheless mentioned that
it can be found \textit{in nuce} in the work of Malcev. More recently it has
been considered by Gromov in the paper \cite{Gr} that lead to the
introduction of sofic groups. The notion of $\mathcal{C}$-approximation is
defined implicitly by Pestov in \cite[Remark 8.6]{pestov} and explicitly by
Thom in \cite[Definition 1.6]{Thom-Higman}, as well as by Arzhantseva and
Cherix in \cite{Ar-Ch}.

A more general notion of metric approximation called \textit{asymptotic
approximation} has been more recently defined by Arzhantseva for \textit{%
finitely generated groups} in \cite[Definition 9]{Ar}: A finitely generated
group $\Gamma $ is asymptotically approximated by a class of invariant
length groups $\mathcal{C}$ if there is a sequence $\left( X_{n}\right)
_{n\in \mathbb{N}}$ of finite generating subsets of $\Gamma $ such that for
every $n\in \mathbb{N}$ there is a $\left( B_{X_{n}}(n),%
\frac{1}{n}\right) $-approximate morphism as in Definition \ref{Definition:
approximate morphism} from $\Gamma $ endowed with the trivial invariant
length function to an element of $\mathcal{C}$, where $B_{X_{n}}(n)$ denotes
the set of elements of $\Gamma $ that have length at most $n$ according to
the word length associated with the generating set $X_{n}$. In the
particular case when the finite generating set $X_{n}$ does not depend on $n$
one obtains the notion of $\mathcal{C}$-approximation as above. It is showed
in \cite[Theorem 11]{Ar} that several important classes of finitely
generated groups, including all hyperbolic groups and one-relator groups,
are asymptotically residually finite, i.e.\ have the asymptotic $\mathcal{C}$%
-approximation property where $\mathcal{C}$ is the class of residually
finite groups endowed with the trivial length function. In particular these
groups are asymptotically sofic. (It is currently not known if these groups
are in fact sofic.)

To this day no countable discrete group is known to \textit{not }have the $%
\mathcal{C}$-approximation property as in Definition \ref{Definition:
C-approximation} when $\mathcal{C}$ is any of the classes of invariant
length groups mentioned above. In an attempt to find an example of a
countable discrete group failing to have the $\mathcal{C}$-approximation
property with respect to some natural large class of invariant length
groups, Thom introduced in \cite{Thom-Higman} the class $\mathcal{F}_{c}$ of
finite commutator-contractive invariant length groups, i.e.\ finite groups
endowed with an invariant length function $\ell $ satisfying%
\begin{equation*}
\ell (xyx^{-1}y^{-1})\leq 4\ell (x)\ell (y)\text{%
\index{length function!commutator contractive}.}
\end{equation*}%
Examples of such groups, besides finite groups endowed with the trivial
length function, are finite subgroups of the unitary group of a C*-algebra $%
A $ endowed with the length function 
\begin{equation*}
\ell (x)=%
\frac{1}{2}\left\Vert 1-x\right\Vert \text{.}
\end{equation*}%
Corollary 3.3 of \cite{Thom-Higman} shows that groups with the $\mathcal{F}%
_{c}$-approximation property form a proper subclass of the class of all
countable discrete groups. More precisely Higman's group $H$ introduced by
Higman in \cite{Higman} does not have the $\mathcal{F}_{c}$-approximation
property. This is one of the few currently know examples of a group failing
to have the $\mathcal{C}$-approximation property for some broad class $%
\mathcal{C}$ of groups endowed with a (nontrivial) invariant length function.

Higman's group%
\index{group!Higman's} $H$ is the group with generators $h_{i}$ for $i\in 
%TCIMACRO{\U{2124} }%
%BeginExpansion
\mathbb{Z}
%EndExpansion
\left/ 4%
%TCIMACRO{\U{2124} }%
%BeginExpansion
\mathbb{Z}
%EndExpansion
\right. $ subject to the cyclic relations%
\begin{equation*}
h_{i+1}h_{i}h_{i+1}^{-1}=h_{i}^{2},
\end{equation*}%
for $i\in 
%TCIMACRO{\U{2124} }%
%BeginExpansion
\mathbb{Z}
%EndExpansion
\left/ 4%
%TCIMACRO{\U{2124} }%
%BeginExpansion
\mathbb{Z}
%EndExpansion
\right. $ where the sum $i+1$ is calculated modulo $4$. It was first
considered by Higman in \cite{Higman}, who showed that $H$ is an infinite
group with no nontrivial finite quotients, thus providing the first example
of a finitely presented group with this property. Higman's proof in fact
also shows that $H$ is not locally embeddable into finite groups. Theorem
3.2 and Corollary 3.3 from \cite{Thom-Higman} strengthen Higman's result,
showing that $H$ does not even have the $\mathcal{F}_{c}$-approximation
property.

We will now prove, following Higman's original argument from \cite{Higman},
that Higman's group $H$ is nontrivial and in fact infinite. Consider for $%
i\in 
%TCIMACRO{\U{2124} }%
%BeginExpansion
\mathbb{Z}
%EndExpansion
\left/ 4%
%TCIMACRO{\U{2124} }%
%BeginExpansion
\mathbb{Z}
%EndExpansion
\right. $ the group $H_{i}$ generated by $h_{i}$ and $h_{i+1}$ subjected to
the relation 
\begin{equation*}
h_{i+1}h_{i}h_{i+1}^{-1}=h_{i}^{2},
\end{equation*}%
where again $i+1$ is calculated modulo $4$. It is not hard to see that every
element of $H_{i}$ can be expressed uniquely as $h_{i}^{n}h_{i+1}^{m}$ for $%
n,m\in 
%TCIMACRO{\U{2124} }%
%BeginExpansion
\mathbb{Z}
%EndExpansion
$. In particular $h_{i}$ and $h_{i+1}$ respectively generate disjoint cyclic
free groups in $H_{i}$. Define $H_{01}$ to be the free product of $H_{0}$
and $H_{1}$ amalgamated over the common free cyclic subgroup generated by $%
h_{1}$, and $H_{23}$ to be the free product of $H_{2}$ and $H_{3}$
amalgamated over the common free cyclic subgroup generated by $h_{3}$. By
Corollary 8.11 from \cite{Neumann} $\left\{ h_{0},h_{2}\right\} $ generates
a free subgroup of both $H_{01}$ and $H_{23}$. Define $H$ to be the free
product of $H_{01}$ and $H_{23}$ amalgamated over the subgroup generated by $%
\left\{ h_{0},h_{2}\right\} $. Again by Corollary 8.11 from \cite{Neumann} $%
\left\{ h_{1},h_{3}\right\} $ generates a free subgroup of $H$. In
particular $H$ is an infinite group.

As mentioned before, Higman showed that the group $H$ is not locally
embeddable into finite groups. Equivalently the system $\mathcal{R}%
_{H}(x_{0},x_{1},x_{2},x_{3})$ consisting of the relations%
\begin{equation*}
x_{i+1}x_{i}x_{i+1}^{-1}=x_{i}^{2},
\end{equation*}%
for $i\in 
%TCIMACRO{\U{2124} }%
%BeginExpansion
\mathbb{Z}
%EndExpansion
\left/ 4%
%TCIMACRO{\U{2124} }%
%BeginExpansion
\mathbb{Z}
%EndExpansion
\right. $ has no nontrivial finite models. This means that if $F$ is any
finite group and $a_{i}\in F$ for $i\in 
%TCIMACRO{\U{2124} }%
%BeginExpansion
\mathbb{Z}
%EndExpansion
\left/ 4%
%TCIMACRO{\U{2124} }%
%BeginExpansion
\mathbb{Z}
%EndExpansion
\right. $ satisfy the system $\mathcal{R}_{H}$, i.e.\ satisfy%
\begin{equation*}
a_{i+1}a_{i}a_{i+1}^{-1}=a_{i}^{2},
\end{equation*}%
for $i\in 
%TCIMACRO{\U{2124} }%
%BeginExpansion
\mathbb{Z}
%EndExpansion
\left/ 4%
%TCIMACRO{\U{2124} }%
%BeginExpansion
\mathbb{Z}
%EndExpansion
\right. $, then%
\begin{equation*}
a_{0}=a_{1}=a_{2}=a_{3}=1_{F}%
\text{.}
\end{equation*}%
In fact suppose by contradiction that one of the $a_{i}$'s is nontrivial.
Define $p$ to be the smallest prime number dividing the order of one of the $%
a_{i}$'s. Without loss of generality we can assume that $p$ divides the
order of $a_{0}$. From the equation 
\begin{equation*}
a_{1}a_{0}a_{1}^{-1}=a_{0}^{2},
\end{equation*}%
one can obtain by induction%
\begin{equation*}
a_{1}^{n}a_{0}a_{1}^{-n}=a_{0}^{2^{n}}.
\end{equation*}%
In particular if $n$ is the order of $a_{1}$ we have%
\begin{equation*}
a_{0}^{2^{n}-1}=1_{F},
\end{equation*}%
which shows that $2^{n}-1$ is a multiple of the order of $a_{0}$ and hence
of $p$. In particular $p$ is an odd prime and%
\begin{equation*}
2^{n}\equiv 1\left( \mathrm{mod}p\right) \text{.}
\end{equation*}%
Thus by Fermat's little theorem $p-1$ divides $n$, contradicting the
assumption that $p$ is the smallest prime dividing the order of one of the $%
a_{i}$'s.

A similar argument from \cite{Tao} shows that the system $\mathcal{R}_{H}$
has no nontrivial solution in $GL_{n}(\mathbb{C})$ for any $n\in 
%TCIMACRO{\U{2115} }%
%BeginExpansion
\mathbb{N}
%EndExpansion
$. In fact suppose that $a_{i}\in GL_{n}(\mathbb{C})$ for $i\in 
%TCIMACRO{\U{2124} }%
%BeginExpansion
\mathbb{Z}
%EndExpansion
\left/ 4%
%TCIMACRO{\U{2124} }%
%BeginExpansion
\mathbb{Z}
%EndExpansion
\right. $ satisfy the system $\mathcal{R}_{H}$. The relation%
\begin{equation*}
a_{i+1}a_{i}a_{i+1}^{-1}=a_{i}^{2}
\end{equation*}%
implies that $a_{i}$ and $a_{i}^{2}$ are conjugate and, in particular, have
the same eigenvalues. This implies that the eigenvalues of $a_{i}$ are roots
of unity. Considering the Jordan canonical form of $a_{i}$ shows that the
absolute value of the entries of $a_{i}^{n}$ grows at most polynomially in $%
n $. Now the relation%
\begin{equation*}
a_{i+1}^{n}a_{i}a_{i+1}^{-n}=a_{i}^{2^{n}}
\end{equation*}%
shows that also the absolute value of the entries of $a_{i}^{2^{n}}$ grows
at most polynomially in $n$. It follows that $a_{i}$ is diagonalizable and,
having roots of unity as eigenvalues, a finite order element of $GL_{n}(%
\mathbb{C})$. The same argument as before now shows that the $a_{i}$'s are
equal to the unity of $GL_{n}(\mathbb{C})$.

Andreas Thom showed in \cite[Corollary 3.3]{Thom-Higman} that $H$ does not
have the $\mathcal{F}_{c}$-approximation property where $\mathcal{F}_{c}$ is
the class of finite groups endowed with a commutator-contractive invariant
length function. This is a strengthening of Higman's result that $H$ is not
locally embeddable into finite groups, since the trivial invariant length
function on a group is commutator-contractive. The following proposition is
the main techinical result involved in the proof; see \cite[Theorem 3.2]%
{Thom-Higman}.

\begin{proposition}
\label{Proposition: Thom Higman} Suppose that $G$ is a finite commutator
contractive invariant length group and $\varepsilon $ is a positive real
number smaller than $\frac{1}{176}$. If $a_{i}$ for $i\in 
%TCIMACRO{\U{2124} }%
%BeginExpansion
\mathbb{Z}
%EndExpansion
\left/ 4%
%TCIMACRO{\U{2124} }%
%BeginExpansion
\mathbb{Z}
%EndExpansion
\right. $ are elements of $G$ satisfying the system $\mathcal{R}_{H}$ up to $%
\varepsilon ,$ then%
\begin{equation*}
\ell (a_{i})<4\varepsilon
\end{equation*}%
for every $i\in 
%TCIMACRO{\U{2124} }%
%BeginExpansion
\mathbb{Z}
%EndExpansion
\left/ 4%
%TCIMACRO{\U{2124} }%
%BeginExpansion
\mathbb{Z}
%EndExpansion
\right. $.
\end{proposition}

In view of Proposition \ref{Proposition: metric approximations} in order to
conclude that Higman's group $H$ does not have the $\mathcal{F}_{c}$%
-approximation property, it is enough to show that there is a universal
sentence $\varphi $ such that $\varphi ^{G}=0$ for every invariant length
group $G$ but $\varphi ^{\Gamma }=1$. Write%
\begin{equation*}
\psi (x_{0},x_{1},x_{2},x_{3})\equiv \max_{i\in \mathbb{Z}\left/ 4\mathbb{Z}%
\right. }\ell (x_{i}) -4\max_{i\in \mathbb{Z}\left/ 4\mathbb{Z}\right. }\ell
(x_{i+1}x_{i}x_{i+1}^{-1}x_{i}^{-2})
\end{equation*}%
where $i+1$ is calculated modulo $4$, and define $\varphi $ to be the
universal sentence%
\begin{equation*}
\sup_{x_0,x_1,x_2,x_3}f(\psi )
\end{equation*}%
where $f$ is the function%
\begin{equation*}
x\mapsto \min \left\{ \max \left\{ x,0\right\} ,1\right\} \text{.}
\end{equation*}%
Observe that by Proposition \ref{Proposition: Thom Higman} the
interpretation of $\varphi $ in any commutator-contractive invariant length
group is $0$, while the interpretation of $\varphi $ in Higman's group
endowed with the trivial length function is $1$.

\section{Rank rings and Kaplansky's direct finiteness conjecture}

\label{Section: rank rings and finiteness conjecture}

Recall that Kaplansky's direct finiteness conjecture for a countable
discrete group $\Gamma $ asserts that if $K$ is a field, then the group
algebra $K\Gamma $ is directly finite. This means that if $a,b$ are elements
of $K\Gamma $ such that $ab=1$ then also $ba=1$. This conjecture has been
confirmed when $\Gamma $ is a sofic%
\index{group!sofic} group by Elek and Szab\'{o} in \cite%
{Elek-Szabo-hyperlinearity}. An alternative proof of this result has been
obtained in \cite[Corollary 1.4]{ceccherini-silberstein_injective_2007}
using the theory of cellular automata; see also \cite[Section 8.15]%
{Ceccherini-Coornaert}. The proof of this result can be naturally
presented within the framework of rank rings.

\begin{definition}
\label{Definiton: rank function}

Suppose that $R$ is a ring. A function $N:R\rightarrow \left[ 0,1\right] $
is a \emph{rank function}%
\index{rank function} if:

\begin{itemize}
\item $N(1)=1$;

\item $N(x)=0$ iff $x=0$;

\item $N(xy)\leq \min \left\{ N(x),N(y)\right\} $;

\item $N(x+y)\leq N(x)+N(y)$.
\end{itemize}
\end{definition}

If $N$ is a rank function on $R$ then%
\begin{equation*}
d(x,y)=N(x-y)
\end{equation*}%
defines a metric that makes the function $x\mapsto x+a$ isometric and the
functions $x\mapsto xa$ and $x\mapsto ax$ contractive for every $a\in R$. A
ring endowed with a rank function is called a \textit{rank ring%
\index{rank ring}}.

In the context of rank rings, a \textit{term} 
\index{term!for rank rings} in the variables $x_{1},\ldots ,x_{n}$ is just a
polynomial in the indeterminates $x_{1},\ldots ,x_{n}$. A \emph{basic formula%
} 
\index{basic formula!for rank rings} is an expression of the form%
\begin{equation*}
N\left( p(x_{1},\ldots ,x_{n})\right)
\end{equation*}%
where $p(x_{1},\ldots ,x_{n})$ is a term in the variables $x_{1},\ldots
,x_{n}$. . \textit{formulae%
\index{formula!for rank rings}}, \textit{sentences 
\index{sentence!for rank rings}} and their \textit{interpretation 
\index{formula!interpretation}} in a rank ring can then be defined starting
from terms analogously as in the case of invariant length groups. Also the
notion of elementary property, and axiomatizable class, saturated and
countably saturated structure carry over without change.%
\index{axiomatizable class!of rank rings}%
\index{elementary property!of rank rings}%
\index{countably saturated!rank ring}

Suppose that $\left( R_{n}\right) _{n\in 
%TCIMACRO{\U{2115} }%
%BeginExpansion
\mathbb{N}
%EndExpansion
}$ is a sequence of rank rings and $\mathcal{U}$ is a free ultrafilter over $%
\mathbb{N}$. The ultraproduct $\prod\nolimits_{\mathcal{U}}R_{n}$%
\index{ultraproduct!of rank rings} is the quotient of the product ring $%
\prod_{n}R_{n}$ by the ideal%
\begin{equation*}
I_{\mathcal{U}}=\left\{ (x_{n})\in \prod\nolimits_{n}R_{n}\left\vert
\lim_{n\rightarrow \mathcal{U}}N_{n}(x_{n})=0\right. \right\} 
\text{.}
\end{equation*}%
The function%
\begin{equation*}
N_{\mathcal{U}}(x_{n})=\lim_{n\rightarrow \mathcal{U}}N_{n}(x_{n})
\end{equation*}%
induces a rank function in the quotient, making $\prod\nolimits_{\mathcal{U}%
}R_{n}$ a rank ring. Then all the $R_{n}$ coincide with the same rank ring $%
R $ the corresponding ultraproduct will be called \emph{ultrapower} of $R$%
\index{ultrapower!of rank rings}. The notion of \emph{representative sequence%
}%
\index{sequence!representative} of an element of an ultraproduct of rank
rings is defined analogously as in the case of length groups. \L o\'s'
theorem%
\index{\L o\'s' theorem!for rank rings} and countable saturation of
ultraproducts can be proved in this context in a way analogous to the case
of invariant length groups. In particular:

\begin{theorem}[\L o\'{s}]
Suppose that $\varphi (x_{1},\ldots ,x_{k})$ is a formula for rank rings
with free variables $x_{1},\ldots ,x_{k}$, $\left( R_{n}\right) _{n\in 
\mathbb{N}}$ is a sequence of rank rings, and $\mathcal{U}$ is a free
ultrafilter over $%
%TCIMACRO{\U{2115} }%
%BeginExpansion
\mathbb{N}
%EndExpansion
$. If $a^{(1)},\ldots ,a^{(k)}$ are elements of $\prod_{n}R_{n}$ then%
\begin{equation*}
\varphi ^{\prod\nolimits_{\mathcal{U}}R_{n}}\left( a^{(1)},\ldots
,a^{(k)}\right) =\lim_{n\rightarrow \mathcal{U}}\varphi ^{R_{n}}\left(
a_{n}^{(1)},\ldots ,a_{n}^{(k)}\right)
\end{equation*}%
where $a^{(i)}$ is any representative sequence of $\left( a_{n}^{(i)}\right)
_{n\in \mathbb{N}}$ for $i=1,2,\ldots ,k$. In particular if $\varphi $ is a
sentence%
\index{sentence!for rank rings} then%
\begin{equation*}
\varphi ^{\prod\nolimits_{\mathcal{U}}R_{n}}=\lim_{n\rightarrow \mathcal{U}%
}\varphi ^{R_{n}}%
\text{.}
\end{equation*}
\end{theorem}

A rank ring $R$ such that for every $x,y\in R$%
\begin{equation*}
N(xy-1)=N\left( yx-1\right)
\end{equation*}%
is called a \textit{finite rank ring} 
\index{rank ring!finite}. Clearly any finite rank ring is a directly finite
ring. Moreover the class of finite rank rings is axiomatizable%
\index{axiomatizable class} by the formula%
\begin{equation*}
\sup_{x}\sup_{y}\left\vert N(xy-1)-N\left( yx-1\right) \right\vert 
\text{.}
\end{equation*}%
It follows that an ultraproduct of finite rank rings is a finite rank ring
and in particular a directly finite ring. Exercise \ref{Exercise: rank and
length} shows that there is a tight connection between rank functions on
rings and length functions on groups.

\begin{exercise}
\label{Exercise: rank and length}Suppose that $N$ is a rank function on a
ring $R$. Define%
\begin{equation*}
\ell (x)=N(x-1),
\end{equation*}%
for $x\in R$. Show that $\ell $ is a length function on the multiplicative
group $R^{\times }$ of invertible elements of $R$, which is invariant if and
only if $R$ is a finite rank ring.
\end{exercise}

A natural example of finite rank rings is given by rings of matrices over an
arbitrary field $K$. Denote by $M_{n}(K)$ the ring of $n\times n$ matrices
with coefficients in $K$ Suppose that $\rho $ is the usual matrix rank on $%
M_{n}(K)$, i.e.\ $\rho (x)$ for $x\in M_{n}(K)$ is the dimension of the
range of $x$ regarded as a linear operator on a $K$-vector space of
dimension $n$. Define%
\begin{equation*}
N(x)=\frac{1}{n}\rho (x)
\end{equation*}%
for every $x\in M_{n}(K)$.

\begin{exercise}
\label{Exercise: matrix rings}Prove that $N$ is a rank function on $M_{n}(K)$
as in Definition \ref{Definiton: rank function}. Show that moreover $%
M_{n}(K) $ endowed with the rank $N$ is a finite rank ring.
\end{exercise}

Another natural example of finite rank rings comes from von Neumann algebra
theory: If $\tau $ is a faithful normalized trace on a von Neumann algebra $%
M $, then%
\begin{equation*}
N_{\tau }(x)=\tau \left( s(x)\right),
\end{equation*}%
where $s(x)$ is the support projection of $x$, is a rank function on $M$.
Moreover $M$ endowed with the rank function $N_{\tau }$ is a finite rank
ring.

Fix a free ultrafilter $\mathcal{U}$ over $%
%TCIMACRO{\U{2115} }%
%BeginExpansion
\mathbb{N}
%EndExpansion
$. In this section the symbol $\prod\nolimits_{\mathcal{U}}M_{n}(K)$ will
denote the ultraproduct with respect to $\mathcal{U}$ of the sequence of
matrix rings $M_{n}(K)$ regarded as rank rings. By Exercise \ref{Exercise:
matrix rings} and \L o\'{s}' theorem on ultraproducts $\prod\nolimits_{%
\mathcal{U}}M_{n}(K)$ is a finite rank ring (and in particular a directly
finite ring).

The rest of this section is dedicated to the proof from \cite%
{Elek-Szabo-finiteness} that sofic groups satisfy Kaplansky's direct
finiteness%
\index{conjecture!direct finiteness} conjecture. The idea is that soficity
of $\Gamma $ together with the algebraic embedding of $S_{n}$ into $M_{n}(K)$
obtained by sending $\sigma $ to the associated permutation matrix $%
P_{\sigma }$ allows one to construct an injective *-homomorphism from the
group algebra $K\Gamma $ to the ultraproduct $\prod\nolimits_{\mathcal{U}%
}M_{n}(K)$. In order to do this one needs some relations between the rank of
a linear combination of permutation matrices and the lengths of the
associated permutations. Explicit upper and lower bounds of the former in
term of the latter ones are established in Exercise \ref{Exercise: rank
distance upper bound} and Exercise \ref{Exercise: rank lower bound}
respectively.

In the following for $\sigma \in S_{n}$ denote by $P_{\sigma }\in M_{n}(K)$
the associated permutation matrix as in Section \ref{Section: definition
hyperlinear groups}.

\begin{exercise}
\label{Exercise: rank distance upper bound}Show that%
\begin{equation*}
N\left( P_{\sigma }-I\right) \leq \ell _{S_{n}}(\sigma )
\end{equation*}%
where $\ell _{S_{n}}$ is the Hamming invariant length function.
\end{exercise}

\begin{hint}
Define $c\left( \sigma \right) $ the number of cycles of $\sigma $
(including fixed points). Show that%
\begin{equation*}
N\left( P_{\sigma }-I\right) =1-%
\frac{c\left( \sigma \right) }{n}
\end{equation*}%
by induction on the number of cycles.
\end{hint}

\begin{exercise}
\label{Exercise: rank lower bound}Suppose that $\sigma _{1},\ldots ,\sigma
_{k}\in S_{n}$ and $\lambda _{1},\ldots ,\lambda _{k}\in K\left\backslash
\left\{ 0\right\} \right. $. Define%
\begin{equation*}
\varepsilon =\min_{1\leq i\leq k}\left( 1-\ell (\sigma _{i})\right) \text{.}
\end{equation*}%
Prove that%
\begin{equation*}
N\left( \sum_{i=1}^{k}\lambda _{i}P_{\sigma _{i}}\right) \geq \frac{%
1-\varepsilon k}{k^{2}}\text{.}
\end{equation*}
\end{exercise}

\begin{hint}
Denote by $\left\{ e_{i}\left\vert \,i\in n\right. \right\} $ the canonical
basis of $K^{n}$. Recall that $S_{n}$ is assumed to act on the set $%
n=\left\{ 0,1,\ldots ,n-1\right\} $. Define $D$ to be a maximal subset of $n$
such that for $s,t\in X$ and $1\leq i,j\leq k$ such that either $s\neq t$ or 
$i\neq j$ one has%
\begin{equation*}
\sigma _{i}\left( s\right) \neq \sigma _{j}\left( t\right) \text{.}
\end{equation*}%
Observe that if $x\in span\left\{ e_{i}\left\vert \,i\in D\right. \right\} $
then%
\begin{equation*}
\sum_{i=1}^{k}\lambda _{i}P_{\sigma _{i}}(x)\neq 0\text{.}
\end{equation*}%
Infer that%
\begin{equation*}
N\left( \sum_{i=1}^{k}\lambda _{i}P_{\sigma _{i}}\right) \geq \frac{%
\left\vert D\right\vert }{n}\text{.}
\end{equation*}%
By maximality of $X$ for every $s\in n$ there are $t\in D$ and $1\leq
i,j\leq k$ such that%
\begin{equation*}
s=\sigma _{i}^{-1}\sigma _{j}\left( t\right)
\end{equation*}%
When $i,,j$ vary between $1$ and $k$ and $t$ varies in $D$ the expression%
\begin{equation*}
\sigma _{i}^{-1}\sigma _{j}\left( t\right)
\end{equation*}%
attains at most $\varepsilon nk+k^{2}\left\vert D\right\vert $ values. Infer
that%
\begin{equation*}
\frac{\left\vert D\right\vert }{n}\geq \frac{1-\varepsilon k}{k^{2}}.
\end{equation*}
\end{hint}

Exercise \ref{Exercise: rank distance upper bound} and Exercise \ref%
{Exercise: rank lower bound} allow one to define a nontrivial morphism from
the group algebra $K\left( \prod\nolimits_{\mathcal{U}}S_{n}\right) $ to $%
\prod\nolimits_{\mathcal{U}}M_{n}(K)$.\ This is the content of Exercise \ref%
{Exercise: embedding matrices}.

\begin{exercise}
\label{Exercise: embedding matrices}Define, using Exercise \ref{Exercise:
rank distance upper bound}, a ring morphism $\Psi :K\left( \prod\nolimits_{%
\mathcal{U}}S_{n}\right) \rightarrow \prod\nolimits_{\mathcal{U}}M_{n}(K)$.
Prove using Exercise \ref{Exercise: rank lower bound} that if $x_{1},\ldots
,x_{k}\in \prod\nolimits_{\mathcal{U}}S_{n}$ are such that $\ell _{\mathcal{U%
}}(x_{i})=1$ for $i=1,2,\ldots ,k$ and $\lambda _{1},\ldots ,\lambda _{k}\in
K\left\backslash \left\{ 0\right\} \right. $ then%
\begin{equation*}
N\left( \lambda _{1}x_{1}+\cdots +\lambda _{k}x_{k}\right) \geq \frac{1}{k}%
\text{.}
\end{equation*}
\end{exercise}

One can now easily prove that a sofic group%
\index{group!sofic} $\Gamma $ satisfies Kaplansky's finiteness conjecture.
In fact if $\Gamma $ is sofic then $\Gamma $ embeds into $\prod\nolimits_{%
\mathcal{U}}S_{n}$ in such a way that the length of any element in the range
of $\Gamma \left\backslash \left\{ 1_{\Gamma }\right\} \right. $ is $1$.
This induces a ring morphism from $K\Gamma $ into $K\left( \prod\nolimits_{%
\mathcal{U}}S_{n}\right) $. By composing it with the ring morphism described
in Exercise \ref{Exercise: embedding matrices}, we obtain a ring morphism
from $K\Gamma $ into $\prod\nolimits_{\mathcal{U}}M_{n}(K)$ that is one to
one by the second statement in Exercise \ref{Exercise: embedding matrices}.
This shows that $K\Gamma $ is isomorphic to a subring of the directly finite
ring $\prod\nolimits_{\mathcal{U}}M_{n}(K)$. In particular $K\Gamma $ is
itself directly finite.

\section{Logic for tracial von Neumann algebras}

\label{Section: logic von neumann algebras}

In the context of tracial von\ Neumann algebras a \textit{term%
\index{term!for tracial von Neumann algebras} }$p(x_{1},\ldots ,x_{n})$ in
the variables $x_{1},\ldots ,x_{n}$ is a \textit{noncommutative *-polynomial}
in $x_{1},\ldots ,x_{n}$, i.e.\ a polynomial in the noncommuting variables $%
x_{1},\ldots ,x_{n}$ and $x_{1}^{\ast },\ldots ,x_{n}^{\ast }$. A basic
formula%
\index{basic formula!for tracial von Neumann algebras} is an expression of
the form%
\begin{equation*}
\tau \left( p(x_{1},\ldots ,x_{n})\right),
\end{equation*}%
where $p(x_{1},\ldots ,x_{n})$ is a noncommutative *-polynomial. General
formulae%
\index{formula!for tracial von Neumann algebras} can be obtained from basic
formulae composing with continuous functions or taking infima and suprema
over norm bounded subsets of the von Neumann algebra or of the field of
scalars. More formally if $\varphi _{1},\ldots ,\varphi _{m}$ are formulae
and $f:\mathbb{C}^{m}\rightarrow \mathbb{C}$ is a continuous function then%
\begin{equation*}
f\left( \varphi _{1},\ldots ,\varphi _{m}\right)
\end{equation*}%
is a formula. Analogously if $\varphi (x_{1},\ldots ,x_{n},y)$ is a formula
then%
\begin{equation*}
\inf_{\left\Vert y\right\Vert \leq 1}\mathrm{Re}\left( \varphi \left(
x_{1},\ldots ,x_{n},y\right) \right)
\end{equation*}%
and%
\begin{equation*}
\inf_{\left\vert \lambda \right\vert \leq 1}\mathrm{Re}\left( \varphi \left(
x_{1},\ldots ,x_{n},\lambda \right) \right)
\end{equation*}%
are formulae. Similarly one can replace $\inf $ with $\sup $. The
interpretation of a formula in a tracial von Neumann algebra is defined in
the obvious way by recursion on the complexity. For example%
\begin{equation*}
\left( \tau (x^{\ast }x)\right) ^{%
\frac{1}{2}}
\end{equation*}%
is a formula usually abbreviated by $\left\Vert x\right\Vert _{2}$ whose
interpretation in a tracial von\ Neumann algebra $(M)$ is the $2$-norm on $M$
associated with the trace $\tau $. Analogously%
\begin{equation*}
\sup_{\left\Vert x\right\Vert \leq 1}\sup_{\left\Vert y\right\Vert \leq
1}\left\Vert x-y\right\Vert _{2}
\end{equation*}%
is a sentence%
\index{sentence!for tracial von Neumann algebras} (i.e.\ a formula without
free variables) that holds in a tracial von\ Neumann algebra $M$ iff $M$ is
abelian, while%
\begin{equation*}
\sup_{\left\Vert x\right\Vert \leq 1}\inf_{\left\vert \lambda \right\vert
\leq 1}\left\Vert x-\lambda \right\Vert _{2}
\end{equation*}%
is a sentence%
\index{sentence!for tracial von Neumann algebras} which holds in $(M)$ iff $%
M $ is one-dimensional (i.e.\ isomorphic to $\mathbb{C}$).

The notion of elementary property%
\index{elementary property!of tracial von Neumann algebras} and
axiomatizable class%
\index{axiomatizable class!of tracial von Neumann algebras} of tracial von
Neumann algebras are defined as in the case of length groups or rank rings.
In particular the previous examples shows that the property of being abelian
and the property of being one-dimensional are elementary. Exercise \ref%
{Exercise: factor} and Exercise \ref{Exercise: II1factor} show that the
property of being a factor and, respectively, the property of being a 
\textup{II}$_{1}$ factor are elementary.

Recall that the \textit{center%
\index{center} }$Z(M)$ of a von Neumann algebra $M$ is\ the set of elements
that commute with any other element of $M$. This is a weakly closed
subalgebra of $M$ and hence it is itself a von Neumann algebra. The von
Neumann algebra $M$ is called a \textit{factor }if its center contains only
the scalar multiples of the identity.

The \emph{unitary group}%
\index{unitary group} $U(M)$ of $M$ is the multiplicative group of unitary
elements of $M$, i.e.\ elements $u$ satisfying $uu^{\ast }=u^{\ast }u=1$.
Recall that it can be seen using the Borel functional calculus \cite[I.4.3]%
{Blackadar} that the set of linear combinations of projections of a von
Neumann algebra is dense in the $\sigma $-weak topology.

\begin{exercise}
\label{Exercise: factor}Suppose that $M$ is a von Neumann algebra endowed
with a faithful trace $\tau $. Show that $M$ is a factor if and only if for
every $x\in M$%
\begin{equation}
\left\Vert x-\tau (x)\right\Vert _{2}\leq \sup_{y\in M_{1}}\left\Vert
xy-yx\right\Vert _{2}%
\text{.\label{eq:factor}}
\end{equation}%
Conclude that the property of being a factor is elementary%
\index{elementary property}.
\end{exercise}

\begin{hint}
If $M$ is not a factor then a nontrivial projection $p$ in $Z(M)$ violates
Equation \eqref{eq:factor}. If $M$ is a factor and $x$ is an element of the
unit ball of $M$ consider the strong closure of the convex hull of the orbit%
\begin{equation*}
\left\{ uxu^{\ast }:u\in U(M)\right\}
\end{equation*}%
of $x$ under the action of $U(M)$ on $M$ by conjugation. Observe that such
set endowed with the Hilbert-Schmidt norm is isometrically isomorphic to a
closed subset of the Hilbert space $L^{2}\left( M,\tau \right) $ obtained
from $\tau $ via the GNS construction \cite[II.6.4]{Blackadar}. It therefore
has a unique element of minimal Hilbert-Schmidt norm $x_{0}$, which by
uniqueness must commute with every element of $U(M)$. Since the convex hull
of $U(M)$ is dense in the unit ball of $M$ by the Russo-Dye theorem \cite[%
II.3.2.17]{Blackadar}, $x_{0}$ belongs to the center of $M$. Moreover by
normality of the trace $x_{0}$ must coincide with $\tau (x)$. Thus $\tau (x)$
can be approximated in the Hilbert-Schmidt norm by convex combinations of
elements of the form $uxu^{\ast }$, with $u\in U(M)$. Equation %
\eqref{eq:factor} easily follows from this fact.
\end{hint}

\begin{exercise}
\label{Exercise: II1factor}Fix an irrational number $\alpha \in \left(
0,1\right) $. Using the type classification of factors prove that a factor $%
M $ is \textup{II}$_{1}$ if and only there is a projection of trace $\alpha $%
. Deduce that the property of being a \textup{II}$_{1}$ factor is elementary%
\index{elementary property}.
\end{exercise}

\begin{hint}
Recall that by Theorem \ref{Theorem: projections II1 factor} the trace in a 
\textup{II}$_{1}$ factor attains on projections all the values between $0$
and $1$.
\end{hint}

Let us now define the ultraproduct $\prod\nolimits_{\mathcal{U}}M_{n}$ of a
sequence $(M_n)$ of tracial von Neumann algebras%
\index{ultraproduct!of tracial von Neumann algebras} with respect to a free
ultrafilter $\mathcal{U}$ over $%
%TCIMACRO{\U{2115} }%
%BeginExpansion
\mathbb{N}
%EndExpansion
$. Define $\ell ^{\infty }(M)$ the C*-algebra of sequences%
\begin{equation*}
(a_{n})_{n\in 
%TCIMACRO{\U{2115} }%
%BeginExpansion
\mathbb{N}
%EndExpansion
}\in \prod\nolimits_{n}M_{n}
\end{equation*}%
such that%
\begin{equation*}
\sup_{n}\left\Vert a_{n}\right\Vert <+\infty
\end{equation*}%
endowed with the norm%
\begin{equation*}
\left\Vert (a_{n})\right\Vert =\sup_{n}\left\Vert a_{n}\right\Vert 
\text{.}
\end{equation*}%
The ultraproduct $\prod\nolimits_{\mathcal{U}}M_{n}$ is the quotient of $%
\ell ^{\infty }(M)$ with respect to the norm closed ideal $I_{\mathcal{U}}$
of sequences $(a)$ such that%
\begin{equation*}
\lim_{n\rightarrow \mathcal{U}}\tau \left( a_{n}^{\ast }a_{n}\right) =0
\end{equation*}%
endowed with the faithful trace%
\begin{equation*}
\tau \left( (a_{n})+I_{\mathcal{U}}\right) =\lim_{n\rightarrow \mathcal{U}%
}\tau (a_{n})\text{.}
\end{equation*}%
Being the quotient of a C*-algebra by a norm closed ideal, $\prod\nolimits_{%
\mathcal{U}}M_{n}$ is a C*-algebra. Exercise \ref{Exercise: complete ball}
shows that in fact $\prod\nolimits_{\mathcal{U}}M_{n}$ is always a von
Neumann algebra. When the sequence $(M_n)$ is constantly equal to a tracial
von Neumann algebra $M$, the ultraproduct $\prod\nolimits_{\mathcal{U}}M_{n}$
is called \emph{ultrapower}%
\index{ultrapower!of tracial von Neumann algebras} of $M$ and denoted by $M^{%
\mathcal{U}}$. The notion of \emph{representative sequence}%
\index{sequence!representative} of an element of an ultraproduct of tracial
von Neumann algebras is defined analogously as in the case of invariant
length groups.

The notions of approximately finitely satisfiable (or consistent) 
\index{set of formulae!approximately finitely satisfiable}%
\index{set of formulae!consistent} and realized set of formulae%
\index{set of formulae!realized} and countably saturated structure%
\index{countably saturated!tracial von Neumann algebra} introduced in
Definition \ref{Definition: consistent and realized} and \ref{Definition:
countably saturated} for invariant length groups admit obvious
generalization to the setting of tracial von Neumann algebra. An adaptation
of the proof of Exercise \ref{Exercise: ultraproducts are countably
saturated} shows that an ultraproduct of a sequence of tracial von Neumann
algebras is countably saturated. If moreover the elements of the sequence
are separable, and the Continuum Hypothesis is assumed, then ultraproducts
are in fact saturated (cf. the discussion after Exercise \ref{Exercise:
ultraproducts are countably saturated}).

\begin{exercise}
\label{Exercise: complete ball}Show that the unit ball of $\prod\nolimits_{%
\mathcal{U}}M_{n}$ is complete with respect to the 2-norm. Infer that $%
\prod\nolimits_{\mathcal{U}}M_{n}$ is a von Neumann algebra.
\end{exercise}

\begin{hint}
Recall that an ultraproduct of a sequence of tracial von Neumann algebra is
countably saturated. Suppose that $(x_{n})_{n\in 
%TCIMACRO{\U{2115} }%
%BeginExpansion
\mathbb{N}
%EndExpansion
}$ is a sequence in the unit ball of $\prod\nolimits_{\mathcal{U}}M_{n}$
which is Cauchy with respect to the 2-norm. Define for every $n\in 
%TCIMACRO{\U{2115} }%
%BeginExpansion
\mathbb{N}
%EndExpansion
$%
\begin{equation*}
\varepsilon _{n}=\sup \left\{ \left\Vert x_{i}-x_{j}\right\Vert _{2}:i,j\geq
n\right\} 
\text{.}
\end{equation*}%
Observe that $\left( \varepsilon _{n}\right) _{n\in 
%TCIMACRO{\U{2115} }%
%BeginExpansion
\mathbb{N}
%EndExpansion
}$ is a vanishing sequence. Consider for every $n\in 
%TCIMACRO{\U{2115} }%
%BeginExpansion
\mathbb{N}
%EndExpansion
$ the formula $\varphi _{n}\left( y\right) $%
\begin{equation*}
\inf_{\left\Vert y\right\Vert \leq 1}\max \left\{ \left\Vert
x_{n}-y\right\Vert _{2}-\varepsilon _{n},0\right\} \text{.}
\end{equation*}%
Argue that the set $\mathcal{X}$ of formulae containing $\varphi _{n}\left(
y\right) $ for every $n\in 
%TCIMACRO{\U{2115} }%
%BeginExpansion
\mathbb{N}
%EndExpansion
$ is approximately finitely satisfiable and hence realized in $%
\prod\nolimits_{\mathcal{U}}M_{n}$. A realization $x$ of $\mathcal{X}$ is
such that%
\begin{equation*}
\left\Vert x_{n}-x\right\Vert _{2}\leq \varepsilon _{n}
\end{equation*}%
for every $n\in 
%TCIMACRO{\U{2115} }%
%BeginExpansion
\mathbb{N}
%EndExpansion
$ and hence the sequence $\left( x_{n}\right) _{n\in 
%TCIMACRO{\U{2115} }%
%BeginExpansion
\mathbb{N}
%EndExpansion
}$ converges to $x$. In order to conclude that $\prod\nolimits_{\mathcal{U}%
}M_{n}$ is a von Neumann algebra it is enough to observe that, by
completeness of its unit ball with respect to the 2-norm and Kaplansky's
density theorem (see \cite[I.9.1.3]{Blackadar}), $\prod\nolimits_{\mathcal{U}%
}M_{n}$ coincides with the von Neumann algebra generated by the GNS
representation of $\prod\nolimits_{\mathcal{U}}M_{n}$ associated with the
canonical trace $\tau $ of $\prod\nolimits_{\mathcal{U}}M_{n}$.
\end{hint}

\L o\'s' theorem on ultraproducts also holds in this context without change.%
\index{\L o\'s' theorem!for tracial von Neumann algebras}

\begin{theorem}[\L o\'{s}' for tracial von Neumann algebras]
\label{Theorem: Los for tracial vN algebras}Suppose that $\varphi \left(
x_{1},\ldots ,x_{k}\right) $ is a formula with free variables $x_{1},\ldots
,x_{k}$, $(M_{n})_{n\in \mathbb{N}}$ is a sequence of tracial von Neumann
algebras, and $\mathcal{U}$ is a free ultrafilter over $%
%TCIMACRO{\U{2115} }%
%BeginExpansion
\mathbb{N}
%EndExpansion
$. If $a^{(1)},\ldots ,a^{(k)}$ are elements of $\prod_{n}M_{n}$ then%
\begin{equation*}
\varphi ^{\prod\nolimits_{\mathcal{U}}M_{n}}\left( a^{(1)},\ldots
,a^{(k)}\right) =\lim_{n\rightarrow \mathcal{U}}\varphi ^{M_{n}}\left(
a_{n}^{(1)},\ldots ,a_{n}^{(k)}\right) ,
\end{equation*}%
where $a^{(i)}$ is any representative sequence of $\left( a_{n}^{(i)}\right)
_{n\in \mathbb{N}}$ for $i=1,2,\ldots ,k$. In particular if $\varphi $ is a
sentence%
\index{sentence!for tracial von Neumann algebras} then%
\begin{equation*}
\varphi ^{\prod\nolimits_{\mathcal{U}}M_{n}}=\lim_{n\rightarrow \mathcal{U}%
}\varphi ^{M_{n}}%
\text{.}
\end{equation*}
\end{theorem}

In particular \L o\'s' theorem implies that a tracial von Neumann algebra $M$
is elementarily equivalent to any ultrapower $M^{\mathcal{U}}$ of $M$. This
means that if $\varphi $ is any sentence%
\index{sentence!for tracial von Neumann algebras}, then $\varphi $ has the
same evaluation in $M$ and in $M^{\mathcal{U}}$. It follows that any two
ultrapowers of $M$ are elementarily equivalent. If moreover $M$ is separable
and the Continuum Hypothesis is assumed, then any two ultrapowers of $M$,
being saturated and elementarily equivalent, are in fact isomorphic by a
standard result in model theory (see Corollary 4.14 in \cite{FHS2}).
Conversely if the Continuum Hypothesis fails then, assuming that $M$ is a II$%
_{1}$ factor, by Theorem 4.8 in \cite{FHS1} there exist nonisomorphic
ultrapowers of $M$ (and in fact $2^{2^{\aleph _{0}}}$ many by Proposition
8.2 of \cite{Fa-Sh}).

Theorem \ref{Theorem: Los for tracial vN algebras} together with the fact
that the property of being a \textup{II}$_{1}$ factor is elementary (see
Exercise \ref{Exercise: II1factor}) shows that an ultraproduct of \textup{II}%
$_{1}$ factors is again a \textup{II}$_{1}$ factor. Analogously an
ultraproduct of factors is a factor and an ultraproduct of abelian tracial
von Neumann algebras is abelian.

%Recall that the \emph{unitary group}%
%\index{unitary group} $U(M)$ of a tracial von Neumann algebra is the set of
%elements $x$ of $M$ such that $xx^{\ast }=x^{\ast }x=1$.

The trace on the von Neumann algebra naturally induces an invariant length
function on $U(M)$, as shown in the following exercise.

\begin{exercise}
\label{Exercise: length unitary group}%
\index{length function!unitary group}Suppose that $M$ is a tracial von
Neumann algebra. Show that the function $\ell :U(M)\rightarrow \left[ 0,1%
\right] $ defined by%
\begin{equation*}
\ell (u)=%
\frac{1}{2}\left\Vert u-1\right\Vert _{2}
\end{equation*}%
is an invariant length function on $U(M)$.
\end{exercise}

In particular if $M$ is the algebra $M_{n}(\mathbb{C})$ of $n\times n$
complex matrices, then $U(M)$ coincides with the group $U_{n}$ of $n\times n$
unitary matrices, and the induced length function on $U_{n}$ coincides with
the one considered in Section \ref{Section: definition hyperlinear groups}.

Recall that in Section \ref{Section: definition hyperlinear groups} we
introduced the following relation%
\begin{equation*}
\max \left\{ \left\Vert x^{\ast }x-1\right\Vert _{2},\left\Vert xx^{\ast
}-1\right\Vert _{2}\right\} \text{.}
\end{equation*}%
This is a formula $\varphi ^{u}$ in the logic for tracial von Neumann
algebras. We have also mentioned that the polar decomposition shows that
such a formula is \emph{stable}, i.e.\ every approximate solution of $%
\varphi ^{u}\left( x\right) =0$ is close to an exact solution (and the
estimate is uniform over all tracial von Neumann algebras). In model
theoretic jargon this means that the zero-set of the interpretation $\varphi
^{u}$ in a given von Neumann algebra $M$---i.e.\ the unitary group $U(M)$%
---is a \emph{definable set%
\index{definable set}} as in \cite[Definition 9.16]{BBHU}; see also \cite[%
Proposition 9.19]{BBHU}. In particular it can be inferred that the unitary
group of an ultraproduct of tracial von Neumann algebras is the ultraproduct
of corresponding unitary groups. This is the content of Exercise \ref%
{Exercise: unitary group ultraproduct}.

\begin{exercise}
\label{Exercise: unitary group ultraproduct}Suppose that $(M_n)_{n\in 
\mathbb{N}}$ is a sequence of tracial von Neumann algebras, and $\mathcal{U}$
is a free ultrafilter on $%
%TCIMACRO{\U{2115} }%
%BeginExpansion
\mathbb{N}
%EndExpansion
$. Show that any element of the unitary group of the ultraproduct $%
\prod\nolimits_{\mathcal{U}}M_{n}$ admits a representative sequence of
unitary elements. Conclude that $U\left( \prod\nolimits_{\mathcal{U}%
}M_{n}\right) $ is isomorphic as invariant length group to the ultraproduct $%
\prod\nolimits_{\mathcal{U}}U(M_n)$ of the sequence of unitary groups of the 
$M_{n}$'s endowed with the invariant length described in Exercise \ref%
{Exercise: length unitary group}.
\end{exercise}

In particular the unitary group of $\prod\nolimits_{\mathcal{U}}M_{n}(%
\mathbb{C})$ can be identified with the group $\prod\nolimits_{\mathcal{U}%
}U_{n}$ introduced in Section \ref{Section: definition hyperlinear groups}.

\begin{exercise}
\label{Exercise: nonseparable ultraproduct II1 factors}The unitary group of $%
\prod\nolimits_{\mathcal{U}}M_{n}(\mathbb{C})$ contains a subset $X$ of size
continuum such that $\left\Vert u-v\right\Vert _{2}=%
\sqrt{2}$ for every pair of distinct elements $u,v$ of $X$. Deduce that the
same conclusion holds for the unitary group of any ultraproduct $%
\prod\nolimits_{\mathcal{U}}M_{n}$ of a sequence of \textup{II}$_{1}$
factors.
\end{exercise}

\begin{hint}
The first statement follows directly from Exercise \ref{Exercise: unitary
group ultraproduct} and Exercise \ref{Exercise: nonseparable universal
hyperlinear}. For the second statement observe that by Theorem \ref{Theorem:
projections II1 factor}, if $\left( M_{n}\right) _{n\in 
%TCIMACRO{\U{2115} }%
%BeginExpansion
\mathbb{N}
%EndExpansion
}$ is a sequence of \textup{II}$_{1}$ factors, and $\mathcal{U}$ is a free
ultrafilter on $%
%TCIMACRO{\U{2115} }%
%BeginExpansion
\mathbb{N}
%EndExpansion
$, then $\prod\nolimits_{\mathcal{U}}M_{n}\left( \mathbb{C}\right) $ embeds
into $\prod\nolimits_{\mathcal{U}}M_{n}$.
\end{hint}

\section{The algebraic eigenvalues conjecture \label{Section: algebraic
eigenvalues}}

Suppose that $\Gamma $ is a (countable, discrete) group. Considering the
particular case of the group algebra construction for the field $\mathbb{C}$
of complex numbers as in Section \ref{Section: vN algebras and II1 factors}
one obtains the complex group algebra 
\index{group algebra!complex} $\mathbb{C}\Gamma $ of formal finite linear
combinations%
\begin{equation*}
\lambda _{1}\gamma _{1}+\cdots +\lambda _{k}\gamma _{k}
\end{equation*}%
where $\lambda _{i}\in \mathbb{C}$ and $\gamma _{i}\in \Gamma $. The group
ring $\mathbb{Z}\Gamma $ is the subring of $\mathbb{C}\Gamma $ of finite
linear combinations%
\begin{equation*}
n_{1}\gamma _{1}+\cdots +n_{k}\gamma _{k}
\end{equation*}%
where $n_{i}\in \mathbb{Z}$ and $\gamma _{i}\in \Gamma $. The natural action
of $\mathbb{C}\Gamma $ on the Hilbert space $\ell ^{2}\Gamma $ defines an
inclusion of $\mathbb{C}\Gamma $ into $B\left( \ell ^{2}\Gamma \right) $. A
conjecture due to Dodziuk, Linnell, Mathai, Schick, and Yates known as 
\textit{algebraic eigenvalues conjectures%
\index{conjecture!algebraic eigenvalues}} (see \cite%
{DodziukLinnellMathaiSchickYates}) asserts that elements $x$ of $\mathbb{Z}%
\Gamma $ regarded as linear operators on $\ell ^{2}\Gamma $ have algebraic
integers as eigenvalues. Recall that a complex number is called an \textit{%
algebraic integer%
\index{algebraic integer}} if it is the root of a monic polynomial with
integer coefficients. The algebraic eigenvalues conjecture has been settled
for sofic%
\index{group!sofic} groups by Andreas Thom in \cite{Thom-diophantine}. The
proof involves the notion of ultraproduct of tracial von Neumann algebras
and can be naturally presented within the framework of logic for metric
structures.

The complex group algebra%
\index{group algebra!complex} $\mathbb{C}\Gamma $ can be endowed with a
linear involutive map $x\mapsto x^{\ast }$ such that%
\begin{equation*}
(\lambda \gamma )^{\ast }=%
\overline{\lambda }\gamma ^{-1}\text{.}
\end{equation*}%
Recall that the trace $\tau $ on $\mathbb{C}\Gamma $ is defined by%
\begin{equation*}
\tau \left( \sum_{\gamma }\lambda _{\gamma }\gamma \right) =\lambda
_{1_{\Gamma }}\text{.}
\end{equation*}%
The weak closure $L\Gamma $ of $\mathbb{C}\Gamma $ in $B\left( \ell
^{2}\Gamma \right) $ is a von Neumann algebra containing $\mathbb{C}\Gamma $
as a *-subalgebra. The trace of $\mathbb{C}\Gamma $ admits a unique
extension to a faithful normalized trace $\tau $ on $L\Gamma $. Moreover $%
\mathbb{C}\Gamma $ is dense in the unit ball of $L\Gamma $ with respect to
the $2$-norm of $L\Gamma $ defined by $\left\Vert x\right\Vert _{2}=\tau
(x^{\ast }x)^{\frac{1}{2}}$.

In the rest of the section the matrix algebra $M_{n}(\mathbb{C})$ of $%
n\times n$ matrices with complex coefficients is regarded as a tracial von
Neumann algebra endowed with the (unique) canonical normalized trace $\tau
_{n}$. If $\mathcal{U}$ is an ultrafilter over $\mathbb{N}$ then $%
\prod\nolimits_{\mathcal{U}}M_{n}(\mathbb{C})$ denotes the ultraproduct of $%
M_{n}(\mathbb{C})$ as tracial von Neumann algebras. (Note that this is
different from the ultraproduct of $M_{n}(\mathbb{C})$ as rank rings.)
Denote by $\prod\nolimits_{\mathcal{U}}M_{n}(\mathbb{Z})$ the closed
self-adjoint subalgebra of $\prod\nolimits_{\mathcal{U}}M_{n}(\mathbb{C})$
consisting of elements admitting representative sequences%
\index{sequence!representative} of matrices with integer coefficients.
Recall that $U_{n}$ denotes the group of unitary elements of $M_{n}(\mathbb{C%
})$. If $\sigma $ is a permutation over $n$, then the associated permutation
matrix $P_{\sigma }$ is a unitary element of $M_{n}(\mathbb{C})$ such that $%
\tau \left( P_{\sigma }\right) =1-\ell (\sigma )$. This fact can be used to
solve Exercise \ref{Exercise: complex algebra sofic}. The argument is
analogous to the proof that sofic groups are hyperlinear (see Exercise \ref%
{Exercise: sofic implies hyperlinear}).

\begin{exercise}
\label{Exercise: complex algebra sofic}Suppose that $\Gamma $ is a sofic%
\index{group!sofic} group and $U$ is a nonprincipal ultrafilter over $N$.
Show that there is a trace preserving *-homomorphism of *-algebras from $%
\mathbb{C}\Gamma $ to $\prod\nolimits_{\mathcal{U}}M_{n}(\mathbb{C})$
sending $\mathbb{Z}\Gamma $ into $\prod\nolimits_{\mathcal{U}}M_{n}(\mathbb{Z%
})$.
\end{exercise}

\begin{hint}
Use, as at the end of Section \ref{Section: definition hyperlinear groups},
the embedding of $S_{n}$ into $M_{n}\left( \mathbb{C}\right) $ sending $%
\sigma $ to its associated permutation matrix $P_{\sigma }$, recalling that $%
\tau (P_{\sigma })=1-\ell (\sigma )$.
\end{hint}

Since the *-homomorphism from $\mathbb{C}\Gamma $ to $\prod\nolimits_{%
\mathcal{U}}M_{n}(\mathbb{C})$ obtained in Exercise \ref{Exercise: complex
algebra sofic} is trace-preserving, it extends to an embedding of $L\Gamma $
into $\prod\nolimits_{\mathcal{U}}M_{n}(\mathbb{C})$. This can be seen using
the GNS construction associated with $\tau $ (see \cite[Section II.6.4]%
{Blackadar}), or observing that by Kaplansky's density theorem \cite[I.9.1.3]%
{Blackadar} the operator norm unit ball of $\mathbb{C}\Gamma $ is dense in
the norm unit ball of $L\Gamma $ with respect to the $\sigma $-strong
topology (which coincides with the topology induced by the Hilbert-Schmidt
norm associated with $\tau $).

Suppose now that $M$ is a von Neumann algebra, and $x$ is an operator in $M$%
. Considering the spectral projection $p\in M$ on $\mathrm{Ker}\left(
x-\lambda I\right) $ shows that a complex number $\lambda $ is an eigenvalue
of $x$ if and only if there is a nonzero projection $p\in M$ such that $%
\left( x-\lambda \right) p=0$. It follows that if $\Psi :M\rightarrow N$ is
an embedding of von Neumann algebra, then $x$ and $\Psi (x)$ have the same
eigenvalues. This observation together with Exercise \ref{Exercise: complex
algebra sofic} and its following observation allow one to conclude that in
order to establish the algebraic eigenvalues conjecture for sofic groups it
is enough to show that the elements of $\prod\nolimits_{\mathcal{U}}M_{n}(%
\mathbb{Z})$ have algebraic eigenvalues. This is proved in the rest of this
section using \L o\'s' theorem on ultraproducts together with a
characterization of algebraic integers due to Thom.

Suppose in the following that $\lambda \in \mathbb{C}$ is not an algebraic
integer. Theorem \ref{Theorem: algebraic integer} is a consequence of
Theorem 3.5 from \cite{Thom-diophantine}.

\begin{theorem}
\label{Theorem: algebraic integer}Suppose that $\varepsilon $ is a positive
real number, and that $N$ is a natural number. There is a positive constant $%
M( \lambda ,N,\varepsilon ) $ depending only on $\lambda $, $N$, and $%
\varepsilon $ with the following property: For every monic polynomial $p $
with integer coefficients such that all the zeros of $p$ have absolute value
at most $N$ the proportion of zeros of $p$ at distance less than $%
\frac{1}{M( \lambda ,N,\varepsilon ) }$ from $\lambda $ is at most $%
\varepsilon $.
\end{theorem}

Corollary \ref{Corollary: eigenvalues matrices} is a direct consequence of
Theorem \ref{Theorem: algebraic integer}.

\begin{corollary}
\label{Corollary: eigenvalues matrices}Suppose that $\varepsilon $ is a
positive real number, and that $N$ is a natural number. Denote by $M(
\lambda ,N,\varepsilon )$ the positive constant given by Theorem \ref%
{Theorem: algebraic integer}. For every finite rank matrix with integer
coefficients $A$ of operator norm at most $N$ there is a complex matrix $B$
of the same size of operator norm at most $M\left( \lambda ,N,\varepsilon
\right) $ such that%
\begin{equation*}
\left\Vert B\left( \lambda -A\right) -I\right\Vert _{2}\leq \varepsilon 
\text{.}
\end{equation*}
\end{corollary}

\begin{proof}
If $A$ is a finite rank matrix with integer coefficients of norm at most $N$%
, then the minimal polynomial $p_{A}$ of $A$ is a monic polynomial with
integer coefficients whose zeros have all absolute value at most $N$. Since $%
\lambda $ is not an algebraic integers, $\lambda $ is not an eigenvalue of $%
p_{A}$ and hence $\lambda -A$ is invertible. Moreover by the choice of $%
M( \lambda ,N,\varepsilon ) $ the proportion of zeros of $p$ at
distance at least $\frac{1}{M( \lambda ,N,\varepsilon ) }$ from $%
\lambda $ is at least $1-\varepsilon $. This means that if $p$ is the
projection on the eigenspace corresponding to these eigenvalues, then $\tau
\left( p\right) >1-\varepsilon $. Define $B=p\left( \lambda -A\right) ^{-1}$%
, and observe that $B$ has operator norm at most $M\left( \lambda ,N,\varepsilon
\right) $ and 
\[
\left\Vert B\left( \lambda -A\right) -1\right\Vert _{2}=\left\Vert
1-p\right\Vert _{2}=\tau \left( 1-p\right) \leq \varepsilon \text{.}
\]
\end{proof}

Define for every $M>0$ the formula $\varphi _{M}(x)$ in the language of
tracial von Neumann algebras by%
\begin{equation*}
\inf_{\left\Vert y\right\Vert \leq 1}\left\Vert My\left( \lambda -x\right)
-1\right\Vert _{2}\text{.}
\end{equation*}%
By Corollary \ref{Corollary: eigenvalues matrices} if $a$ is any element of $%
M_{n}(\mathbb{Z})$ of operator norm at most $N$, then%
\begin{equation*}
\varphi _{M(\lambda ,N,\varepsilon )}(a)\leq \varepsilon
\end{equation*}%
where $M(\lambda ,N,\varepsilon )$ is the constant given by Theorem \ref%
{Theorem: algebraic integer}. By \L o\'{s}' theorem on ultraproducts the
same is true for any element $a$ of $\prod\nolimits_{\mathcal{U}}M_{n}(%
\mathbb{Z})$ of operator norm at most $N$. It follows that for every $a\in
\prod\nolimits_{\mathcal{U}}M_{n}(\mathbb{Z})$ and every positive real
number $\varepsilon $ there is $b\in \prod\nolimits_{\mathcal{U}}M_{n}(%
\mathbb{C})$ such that $\left\Vert b\right\Vert \leq M(\lambda ,\left\Vert
a\right\Vert ,\varepsilon )$ and%
\begin{equation*}
\left\Vert b\left( a-\lambda \right) -1\right\Vert _{2}\leq \varepsilon 
\text{.}
\end{equation*}%
Therefore if $p$ is a projection such that $\left( a-\lambda \right) p=0$,
then%
\begin{eqnarray*}
\left\Vert p\right\Vert _{2} &\leq &\left\Vert \left( b\left( a-\lambda
\right) -1\right) p\right\Vert _{2}+\left\Vert b\left( a-\lambda \right)
p\right\Vert \\
&\leq &\left\Vert b\left( a-\lambda \right) -1\right\Vert _{2} \\
&\leq &\varepsilon \text{.}
\end{eqnarray*}%
Being this true for every positive real number $\varepsilon $, $p=0$. This
shows that $\lambda $ is not an eigenvalue of $a$ for any element $a$ of $%
\prod\nolimits_{\mathcal{U}}M_{n}(\mathbb{Z})$, concluding the proof that
the elements of $\prod\nolimits_{\mathcal{U}}M_{n}\left( 
%TCIMACRO{\U{2124} }%
%BeginExpansion
\mathbb{Z}
%EndExpansion
\right) $ have algebraic eigenvalues.

\section{Entropy}

\label{Section: entropy}

\subsection{Entropy of a single transformation}

Suppose that $X$ is a zero-dimensional compact metrizable space, i.e.\ a
compact space with a countable clopen basis. Denote by $T$ a transformation
of $X$, i.e.\ a homeomorphism $T:X\rightarrow X$. A (necessarily finite)
clopen partition $\mathcal{P}$ of $X$ is \textit{generating%
\index{generating clopen partition} }for $T$ if for every pair of distinct
points $x,y$ of $X$ there are $C\in \mathcal{P}$ and $n\in \mathbb{N}$ such
that $T^{n}x\in C$ and $T^{n}y\notin C$. Suppose that $T$ has a generating
clopen partition $\mathcal{P}$ (observe that in general this might not
exist). Consider $\mathcal{P}$ as a coloring of $X$ and $\mathcal{P}^{n}$ as
a coloring of $X^{n}$. Denote for $n\geq 1$ by $H_{n}\left( \mathcal{P}%
,T\right) $ the number of possible colors of partial orbits of the form%
\begin{equation*}
\left( x,Tx,T^{2}x,\ldots ,T^{n}x\right) \in X^{n}%
\text{.}
\end{equation*}%
Equivalently $H_{n}\left( \mathcal{P},T\right) $ is the cardinality of the
clopen partition of $X$ consisting of sets%
\begin{equation*}
C_{0}\cap T^{-1}C_{1}\cap \ldots \cap T^{-n}C_{n},
\end{equation*}%
where $C_{i}\in \mathcal{P}$ for $i=0,1,2,\ldots ,n$.

A sequence $(a_n)_{n\in \mathbb{N}}$ of real numbers is called
submultiplicative%
\index{sequence!submultiplicative} if $a_{n+m}\leq a_{n}\cdot a_{m}$ for
every $n,m\in \mathbb{N}$.

\begin{exercise}
\label{Exercise: submultiplicative}Show that the sequence%
\begin{equation*}
\left( H_{n}\left( \mathcal{P},T\right) \right) _{n\in \mathbb{N}}
\end{equation*}%
is submultiplicative.
\end{exercise}

Fekete's lemma from \cite{Fekete} 
\index{Fekete's lemma}asserts that if $(a_n)_{n\in \mathbb{N}}$ is a
submultiplicative sequence then the sequence $\left( 
\frac{1}{n}\mathrm{log}(a_n)\right) _{n\in \mathbb{N}}$ converges to $%
\inf_{n\in \mathbb{N}}\frac{1}{n}\mathrm{log}(a_n)$.

\begin{exercise}
\label{Exercise: Fekete}Prove Fekete's lemma.
\end{exercise}

It follows from Fekete's lemma and Exercise \ref{Exercise: submultiplicative}
that the sequence%
\begin{equation*}
\left( \frac{1}{n}\mathrm{log}\left( H_{n+m}\left( \mathcal{P},T\right)
\right) \right) _{n\in \mathbb{N}}
\end{equation*}%
has a limit $h\left( T\right) $ that is by definition the \textit{entropy%
\index{entropy!single transformation}} of the transformation $T$. Despite
being defined in terms of the partition $\mathcal{P}$, the entropy $h\left(
T\right) $ of $T$ does not in fact depend on the choice of the clopen
partition of $X$ generating for $T$.

\begin{exercise}
Show that the entropy of $T$ does not depend from the choice of $P$.
\end{exercise}

Observe that $h\left( T\right) $ is at most $\mathrm{log}\left\vert \mathcal{%
P}\right\vert $ for any clopen partition $\mathcal{P}$ of $X$ generating for 
$T$. Recall that transformations $T$ and $T^{\prime }$ of compact Hausdorff
spaces $X$ and $X^{\prime }$ are \textit{topologically conjugate} if there
is a homeomorphism $f:X\rightarrow X^{\prime }$ such that $T^{\prime }\circ
f=f\circ T$ for every $x\in X$. It is not difficult to verify that entropy
is a topological conjugacy invariant, that is, topologically conjugate
transformations have the same entropy.

\subsection{Entropy of an integer Bernoulli shift}

Suppose that $A$ is a finite alphabet of cardinality $k$ and $X$ is the
space $A^{\mathbb{Z}}$. Consider the Bernoulli shift%
\index{Bernoulli shift!integer} $T$ on $X$ defined by%
\begin{equation*}
T(a_{i})_{i\in \mathbb{Z}}=\left( a_{i-1}\right) _{i\in \mathbb{Z}}%
\text{.}
\end{equation*}%
For $a\in A$ consider the clopen set%
\begin{equation*}
X_{a}=\left\{ (a_{i})_{i\in \mathbb{Z}}:a_{0}=a\right\}
\end{equation*}%
and observe that $\mathcal{P}=\left\{ X_{a}\left\vert \,a\in A\right.
\right\} $ is a clopen partition of $X$ generating for $T$. It is not hard
to verify that%
\begin{equation*}
H_{n}\left( \mathcal{P},T\right) =k^{n},
\end{equation*}%
for every $n\in \mathbb{N}$, and hence%
\begin{equation*}
h\left( T\right) =\mathrm{log}(k)\text{.}
\end{equation*}

Gottschalk's conjecture%
\index{conjecture!Gottschalk's surjunctivity} for $\mathbb{Z}$ asserts that
if $f:X\rightarrow X$ is a continuous injective function such that $T\circ
f=f\circ T$ then $f$ is surjective. In view of the conjugation invariance of
entropy, in order to establish Gottschalk's conjecture for $\mathbb{Z}$ it
is enough to show that if $Y$ is any proper closed $T$-invariant subspace of 
$X$, then the entropy $h\left( T|_{Y}\right) $ of the restriction of $T$ to $%
Y$ is strictly smaller then the entropy $h\left( T\right) $ of $X$. Suppose
that $Y$ is a proper closed $T$-invariant subspace of $X$. Observe that 
\begin{equation*}
\mathcal{Q}=\left\{ X_{a}\cap Y:a\in A\right\}
\end{equation*}%
is a generating partition for $T|_{Y}$. It is easy to see that%
\begin{equation*}
H_{n}\left( \mathcal{Q},T|_{Y}\right)
\end{equation*}%
is the number of $A$-words of length $n$ that appear in elements of $Y$.
Since $Y$ is a proper closed $T$-invariant subspace of $X$ there is some $%
n\in \mathbb{N}$ such that%
\begin{equation*}
H_{n}\left( \mathcal{Q},T|_{Y}\right) <k^{n}
\end{equation*}%
and hence%
\begin{equation*}
h\left( T|_{Y},Y\right) =\inf_{n}%
\frac{1}{n}\mathrm{\mathrm{log}}\left( H_{n}\left( \mathcal{Q},T|_{Y}\right)
\right) <\mathrm{log}(k)=h\left( T,X\right) \text{.}
\end{equation*}%
This concludes the proof of Gottschalk's conjecture for $\mathbb{Z}$.

\subsection{Tilings on amenable groups}

Suppose in the following that $\Gamma $ is a (countable, discrete) group. If 
$E,F$ are subsets of $\Gamma $, then define

\begin{itemize}
\item $F^{-E}=\bigcap_{\gamma \in E}F\gamma ^{-1}=\left\{ x\in \Gamma
\left\vert xE\subset F\right. \right\} $;

\item $F^{+E}=\bigcup_{\gamma \in E}F\gamma =\left\{ x\in \Gamma \left\vert
xE\cap F\neq \varnothing \right. \right\} $;

\item $\partial _{E}F=F^{+E}\left\backslash F^{-E}\right. $.
\end{itemize}

The subset $F$ of $\Gamma $ is $\left( E,\delta \right) $-invariant for some
positive real number $\delta $ if 
\begin{equation*}
\frac{\left\vert \partial _{E}F\right\vert }{\left\vert F\right\vert }%
<\delta \text{.}
\end{equation*}%
The group $\Gamma $ is \textit{amenable%
\index{group!amenable}} if and only if for every finite subset $E$ of $%
\Gamma $ and every positive real number $\delta $ there is a finite $\left(
E,\delta \right) $-invariant subset $F$ of $\Gamma $. This is equivalent to
the existence of a F\o lner sequence, i.e.\ a sequence $\left( F_{n}\right)
_{n\in \mathbb{N}}$ of finite subsets of $\Gamma $ such that 
\begin{equation*}
\lim_{n}%
\frac{\left\vert \partial _{E}F_{n}\right\vert }{\left\vert F_{n}\right\vert 
}=0
\end{equation*}%
for every finite subset $E$ of $\Gamma $.

Suppose in the following that $\Gamma $ is an amenable group. A function $%
\phi $ from the set $\left[ \Gamma \right] ^{<\aleph _{0}}$ of finite
subsets of $\Gamma $ to the set $\mathbb{R}_{+}$ of positive real numbers is
called:

\begin{itemize}
\item \textit{subadditive }if $\phi (a\cup b)\leq \phi (a)+\phi \left(
b\right) $ and $\phi \left( \varnothing \right) =0$;

\item \textit{right invariant }if $\phi (a)=\phi (a\gamma)$ for every $%
\gamma \in \Gamma $.
\end{itemize}

Lindenstrauss and Weiss proved in \cite{Lindenstrauss-Weiss} using the
theory of quasi-tilings for amenable groups introduced by Ornstein and Weiss
in \cite{Ornstein-Weiss} that if $\phi :\left[ \Gamma \right] ^{<\aleph
_{0}}\rightarrow \mathbb{R}_{+}$ is a right-invariant subadditive function,
then the function $E\mapsto \frac{\left\vert \phi (E)\right\vert }{%
\left\vert E\right\vert }$ has a F\o lner limit $\ell (\phi )$. This means
that for every $\varepsilon >0$ there is a finite subset $E$ of $\Gamma $
and a positive real number $\delta $ such that for every $\left( E,\delta
\right) $-invariant finite subset $F$ of $\Gamma $%
\begin{equation*}
\left\vert \frac{\phi (F)}{\left\vert F\right\vert }-\ell (\phi )\right\vert
<\varepsilon \text{.}
\end{equation*}%
If $E$ and $E^{\prime }$ are subsets of $\Gamma $, then an $\left(
E,E^{\prime }\right) $-tiling%
\index{tiling} is a subset $T$ of $\Gamma $ such that the family $\left\{
\gamma E\left\vert \gamma \in T\right. \right\} $ is made of pairwise
disjoint sets, while the family $\left\{ \gamma E^{\prime }\left\vert \gamma
\in T\right. \right\} $ is a cover of $\Gamma $.

\begin{exercise}
\label{Exercise:tilings} If $1\in E$ and $E^{\prime }=EE^{-1}$, then there
is an $\left( E,E^{\prime }\right) $-tiling of $\Gamma $. Moreover if $T$ is
any $\left( E,E^{\prime }\right) $-tiling, then for every finite subset $F$
of $\Gamma $%
\begin{equation*}
\frac{\left\vert T\cap F^{-E}\right\vert }{\left\vert F\right\vert }\geq 
\frac{1}{\left\vert E^{\prime }\right\vert }-\frac{\left\vert \partial
_{E^{\prime }}(F)\right\vert }{\left\vert F\right\vert }\text{.}
\end{equation*}
\end{exercise}

\begin{hint}
To show existence consider any maximal set $T$ with the property that the
family $\left\{ T\gamma \left\vert \gamma \in E\right. \right\} $ contains
pairwise disjoint elements. For the second statement, observe that%
\begin{equation*}
\left\{ \gamma E^{\prime }\left\vert \gamma \in T\cap F^{+E^{\prime
}}\right. \right\}
\end{equation*}%
covers $F$, while%
\begin{equation*}
\left( T\cap F^{+E^{\prime }}\right) \left\backslash \left( T\cap
F^{-E}\right) \right. \subset F^{+E^{\prime }}\left\backslash F^{-E^{\prime
}}\right. =\partial _{E^{\prime }}(F)\text{.}
\end{equation*}%
Thus%
\begin{eqnarray*}
\frac{\left\vert T\cap F^{-E}\right\vert }{\left\vert F\right\vert } &\geq &%
\frac{\left\vert T\cap F^{+E^{\prime }}\right\vert -\left\vert \partial
_{E^{\prime }}(F)\right\vert }{\left\vert F\right\vert } \\
&\geq &\frac{1}{\left\vert E^{\prime }\right\vert }-\frac{\left\vert
\partial _{E^{\prime }}(F)\right\vert }{\left\vert F\right\vert }\text{.}
\end{eqnarray*}
\end{hint}

\subsection{Entropy of actions of amenable groups}

Suppose that $\Gamma \curvearrowright X$ is an action of an amenable group $%
\Gamma $ on a zero-dimensional compact space $X$ admitting a generating
clopen partition%
\index{generating clopen partition}. This means that $\mathcal{P}$ is a
clopen partition of $X$ such that for every $x,y\in X$ distinct there are $%
\gamma \in \Gamma $ and $C\in \mathcal{P}$ such that $\gamma \cdot x\in C$
and $\gamma \cdot y\notin C$. Regard $\mathcal{P}$ as a coloring of $X$ and,
for any finite subset $F$ of $\Gamma $, $\mathcal{P}^{F}$ as a coloring of $%
X^{F}$. Denote by $H_{F}\left( \mathcal{P},\Gamma ,X\right) $ the number of
possible colors of partial orbits%
\begin{equation*}
\left( \gamma \cdot x\right) _{\gamma \in F}\in X^{F}%
\text{.}
\end{equation*}%
Equivalently $H_{F}\left( \mathcal{P},\Gamma ,X\right) $ is the cardinality
of the clopen partition of $X$ consisting of sets of the form%
\begin{equation*}
\bigcap_{\gamma \in F}\gamma ^{-1}C_{\gamma }
\end{equation*}%
for $C_{\gamma }\in \mathcal{P}$.

\begin{exercise}
\label{Exercise: amenable subadditive}The function $F\mapsto \log\left(
H_{F}\left( \mathcal{P},\Gamma ,X\right) \right) $ is right-invariant and
subadditive.
\end{exercise}

It follows from Exercise \ref{Exercise: amenable subadditive} and the
Lindenstrauss-Weiss theorem on right-invariant subadditive functions that
there is a positive real number $h\left( \Gamma ,X\right) $, called \textbf{%
entropy} of the action $\Gamma \curvearrowright X$%
\index{entropy!amenable group action}, such that for every $\varepsilon >0$
there is a finite subset $E$ of $\Gamma $ and a positive real number $\delta 
$ such that if $F$ is an $\left( E,\delta \right) $-invariant subset of $%
\Gamma $ then%
\begin{equation*}
\left\vert h\left( \Gamma ,X\right) -%
\frac{1}{\left\vert F\right\vert }\mathrm{log}\left( H_{F}\left( \mathcal{P}%
,\Gamma ,X\right) \right) \right\vert <\varepsilon \text{.}
\end{equation*}%
As for the case of integer actions, it can be verified that the entropy $%
h\left( \Gamma ,X\right) $ does not depend on the chosen generating clopen
partition. Recall that two actions $\Gamma \curvearrowright X$ and $\Gamma
\curvearrowright X^{\prime }$ are topologically conjugate if there is a
homeomorphism $f:X\rightarrow Y$ such that for every $\gamma \in \Gamma $%
\begin{equation*}
f(\gamma x)=\gamma f(x)\text{.}
\end{equation*}%
It is easy to verify that topologically conjugate actions have the same
entropy.

\subsection{Entropy of Bernoulli shifts of an amenable group}

Suppose that $\Gamma $ is an amenable group, $A$ is a finite alphabet of
cardinality $k$, and $A^{\Gamma }$ is the space of $\Gamma $-sequences of
elements of $A$. The\textit{\ Bernoulli action%
\index{Bernoulli shift!amenable}} of $\Gamma $ on $X$ is defined by%
\begin{equation*}
\gamma \left( a_{h}\right) _{h\in \Gamma }=\left( a_{\gamma ^{-1}h}\right)
_{h\in \Gamma }%
\text{.}
\end{equation*}%
The clopen partition $\mathcal{P}$ of $X$ containing for every $a\in A$ the
set%
\begin{equation*}
X_{a}=\left\{ \left( a_{h}\right) _{h\in \Gamma }\left\vert a_{1}=a\right.
\right\}
\end{equation*}%
is generating for the Bernoulli action. It is not hard to verify that for
every finite subset $F$ of $\Gamma $%
\begin{equation*}
H_{F}\left( \mathcal{P},\Gamma ,A^{\Gamma }\right) =k^{\left\vert
F\right\vert }
\end{equation*}%
and hence%
\begin{equation*}
h\left( \Gamma ,A^{\Gamma }\right) =\mathrm{log}(k)\text{.}
\end{equation*}

Gottschalk's conjecture%
\index{conjecture!Gottschalk's surjunctivity} for the group $\Gamma $
asserts that if $f:A^{\Gamma }\rightarrow A^{\Gamma }$ is a continuous
injective function such that $f\left( \gamma x\right) =\gamma f(x)$ for
every $\gamma \in \Gamma $ and $x\in A^{\Gamma }$ then $f$ is surjective. As
in the case of integers, in order to establish Gottschalk's conjecture it is
enough to show that if $Y$ is any proper closed invariant subspace of $X$
then the entropy $h\left( \Gamma ,Y\right) $ of the Bernoulli action of $%
\Gamma $ on $Y$ is strictly smaller than $\mathrm{log}(k)$. The theory of
tilings is useful to show that a proper Bernoulli subshift has strictly
smaller entropy.

Suppose that $Y$ is a proper Bernoulli subshift. If $F\subset \Gamma $ is
finite, define $Y_{F}$ to be the set of restrictions of elements of $Y$ to $%
F $. Observe that $H_{F}\left( \Gamma ,Y\right) =\left\vert Y_{F}\right\vert 
$. Moreover since $Y$ is a proper subshift of $X$ there is a finite subset $%
E $ of $\Gamma $ such that $1\in E$ and $Y_{E}$ is a proper subset of $A^{E}$%
. Define $E^{\prime }=EE^{-1}$.

\begin{exercise}
\label{Exerise: entropy tiling} Show that for any finite subset $F$ of $%
\Gamma $ 
\begin{equation*}
\frac{1}{\left\vert F\right\vert }\mathrm{log}\left\vert Y_{F}\right\vert
\leq \mathrm{log}k-\left( \frac{1}{\left\vert E^{\prime }\right\vert }-\frac{%
\left\vert \partial _{E^{\prime }}(F)\right\vert }{\left\vert F\right\vert }%
\right) \mathrm{log}\left( \frac{k^{\left\vert E\right\vert }}{k^{\left\vert
E\right\vert }-1}\right) .
\end{equation*}
\end{exercise}

\begin{hint}
Pick an $\left( E,E^{\prime }\right) $-tiling $T$. If $F$ is a finite subset
of $\Gamma $, then define%
\begin{equation*}
T^{-}=T\cap F^{-E}
\end{equation*}%
and%
\begin{equation*}
F^{\ast }=F\backslash \bigcup_{\gamma \in T^{-}}\gamma E\text{.}
\end{equation*}%
Observe that%
\begin{equation}
Y_{F}\subset A^{F^{\ast }}\times \prod_{g\in T^{-}}Y_{gE} .
\label{eq:tiling}
\end{equation}%
Deduce from Equation \eqref{eq:tiling} and from Exercise \ref%
{Exercise:tilings} that the conclusion holds.
\end{hint}

It follows from Exercise \ref{Exerise: entropy tiling} that%
\begin{equation*}
h\left( \Gamma ,Y\right) \leq \mathrm{log}(k)-\frac{1}{\left\vert E^{\prime
}\right\vert }\mathrm{\mathrm{log}}\left( \frac{k^{\left\vert E\right\vert }%
}{k^{\left\vert E\right\vert }-1}\right) <\mathrm{log}(k)\text{.}
\end{equation*}

\subsection{Entropy of actions of sofic groups}

Suppose that $\Gamma \curvearrowright X$ is an action of a sofic%
\index{group!sofic} group on a zero-dimensional space $X$. As before we will
assume that there is a clopen partition $\mathcal{P}$ that is generating for
the action%
\index{generating clopen partition}, and we will regard $\mathcal{P}$ as a
coloring of $X$ and $\mathcal{P}^{n}$ as a coloring of $X^{n}$. If $F$ is a
finite subset of $\Gamma $ and $\sigma $ is a function from $\Gamma $ to $%
S_{n}$, define%
\begin{equation*}
H_{F,\delta }\left( \sigma ,\mathcal{P},\Gamma ,X\right)
\end{equation*}%
to be the number of colors $\left( C_{i}\right) _{i\in n}$ in $\mathcal{P}%
^{n}$ such that there is a sequence $\left( x_{i}\right) _{i\in n}\in X^{n}$
such that for every $\gamma \in F$ and for a proportion of at least $\left(
1-\delta \right) $ indexes $i\in n$, $\gamma ^{-1}x_{i}$ has color $%
C_{\sigma _{\gamma }^{-1}(i)}$. Suppose that $\Sigma $ is a\emph{\ sofic
approximation sequence}\textit{%
\index{sequence!sofic approximation} }of $\Gamma $, i.e.\ a sequence $%
(\sigma _{n})_{n\in 
%TCIMACRO{\U{2115} }%
%BeginExpansion
\mathbb{N}
%EndExpansion
}$ of maps $\sigma _{n}:\Gamma \rightarrow S_{n}$ such that for every $%
\gamma ,\gamma ^{\prime }\in \Gamma $%
\begin{equation*}
\lim_{n\rightarrow +\infty }d\left( \sigma _{n}\left( \gamma \gamma ^{\prime
}\right) ,\sigma _{n}(\gamma )\sigma _{n}\left( \gamma ^{\prime }\right)
\right) =0
\end{equation*}%
and 
\begin{equation*}
\lim_{n\rightarrow +\infty }d\left( \sigma _{n}(\gamma ),1\right) =1,
\end{equation*}%
for all $\gamma \neq 1_{\Gamma }$. Define $h_{\Sigma ,F,\delta }\left( 
\mathcal{P},\Gamma ,X\right) $ to be%
\begin{equation*}
\limsup_{n\rightarrow +\infty }%
\frac{1}{n}\mathrm{log}\left( H_{F,\delta }\left( \sigma _{n},\mathcal{P}%
,\Gamma ,X\right) \right) \text{.}
\end{equation*}%
The \emph{entropy} $h_{\Sigma }\left( \Gamma ,X\right) $ 
\index{entropy!sofic group action}of the action $\Gamma \curvearrowright X$
relative to the sofic approximation sequence $\Sigma $ is the infimum of $%
h_{\Sigma ,F,\delta }\left( \mathcal{P},\Gamma ,X\right) $ when $F$ varies
among all finite subsets of $\Gamma $ and $\delta $ varies among all
positive real numbers. Observe that, as before, $h_{\Sigma }\left( \Gamma
,X\right) $ is at most $\mathrm{log}\left\vert \mathcal{P}\right\vert $, it
does not depend on the generating finite clopen partition chosen, and it is
invariant by topological conjugation. It is shown in Section 5 of \cite%
{KerrLiDynamical} using the so called Rokhlin lemma for sofic approximations
of countable amenable groups (see Section 4 of \cite{KerrLiDynamical}) that
the sofic entropy associated with any sofic approximation sequence of an
amenable group coincide with the classical notion of entropy for actions of
amenable groups. Nonetheless the entropy of an action of a nonamenable sofic
group can in general depend on the choice of the sofic approximation
sequence.

\subsection{Bernoulli actions of sofic groups}

Suppose that $\Gamma $ is a sofic%
\index{group!sofic} group, $A$ is a finite alphabet, and $\Gamma
\curvearrowright A^{\Gamma }$ is the Bernoulli action%
\index{Bernoulli shift!sofic} of $\Gamma $ with alphabet $A$ of cardinality $%
k$. The entropy $h_{\Sigma }\left( \Gamma ,A^{\Gamma }\right) $ of the
Bernoulli action with respect to any sofic approximation sequence $\Sigma $
is $\mathrm{log}(k)$. In fact consider as before the clopen partition $%
\mathcal{P}$ of $A^{\Gamma }$ consisting of the sets 
\begin{equation*}
X_{a}=\left\{ \left( a_{\gamma }\right) _{\gamma \in \Gamma }\left\vert
a_{1_{\Gamma }}=a\right. \right\}
\end{equation*}%
for $a\in A$.

In the following we will say that $\sigma :\Gamma \rightarrow S_{n}$ is a 
\emph{good enough sofic approximation} if 
\begin{equation*}
d\left( \sigma _{\gamma \eta },\sigma _{\gamma }\sigma _{\eta }\right)
<\varepsilon
\end{equation*}%
for every $\eta $,$\gamma \in F$, where $F$ is a large enough finite subset
of $\Gamma $ and $\varepsilon $ is a small enough positive real number. It
is not difficult to see that if $\sigma $ is a good enough sofic
approximation then 
\begin{equation*}
H_{F,\delta }\left( \mathcal{P},\Gamma ,A^{\Gamma }\right) =k^{n}%
\text{.}
\end{equation*}%
It follows that%
\begin{equation*}
h_{\Sigma }\left( \Gamma ,A^{\Gamma }\right) =\mathrm{log}(k)\text{.}
\end{equation*}%
As in the case of amenable groups, Gottschalk's conjecture%
\index{conjecture!Gottschalk's surjunctivity} for sofic%
\index{group!sofic} groups can be proved by showing that a proper subshift
of the Bernoulli shift has entropy strictly smaller than $\mathrm{\mathrm{log%
}}(k)$. Suppose thus that $Y$ is a proper closed invariant subspace of $%
A^{\Gamma }$. It is easy to see that if some element of $A$ does not appear
as a digit in any element of $Y$ then $h_{\Sigma }\left( \Gamma ,Y\right)
\leq \mathrm{log}\left( k-1\right) <\mathrm{log}(k)$. Thus without loss of
generality we can assume that all elements of $A$ appear as digits in some
element of $Y$. Since $Y$ is a proper closed subset of $X$ there is a finite
subset $F$ of $\Gamma $ such that the set $Y_{F}$ of restrictions of
elements of $Y$ to $F$ is a proper subset of $A^{F}$. We will prove that, if 
$N$ is the cardinality of $F$, then%
\begin{equation*}
\inf_{\delta >0}h_{\Sigma ,F,\delta }\left( \mathcal{P},\Gamma ,Y\right)
\leq \mathrm{log}(k)-%
\frac{1}{N^{2}}\mathrm{log}\left( \frac{k^{N}}{k^{N}-1}\right) \text{.}
\end{equation*}%
Fix an element $\left( b_{\gamma }\right) _{\gamma \in F}$ of $%
A^{F}\left\backslash Y_{F}\right. $, a function $\sigma :\Gamma \rightarrow
S_{n}$ for some $n\in 
%TCIMACRO{\U{2115} }%
%BeginExpansion
\mathbb{N}
%EndExpansion
$, and $\eta ,\delta >0$ such that $\delta \left\vert F\right\vert <\eta <%
\frac{1}{2\left\vert F\right\vert ^{2}+1}$

\begin{lemma}
\label{Lemma: bad colors}Suppose that $\left( c_{i}\right) _{i\in n}\in
A^{n} $. If there is $\left( x_{i}\right) _{i\in n}\in Y^{n}$ such that for
every $\gamma \in F$%
\begin{equation*}
\gamma ^{-1}x_{i}\in X_{c_{\sigma _{\gamma }^{-1}(i)}}
\end{equation*}%
for a proportion of $i\in n$ larger than $\left( 1-\delta \right) $ then%
\begin{equation*}
\frac{1}{n}\left\vert \bigcap_{\gamma \in F}\sigma _{\gamma }\left[ \left\{
i\in n:c_{i}=b_{\gamma }\right\} \right] \right\vert <\delta \left\vert
F\right\vert \text{.}
\end{equation*}
\end{lemma}

\begin{proof}
Define%
\begin{equation*}
B=\bigcap_{\gamma \in F}\sigma _{\gamma }\left[ \left\{ j\in
n:c_{j}=b_{\gamma }\right\} \right] 
\end{equation*}%
and suppose by contradiction that%
\begin{equation*}
\frac{\left\vert B\right\vert }{n}\geq \delta \left\vert F\right\vert \text{.%
}
\end{equation*}%
Observe that there is a subset $C$ of $n$ such that%
\begin{equation*}
\frac{1}{n}\left\vert C\right\vert >1-\delta \left\vert F\right\vert 
\end{equation*}%
and for every $i\in C$ and $\gamma \in F$ 
\begin{equation*}
\gamma ^{-1}x_{i}\in X_{c_{\sigma _{\gamma }^{-1}(i)}}\text{.}
\end{equation*}%
It follows that there is $i\in C\cap B$. Define $y=x_{i}$ and observe that
for every $\gamma \in F$%
\begin{equation*}
y\in \bigcap_{\gamma \in F}\gamma X_{\sigma _{\gamma }^{-1}(i)}
\end{equation*}%
where%
\begin{equation*}
c_{\sigma _{\gamma }^{-1}(i)}=b_{\gamma }\text{.}
\end{equation*}%
This contradicts the fact that $\left( b_{\gamma }\right) _{\gamma \in
F}\notin Y_{F}$.
\end{proof}

Denote by $Z$ the set of $i\in n$ such that for every distinct $\gamma
,\gamma ^{\prime }\in F$ one has%
\begin{equation*}
\sigma _{\gamma }(i)\neq \sigma _{\gamma }(i)\text{.}
\end{equation*}%
Assuming that $\sigma $ is a good enough sofic approximation we have%
\begin{equation*}
\frac{1}{n}\left\vert Z\right\vert >1-\eta
\end{equation*}%
For every $i\in Z$ consider the set%
\begin{equation*}
V_{i}=\left\{ \sigma _{\gamma }^{-1}(i):\gamma \in F\right\}
\end{equation*}%
and observe that $\left\vert V_{i}\right\vert =F$. Take a maximal subset $%
Z^{\prime }$ of $Z$ subject to the condition that $V_{i}$ and $V_{j}$ are
disjoint for distinct $i$ and $j$ in $Z^{\prime }$. Then by maximality%
\begin{equation*}
Z\subset \bigcup_{\gamma ,\gamma ^{\prime }\in F}\sigma _{\gamma
}^{-1}\sigma _{\gamma ^{\prime }}\left[ Z^{\prime }\right]
\end{equation*}%
and hence 
\begin{equation*}
\frac{1}{n}\left\vert Z^{\prime }\right\vert \geq \frac{\left\vert
Z\right\vert }{n\left\vert F\right\vert ^{2}}\geq \frac{1-\eta }{\left\vert
F\right\vert ^{2}}
\end{equation*}%
Denote by $S$ the set of choices of colors $c=\left( c_{i}\right) _{i\in
n}\in A^{n}$ for which there is some $Z^{\prime \prime }\subset Z^{\prime }$
such that%
\begin{equation*}
\frac{1}{n}\left\vert Z^{\prime \prime }\right\vert >\eta
\end{equation*}%
and for every $\gamma \in F$ and $i\in Z^{\prime \prime }$ one has that $%
c_{\sigma _{\gamma }^{-1}(i)}=b_{\gamma }$. For any such $c$ one has that%
\begin{equation*}
\bigcap_{\gamma \in F}\sigma _{\gamma }^{-1}\left( \left\{ j\in
n:c_{j}=b_{\gamma }\right\} \right) \supset Z^{\prime \prime }
\end{equation*}%
and hence%
\begin{equation*}
\frac{1}{n}\left\vert \bigcap_{\gamma \in F}\sigma _{\gamma }^{-1}\left(
\left\{ j\in n:c_{j}=b_{\gamma }\right\} \right) \right\vert \geq \frac{1}{n}%
\left\vert Z^{\prime \prime }\right\vert >\eta \text{.}
\end{equation*}

By Lemma \ref{Lemma: bad colors} this shows that when $\sigma $ is a good
enough sofic approximation

\begin{equation*}
H\left( Y,F,\delta ,\sigma \right) \leq \left\vert A^{n}\left\backslash
S\right. \right\vert \text{.}
\end{equation*}%
Observe that if $c=\left( c_{i}\right) _{i\in n}\in A^{n}\left\backslash
S\right. $ then for every $Z^{\prime \prime }\subset Z^{\prime }$ such that $%
\left\vert Z^{\prime \prime }\right\vert >\eta n$ there is $i\in Z^{\prime
\prime }$ such that the sequences $\left( b_{\gamma }\right) _{\gamma \in F}$
and $\left( c_{\sigma _{\gamma }^{-1}\left( i\right) }\right) _{\gamma \in
F} $ are distinct. This implies that there is $W\subset Z^{\prime }$ such
that $\left\vert W\right\vert =\left\vert Z^{\prime }\right\vert
-\left\lfloor \eta n\right\rfloor $ and for every $i\in W$ the sequence $%
\left( \sigma _{\gamma }^{-1}(i)\right) _{\gamma \in F}$ differs from the
sequence $\left( b_{\gamma }\right) _{\gamma \in F}$. Therefore the number
of elements of $A^{n}\left\backslash S\right. $ is bounded from above by

\begin{equation}
\binom{\left\vert Z^{\prime }\right\vert }{\left\vert Z^{\prime }\right\vert
-\left\lfloor \eta n\right\rfloor }\left( \left\vert A\right\vert
^{\left\vert F\right\vert }-1\right) ^{\left\vert Z^{\prime }\right\vert
-\left\lfloor \eta n\right\rfloor }\left\vert A\right\vert ^{n-\left(
\left\vert Z^{\prime }\right\vert -\left\lfloor \eta n\right\rfloor \right)
\left\vert F\right\vert }\text{\label{Equation: entropy 1}.}
\end{equation}%
Define the function $\xi \left( t\right) =-t\mathrm{log}\left( t\right) $
for $t\in \left[ 0,1\right] $, and observe that $\xi $ is a concave
function. By Stirling's approximation formula (see \cite{Romik} for a very
short proof) the expression \ref{Equation: entropy 1} is in turn bounded
from above by

\begin{equation}
C\mathrm{exp}\left( \left\vert Z^{\prime }\right\vert \xi \left( 1-\frac{%
\eta n}{\left\vert Z^{\prime }\right\vert }\right) +\left\vert Z^{\prime
}\right\vert \xi \left( \frac{\eta n}{\left\vert Z^{\prime }\right\vert }%
\right) \right) \left\vert A\right\vert ^{n}\left( \frac{\left\vert
A\right\vert ^{\left\vert F\right\vert }}{\left\vert A\right\vert
^{\left\vert F\right\vert }-1}\right) ^{-\left( \left\vert Z^{\prime
}\right\vert -\eta n\right) }\text{\label{Equation: entropy 2}}
\end{equation}%
for some constant $C$ not depending on $\left\vert Z^{\prime }\right\vert $
or $n$. From the fact that $\xi $ is a concave function and 
\begin{equation*}
\frac{1}{n}\left\vert Z^{\prime }\right\vert \geq \frac{\left( 1-\eta
\right) }{\left\vert F\right\vert ^{2}}>2\eta
\end{equation*}%
we obtain the estimate%
\begin{equation*}
\xi \left( 1-\frac{\eta n}{\left\vert Z^{\prime }\right\vert }\right) +\xi
\left( \frac{\eta n}{\left\vert Z^{\prime }\right\vert }\right) \leq \xi
\left( 1-\frac{\eta \left\vert F\right\vert ^{2}n}{1-\eta }\right) +\xi
\left( \frac{\eta \left\vert F\right\vert ^{2}n}{1-\eta }n\right) \text{.}
\end{equation*}%
It follows that the quantity \eqref{Equation: entropy 2} is smaller than or
equal to%
\begin{equation*}
C\mathrm{exp}\left( n\xi \left( 1-\frac{\eta \left\vert F\right\vert ^{2}}{%
1-\eta }\right) +n\xi \left( \frac{\eta \left\vert F\right\vert ^{2}}{1-\eta 
}\right) \right) \left\vert A\right\vert ^{n}\left( \frac{\left\vert
A\right\vert ^{\left\vert F\right\vert }}{\left\vert A\right\vert
^{\left\vert F\right\vert }-1}\right) ^{-\left( \frac{1-\eta }{\left\vert
F\right\vert ^{2}}-\eta \right) n}\text{.}
\end{equation*}%
Thus%
\begin{eqnarray*}
&&\frac{1}{n}\mathrm{log}\left( H_{F,\delta }\left( \sigma ,\mathcal{P}%
,\Gamma ,Y\right) \right) \\
&\leq &\xi \left( 1-\frac{\eta \left\vert F\right\vert ^{2}}{1-\eta }\right)
+\xi \left( \frac{\eta \left\vert F\right\vert ^{2}}{1-\eta }\right) +%
\mathrm{log}\left\vert A\right\vert -\left( \frac{1-\eta }{\left\vert
F\right\vert ^{2}}-\eta \right) \mathrm{log}\left( \frac{\left\vert
A\right\vert ^{\left\vert F\right\vert }-1}{\left\vert A\right\vert
^{\left\vert F\right\vert }}\right) +o(1)\text{.}
\end{eqnarray*}%
Since this is true for every good enough sofic approximation $\sigma $, if $%
\Sigma $ is any sofic approximation sequence then%
\begin{align*}
& h_{\Sigma ,\delta ,F}\left( \mathcal{P},\Gamma ,Y\right) \\
& =\limsup_{n\rightarrow +\infty }\frac{1}{n}\mathrm{log}\left( H_{F,\delta
}\left( \sigma _{n},\mathcal{P},\Gamma ,Y\right) \right) \\
& \leq \xi \left( 1-\frac{\eta \left\vert F\right\vert ^{2}}{1-\eta }\right)
+\xi \left( \frac{\eta \left\vert F\right\vert ^{2}}{1-\eta }\right) +%
\mathrm{log}\left\vert A\right\vert -\left( \frac{1-\eta }{\left\vert
F\right\vert ^{2}}-\eta \right) \mathrm{log}\left( \frac{\left\vert
A\right\vert ^{\left\vert F\right\vert }-1}{\left\vert A\right\vert
^{\left\vert F\right\vert }}\right) \text{.}
\end{align*}%
Being this true for every $\delta ,\eta <0$ such that $\delta \left\vert
F\right\vert <\eta <\frac{1}{2\left\vert F\right\vert ^{2}+1}$, it follows
that%
\begin{equation*}
\inf_{\delta >0}h_{\Sigma ,\delta ,F}\left( \mathcal{P},\Gamma ,Y\right)
\leq \mathrm{log}\left\vert A\right\vert -\frac{1}{\left\vert F\right\vert
^{2}}\mathrm{log}\left( \frac{\left\vert A\right\vert ^{\left\vert
F\right\vert }-1}{\left\vert A\right\vert ^{\left\vert F\right\vert }}\right)
\end{equation*}%
as desired.

\chapter{Connes' embedding conjecture}
\chapterauthor{Valerio Capraro}

\section{The hyperfinite \texorpdfstring{\textup{II}$_{1}$}{II1}
factor\label{Section: hyperfinite}}

Recall that a von Neumann algebra, as defined in Section \ref{Section: vN
algebras and II1 factors}, is a weakly closed *-subalgebra of the algebra $%
B\left( H\right) $ of bounded linear operators on a Hilbert space $H$. A
factor is just a von Neumann algebra whose center consists only of scalar
multiples of the identity. An infinite-dimensional factor endowed with a
(necessarily unique and weakly continuous) faithful normalized trace is a II$%
_{1}$ factor.

\begin{definition}
Suppose that $M$ and $N$ are two factors. If $F$ is a subset of $M$ and $%
\varepsilon $ is a positive real number, then an $\left( F,\varepsilon
\right) $\emph{-approximate morphism}%
\index{approximate morphism!of factors} from $M$ to $N$ is a function $\Phi
:M\rightarrow N$ such that $\Phi \left( 1\right) =1$ and for every $x,y\in F$%
:
\end{definition}

\begin{itemize}
\item $\left\Vert \Phi (x+y)-\left( \Phi (x)+\Phi (y)\right) \right\Vert
_{2}<\varepsilon $;

\item $\left\Vert \Phi (xy)-\Phi (x)\Phi (y)\right\Vert
_{2}<\varepsilon $;

\item $\left\vert \tau _{M}(x)-\tau _{N}\left( \Phi (x)\right) \right\vert
<\varepsilon $.
\end{itemize}

A \textup{II}$_{1}$ factor $M$ satisfies Connes' embedding\index{conjecture!Connes' embedding} conjecture (or CEC\ for
short) if for every finite subset $F$ of $M$ and every positive real number $%
\varepsilon $ there is a natural number $n$ and an $\left( F,\varepsilon
\right) $-approximate morphism\index{approximate morphism!of factors} $\Phi :M\rightarrow M_{n}(\mathbb{C})$.%
\index{conjecture!Connes' embedding}

\begin{definition}
\label{Definition: hyperfinite}A finite von Neumann algebra is called \emph{%
hyperfinite} 
\index{von Neumann algebra!hyperfinite} if it contains an increasing chain
of copies of matrix algebras whose union is weakly dense.
\end{definition}

\begin{exercise}
\label{Exercise: hyperfinite factor}Show that a separable hyperfinite II$%
_{1} $ factor satisfies the CEC.
\end{exercise}

Hyperfiniteness is a much stronger property than satisfying the CEC. In fact
it is a cornerstone result of Murray and von Neumann from \cite{Mu-vN4} that
there is a unique separable hyperfinite \textup{II}$_{1}$ factor up to isomorphism,
usually denoted by $\mathcal{R}$. The separable hyperfinite \textup{II}$_{1}$ factor
admits several different characterizations. It can be seen as the group von
Neumann algebra (as defined in Section \ref{Section: vN algebras and II1
factors}) of the group $S_{\infty }^{%
\text{fin}}$ of finitely supported permutations of $\mathbb{N}$.
Alternatively it can be described as the von Neumann algebra tensor product $%
\overline{\bigotimes }_{n}M_{2}(\mathbb{C})$ of countably many copies of the
algebra of $2\times 2$ matrices. This description enlightens the useful
property that $\mathcal{R}$ is tensorially self-absorbing, i.e.\ $\mathcal{R}%
\simeq \mathcal{R}\overline{\otimes }\mathcal{R}$. A deep result of Connes
from \cite{Co76} asserts that $\mathcal{R}$ is also the unique \textup{II}$_{1}$
factor that embeds in any other separable \textup{II}$_{1}$ factor. %\footnote{%
%The tensor product of von Neumann algebras can be defined very easily. Let $%
%H_{1},\ldots ,H_{n}$ be Hilbert spaces with inner products $\langle \cdot
%,\cdot \rangle _{1},\ldots ,\langle \cdot ,\cdot \rangle _{n}$,
%respectively. We can define an inner product on the algebraic tensor product
%of the $H_{i}$'s as follows 
%\begin{equation*}
%\langle \xi _{1}\otimes \ldots \otimes \xi _{n},\eta _{1}\otimes \ldots
%\otimes \eta _{n}\rangle =\langle \xi _{1},\eta _{1}\rangle _{1}\cdot \ldots
%\cdot \langle \xi _{n},\eta _{n}\rangle _{n}
%\end{equation*}%
%for all $\xi _{i},\eta _{i}\in H_{i}$. The completion of the algebraic
%tensor product of the $H_{i}$'s with respect to this inner product is a
%Hilbert space, called tensor product of the $H_{i}$'s and denoted by $%
%H_{1}\otimes \ldots \otimes H_{n}$. Now, let $M_{1},\ldots ,M_{n}$ be von
%Neumann algebras acting on $H_{1},\ldots ,H_{n}$, respectively. Let $M_{0}$
%be the *-algebra acting on $H_{1}\otimes \ldots \otimes H_{n}$ of all finite
%sums of operators of the form $x_{1}\otimes \ldots \otimes x_{n}$, with $%
%x_{i}\in M_{i}$. This is a unital *-subalgebra of $B(H_{1}\otimes \ldots
%\otimes H_{n})$ and therefore its weak closure is a von Neumann algebra,
%which is called \emph{tensor product} of the $M_{i}$'s and denoted by $M_{1}%
%\bar{\otimes}\ldots \bar{\otimes}M_{n}$.\newline
%}.

In the following \emph{all }\textup{II}$_{1}$\emph{\ factors are assumed to be
separable}, apart from ultrapowers of separable \textup{II}$_{1}$ factors that are
never separable by Exercise \ref{Exercise: nonseparable ultraproduct II1
factors}. Moreover $\mathcal{R}$ will denote the (unique up to isomorphism)
hyperfinite separable \textup{II}$_{1}$ factor.

The CEC can be equivalently reformulated in terms of local representability
into $\mathcal{R}$.

\begin{definition}
A \textup{II}$_{1}$ factor $M$ is \emph{locally representable%
\index{factor!locally representable} }in a \textup{II}$_{1}$ factor $N$ if for every finite
subset $F$ of $M$ and for every positive real number $\varepsilon $ there is
an $\left( F,\varepsilon \right) $-approximate morphism\index{approximate morphism!of factors} from $M$ to $N$.
\end{definition}

\begin{exercise}
\label{Exercise: equivalence CEC 1}%
\index{conjecture!Connes' embedding}Show that a separable \textup{II}$_{1}$ factor
satisfies the CEC if and only if it is locally representable into $\mathcal{R%
}$.
\end{exercise}

Local representability can be equivalently reformulated in terms of
embedding into an ultrapower, or in terms of values of universal sentences.
(The notion of formula and sentence for tracial von Neumann algebras has been defined in
Section \ref{Section: logic von neumann algebras}.) The arguments are
analogous to the ones seen in Section \ref{Section: logic invariant length
groups} and are left to the reader as Exercise \ref{Exercise: finitely
representable}. As in Section \ref{Section: logic invariant length groups} a
formula is called $\emph{universal}$%
\index{formula!universal} if it is of the form%
\begin{equation*}
\sup_{x_{1}}\sup_{x_{2}}\ldots \sup_{x_{n}}\psi (x_{1},\ldots ,x_{n})
\end{equation*}%
where no $\sup $ or $\inf $ appear in $\psi $.

\begin{exercise}
\label{Exercise: finitely representable}Suppose that $N$ and $M$ are a
separable \textup{II}$_{1}$ factors. Show that the following statements are
equivalent:

\begin{enumerate}
\item $M$ is locally representable in $N$;

\item $M$ embeds into some or, equivalently, every ultrapower of $N$;

\item $\varphi ^{M}\leq \varphi ^{N}$ for every universal sentence $\varphi $.
\end{enumerate}
\end{exercise}

The \emph{universal theory}%
\index{theory!universal} of a \textup{II}$_{1}$ factor $M$ is the function
associating with any universal sentence its value in $M$. The \emph{existential theory}\index{theory!existential} of $M$ is defined similarly, where universal sentences are replaced by existential sentences. Since $\mathcal{R}$
embeds in any other \textup{II}$_{1}$ factor, the universal theory of $\mathcal{R}$
is \emph{minimal} among the universal theory of \textup{II}$_{1}$ factors, i.e.\ $%
\varphi ^{\mathcal{R}}\leq \varphi ^{M}$ for any \textup{II}$_{1}$ factor $M$. In
view of this observation and Exercise \ref{Exercise: equivalence CEC 1}, the
particular instance of Exercise \ref{Exercise: finitely representable} when $%
N$ is the hyperfinite \textup{II}$_{1}$ factor implies that the following statements
are equivalent:%
\index{conjecture!Connes' embedding}

\begin{enumerate}
\item $M$ satisfies the CEC;

\item $M$ embeds into some or, equivalently, every ultrapower of $\mathcal{R}
$;

\item $M$ has the same universal theory as $\mathcal{R}$.
\end{enumerate}

In the terminology of \cite{FHS3}, a \textup{II}$_{1}$ factor $N$ is called \emph{%
locally\ universal}%
\index{factor!locally universal}\emph{\ }if every \textup{II}$_{1}$ factor is finitely
representable in $N$. Thus CEC asserts that the separable hyperfinite II$_{1}
$ is locally universal.%
\index{conjecture!Connes' embedding}. The existence of a locally universal II%
$_{1}$ factor, which can be regarded as a sort of weak resolution of the
CEC, is established in \cite[Example 6.4]{FHS3}. It should be noted that on
the other hand by \cite[Theorem 2]{Oz2} there is no universal \emph{separable%
} \textup{II}$_{1}$ factor, i.e.\ there is no separable \textup{II}$_{1}$ factor containing
any other separable \textup{II}$_{1}$ factor as a subfactor

\section{Hyperlinear groups and \texorpdfstring{R\u{a}dulescu's}{Radulescu's} Theorem\label{Section:
hyperlinear and radulescu}}

Hyperlinear groups%
\index{group!hyperlinear} were defined in Section \ref{Section: definition
hyperlinear groups} in terms of local approximations by unitary groups of
matrix algebras (see Definition \ref{Definition: hyperlinear}). They can
also be equivalently characterized in terms of local approximations by the
unitary group of the hyperfinite \textup{II}$_{1}$ factor. This is discussed in
Exercise \ref{Exercise: hyperlinear equivalence 3}.

\begin{exercise}
\label{Exercise: hyperlinear equivalence 3}Suppose that $\Gamma $ is a
countable discrete group. Regard $\Gamma $ as an invariant length group with
respect to the trivial length function, and the unitary group $U(\mathcal{R})
$ of $\mathcal{R}$ as an invariant length group with respect to the
invariant length described in Exercise \ref{Exercise: length unitary group}.
Then the following statements are equivalent:

\begin{enumerate}
\item $\Gamma $ is hyperlinear\index{group!hyperlinear};

\item for every finite subset $F$ of $\Gamma $ and every positive real
number $\varepsilon $ there is an $\left( F,\varepsilon \right) $%
-approximate morphism\index{approximate morphism!of invariant length groups} (as in Definition \ref{Definition: approximate
morphism}) from $\Gamma $ to $U(\mathcal{R})$;

\item there is an injective group homomorphism from $\Gamma $ to $U(\mathcal{%
R})^{\mathcal{U}}$ for some or, equivalently, any free ultrafilter $\mathcal{%
U}$ over $%
%TCIMACRO{\U{2115} }%
%BeginExpansion
\mathbb{N}
%EndExpansion
$;

\item there is a length preserving group homomorphism from $\Gamma $ to $U(%
\mathcal{R})^{\mathcal{U}}$ for some or, equivalently, any free ultrafilter $%
\mathcal{U}$ over $%
%TCIMACRO{\U{2115} }%
%BeginExpansion
\mathbb{N}
%EndExpansion
$.
\end{enumerate}
\end{exercise}

\begin{hint}
The equivalence of 1 and 2 can be can be shown using the characterization of
hyperlinear\index{group!hyperlinear} groups given by Proposition \ref{Proposition: amplification
hyperlinear}. The equivalence of 1, 3, and 4 can be established as in
Exercise \ref{Exercise: hyperlinear equivalence 1} and Exercise \ref{Exercise: hyperlinear equivalence 2}.
\end{hint}

Observe that, by Exercise \ref{Exercise: unitary group ultraproduct}, the
ultrapower $U(\mathcal{R})^{\mathcal{U}}$ of the unitary group of $\mathcal{R%
}$ with respect to the free ultrafilter $\mathcal{U}$ can be identified with
the unitary group $U(\mathcal{R})$ of the corresponding ultrapower of $%
\mathcal{R}$. Therefore, in view of Exercise \ref{Exercise: hyperlinear
equivalence 3}, $U(\mathcal{R})$ can be regarded as a \emph{universal
hyperlinear group}%
\index{group!universal hyperlinear}, analogously as the ultraproduct $%
\prod\nolimits_{\mathcal{U}}U_{n}$ of the finite rank unitary groups as we
have seen in Section \ref{Section: definition hyperlinear groups}.
Proposition \ref{Proposition: external automorphisms} allows one to infer
(cf.\ Corollary \ref{Corollary: external automorphism universal}) that if
the Continuum Hypothesis holds then the unitary group of any ultrapower of $%
\mathcal{R}$ has $2^{\aleph _{1}}$ outer automorphisms.

It is a consequence of condition (4) in Exercise \ref{Exercise: hyperlinear
equivalence 3} that if $\Gamma $ is a \textit{countable} discrete
hyperlinear\index{group!hyperlinear} group, then the group von Neumann algebra\index{group von Neumann algebra} $L\Gamma $ of $\Gamma $
(se Definition \ref{Definition: group vN algebra}) satisfies the CEC (this a
result of R\u{a}dulescu from \cite{Radulescu}). In fact a length preserving
homomorphism from $\Gamma $ to $U(\mathcal{R})$ extends by linearity to a
trace preserving embedding of $\mathbb{C}\Gamma $ to $\mathcal{R}^{\mathcal{U%
}}$. This in turn induces an embedding of $L\Gamma $ into $\mathcal{R}^{%
\mathcal{U}}$, witnessing that $L\Gamma $ satisfies the CEC.

With a little extra care one can show that any (not necessarily countable)
subgroup $\Gamma $ of the unitary group of some ultrapower of $\mathcal{R}$
has the property that the group von Neumann algebra\index{group von Neumann algebra} $L\Gamma $ of $\Gamma $
embeds into a (possibly different) ultrapower of $\mathcal{R}$. This is the
content of Theorem \ref{th:radulescu}, established by the first-named author
and P\u aunescu in \cite{Ca-Pa}.

\begin{theorem}
\label{th:radulescu} For any group $\Gamma$, the following conditions are
equivalent:

\begin{enumerate}
\item $\Gamma $ admits a group monomorphism into $U(\mathcal{R})$ for some
free ultrafilter $\mathcal{U}$ on $%
%TCIMACRO{\U{2115} }%
%BeginExpansion
\mathbb{N}
%EndExpansion
$;

\item The group von Neumann algebra\index{group von Neumann algebra} $L\Gamma $ embeds into $\mathcal{R}^{%
\mathcal{V}}$ for some (possibly different) free ultrafilter $\mathcal{V}$
on $%
%TCIMACRO{\U{2115} }%
%BeginExpansion
\mathbb{N}
%EndExpansion
$.
\end{enumerate}
\end{theorem}

If the Continuum Hypothesis holds then, as discussed in Section \ref%
{Section: logic von neumann algebras}, all the ultrapowers of $\mathcal{R}$
as isomorphic. In particular one can always pick the same ultrafilter. It is
not clear that this is still possible under the failure of the Continuum
Hypothesis (see \cite{MO2}).

In the rest of this section we will present the proof of Theorem \ref%
{th:radulescu} that involves the notion of product of ultrafilters. Let us
denote for $k\in \mathbb{N}$ and $B\subset \mathbb{N}$ the vertical section $%
\left\{ n\in \mathbb{N}:\left( k,n\right) \in B\right\} $ of $B$
corresponding to $k$ by $B_{k}$.

\begin{definition}
\label{productultrafilter}Suppose that $\mathcal{U}$, $\mathcal{V}$ are free
ultrafilters on $\mathbb{N}$. The \emph{(Fubini) product} 
\index{product!of ultrafilters} $\mathcal{U}\times \mathcal{V}$ is the free
ultrafilter on $\mathbb{N}\times \mathbb{N}$ such that $B\in \mathcal{U}%
\times \mathcal{V}$ if and only if the set of $k\in \mathbb{N}$ such that $%
B_{k}\in \mathcal{V}$ belongs to $\mathcal{U}$.
\end{definition}

\begin{exercise}\label{ex:product non commutative}
Show that the operation $\times$ is not commutative.
\end{exercise}

\begin{exercise}
\label{Exercise: product ultrafilters}Suppose that $\mathcal{U}$, $\mathcal{V%
}$ are free ultrafilters on $\mathbb{N}$. Show that $\mathcal{U}\times 
\mathcal{V}$ is a free ultrafilter on $\mathbb{N}\times \mathbb{\mathbb{N}}$%
. Moreover if $\left( a_{n,m}\right) _{n,m\in \mathbb{N}}$ is a
double-indexed sequence in $\mathbb{R}$, then%
\begin{equation*}
\lim_{n\rightarrow \mathcal{U}}\lim_{m\rightarrow \mathcal{V}%
}a_{n,m}=\lim_{\left( n,m\right) \rightarrow \mathcal{U}\times \mathcal{V}%
}a_{n,m}%
\text{.}
\end{equation*}
\end{exercise}

It follows from Exercise \ref{Exercise: product ultrafilters} that an \emph{%
iterated ultrapower} can be regarded as a single ultrapower. More generally:

\begin{proposition}
\label{prop:productultra}If $\left( M_{n,m},\tau _{n,m}\right) $ is a
double-indexed sequence of tracial von Neumann algebras, then%
\begin{equation*}
\prod\nolimits_{\mathcal{U}}\left( \prod\nolimits_{\mathcal{V}%
}M_{n,m}\right) \simeq \prod\nolimits_{\mathcal{U}\times \mathcal{V}}M_{n,m}%
\text{.}
\end{equation*}%
In particular%
\begin{equation*}
\left( M^{\mathcal{V}}\right) ^{\mathcal{U}}\simeq M^{\mathcal{U}\times 
\mathcal{V}}\text{.}
\end{equation*}
\end{proposition}

\begin{exercise}
\label{ex:one} This exercise describes an amplification trick to be used in
the proof of Theorem \ref{th:radulescu}.

\begin{enumerate}
\item Let $z\in C$ such that $|z|\leq 1$ and $|z+1|=2$. Prove that $z=1$.

\item Let $u$ be a unitary in a \textup{II}$_{1}$ factor with trace $1$. Show that $%
u=1$.

\item Let $u$ be a unitary in a \textup{II}$_{1}$ factor $N$ and let $M_{2}(N)$ be
the factor of two-by-two matrices with entries in $N$ and denote by $\tau $
its normalized trace. Consider the element 
\begin{equation*}
u^{\prime }=\left( 
\begin{array}{cc}
u & 0 \\ 
0 & 1%
\end{array}%
\right) .
\end{equation*}%
Show that $|\tau (u^{\prime })|<1$.
\end{enumerate}
\end{exercise}

\begin{proof}[Proof of Theorem \protect\ref{th:radulescu}]
(2) implies (1) is trivial. Conversely, let $\Gamma $ be a group and $\theta
:\Gamma \rightarrow U(\mathcal{R}^{\mathcal{U}})$ a group monomorphism.
Define the new group homomorphism $\theta _{1}:\Gamma \rightarrow U(M_{2}(%
\mathcal{R}^{\mathcal{U}}))$ by%
\begin{equation*}
\theta ^{\prime }(\gamma )=\left( 
\begin{array}{cc}
\theta (\gamma ) & 0 \\ 
0 & 1_{\Gamma }%
\end{array}%
\right) \text{.}
\end{equation*}%
By Exercise \ref{ex:one}, this monomorphism verifies the property that $%
|\tau (\theta ^{\prime }(\gamma ))|<1$, for all $\gamma \neq 1$. Recall that 
$\mathcal{R}$ is up to isomorphism, the unique (separable) hyperfinite II$%
_{1}$ factor (see Definition \ref{Definition: hyperfinite}). In particular $%
\mathcal{R}$ is isomorphic to $\mathcal{R}\bar{\otimes}\mathcal{R}$, as well
as to the algebra $M_{2}(\mathcal{R})$ of $2\times 2$ matrices over $%
\mathcal{R}$. As a consequence one obtains an isomorphism between $\mathcal{R%
}^{\mathcal{U}}$, $(\mathcal{R}\bar\otimes\mathcal R)^{\mathcal{U}}$, and $M_{2}(\mathcal{R})$.
This allows one to regard $\theta _{1}$ as a group monomorphism from $\Gamma 
$ to the unitary group of $\mathcal{R}^{\mathcal{U}}$ itself.

Now define $\theta _{2}$ to be the map from $\Gamma $ to the unitary group
of $(\mathcal{R})^{\mathcal{U}}$ defined in the following way: If $\theta
_{1}(\gamma )$ has a representative sequence of unitaries $\left(
u_{n}\right) _{n\in \mathbb{N}}\in U(\mathcal{R})$ (this exists by Exercise %
\ref{Exercise: unitary group ultraproduct}) then $\theta _{2}(\gamma )$ has
representative sequence%
\begin{equation*}
\left( u_{n}\otimes u_{n}\right) _{n\in \mathbb{N}}\text{.}
\end{equation*}%
Identifying $(\mathcal{R}\bar\otimes\mathcal R)^{\mathcal{U}}$ with $\mathcal{R}^{\mathcal{U}}$
one can regard $\theta _{2}$ as a map from $\Gamma $ to the unitary group of 
$\mathcal{R}^{\mathcal{U}}$. Define analogously $\theta _{n}:\Gamma
\rightarrow U(\mathcal{R})$ for every natural number $n$ taking the tensor
product of $\theta _{1}$ by itself (coordinatewise). Observe that $\theta
_{n}$ is a monomorphism that moreover has the property that 
\begin{equation*}
|\tau (\theta _{n}(\gamma ))|=|\tau (\theta _{1}(\gamma ))|^{n}
\end{equation*}%
for every $\gamma \in \Gamma $. Next define $\theta _{\infty }$ from $\Gamma 
$ the unitary group of the iterated ultrapower $(\mathcal{R}^\mathcal U)^{\mathcal{U}}$
(which by Exercise \ref{Exercise: product ultrafilters} is isomorphic to $%
\mathcal{R}^{\mathcal{U}\times \mathcal{U}}$) assigning $\gamma $ to the
element of $(\mathcal{R}^\mathcal U)^{\mathcal{U}}$ having%
\begin{equation*}
\left( \theta _{n}(\gamma )\right) _{n\in \mathbb{N}}
\end{equation*}%
as representative sequence. It is easy to check that $\theta _{\infty }$ is
a group monomorphism verifying the additional property that $\tau \left(
\theta (\gamma )\right) =0$ for every $\gamma \neq 1_{\Gamma }$. It is easy
to infer from this, as in the discussion after Exercise \ref{Exercise:
hyperlinear equivalence 3}, that there is a trace-preserving embedding of $%
L\Gamma $ into $(\mathcal{R}^\mathcal U)^{\mathcal{U}}\simeq \mathcal{R}^{\mathcal{U}%
\times \mathcal{U}}$.
\end{proof}

\section{Kirchberg's theorem}\label{se:kirchberg}

As mentioned in the Introduction, Connes' embedding\index{conjecture!Connes' embedding} conjecture can be reformulated in several different ways. The purpose of this section is to present probably the most unexpected of such reformulations, originally proved by Kirchberg in \cite{Ki}.
\begin{theorem}\label{th:kirchberg}
The following statements are equivalent:
\begin{enumerate}
\item Connes' embedding conjecture holds true,
\item 
$
C^*(\mathbb F_\infty)\otimes _{\mathrm{max}} C^*(\mathbb F_\infty)=C^*(\mathbb F_\infty)\otimes _{\mathrm{min}} C^*(\mathbb F_\infty).
$
\end{enumerate}
\end{theorem}

We refer the reader to Appendix \ref{suse:tensor product} for all definitions needed to understand the statement of Kirchberg's theorem, as the minimal and maximal tensor product of C*-algebras and the full C*-algebra associated to a locally compact group.

We are not going to present the original proof, but a more recent and completely different one provided by Haagerup and Winsl\o w \cite{Ha-Wi1}, \cite{Ha-Wi2}. This is fundamentally a topological proof relying on the use of the Effros-Marechal topology on the space of von Neumann algebras. To present a completely self-contained proof of this theorem is pretty much impossible, since it relies on the Tomita-Takesaki modular theory. However, we will present the idea in quite many details, leaving out a few technical lemmas, whose proof is highly non-trivial. 

\subsection{Effros-Marechal topology on the space of von Neumann algebras}\label{suse:effros}

Let $H$ be a Hilbert space and $vN(H)$ be the set of von Neumann algebras
acting on $H$, that is the set of von Neumann subalgebras of $B(H)$. The Effros-Marechal topology on $vN(H)$, first introduced in Effros in \cite{Ef} and Marechal in \cite{Ma}, can be defined is several different equivalent ways and we will in fact be playing with two different definitions. 

Let us first recall some terminology. Given a net $\{C_a\}_{a\in A}$ on a direct set $A$, we say that a property P is \emph{eventually}\index{eventually} satisfied if there is $a\in A$ such that, for all $b\geq a$, $C_b$ satisfies property P. We say that P is \emph{frequently}\index{frequently} satisfied if for all $a\in A$ there is $b\geq a$ such that $C_b$ satisfies property $P$.\\

\textbf{First definition of the Effros-Marechal topology.}\index{topology!Effros-Marechal} Let $X$ be a compact Hausdorff space, $c(X)$ be the set of closed subsets of $X$ and $\mathcal N(x)$ be the filter of neighborhoods of a point $x\in X$. $\mathcal N(x)$ is a direct set, ordered by inclusion. Let $\{C_a\}$ be a net of subsets of $c(X)$, define

\begin{align}
\underline{\lim}C_a=\left\{x\in X:\forall N\in\mathcal N(x), N\cap
C_a\neq\emptyset\text{ eventually}\right\},
\end{align}
\begin{align}
\overline{\lim}C_a=\left\{x\in X:\forall N\in\mathcal N(x), N\cap
C_a\neq\emptyset\text{ frequently}\right\}.
\end{align}

It is clear that $\underline{\lim}C_\alpha\subseteq\overline{\lim}C_\alpha$, but the other inclusion is far from being true in general.

\begin{exercise}
Find an explicit example of a sequence $C_n$ of closed subsets of the real interval $[0,1]$ such that $\underline{\lim}C_n\subsetneq\overline{\lim}C_n$.
\end{exercise}

Effros proved in \cite{Ef2} that there is only one topology on $c(X)$, whose
convergence is described by the condition:
\begin{align}
C_a\rightarrow
C\qquad\text{ if and only if }\qquad\overline{\lim}C_a=\underline{\lim}C_a=C.
\end{align}

\begin{exercise}\label{exer:compactness}
Let $M$ be a von Neumann subalgebra of $B(H)$. Denote by $Ball(M)$ the unit ball of $M$. Prove that $Ball(M)$ is weakly compact in $Ball(B(H))$.
\end{exercise}

The previous exercise allows to use Effros' convergence in our setting.

\begin{definition}\label{second}
Let $\{M_a\}\subseteq vN(H)$ be a net. The Effros-Marechal topology
is described by the following notion of convergence:
\begin{align}
M_a\rightarrow M
\qquad\text{ if and only if }\qquad\overline{\lim}Ball(M_a)=\underline{\lim}Ball(M_a)=Ball(M).
\end{align} 
\end{definition}

\textbf{Second definition of the Effros-Marechal topology.} Recall that the strong* topology on $B(H)$ is the weakest locally convex topology making the maps%
\begin{equation*}
x\mapsto \left\Vert x\xi \right\Vert\qquad\text{and}\qquad x\mapsto \left\Vert x^*\xi \right\Vert
\end{equation*}%
continuous for every $\xi \in H$. Let $x\in B(H)$ and let $so^*(x)$ denote the filter of neighborhoods of $x$ with respect to the strong* topology.

\begin{definition}
Let $\{M_a\}\subseteq vN(H)$ be a net. We set
\begin{align}
\lim\inf M_a=\left\{x\in B(H):\forall U\in so^*(x), U\cap
M_a\neq\emptyset\text{ eventually}\right\}.
\end{align} 
\end{definition}

Observe that $\lim\inf M_a$ is obviously so*-closed, contains the identity and it is closed under involution, scalar product and sum, since these operations are so*-continuous. Haagerup and Winsl\o w proved in \cite{Ha-Wi1}, Lemma 2.2 and Theorem 2.3, that $\lim\inf M_a$ is also closed under multiplication. Therefore $\lim\inf M_a$ is a von Neumann algebra. Moreover, by Theorem 2.6 in \cite{Ha-Wi1}, $\lim\inf M_a$ can be seen as the largest von Neumann algebra whose unit ball is contained in
$\underline{\lim}Ball(M_a)$. This suggests to define $\lim\sup M_a$ as
the smallest von Neumann algebra whose unit ball contains
$\overline{\lim}Ball(M_a)$, that is clearly
$(\overline{\lim}Ball(M_a))''$. Indeed, the double commutant of a subset of $B(H)$ is always a *-algebra and the double commutant theorem of von Neumann states that this is the smallest von Neumann algebra containing the set. So we are led to the following

\begin{definition}
Let $\{M_a\}\subseteq vN(H)$ be a net. We set
\begin{align}
\lim\sup M_a:=\left(\overline{\lim}Ball(M_a)\right)''.
\end{align} 
\end{definition}

\begin{definition}
The Effros-Marechal topology\index{topology!Effros-Marechal} on $vN(H)$ is described by the
following notion of convergence:
$$
M_a\rightarrow
M\qquad\text{ if and only if }\qquad\lim\inf
M_a=\lim\sup M_a=M.
$$
\end{definition}

These two definitions of the Effros-Marechal topology are shown to be equivalent in \cite{Ha-Wi1}, Theorem 2.8.\\

Connes' embedding conjecture\index{conjecture!Connes' embedding} regards separable
II$_1$ factors, that is, factors with a faithful representation into $B(H)$, with $H$ is separable. Assuming separability on $H$, the Effros-Marechal topology on $vN(H)$
turns out to be separable and induced by a complete metric (i.e.\ $vN(H)$ is a
Polish space). One possible distance is given by the
Hausdorff distance between the unit balls:

\begin{align}
d(M,N)=\max\left\{\sup_{x\in Ball(M)}\left\{\inf_{y\in Ball(N)}d(x,y)\right\},\sup_{x\in
Ball(N)}\left\{\inf_{y\in Ball(M)}d(x,y)\right\}\right\},
\end{align}

where $d$ is a metric on the unit ball of $B(H)$ which induces the
weak topology \cite{Ma}. Since $vN(H)$ is separable, we may express its topology by using sequences instead of nets. This turns out to be particularly useful, since, if $\{M_a\}=\{M_n\}$ is a sequence, then the definition of the Effros-Marechal topology may be
simplified by making use of the second definition. One has

\begin{align}\label{eq:strong}
\liminf M_n=\left\{x\in B(H): \exists\{x_n\}\in\prod
M_n\text{ such that } x_n\rightarrow^{s^*}x\right\}.
\end{align}

We now state the main theorem of \cite{Ha-Wi2}. A part from being of intrinsic interest, it allows to reformulate Kirchberg's theorem in the form that we are going to prove.

Let us fix some notation: $\mathfrak F_{\text{I}_{\text{fin}}}$ is the set of finite factors of \emph{type I}\index{factor!type I$_{\text{fin}}$} acting on $H$, that is the set of von Neumann factors that are isomorphic to a matrix algebra. $\mathfrak F_{\text{I}}$ is the set of type I factors\index{factor!type I} acting on $H$, that is the set of von Neumann factors that are isomorphic to some $B(H)$, with $H$ separable. $\mathfrak F_{\text{AFD}}$ is the set of approximately finite dimensional or AFD factors\index{factor!AFD} acting on $H$, that is the set of factors containing an increasing chain of matrix algebras whose union is weakly dense. Finally, $\mathfrak F_{\text{inj}}$ is the set of injective factors\index{factor!injective} acting on $H$, that is the set of factors that are the image of a bounded linear projection of norm $1$, $P:B(H)\rightarrow M$.

\begin{theorem}{\bf(Haagerup-Winsl\o w)}\label{haa}
The following statements are equivalent:
\begin{enumerate}
\item $\mathfrak F_{\text{I}_{\text{fin}}}$ is dense in $vN(H)$,
\item $\mathfrak F_{\text{I}}$ is dense in $vN(H)$,
\item $\mathfrak F_{\text{AFD}}$ is dense in $vN(H)$,
\item $\mathfrak F_{\text{inj}}$ is dense in $vN(H)$,
\item Connes' embedding conjecture is true\index{conjecture!Connes' embedding}.
\end{enumerate}
Moreover, a separable II$_1$ factor $M$ is embeddable into
$\mathcal{R}^{\mathcal U}$ if and only if $M\in\overline{\mathfrak F_{\text{inj}}}$.
\end{theorem}
Since $\mathfrak F_{I_{\text{fin}}}\subseteq\mathfrak F_I\subseteq\mathfrak F_{\text{AFD}}$, the implications
$(1)\Rightarrow(2)\Rightarrow(3)$ are trivial. The implication
$(3)\Rightarrow(1)$ follows from the fact that AFD factors contain, by mere definition, an increasing chain of type I$_{\text{fin}}$ factors, whose union is weakly
dense and from the first definition of the Effros-Marechal topology
(Definition \ref{second}). The equivalence between (3) and (4) is a theorem
by Alain Connes proved in \cite{Co76}. What is really new and important in Theorem \ref{haa} is the equivalence between (4) and (5), proved in \cite{Ha-Wi2}, Corollary 5.9, and the proof of the \emph{last sentence}, proved in \cite{Ha-Wi2}, Theorem 5.8. 

\subsection{Proof of Kirchberg's theorem}\label{suse:proof}

By using Theorem \ref{haa}, it is enough to prove the following statements:
\begin{enumerate}
\item If $\mathfrak F_{\text{I}_{\text{fin}}}$ is dense in $vN(H)$, then $C^*(\mathbb F_{\infty})\otimes_{\min}C^*(\mathbb F_{\infty})=C^*(\mathbb F_{\infty})\otimes_{\max}C^*(\mathbb F_{\infty})$.
\item If
$C^*(\mathbb F_{\infty})\otimes_{\min}C^*(\mathbb
F_{\infty})=C^*(\mathbb F_{\infty})\otimes_{\max}C^*(\mathbb
F_{\infty})$, then $\mathfrak F_{\text{inj}}$ is dense in $vN(H)$.\\
\end{enumerate}

\begin{proof}[Proof of (1)]
Let $\pi$ be a *-representation of the algebraic tensor product
$C^*(\mathbb F_{\infty})\odot C^*(\mathbb F_{\infty})$ into
$B(H)$. Since $C^*(\mathbb F_{\infty})$ is separable, we may assume
that $H$ is separable. In this way
$$
A=\pi(C^*(\mathbb F_{\infty})\odot\mathbb
C1)\qquad\text{ and }\qquad B=\pi(\mathbb C1\odot
C^*(\mathbb F_{\infty}))
$$
belong to $B(H)$, with $H$ separable. Let $\{u_n\}$ be the
universal unitaries in $C^*(\mathbb F_{\infty})$, as in Exercises \ref{exer:universalunitaries} and \ref{exer:universalunitariestotal}. Let
$$
v_n=\pi(u_n\otimes1)\in
A\qquad\text{ and }\qquad w_n=\pi(1\otimes u_n)\in
B.
$$
Set $M=A''\in vN(H)$. By hypothesis, there exists a sequence
$\{F_m\}\subseteq\mathfrak F_{\text{I}_{\text{fin}}}$ such that $F_m\rightarrow M$ in the Effros-Marechal topology. Therefore,
$A\subseteq M=\liminf F_m=\limsup F_m$. Thus, we have
$$
A\subseteq \liminf F_m.
$$
It follows that
$$
\{v_n\}\subseteq U(A)\subseteq U(\liminf
F_m)=\underline{\lim}Ball(F_m)\cap U(B(H)),
$$
where the equality follows from \cite{Ha-Wi1}, Theorem 2.6. Let $w(x)$ and
$so^*(x)$ respectively the filters of weakly and strong* open
neighborhoods of an element $x\in B(H)$. We have just proved that
for every $n\in\mathbb N$ and $W\in w(v_n)$, one has $W\cap
Ball(F_m)\cap U(B(H))\neq\emptyset$ eventually in $m$. Let $S\in
so^*(v_n)$, by \cite{Ha-Wi1} Lemma 2.4, there exists $W\in w(v_n)$
such that $W\cap Ball(F_m)\cap U(B(H))\subseteq S\cap Ball(F_m)\cap
U(B(H))$. Now, since the first set must be eventually non empty,
also the second one must be the same. This means that we can
approximate in the strong* topology $v_n$ with elements $v_{m,n}\in U(F_m)$.

In a similar way we can find unitaries $w_{m,n}$ in $F_m'$ such that
$w_{m,n}\rightarrow^{so^*}w_n$. Indeed, we have
$$
B\subseteq A'=M'=(\limsup F_m)'=\liminf F_m',
$$
where the last equality follows by the \emph{commutant theorem} (\cite{Ha-Wi1}, Theorem 3.5). Therefore, we may repeat word-by-word the previous argument.

Now let $m$ be fixed, $\pi_{m,1}$ be a
representation of $C^*(\mathbb F_{\infty})$ mapping $u_n$ to
$v_{m,n}$ and $\pi_{m,2}$ be a representation of $C^*(\mathbb
F_{\infty})$ mapping $u_n$ to $w_{m,n}$. We can find these
representations because the $u_n's$ are free and because any representation of $G$ extends to a
representation of $C^*(G)$. Notice that the ranges of these
representations commute, since $v_n\in A$ and $w_n\in B$ and $A,B$
commute. More precisely, the image of $\pi_{m,1}$ belongs into $C^*(F_m)$
and the image of $\pi_{m,2}$ belongs to $C^*(F_m')$. So, by the
universal property in Proposition \ref{universal}, there are unique
representations $\pi_m$ of $C^*(\mathbb
F_{\infty})\otimes_{\max}C^*(\mathbb F_{\infty})$ such that
$$
\pi_m(u_n\otimes1)=v_{m,n}\qquad\text{ and }\qquad
\pi_m(1\otimes u_n)=w_{m,n},\qquad m,n\in\mathbb N,
$$
whose image lies into $C^*(F_m,F_m')$.

\begin{exercise}\label{exer:nuclear}
Prove that $C^*(F_m,F_m')=F_m\otimes_{\min}F_m'$ (Hint: use Theorem 4.1.8(iii) in \cite{Ka-Ri1} and Lemma 11.3.11 in \cite{Kadison-RingroseII}).
\end{exercise}

By Exercise \ref{exer:nuclear}, $C^*(F_m,F_m')=F_m\otimes_{\min}F_m'$ and
thus $\pi_m$ splits: $\pi_m=\sigma_m\otimes\rho_m$, for some
$\sigma_m,\rho_m$ representation of $C^*(\mathbb F_{\infty})$ in
$C^*(F_m,F_m')$. Consequently, by the very definition of minimal C*-norm, $||\pi_m(x)||\leq||x||_{\min}$ for all
$m\in\mathbb N$ and $x\in C^*(\mathbb F_{\infty})\odot C^*(\mathbb
F_{\infty})$. Now, by Exercise \ref{exer:universalunitariestotal}, the sequence $\{u_n\}$ is total and therefore $\pi_m$ converges to $\pi$ in the strong* point-wise sense. Namely, for all $x\in C^*(\mathbb F_\infty)\odot C^*(\mathbb F_\infty)$, one has $\pi_m(x)\rightarrow^{so^*}\pi(x)$.

\begin{exercise}
Prove that if $x_n\in B(H)$ converges to $x\in B(H)$ in the strong* topology, then $||x||\leq\liminf||x_n||$.
\end{exercise}

Therefore
$$
||\pi(x)||\leq
\liminf||\pi_m(x)||\leq||x||_{\min},\qquad \forall x\in
C^*(\mathbb F_{\infty})\odot C^*(\mathbb F_{\infty}).
$$
Since $\pi$ is arbitrary, it follows that
$||x||_{\max}\leq||x||_{\min}$ and the proof of the first implication is complete.
\end{proof}

In order to prove (2) we need a few more definitions and preliminary results. Given two *-representation $\pi$ and $\rho$ of the same C*-algebra $A$ in the same $B(H)$, we say that they are unitarily equivalent\index{representation!unitarily equivalent}, and we write $\pi\sim\rho$, if there is $u\in U(B(H))$ such that, for all $x\in A$, one has $u\pi(x)u^*=\rho(x)$. Given a family of representations $\pi_a$ of the same C*-algebra in possibly different $B(H_a)$, we may define the direct sum\index{representation!direct sum} $\bigoplus_a\pi_a$ to be a representation of $A$ in $B(\bigoplus H_a$) in the obvious way, that is,
$$
\left(\bigoplus\pi_a\right)(x)\xi=\bigoplus\left(\pi_a(x)\xi_a\right),\qquad\text{ for all } \xi=(\xi_a)\in\bigoplus H_a.
$$
A family of representation is called \emph{separating}\index{representation!separating family of} if their direct sum is faithful (i.e.\ injective).\\
We state the following deep lemma, proved by Haagerup and Winsl\o w in \cite{Ha-Wi2}, Lemma 4.3, making use of Voiculescu's Weyl-von Neumann theorem.
\begin{lemma}\label{haag}
Let $A$ be a unital C*-algebra and $\lambda,\rho$ representations
of $A$ in $B(H)$. Assume $\rho$ is faithful and satisfies
$\rho\sim\rho\oplus\rho\oplus \cdots $. Then there exists a sequence
$\{u_n\}\subseteq U(B(H))$ such that
$$
u_n\rho(x)u_n^*\rightarrow^{s^*}\lambda
(x),\qquad \forall x\in A.
$$
\end{lemma}

The other preliminary result is a classical theorem by Choi (see \cite{Ch}, Theorem 7).

\begin{theorem}\label{ch}
Let $\mathbb F_2$ be the free group with two generators. Then
$C^*(\mathbb F_2)$ has a separating family of finite dimensional
representations.
\end{theorem}

\begin{exercise}\label{ex:freegroup}
Show that $\mathbb F_\infty$ embeds into $\mathbb F_2$.
\end{exercise}

\begin{proof}[Proof of (2)]
By using Choi's theorem and Exercise \ref{ex:freegroup} we can find a sequence
$\sigma_n$ of finite dimensional representations of $C^*(\mathbb
F_{\infty})$ such that $\sigma=\sigma_1\oplus\sigma_2\oplus \cdots $ is
faithful. Replacing $\sigma$ with the direct sum of countably many
copies of itself, we may assume that
$\sigma\sim\sigma\oplus\sigma\oplus \cdots $. Now, $\rho=\sigma\otimes\sigma$ is a faithful (by \cite{Ta1} IV.4.9) representation of
$C^*(\mathbb F_{\infty})\otimes_{\min}C^*(\mathbb F_{\infty})$, because $\rho$ splits. Moreover, $\rho$ still
satisfies $\rho\sim\rho\oplus\rho\oplus \cdots $.
Now, given $M\in vN(H)$, let $\{v_n\},\{w_n\}$ be strong* dense
sequences of unitaries in $Ball(M)$ and $Ball(M')$, respectively.
Let $\{z_n\}$ be the universal unitaries representing free generators of $\mathbb
F_{\infty}$ in $C^*(\mathbb F_{\infty})$, as in Exercises \ref{exer:universalunitaries} and \ref{exer:universalunitariestotal}. Since the $z_n$'s are free, we may find *-representations $\lambda_1$ and $\lambda_2$ of $C^*(\mathbb F_\infty)$ in $B(H)$ such that
$$
\lambda_1(z_n)=v_n\qquad\text{ and }\qquad\lambda_2(z_n)=w_n.
$$
Since the ranges of these representations commute, we can apply the universal property in Proposition \ref{universal} to find a representation
$\lambda$ of $C^*(\mathbb F_{\infty})\otimes_{\min}C^*(\mathbb
F_{\infty})$ such that
$$
\lambda(z_n\otimes1)=v_n\qquad\text{ and }\qquad \lambda(1\otimes
z_n)=w_n,\qquad \forall n\in\mathbb N,
$$
where we can use the minimal norm instead of the maximal one thank to the hypothesis of the theorem. This means that $\lambda$ and $\rho$ satisfy the hypotheses of Lemma \ref{haag} and therefore there are unitaries $u_n\in
U(B(H))$ such that
$$
u_n\rho(x)u_n^*\rightarrow^{so^*}\lambda(x),\qquad\forall
x\in C^*(\mathbb F_{\infty})\otimes_{\min}C^*(\mathbb F_{\infty}).
$$
Define
$$
M_n=u_n\rho(C^*(\mathbb F_{\infty})\otimes\mathbb C1)''u_n^*.
$$
%then
%$$
%u_n\rho(\mathbb C1\otimes C^*(\mathbb F_{\infty}))u_n^*\subseteq
%M_n'
%$$
Therefore, using also Equation \eqref{eq:strong}, we have
\begin{align}\label{eq:appr1}
\lambda(C^*(\mathbb F_{\infty})\otimes\mathbb
C1)=\liminf u_n\rho(C^*(\mathbb F_{\infty})\otimes\mathbb
C1)u_n^*\subseteq \liminf M_n.
\end{align}

Now, observe that
$$
u_n\rho(\mathbb C1\otimes C^*(\mathbb F_{\infty}))u_n^*\subseteq
M_n'
$$
and therefore
\begin{align}\label{eq:appr2}
\lambda(\mathbb C1\otimes C^*(\mathbb F_{\infty}))\subseteq \liminf
M_n'.
\end{align}
Since $\liminf M_a$ is always a von Neumann
algebra, the inclusions in \eqref{eq:appr1} and \eqref{eq:appr2} still hold passing to the weak closure. Therefore, by \eqref{eq:appr1}, we obtain
\begin{align}\label{eq:appr3}
M=\lambda(C^*(\mathbb F_{\infty})\otimes\mathbb C1)''\subseteq
\liminf M_n,
\end{align}
where the equality with $M$ follows from the fact that the $v_n$'s are strong* dense in $Ball(M)$. Analogously, by \eqref{eq:appr2}, we obtain
\begin{align}\label{eq:appr4}
M'=\lambda(\mathbb C1\otimes C^*(\mathbb F_{\infty}))''\subseteq
\liminf M_n',
\end{align}
where the equality with $M'$ follows from the fact that the $w_n$'s are strong* dense in $Ball(M')$.

Now, using \eqref{eq:appr3} and \eqref{eq:appr4} and applying the commutant theorem (\cite{Ha-Wi1}, Theorem 3.5), we get

$$
M=M''\supseteq\left(\liminf M_n'\right)'=\limsup(M_n)''=\limsup M_n.
$$

Therefore

$$
\limsup M_n\subseteq M\subseteq\liminf M_n,
$$

i.e.\ $M_n\rightarrow M$ in the Effros-Marechal topology.
Therefore, we have proved that every von Neumann algebra $M$ can be approximated by von Neumann algebras $M_n$ that are constructed as strong closure of faithful representations that are direct sum of countably many finite-dimensional representations. Such von Neumann algebras are injective and therefore, we have proved that $vN_{\text{inj}}(H)$ is dense in $vN(H)$. Now, by \cite{Ha-Wi1}, Theorem 5.2, $vN_{\text{inj}}$ is a $G_\delta$ subset of $vN(H)$ and by \cite{Ha-Wi1}, Theorem 3.11, and \cite{Ha-Wi2}, Theorem 2.5, the set of all factors $\mathfrak F(H)$ is a dense $G_\delta$-subset of vN(H). Since $vN(H)$ is a Polish space, we can apply Baire's theorem and conclude that
also the intersection $vN_{\text{inj}}(H)\cap\mathfrak F(H)=\mathfrak F_{\text{inj}}(H)$ must be
dense.
\end{proof}

Let $\mathcal{K}$ denote the class of groups $\Gamma$ satisfying \emph{Kirchberg's property}\index{Kirchberg's property}
$$
C^*(\Gamma)\otimes_{\min}C^*(\Gamma)=C^*(\Gamma)\otimes_{\max}C^*(\Gamma)
$$

The following seems to be an open and interesting problem \cite{MO}.

\begin{problem}
Does $\mathcal{K}$ contain a countable discrete non-amenable group?
\end{problem}

Another interesting question comes from the observation that the previous proof as well as Kirchberg's original proof uses free groups in a very strong way.

\begin{problem}
For which groups $\Gamma$, does the property $\Gamma\in \mathcal{K}$ implies Connes' embedding\index{conjecture!Connes' embedding} conjecture? 
\end{problem}

\section{Connes' embedding conjecture and Lance's WEP}

Kirchberg's theorem \ref{th:kirchberg} is an astonishing
link between the theory of von Neumann algebras and the theory of C*-algebras. In the same paper, Kirchberg found another profound link
between these two theories: Connes' embedding\index{conjecture!Connes' embedding} conjecture is a
particular case of a conjecture regarding the structure of
C*-algebras, the QWEP conjecture, asking whether any
C*-algebra is a quotient of a C*-algebra with
Lance's WEP. It is then natural to ask whether there is a direct relation
between Connes' embedding\index{conjecture!Connes' embedding} conjecture and WEP. Nate Brown answered this question in the affirmative, proving that Connes' embedding\index{conjecture!Connes' embedding} conjecture is equivalent
to the analogue of Lance's WEP for separable type II$_1$ factors.

\begin{definition}
Let $A,B$ be C*-algebras. A linear map $\phi:A\to B$ is called \emph{positive}\index{linear map!positive} if $\phi(a^*a)\geq0$, for all $a\in A$.
\end{definition}

\begin{definition}
Let $A,B$ be C*-algebras and $\phi:A\rightarrow B$ be a linear map.
For every $n\in\mathbb N$ we can define a map
$\phi_n:M_n(A)\rightarrow M_n(B)$ by setting
$$
\phi_n[a_{ij}]=[\phi(a_{ij})]
$$
$\phi$ is called \emph{completely positive}\index{linear map!completely positive} if $\phi_n$ is positive for
every $n$.
\end{definition}
\begin{exercise}\label{ex:homarecp}
Show that any *-homomorphism between C*-algebras is completely positive.
\end{exercise}
\begin{definition}
Let $A\subseteq B$ be two C*-algebras. We say that $A$ is \emph{weakly
cp complemented}\index{C*-algebra!weakly cp complemented} in $B$ if there exists a unital completely positive
map $\phi:B\rightarrow A^{**}$ such that $\phi|_A=Id_A$.
\end{definition}

Let $A$ be a C*-algebra. We recall that there is always a faithful representation of $A$ into a suitable $B(H)$, for instance the GNS representation. In other words, one can always regard an abstract C*-algebra as a sub-C*-algebra of $B(H)$ through a faithful representation. Hereafter, we use the notation $A\subseteq B(H)$ to say that we have fixed one particular such representation. As we will see, the concepts we are going to introduce are independent of the representation.

\begin{definition}
A C*-algebra $A$ has the \emph{weak expectation
property}\index{weak expectation property} (WEP, for short) if it is weakly cp complemented in $B(H)$ for a faithful
representation $A\subseteq B(H)$.
\end{definition}

The weak expectation property was introduced by Lance in\cite{La}.

\begin{exercise}\label{ex:wepindependent}
Show that WEP does not depend on the faithful representation $A\subseteq B(H)$.
\end{exercise}

Representation-free characterizations of WEP have been recently proven in \cite{FKP}, \cite{K12}, and \cite{FKPT} leading to new formulations of Connes' Embedding Problem \cite{FKPT}.

\begin{definition}
A C*-algebra $A$ has \emph{QWEP}\index{QWEP} if it is a quotient of a C*-algebra with
WEP.
\end{definition}
QWEP conjecture states that every C*-algebra has QWEP. As mentioned above, in\cite{Ki}, Kirchberg proved also the following theorem.
\begin{theorem}
Let $M$ be a separable $II_1$ factor. The following statements are equivalent
\begin{enumerate}
\item $M$ is embeddable into some $\mathcal{R}^{\mathcal{U}}$,
\item $M$ has QWEP.
\end{enumerate}
\end{theorem}

We refer to the original paper by
Kirchberg \cite{Ki} and to the more recent survey by Ozawa \cite{Oz} for the proof of this theorem. In these notes we focus on a technically easier but equally interesting topic: the von Neumann algebraic analogue of Lance's WEP and the proof of Brown's theorem. 

\begin{definition}
Let $M\subseteq B(H)$ be a von Neumann algebra and $A\subseteq M$ a
weakly dense C*-subalgebra. We say that $M$ has a \emph{weak
expectation relative}\index{weak expectation property!relative} to $A$ if there exists a ucp map
$\Phi:B(H)\rightarrow M$ such that $\Phi|_A=\text{Id}_A$.
\end{definition}

\begin{theorem}{\bf (Brown, \cite{Br} Th.1.2)}\label{wep}
For a separable type II$_1$ factor $M$ the following conditions are
equivalent:
\begin{enumerate}
\item $M$ is embeddable into $\mathcal{R}^{\mathcal U}$.
\item $M$ has a weak expectation relative to some weakly dense
subalgebra.
\end{enumerate}
\end{theorem}

Observe that the notion of injectivity for von Neumann algebras can be stated also
this way: $M\subseteq B(H)$ is \emph{injective}\index{factor!injective} if there exists
a ucp map $\Phi:B(H)\rightarrow M$ such that $\Phi(x)=x$, for all
$x\in M$. So relative weak expectation property is something just a bit less
than injectivity. In fact Brown's
theorem can be read by saying that weak expectation relative
property is the ``limit property of injectivity''.
\begin{corollary}
For a separable type II$_1$ factor the following conditions are
equivalent:
\begin{enumerate}
\item $M$ has a relative weak expectation property,
\item $M$ has QWEP,
\item $M$ is Effros-Marechal limit of injective factors.
\end{enumerate}
\begin{proof}
It is an obvious consequence of Theorems \ref{wep} and \ref{haa}.
\end{proof}
\end{corollary} 
The equivalence between 2. and 3. has been recently generalized to every von Neumann algebra (see \cite{AHW13}, Theorem 1.1).\\

We now move towards the proof of Theorem \ref{wep}. Let $A$ be a separable C*-algebra.
\begin{definition}
A \emph{tracial state}\index{tracial state} on $A$ is map $\tau:A_+\rightarrow[0,\infty]$ such
that
\begin{enumerate}
\item $\tau(x+y)=\tau(x)+\tau(y)$, for all $x,y\in A_+$;
\item $\tau(\lambda x)=\lambda\tau(x)$, for all $\lambda\geq0,x\in
A_+$;
\item $\tau(x^*x)=\tau(xx^*)$ for all $x\in A$;
\item $\tau(1)=1$.
\end{enumerate}
A tracial state extends to a positive functional on the whole $A$. We will often identify the tracial state with its extension.
\end{definition}
\begin{definition}
A tracial state $\tau$ on $A\subseteq B(H)$ is called
\emph{invariant mean}\index{invariant mean} if there exists a state $\psi$ on $B(H)$ such
that
\begin{enumerate}
\item $\psi(uTu^*)=\psi(T)$, for all $u\in U(A)$ and $T\in B(H)$;
\item $\psi|_A=\tau$.
\end{enumerate} 
\end{definition}
\begin{note}
Theorem \ref{invariant} shows that the notion of invariant
mean is independent of the choice of the faithful representation
$A\subseteq B(H)$.
\end{note}
In order to prove Brown's theorem we need a characterization of
invariant means that will be proven in Theorem \ref{invariant}. 

In order to prove such result we need two classical theorems and one lemma. Recall that $L^p(B(H))$ stands for the ideal of linear and bounded operators $a$ such that $||a||_p:=Tr\left((a^*a)^{\frac{p}{2}}\right)^{\frac{1}{p}}$ is finite, where $Tr$ is the canonical semifinite trace on $B(H)$. The notation $L^p(B(H))_+$ stands for the cone of operators $a$ in $L^p(B(H))$ that are positive, that is, $(a\xi,\xi)\geq0$, for all $\xi\in H$. The following well known inequality is sometimes called Powers-St{\o}rmer inequality (see\cite{Po-St}).
\begin{theorem}
Let $h,k\in L^1(B(H))_+$. Then
$$
||h-k||_2^2\leq||h^2-k^2||_1,
$$
In particular, if $u\in
U(B(H))$ and $h\geq0$ has finite rank, then
$$
||uh^{1/2}-h^{1/2}u||_2=||uh^{1/2}u^*-h^{1/2}||_2\leq||uhu^*-h||_1^{1/2}.
$$
\end{theorem}

We now use Powers-St{\o}rmer's inequality to prove a lemma. Recall that $M_{q}\left( \mathbb{C}\right) $ stands for the von Neumann algebra of q-by-q matrices with complex entries.

\begin{lemma}\label{brown1}
Let $H$ be a separable Hilbert space and $h\in B(H)$ be a positive,
finite rank operator with rational eigenvalues and $Tr(h)=1$. Then
there exists a ucp map $\Phi:B(H)\rightarrow M_{q}\left( \mathbb{C}\right) $ such
that
\begin{enumerate}
\item $tr(\Phi(T))=Tr(hT)$, for all $T\in B(H)$;
\item $|tr(\Phi(uu^*)-\Phi(u)\Phi(u^*))|<2||uhu^*-h||_1^{1/2}$, for
all $u\in U(B(H))$;
\end{enumerate}
where $tr$ stands for the normalized trace on $M_{q}\left( \mathbb{C}\right) $.
\end{lemma}

We prove the two statements separately.

\begin{proof}[Proof of (1).]
Let $v_1,\ldots ,v_k\in H$ be the eigenvectors of $H$ and
$\frac{p_1}{q},\ldots ,\frac{p_k}{q}$ be the corresponding eigenvalues. Thus
\begin{enumerate}
\item $hv_i=\frac{p_i}{q}$;
\item $\sum_{i=1}^k\frac{p_i}{q}=tr(h)=1$. It follows that $\sum
p_i=q$.
\end{enumerate}
Let $\{w_m\}$ be an orthonormal basis for $H$ and let
$$
B=\{v_1\otimes w_1,\ldots ,v_1\otimes v_{p_1}\}\cup\{v_2\otimes w_1,\ldots ,v_2\otimes w_{p_2}\}\cup \cdots \cup\{v_k\otimes w_1,\ldots v_k\otimes
w_{p_k}\}
$$
Let $V$ be the subspace of $H\otimes H$ spanned by $B$ and let
$P:H\otimes H\rightarrow V$ be the orthogonal projection. Given $T\in
B(H)$, observe that the following formula holds
$$
Tr(P(T\otimes1)P)=\sum_{i=1}^kp_i\langle Tv_i,v_i\rangle.
$$
Indeed $P(T\otimes1)P$ is representable (in the basis $B$) by a
$q\times q$ block diagonal matrix whose blocks have dimension $p_i$
with entries $ETE$, where $E:H\rightarrow \text{span}\{v_1,\ldots v_k\}$ is the orthogonal
projection. 

Now define $\Phi:B(H)\rightarrow M_{q}\left( \mathbb{C}\right) $ by
setting $\Phi(T)=P(T\otimes1)P$. We have
\begin{align*}
tr(\Phi(T))\\
&=\frac{1}{q}Tr(P(T\otimes1)P)\\
&=\sum_{i=1}^k\langle T\frac{p_i}{q}v_i,v_i\rangle\\
&=\sum_{i=1}^k\langle Thv_i,v_i\rangle=Tr(Th).
\end{align*}
The following exercise concludes the proof of the first statement.
\end{proof}

\begin{exercise}
Prove that $\Phi$ is ucp.
\end{exercise}

\begin{proof}[Proof of (2)]
By writing down the matrix of $P(T\otimes1)P(T^*\otimes1)P$ in
the basis $B$ we have
$$
Tr(P(T\otimes1)P(T^*\otimes1)P)=\sum_{i,j=1}^k|T_{i,j}|^2\min(p_i,p_j),
$$
where $T_{i,j}=\langle Tv_j,v_i\rangle$. Analogously, by writing down the
matrices of $h^{1/2}T, Th^{1/2}$ and $h^{1/2}Th^{1/2}T^*$ in any
orthonormal basis beginning with $\{v_1,\ldots  v_k\}$ we have
$$
Tr(h^{1/2}Th^{1/2}T^*)=\sum_{i,j=1}^k\frac{1}{q}(p_ip_j)^{1/2}|T_{i,j}|^2.
$$
By using these formulae, we can make the following preliminary
calculation
\begin{align*}
&|Tr(h^{1/2}Th^{1/2}T^*)-tr(\Phi(T)\Phi(T^*))|\\
&=\left|\sum_{i,j=1}^k\frac{1}{q}(p_ip_j)^{1/2}|T_{i,j}|^2-\frac{1}{q}Tr(P(T\otimes1)P(T^*\otimes1)P)\right|\\
&=\left|\sum_{i,j=1}^k\frac{1}{q}|T_{i,j}|^2((p_ip_j)^{1/2}-min(p_i,p_j)\right|\\
&\leq\sum_{i,j=1}^k\frac{1}{q}|T_{i,j}|^2p_i^{1/2}\left|p_j^{1/2}-p_i^{1/2}\right|\\
&\leq\left(\sum_{i,j=1}^k\frac{1}{q}|T_{i,j}|^2p_i\right)^{1/2}\left(\sum_{i,j=1}^k\frac{1}{q}|T_{i,j}|^2\left(p_i^{1/2}-p_j^{1/2}\right)\right)^{1/2}\\
&=||Th^{1/2}||_2||h^{1/2}T-Th^{1/2}||_2,
\end{align*}
where the first inequality follows by the property $min(p_i,p_j)\leq p_i$ and the second one is obtained by applying the classical H\"older inequality.
Now suppose that $T\in U(B(H))$, so that
$||Th^{1/2}||_2=||h^{1/2}||_2=1$. Using again the Powers-St\o rmer inequality, we can continue the previous computation as follows:
\begin{align*}
&||h^{1/2}T-Th^{1/2}||_2\\
&=||Th^{1/2}T^*-h^{1/2}||_2\\
&\leq||ThT^*-T||_1^{1/2}
\end{align*}
We can now conclude the proof. Indeed, using the triangle inequality, the previous computation and the Cauchy-Schwarz inequality, we find
\begin{align*}
&|Tr(\Phi(TT^*)-\Phi(T)\Phi(T^*))|\\
&\leq|1-Tr(h^{1/2}Th^{1/2}T^*)|+||ThT^*-h||_1^{1/2}\\
&=|Tr(ThT^*)-Tr(h^{1/2}Th^{1/2}T^*)|+||ThT^*-h||_1^{1/2}\\
&=|Tr((Th^{1/2}-h^{1/2}T)h^{1/2}T^*)|+||ThT^*-h||_1^{1/2}\\
&\leq||h^{1/2}T^*||_2||Th^{1/2}-h^{1/2}T||_2+||ThT^*-h||_1^{1/2}\\
&\leq2||ThT-h||_1^{\frac{1}{2}},
\end{align*}
where the last inequality follows by using the fact that $T$ is unitary and by applying Powers-St{\o}rmer inequality once again.
\end{proof}

The last preliminary result needed for the proof of Theorem \ref{invariant} is classical theorem by Choi \cite{Ch2}.

\begin{theorem}\label{choi}
Let $A,B$ be two C*-algebras and $\Phi:A\rightarrow B$ be a ucp
map. Then
$$
\{a\in A:\Phi(aa^*)=\Phi(a)\Phi(a^*),\Phi(a^*a)=\Phi(a^*)\Phi(a)\}
$$
$$
=\{a\in A:\Phi(ab)=\Phi(a)\Phi(b),\Phi(ba)=\Phi(b)\Phi(a),\forall
b\in A\}.
$$
\end{theorem}

We are now ready to prove a useful characterization of invariant means, first proved in \cite{Kirchberg}, Proposition 3.2.

\begin{theorem}\label{invariant}
Let $\tau$ be a tracial state on $A\subseteq B(H)$. Then the
following are equivalent:
\begin{enumerate}
\item $\tau$ is an invariant mean.
\item There exists a sequence of ucp maps $\Phi_n:A\rightarrow
\mathbb M_{k(n)}$ such that
\begin{enumerate}
\item $||\Phi_n(ab)-\Phi_n(a)\Phi_n(b)||_2\rightarrow0$ for all
$a,b\in A$,
\item $\tau(a)=lim_{n\rightarrow\infty}tr(\Phi_n(a))$, for all $a\in
A$.
\end{enumerate}
\item For any faithful representation $\rho:A\rightarrow B(H)$ there exists
a ucp map $\Phi:B(H)\rightarrow\pi_{\tau}(A)''$ such that
$\Phi(\rho(a))=\pi_{\tau}(a)$, for all $a\in A$, where $\pi_{\tau}$
stands for the GNS representation associated to $\tau$\footnote{We briefly recall the GNS construction (see \cite{GN43} and \cite{Se47}). Setting $\langle x,y\rangle:=\tau(x^*y)$, one obtains a possibly singular inner product on $A$. Let $I$ be the subspace of elements $x$ such that $\langle x,x\rangle=0$ and consider the Hilbert space $H$ obtained by completing $A/I$ with respect to $\langle\cdot,\cdot\rangle$. Using the fact that $I$ is a left ideal, one sees that the operator $\pi_\tau(x)$ defined by $\pi_\tau(x)(y)=xy$ passes to a linear and bounded operator from $H$ to itself and then the mapping $x\to\pi_{\tau}(x)$ is a representation of $A$ into $B(H)$. It turns out that this representation is always faithful: if $A$ is unital, then this is immediate. Otherwise, one has to use a more subtle argument using approximate identities in C*-algebras.}.
\end{enumerate}
\end{theorem}

\begin{proof}[Proof of $1)\Rightarrow 2)$]
Let $\tau$ be an invariant mean with respect to the faithful
representation $\rho:A\rightarrow B(H)$. Thus we can find a state
$\psi$ on $B(H)$ which extends $\tau$ and such that
$\psi(uTu^*)=\psi(T)$, for all $u\in U(A)$ and for all $T\in B(H)$.
Since the normal states are weak* dense in the dual of $B(H)$ and they can be represented
in the form $Tr(h\cdot)$, with $h\in L^1(B(H))$, we can find a net
$h_{\lambda}\in L^1(B(H))$ such that
$Tr(h_{\lambda}T)\rightarrow\psi(T)$, for all $T\in B(H)$. Moreover, the representatives $h_{\lambda}$ can be chosen to be positive and with trace $1$. Now,
since $\psi(uTu^*)=\psi(T)$, it follows that
$Tr(uh_{\lambda}u^*T)=Tr(h_{\lambda}u^*Tu)\rightarrow\psi(u^*Tu)=\psi(T)$
and thus $Tr(h_{\lambda}T)-Tr((uh_{\lambda}u^*)T)\rightarrow 0$, for
all $T\in B(H)$, i.e.\ $h_{\lambda}-uh_{\lambda}u^*\rightarrow0$ in
the weak topology of $L^1(B(H))$. Now let $\{U_n\}$ be an increasing
family of finite sets of unitaries whose union have dense linear
span in $A$ and $\varepsilon=\frac{1}{n}$. Let $U_{n}=\{u_1,\ldots,
u_n\}$. Fixed $n$, let us consider the convex hull of the set
$\{u_1h_{\lambda}u_1^*-h_{\lambda},\ldots,u_nh_{\lambda}u_n^*-h_{\lambda}\}$.
Its weak closure contains $0$ (because of the previous observation)
and coincide with the $1$-norm closure, by the Hahn-Banach
separation theorem. Thus there exists a convex combination of the
$h_{\lambda}$'s, say $h$, such that
\begin{enumerate}
\item $Tr(h)=1$,
\item $||uhu^*-h||_1<\varepsilon, \forall u\in U_n$,
\item $|Tr(uh)-\tau(u)|<\varepsilon, \forall u\in U_n$.
\end{enumerate}
Moreover, since finite rank operators are norm dense in $L^1(B(H))$,
we can assume without loss of generality that $h$ has finite rank with rational eigenvalues.
Now we can apply Lemma \ref{brown1} in order to construct a sequence
of ucp maps $\Phi_n:B(H)\rightarrow\mathbb M_{k(n)}$ such
that
\begin{enumerate}
\item $Tr(\Phi_n(u))\rightarrow\tau(u)$,
\item $|Tr(\Phi_n(uu^*))-\Phi_n(u)\Phi_n(u^*)|\rightarrow0$,
\end{enumerate}
for every unitary in a countable set whose linear span is dense in
$A$. So the property 2.(b) is true for unitaries and consequently, being a linear property, it holds for all operators.

In order to prove 2.(a), we observe that
$\Phi_n(uu^*)-\Phi_n(u)\Phi_n(u^*)\geq0$ and thus the following
inequality holds
$$
||1-\Phi_n(u)\Phi_n(u^*)||_2^2\leq||1-\Phi_n(u)\Phi_n(u^*)||tr(\Phi_n(uu^*)-\Phi_n(u)\Phi_n(u^*))
$$
Since the right hand side tends to zero, also the left hand side must converge to $0$. Now define
$\Phi=\oplus\Phi_n:A\rightarrow\prod\mathbb M_{k(n)}\subseteq
\ell^{\infty}(\mathcal{R})$ and compose with the quotient map
$p:\ell^{\infty}(\mathcal{R})\rightarrow \mathcal{R}^{\mathcal U}$. The previous inequality
shows that if $u$ is a unitary such that
$||\Phi_n(uu^*)-\Phi_n(u)\Phi_n(u^*)||_2\rightarrow0$ and
$||\Phi_n(u^*u)-\Phi_n(u^*)\Phi_n(u)||_2\rightarrow0$, then $u$
falls in the multiplicative domain of $p\circ\Phi$. But such
unitaries have dense linear span in $A$ and hence the whole $A$
falls in the multiplicative domain of $p\circ\Phi$ (by Choi's
theorem \ref{choi}). By definition of ultraproduct this just means
that $||\Phi_n(ab)-\Phi_n(a)\Phi_n(b)||_2\rightarrow0$, for all
$a\in A$.
\end{proof}

\begin{proof}[Proof of $2)\Rightarrow3)$]
Let $\Phi_n:A\rightarrow\mathbb M_{k(n)}$ be a sequence of
ucp maps with the properties stated in the theorem. By identifying each
$\mathbb M_{k(n)}$ with a unital subfactor of $\mathcal{R}$ we can
define a ucp map $\tilde{\Phi}:A\rightarrow\ell^{\infty}(\mathcal{R})$ by
$x\rightarrow(\Phi_n(x))_n$. Since the $\Phi_n's$ are asymptotically
multiplicative in the 2-norm one gets a $\tau$-preserving *-homomorphism
$A\rightarrow \mathcal{R}^{\mathcal U}$ by composing with the quotient map
$p:l^{\infty}(\mathcal{R}\rightarrow \mathcal{R}^{\mathcal U}$. Note that the weak closure
of $p\circ\tilde{\Phi}(A)$ into $\mathcal{R}^{\mathcal U}$ is isomorphic to
$\pi_{\tau}(A)''$. Thus we are in the following situation
$$
\xymatrix{A\ar[r]^{\tilde{\Phi}}\ar[d]^{\rho} & \ell^{\infty}(\mathcal{R})\ar[r]^p & \mathcal{R}^{\mathcal U} & \supseteq & \overline{p\circ\tilde{\Phi}(A)}^w\cong \pi_{\tau}(A)''\\
B(H)\ar[d]^i\\
B(K)}
$$
where $K$ is a representing Hilbert space for $\ell^{\infty}(\mathcal{R})$ and
$i$ is a natural embedding induced by an embedding $H\to K$, whose existence is guaranteed by the fact that $H$ is separable. Since
$\ell^{\infty}(\mathcal{R})$ is injective, there is a projection $E:B(K)\rightarrow
\ell^{\infty}(\mathcal{R})$ of norm 1. Let
$F:\mathcal{R}^{\mathcal U}\rightarrow\pi_{\tau}(A)''$ be a conditional expectation
(see\cite{Ta1}, Proposition 2.36), one has
$$
\xymatrix{A\ar[r]^{\tilde{\Phi}}\ar[d]^{\rho} & l^{\infty}(\mathcal{R})\ar[r]^p & \mathcal{R}^{\mathcal U}\ar[r]^F & \pi_{\tau}(A)''\cong\overline{p\circ\tilde{\Phi}(A)}^w\\
B(H)\ar[d]^i\\
B(K)\ar[uur]_E}
$$
Define $\Phi:B(H)\rightarrow\pi_{\tau}(A)''$ by setting $\Phi=FpEi$.
One has $\Phi(\rho(a))=\pi_{\tau}(a)$.
\end{proof}
\begin{proof}[Proof of
$3)\Rightarrow1)$]
The hypothesis $\Phi(a)=\pi_{\tau}(a)$ guarantees that $\Phi$ is
multiplicative on $A$, since $\pi_\tau$ is a representation. By Choi's theorem \ref{choi} it follows that
$\Phi(aTb)=\pi_{\tau}(a)\Phi(T)\pi_{\tau}(b)$, for all $a,b\in A,
T\in B(H)$. Let $\tau''$ be the vector trace on $\pi_{\tau}(A)''$
and consider $\tau''\circ\Phi$. Clearly it extends $\tau$. Moreover
it is invariant under the action of $U(A)$, indeed
$$
(\tau''\circ\Phi)(u^*Tu)=\tau''(\pi_{\tau}(u)^*\Phi(T)\pi_{\tau}(u))=\tau''(\Phi(T))=(\tau''\circ\Phi)(T)
$$
Hence $\tau$ is an invariant mean.
\end{proof}

The following proposition was also proved by Nate Brown in\cite{Br}.

\begin{proposition}\label{densealgebra}
Let $M$ be a separable II$_1$ factor. There exists a
*-monomorphism $\rho:C^*(\mathbb F_{\infty})\rightarrow M$ such that
$\rho(C^*(\mathbb F_{\infty}))$ is weakly dense in $M$.
\begin{proof}
Observe that $C^*(\mathbb F_{\infty})$ can be viewed as inductive limit
of free products of copies of itself. This can be proven by partitioning the
set of generators in a sequence $X_n$ of countable sets, by defining $A_n=C^*(X_1,\ldots ,X_n)$ and by observing that
$A_n=A_{n-1}*C^*(X_n)\cong C^*(\mathbb F_\infty)$, where $*$ stands for the free product with
amalgamation over the scalars. Now, by
Choi's theorem \ref{ch} we can find a sequence of integers
$\{k(n)\}_{n\in\mathbb N}$ and a unital
*-monomorphism $\sigma:A\rightarrow\prod_{n\in\mathbb N}\mathbb M_{k(n)}$. Note
that we may naturally identify each $A_i$ with a subalgebra of $A$
and hence, restricting $\sigma$ to this copy, get an injection of
$A_i$ into $\prod\mathbb M_{k(n)}$. 

Assume we can prove the existence
of a sequence of unital *-homomorphism $\rho_i:A_i\rightarrow M$ such
that:
\begin{enumerate}
\item Each $\rho_i$ is injective;
\item $\rho_{i+1}|_{A_i}=\rho_i$ where we identify $A_i$ with the
''left side'' of $A_i*C^*(\mathbb F_{\infty})=A_{i+1}$;
\item The union of $\{\rho_i(A_i)\}$ is weakly dense in $M$.
\end{enumerate}
then we would have completed the proof. Indeed, it would be enough to define $\rho$ as the
union of the $\rho_i$'s.

The purpose is then to prove existence of such a sequence $\rho_i$. To this end we first choose an increasing sequence of projections of $M$ such
that $\tau_M(p_i)\rightarrow1$. Then we define the orthogonal
projections $q_n=p_n-p_{n-1}$ and consider the II$_1$ factors
$Q_i=q_iMq_i$. Now, by the division property of II$_1$ factors (see Theorem \ref{Theorem: projections II1 factor}),
we can find a unital embedding $\prod\mathbb M_{k(n)}\rightarrow Q_i\subseteq
M$. By composing with $\sigma$, we get a sequence of embeddings
$A\rightarrow M$, which will be denoted by $\sigma_i$. Now $p_iMp_i$
is separable and thus there is a countable family of
unitaries whose finite linear combinations are dense in the weak topology. Hence we can find a *-homomorphism $\pi_i:C^*(\mathbb
F_{\infty})\rightarrow p_iMp_i$ with weakly dense range (take the
generators of $\mathbb F_{\infty}$ into $C^*(\mathbb F_{\infty})$
and map them into that total family of unitaries). Now we define
$$
\rho_1=\pi_1\oplus\left(\bigoplus_{j\geq2}\sigma_j|_{A_1}\right):A_1\rightarrow
p_1Mp_1\oplus(\Pi_{j\geq2}Q_j)\subseteq M
$$
$\rho_1$ is a *-monomorphism, since each $\sigma_i$ is already faithful on
the whole $A$. 

Now define a *-homomorphism
$\theta_2:A_2=A_1*C^*(\mathbb F_{\infty})\rightarrow p_2Mp_2$ as the
free product of the *-homomorphism $A_1\rightarrow p_2Mp_2$ defined by
$x\rightarrow p_2\rho_1(x)p_2$ and $\pi_2:C^*(\mathbb
F_{\infty})\rightarrow p_2Mp_2$. We then set
$$
\rho_2=\theta_2\oplus\left(\bigoplus_{j\geq3}\sigma_j|_{A_2}\right):A_2\rightarrow
p_2Mp_2\oplus(\Pi_{j\geq3}Q_j)\subseteq M
$$
Clearly $\rho_2|_{A_1}=\rho_1$. In general, we construct a map
$\theta_{n+1}:A_n*C^*(\mathbb F_{\infty})\rightarrow
p_{n+1}Mp_{n+1}$ as the free product of the cutdown (by $p_{n+1}$)
of $\rho_n$ and $\pi_n$. This map need not be injective and hence we
take a direct sum with $\oplus_{j\geq n+2}\sigma_j|_{A_{n+1}}$ to
remedy this deficiency. These maps have all the required properties
and hence the proof is complete (note that the last property follows
from the fact that the range of each $\theta_n$ is weakly dense in
$p_{n+1}Mp_{n+1}$).
\end{proof}
\end{proposition}

\begin{theorem}{\bf (N.P. Brown \cite{Br})}\label{marrone}
Let $M$ be a separable II$_1$ factor and $\mathcal U$ be a free ultrafilter on the natural numbers. The following conditions are equivalent:
\begin{enumerate}
\item $M$ is embeddable into $\mathcal{R}^{\mathcal U}$.
\item $M$ has the weak expectation property relative to some weakly
dense subalgebra.
\end{enumerate}
\begin{proof}[Prood of $(1)\Rightarrow(2)$]
Let $M$ be embeddable into $\mathcal{R}^{\mathcal U}$. By Prop.\ref{densealgebra},
we may replace $M$ with a weakly dense subalgebra $A$ isomorphic to $C^*(\mathbb F_\infty)$. We want to prove that $M$ has the weak
expectation property relative to $A$. Let $\tau$ the unique
normalized trace on $M$, more precisely we will prove that
$\pi_{\tau}(M)$ has the weak expectation property relative to
$\pi_{\tau}(A)$. Indeed $\tau$ is faithful and w-continuous and
hence $\pi_\tau(M)$ and $\pi_\tau(A)$ are respectively copies of $M$
and $A$ and $\pi_\tau(A)$ is still weakly dense in $\pi_\tau(M)$. We
first prove that $\tau|_A$ is an invariant mean. Take $\{u_n\}$
universal generators of $\mathbb F_{\infty}$ into $A$. Let $n$ be
fixed, since $u_n\in \mathcal{R}^{\mathcal{U}}$, then $u_n$ is $||\cdot||_2$-ultralimit of
unitaries in $\mathcal{R}$ (see Exercise \ref{Exercise: unitary group ultraproduct}). On the other hand, the unitary matrices are weakly
dense in $U(\mathcal{R})$ and hence they are $||\cdot||_2$-dense in $U(\mathcal{R})$
(since weakly closed convex subsets coincide with the
$||\cdot||_2$-closed convex ones (see, e.g.,\cite{Jo})). Thus we can find
a sequence of unitary matrices which converges to $u_n$ in norm
$||\cdot||_2$. Let $\sigma$ be the mapping which sends each $u_n$ to
such a sequence. Since the $u_n$'s have no relations, we can extend
$\sigma$ to a
*-homomorphism $\sigma:C^*(\mathbb F_{\infty})\rightarrow\prod M_{k}\left( \mathbb{C}\right) \subseteq \ell^{\infty}(\mathcal{R})$. Let
$p:\ell^{\infty}(\mathcal{R})\rightarrow \mathcal{R}^{\mathcal U}$ be the quotient mapping. By
the $2$-norm convergence we have $(p\circ\sigma)(x)=x$ for all
$x\in\mathbb C^*(\mathbb F_{\infty})$. Let $p_n:\prod_{k=1}^\infty M_{k}\left( \mathbb{C}\right) \rightarrow M_{n}\left( \mathbb{C}\right) $ be the projection, by the definition
of the trace in $\mathcal{R}^{\mathcal U}$, we have
$$
\tau(x)=\lim_{n\rightarrow\mathcal U}tr_n(p_n(\sigma(x))),
$$
where $tr_n$ is the normalized trace on $M_{n}\left( \mathbb{C}\right) $. Now we can
apply Theorem \ref{invariant},2) by setting $\phi_n=p_n\circ\sigma$ (they
are ucp since they are *-homomorphisms) and conclude that
$\tau|_A$ is an invariant mean. Now consider $\pi_{\tau}(M)\subseteq
B(H)$ and $\pi_{\tau}(A)=\pi_{\tau|_A}(A)\subseteq B(H)$. By
Theorem \ref{invariant} there exists a ucp map
$\Phi:B(H)\rightarrow\pi_{\tau}(A)''=\pi_{\tau}(M)$ such that
$\Phi(a)=\pi_{\tau}(a)$. Thus $M$ has the weak expectation
property relative to $C^*(\mathbb F_{\infty})$.
\end{proof}
\begin{proof}[Proof of $(2)\Rightarrow(1)$]
Let $A\subseteq M\subseteq B(H)$, with $A$ weakly dense in $M$,  and
$\Phi:B(H)\rightarrow M$ a ucp map which restricts to the identity on $A$. Let $\tau$ be
the unique normalized trace on $M$. After identifying $A$ with
$\pi_\tau(A)$, we are under the hypothesis of Theorem \ref{invariant}.3) and
thus $\tau|_A$ is an invariant mean. By Theorem \ref{invariant} it
follows that there exists a sequence $\phi_n:A\rightarrow\mathbb M_{k(n)}$ such that
\begin{enumerate}
\item $||\phi_n(ab)-\phi_n(a)\phi_n(b)||_2\rightarrow0$ for all
$a,b\in A$,
\item $\tau(a)=\lim_{n\rightarrow\infty}tr_n(\phi_n(a))$, for all
$a\in A$.
\end{enumerate}
Let $p:\ell^{\infty}(\mathcal{R})\rightarrow \mathcal{R}^{\mathcal U}$ be the quotient mapping.
The previous properties guarantee that the ucp mapping
$\Phi: A\rightarrow \mathcal{R}^{\mathcal U}$, defined by setting $\Phi(x)=p(\{\phi_n(x)\})$ is a
*-homomorphism which preserves $\tau|_A$. It follows that $\Phi$ is injective. Indeed,
$\Phi(x)=0\Rightarrow\Phi(x^*x)=0\Rightarrow\tau(x^*x)=0\Rightarrow x=0$. Observe now that the weak closure of $A$
into $\mathcal{R}^{\mathcal U}$ is isomorphic to $M$ (they are algebraically
isomorphic and have the same trace) and hence $M$ embeds into
$\mathcal{R}^{\mathcal U}$.
\end{proof}
\end{theorem}

\section{Algebraic reformulation of the conjecture}

In this section we present a new line of research that has been initially designed by R\u adulescu, with his proof that Connes' embedding conjecture is equivalent to a non-commutative analogue of Hilbert's 17th problem, and continued by Klep and Schweighofer first and Juschenko and Popovich afterwards, who arrived to a purely algebraic reformulation of Connes embedding conjecture. This section is merely descriptive and serves to introduce the reader to a new field of research, which, though motivated by the Connes embedding conjecture, is quite far from geometric group theory and operator theory, which are the main topics of this monography. The reader interested in technical details of this approach, is referred to the original papers \cite{Ra2},\cite{Kl-Sc},\cite{Ju-Po}.

We begin with a short description of the original formulation of Hilbert's problem and we show hot to get to R\u adulescu's formulation through a series of generalizations. 

Let $\mathbb R[x_1,\ldots,x_n]$ denote the ring of polynomials with $n$ indeterminates and real coefficients and $\mathbb R(x_1,\ldots,x_n)$ denote its quotient field. A polynomial $f\in\mathbb R[x_1,\ldots,x_n]$ is called \emph{non-negative} if, for all $(x_1,\ldots,x_n)\in\mathbb R^n$, one has $f(x_1,\ldots,x_n)\geq0$.

\begin{problem}[Hilbert's 17th problem]
{\rm Can every non-negative polynomial be expressed as sum of squares of elements belonging to $\mathbb R(x_1,\ldots,x_n)$?}
\end{problem}

This problem was solved in the affirmative by Emil Artin \cite{Art}, who provided an abstract proof of existence of such a sum. More recently, Delzell\cite{Delzell} provided an explicit algorithm.

More recently, scholars have been looking for challenging generalizations of this problem, the most intuitive of which is the one concerning matrices. Consider positive semi-definite matrices with entries in $\mathbb R[x_1,\ldots,x_n]$, that is, matrices that are positive semi-definite for all substitution $(x_1,\ldots,x_n)$. Observe that the matrix analogue of a square, is a positive semi-definite symmetric matrix. Indeed, every matrix $B$ such that $B=A^*A$, is also symmetric, that is $B^*=B$; conversely, every positive semi-definite symmetric matrix $B$ can be rooted (by functional calculus) and so it is a square: $B=(\sqrt B)^2$.

\begin{problem}
{\rm Can all positive semi-definite matrices with entries in $\mathbb R[x_1,\ldots,x_n]$ be written as sum of squares of symmetric matrices with entries in $\mathbb R(x_1,\ldots,x_n)$?}
\end{problem}

This problem was solved in the affirmative independently by Gondard and Ribenoim \cite{Go-Ri} and Procesi and Schacher \cite{Pr-Sc}. A constructive solution has been provided much later by Hillar and Nie \cite{Hi-Ni}.

In order to present the version of this problem using operator theory, we need to pass through a geometric analogue. To this end, observe that a polynomial $f\in\mathbb R[x_1,\ldots,x_n]$ is just a function from $\mathbb R^n$ to $\mathbb R$. So one may attempt to extend Hilbert's 17th problem from polynomial on $\mathbb R^n$ to more general functions defined on manifolds. Recall that an $n$-manifold $M$ is called irreducible if for any embedding of the $n$-sphere $S^{n-1}$ into $M$ there exists an embedding of the $n$-ball $B^n$ into $M$ such that the image of the boundary of $B^n$ coincides with the image of $S^{n-1}$. 

\begin{problem}[Geometric analogue of Hilbert's 17th problem]
{\rm Let $M$ be a paracompact irreducible analytic manifold and $f:M\to\mathbb R$ be a non-negative analytic function. Can $f$ be expressed as a sum of squares of meromorphic functions?}
\end{problem}

Recall that meromorphic functions are those functions which are analytic on the whole domain expect for a set of isolated points, which are their poles. So, rational functions are meromorphic and one can thus recognize a generalization of Hilbert's 17th problem. In its generality, this problem is still open. A complete solution is known only for dimension $n=2$ (see \cite{Castilla}) and for compact manifolds (see \cite{Ruiz}). 

The basic idea to get to R\u adulescu's analogue in terms of operator theory is to generalize analytic functions with formal series. Let $Y_1,\ldots,Yn$ be $n$ indeterminates. Define
$$
I_n=\{(i_1,\ldots,i_p), p\in\mathbb N, i_1,\ldots,i_p\in\{1,\ldots,n\}\}.
$$
For each $I=(i_1,\ldots,i_p)\in I_n$, set $Y_I=Y_{i_1}\cdot\ldots\cdot Y_{i_p}$ and define the space
$$
V=\left\{\sum_{I\in I_n}\alpha_IY_I, \alpha_I\in\mathbb C : \forall R>0, \left\|\sum_{I\in I_n}\alpha_IY_I\right\|_R:=\sum_I|\alpha_I|R^{|I|}<\infty\right\}.
$$
R\u adulescu showed in \cite{Ra2}, Proposition 2.1, that $V$ is a Fr\'echet space and so it carries a natural weak topology $\sigma(V,V^*)$ induced by its dual space. 
\commentout{
\begin{exercise}
After identifying $V$ with the vector space of sequences $(\alpha_I)$ such that $\sum\alpha_IR^{|I|}$, for all $R>0$, show that $V^*$ can be identified with the space of all sequences of complex number $(\beta_I)$ such that there is $R>0$ such that $\sup_{I\in I_n}|\beta_I|R^{-|I|}<\infty$, through the duality
$$
(\alpha,\beta)=\sum_I\alpha_I\beta_I.
$$
\end{exercise}}
To generalize the notion of `square' and `sum of squares' to this setting, we need to introduce a notion of symmetry in $V$. Since $V$ is a space of formal series, this is done in the obvious way. We set $(Y_{i_1}\cdot\ldots\cdot Y_{i_p})^*:=Y_{i_p}\cdot\ldots\cdot Y_{i_1}$ and $\alpha^*:=\bar\alpha$. This mapping can clearly be extended to an adjoint operation on $V$. Also, we observe that formal series are, by definition, possibly infinite sums and so there is no hope, in general, to have them expressed as a finite sum of squares. This motivates the need to use weak limits of sum of squares, instead of just the finite sums of squares. 

\begin{definition}\label{defin:sum of squares}
We say that a formal series $q\in V$ is a sum of squares if it is in the weak closure of the set of the elements of the form $p^*p$, for $p\in V$.
\end{definition}

We now observe that the original formulation of Hilbert's 17th problem concerns matrices with \emph{real} entries and its geometric variant concerns \emph{real} valued analytic functions. Recalling that the operator analogue of real valued functions are self-adjoint operators, it follows that the right setting in which Hilbert's 17th problem can be generalized is that of self-adjoint operators. We then introduce the space $V_{sa}=\{v\in V : v^*=v\}$. It remains only to generalize the notion of positivity. 
\begin{definition}\label{defin:positive semidefinite operator}
A self-adjoint operator $v\in V_{sa}$ is called positive semidefinite if for every $N\in\mathbb N$ and for every $N$-tuple of self-adjoint matrices $X_1,\ldots,X_N$, one has
$$
tr(p(X_1,\ldots,X_N))\geq0.
$$
The cone of positive semi-definite operators is denoted $V_{sa}^+$.
\end{definition}

One last step is needed to get to the operator analogue of Hilbert's 17th problem. Indeed, in case of polynomials, one has $Y_I-Y_{\tilde I}=0$, for every permutation $\tilde I$ of $I$. Since this is no longer the case in the non-commutative world of formal series, we need to identify series which differ by a permutation.

\begin{definition}
Two elements $p,q\in V_{sa}^+$ are called \emph{cyclic equivalent} if $p-q$ is weak limit of sums of scalars multiples of monomials of the form $Y_I-Y_{\tilde I}$, where $\tilde I$ is a cyclic permutation of $I$.
\end{definition}

\begin{problem}[Non-commutative analogue of Hilbert's 17th problem]\label{prob:noncommutativeHilbert}
{\rm Is every element of $V_{sa}^+$ cyclic equivalent to a weak limit of sums of squares?}
\end{problem}

As mentioned, the interest in this problem comes from the fact that it is equivalent to Connes' embedding conjecture.

\begin{theorem}[R\u adulescu]
The following statements are equivalent:
\begin{enumerate}
\item Connes' embedding conjecture is true;
\item Problem \ref{prob:noncommutativeHilbert} has a positive answer.
\end{enumerate} 
\end{theorem}

The proof of this theorem is quite involved and we refer the interested reader to the original paper by R\u adulescu. Here we move on to further developments, which led to a purely algebraic reformulation of Connes' embedding conjecture. Such a purely algebraic reformulation is unexpected since all formulations we have discussed so far have a strong topological component. We discuss two different, though similar, purely algebraic reformulations of Connes' embedding conjecture; one due to Klep and Schweighofer \cite{Kl-Sc}, the other to Juschenko and Popovich \cite{Ju-Po}.

We start by discussing Klep and Schweighofer's approach. This differs from R\u adulescu in many parts, the most important of which is that they do not use formal series but they get back to polynomials. Using finite objects instead of infinite objects is the ultimate reason why they are able to get rid of every topological condition. 

Let $K$ be either the real or the complex field and let $V$ denote the ring of polynomials on $n$ indeterminates with coefficients in $K$. Instead of using R\u adulescu's adjoint operation they define the adjoint operation acting identically on monomials and switching each coefficient with its conjugate. As before, the set of self-adjoint elements is senoted $V_{sa}$. Moreover, instead of using the cyclic equivalence, they define two polynomial $p$ and $q$ to be equivalent when their difference is a sum of commutators.

\begin{definition}
A polynomial $f\in V$ is called positive semidefinite if for every $N\in\mathbb N$ and for every contractions $A_1,\ldots,A_n\in\mathbb M_N(\mathbb R)$ one has
$$
tr(f(A_1,\ldots,A_n))\geq0.
$$
The set of positive semidefinite elements is denoted by $V^+$.
\end{definition}

The introduction of the quadratic module is the major difference with R\u adulescu's approach.

\begin{definition}
A subset $M\subseteq V_{sa}$ is called quadratic module if the following hold:
\begin{enumerate}
\item $1\in M$;
\item $M+M\subseteq M$;
\item $p^*Mp\subseteq M$, for all $p\in V$.
\end{enumerate}
The quadratic module generated by the elements $1-X_1^2,\ldots,1-X_n^2$ is denoted by $Q$.
\end{definition}

\begin{theorem}[Klep-Schweighofer]
The following statements are equivalent:
\begin{enumerate}
\item Connes' embedding conjecture is true;
\item For every $f\in V^+$ and for every $\varepsilon>0$, there exists $q\in Q$ such that $f+\varepsilon$ is equivalent to $q$, in the sense that $f+\varepsilon-q$ is a sum of commutators.
\end{enumerate} 
\end{theorem}

One may wish to be able to replace the quadratic module by the more standards squares, namely elements of the form $v^*v$. We conclude this section by describing the approach by Juschenko and Popovich, which is indeed aimed to this.

One way to reformulate Klep and Schweighofer's approach is by considering the free associative algebra $K(X)$ generated by a countable family of self-adjoint elements $X=(X_1,X_2\ldots\}$. Thus, a polynomial $f\in V$ is just an element of $K(X)$. Klep-Schweighofer's theorem affirms that Connes' embedding conjecture is equivalent to the statement that every positive semidefinite element of $K(X)$ cab be written as an element of the quadratic module, up to a sum of commutators and $\varepsilon$, with $\varepsilon$ arbitrarily small.

Instead of considering the free associative algebra $K(X)$, Juschenko and Popovich considered the group *-algebra $\mathcal F$ of the countably generated free group $\mathbb F_{\infty}=\langle u_1,u_2,\ldots\rangle$. With this choice they were able to simplify Klep and Schweighofer's theorems in two ways. First, instead of using the quadratic module, they were able to use standard squares; second, instead of considering every polynomial $f$, they were able to consider only polynomials of degree at most two, in the variables $u_i$. Before presenting the theorem, we redefine the notion of positivity in this new context.

\begin{definition}
An element $f\in\mathcal F$ with $n$ indeterminates is called positive semidefinite if for all $m\geq1$ and all $n$-tuples of unitary matrices $U_1,\ldots,U_n$ of dimension $m$, one has 
$$
tr(f(U_1,\ldots,U_n))\geq0.
$$
\end{definition}

\begin{theorem}[Juschenko-Popovich]
The following statements are equivalent:
\begin{enumerate}
\item Connes' embedding conjecture is true;
\item For every self-adjoint positive semidefinite $f\in\mathcal F$ and for every $\varepsilon>0$, one has $f+\varepsilon=g+c$, where $g$ is a sum of squares (elements of the forv $v^*v$) and $c$ is a sum of commutators.
\end{enumerate}
\end{theorem}

\section{Brown's invariant}

Most of research about Connes' embedding\index{conjecture!Connes' embedding} conjecture has been focusing on impressive reformulations of it, that is, on finding apparently very far statements that turn out to be eventually equivalent to the original conjecture. 

Over the last couple of years another point of view has been also taken, mostly due to Nate Brown's paper\cite{Br2}. He assumes that a fixed separable II$_1$ factor $M$ verifies Connes' embedding\index{conjecture!Connes' embedding} conjecture and tries to tell something interesting about $M$. In particular, he managed to associate an invariant to $M$, now called Brown's invariant, that carries information about rigidity properties of $M$. The purpose of this section is to introduce the reader to this invariant.  

\subsection{Convex combinations of representations into \texorpdfstring{$\mathcal{R}^{\mathcal{U}}$}{RU} }\label{se:standard}

Let $M$ be a separable II$_1$ factor verifying Connes' embedding\index{conjecture!Connes' embedding} conjecture and fix a free ultrafilter $\mathcal U$ on the natural numbers. The set $\mathbb Hom(M,\mathcal{R}^{\mathcal{U}})$ of unital morphisms $M\to \mathcal{R}^{\mathcal{U}}$ modulo unitary equivalence is
non-empty. We shall show that this set, that is in fact Brown's invariant, has a surprisingly rich structure.

We can equip $\mathbb Hom(M,\mathcal{R}^{\mathcal{U}})$ with a metric in a reasonably simple way. Since $M$ is separable, it is topologically generated by countably many elements $a_1,a_2\ldots$, that we may assume to be contractions, that is $||a_i||\leq1$, for all $i$. So we can define a metric on $\mathbb Hom(M,\mathcal{R}^{\mathcal{U}})$ as follows
$$
d([\pi],[\rho])=\inf_{u\in
U(\mathcal{R}^{\mathcal{U}})}\left(\sum_{n=1}^\infty\frac{1}{2^{2n}}||\pi(a_n)-u\rho(a_n)u^*||_2^2\right)^{\frac{1}{2}},
$$
since the series in the right hand side is convergent. A priori, $d$ is just a pseudo-metric, but we can use Theorem 3.1 in \cite{Sh} to say that approximately unitary equivalence is the same as unitary equivalence in separable subalgebras of $\mathcal{R}^{\mathcal{U}}$. This means that $d$ is actually a metric. Moreover, while this metric may depend on the generating set $\{a_1,a_2,\ldots\}$, the induced topology does not. It is indeed the point-wise convergence topology.\\

$\mathbb Hom(M,\mathcal{R}^{\mathcal{U}})$ does not carry any evident vector space structure, but Nate Brown's intuition was that one can still do convex combinations inside $\Hom(M,\mathcal{R}^{\mathcal{U}})$ in a formal way. There is indeed an obvious (and wrong) way to proceed: given *-homomorphisms $\pi, \rho \colon M \to \mathcal{R}^{\mathcal{U}}$
and $0 < t < 1$, take a projection $p_t \in (\pi(M) \cup \rho(M))'
\cap \mathcal{R}^{\mathcal{U}}$ such that $\tau(p_t) = t$ and define the ``convex combination''
$t\pi + (1-t)\rho$ to be $$x \mapsto \pi(x)p_t + \rho(x)
p_t^{\perp}.$$

Since the projection $p_t$ is chosen in $(\pi(M) \cup \rho(M))'$, then $t\pi+(1-t)\rho$ is certainly a new unital morphism of $M$ in $\mathcal{R}^{\mathcal{U}}$. Unfortunately this procedure is not well defined on classes in $\Hom(M, \mathcal{R}^{\mathcal{U}})$ and the reason can be explained as follows: if $p \in \mathcal{R}^{\mathcal{U}}$ is a nonzero projection, then the corner $p  \mathcal{R}^{\mathcal{U}} p$ is still a hyperfinite $II_1$-factor and so, by uniqueness, it is isomorphic to $\mathcal{R}^{\mathcal{U}}$. Thus the cut-down $p\pi$ can be seen as a new morphism $M \to \mathcal{R}^{\mathcal{U}}$.  The problem is that the isomorphism $p\mathcal{R}^{\mathcal{U}} p\to \mathcal{R}^{\mathcal{U}}$ is not canonical and this reflects on the fact that convex combinations as defined above are not well-defined on classes in $\Hom(M,\mathcal{R}^{\mathcal{U}})$. The idea is to allow only particular isomorphisms $p\mathcal{R}^{\mathcal{U}} p\to \mathcal{R}^{\mathcal{U}}$ that are somehow fixed by conjugation by a unitary. This is done by using the so-called standard isomorphisms, that represent Nate Brown's main technical innovation.

\begin{definition}
\label{defin:standard} 
Let $p \in \mathcal{R}^{\mathcal{U}}$ be a non-zero projection.  A
\emph{standard isomorphism}\index{standard isomorphism} is any map $\theta_p\colon p\mathcal{R}^{\mathcal{U}} p \to
\mathcal{R}^{\mathcal{U}}$ constructed as follows. Lift $p$ to a
projection $(p_n) \in \ell^\infty(\mathcal{R})$ such that $\tau_\mathcal{R}(p_n) =
\tau_{\mathcal{R}^{\mathcal{U}}}(p)$, for all $n \in \mathbb{N}$, fix isomorphisms
$\theta_n\colon p_n \mathcal{R}p_n \to \mathcal{R}$, and define $\theta_p$ to be the
isomorphism on the right hand side of the following commutative diagram

$$
\xymatrix{\ell^\infty(p_n\mathcal{R}p_n)\ar[r]\ar[d]^{\oplus\theta_n} & p\mathcal{R}^{\mathcal{U}} p\ar[d]^\cong\\
\ell^\infty(\mathcal{R})\ar[r] & \mathcal{R}^{\mathcal{U}}}
$$

\end{definition}

\begin{definition}\label{defin:convex combinations}
\label{defin:convex-like} 
Given $[\pi_1], \ldots, [\pi_n] \in
\Hom(N,\mathcal{R}^{\mathcal{U}})$ and $t_1,\ldots, t_n \in [0,1]$ such that $\sum t_i =
1$, we define $$\sum_{i=1}^n t_i [\pi_i] := \left[ \sum_{i=1}^n
\left(\theta_i^{-1}\circ\pi_i\right)\right],$$ where $\theta_i \colon p_i \mathcal{R}^{\mathcal{U}} p_i \to
\mathcal{R}^{\mathcal{U}}$ are standard isomorphisms and $p_1, \ldots, p_n \in \mathcal{R}^{\mathcal{U}}$ are
orthogonal projections such that $\tau(p_i) = t_i$ for $i \in
\{1,\ldots, n\}$.
\end{definition}

We can explain in a few words why this procedure of using standard isomorphisms works. It has been originally proven by Murray and von Neumann that there is a unique unital embedding of $M_{n}\left( \mathbb{C}\right) $ into $\mathcal{R}$ up to unitary equivalence. Since $\mathcal{R}$ contains an increasing chain of matrix algebras whose union is weakly dense, it follows that all unital endomorphisms of $\mathcal{R}$ are approximately inner. Now, if we take an automorphism $\Theta$ of $\mathcal{R}^{\mathcal{U}}$ that can be lifted (i.e.\ it is of the form $(\theta_n)_{n\in\mathbb N}$ where $\theta_n$ is an automorphism of $\ell^\infty(\mathcal{R})$), it follows that $\Theta$ is just the conjugation by some unitary, when restricted to a separable subalgebras or $\mathcal{R}^{\mathcal{U}}$. Now, Nate Brown's standard isomorphisms are exactly those isomorphisms $p\mathcal{R}^{\mathcal{U}} p\to \mathcal{R}^{\mathcal{U}}$ that are liftable and therefore it is intuitively clear that, after passing to the quotient by the relation of unitary equivalence, the choice of the standard isomorphism should not affect the result. The formalization of this rough idea leads to the following theorem.

\begin{theorem}{\bf (N.P. Brown\cite{Br2})}\label{th:convexcombwelldefined}
$\sum_{i=1}^n t_i [\pi_i]$ is well defined, i.e.\ independent of the projections $p_i$, the standard isomorphisms $\theta_i$ and the representatives $\pi_i$.
\end{theorem}

To prove this result we need some preliminary observations.

\begin{lemma}\label{lem:partialisometries}
Let $p,q\in \mathcal{R}$ be projections of the same trace and $\theta:p\mathcal{R}p\to q\mathcal{R}q$ be a unital *-homomorphism, that is $\theta(p)=q$. Then there is a sequence of partial isometries $v_n\in \mathcal{R}$ such that:
\begin{enumerate}
\item $v_n^*v_n=p$,
\item $v_nv_n^*=q$,
\item $\theta(x)=\lim_{n\to\infty}v_nxv_n^*$,
\end{enumerate}
where the limit is taken in the 2-norm.
\end{lemma}

\begin{proof}
Since $p,q$ have the same trace, we can find a partial isometry $w$ such that $w^*w=q$ and $ww^*=p$. Consider the unital endomorphism $\theta_w:p\mathcal{R}p\to p\mathcal{R}p$ defined by $\theta_w(x)=w\theta(x)w^*$. Since $\mathcal{R}$ is hyperfinite, every endomorphism is approximately inner in the 2-norm, that is, we can find unitaries $u_n\in p\mathcal{R}p$ such that $w\theta(x)w^*=\lim_{n\to\infty}u_nxu_n^*$. Defining $v_n=w^*u_n$ completes the proof.
\end{proof}

\begin{proposition}\label{prop:unitary}
Assume $p,q\in \mathcal{R}^{\mathcal{U}}$ are projections with the same trace, $M\subseteq p\mathcal{R}^{\mathcal{U}} p$ is a separable von Neumann subalgebra and $\Theta:p\mathcal{R}^{\mathcal{U}} p\to q\mathcal{R}^{\mathcal{U}} q$ is a unital *-homomorphism. Let $(p_i), (q_i)\in\ell^\infty(\mathcal{R})$ lifts of $p$ and $q$, respectively, such that $\tau_\mathcal{R}(p_i)=\tau_\mathcal{R}(q_i)=\tau_{\mathcal{R}^{\mathcal{U}}}(p)$, for all $i\in\mathbb N$. If there are unital *-homomorphisms $\theta_i:p_i\mathcal{R}p_i\to q_i\mathcal{R}q_i$ such that $(\theta_i(x_i))$ is a lift of $\Theta(x)$, whenever $(x_i)\in\Pi p_i\mathcal{R}p_i$ is a lift of $x\in M$, then there are a partial isometry $v\in \mathcal{R}^{\mathcal{U}}$ such that:
\begin{enumerate}
\item $v^*v=p$,
\item $vv^*=q$,
\item $\Theta(x)=vxv^*$, for all $x\in M$.
\end{enumerate}
\end{proposition}

\begin{proof}
We shall prove the proposition only in the case $M=W^*(X)$ is singly generated.

Let $(x_i)\in\Pi p_i\mathcal{R}p_i$ be a lift of $X$. By Lemma \ref{lem:partialisometries}, there are partial isometries $v_i\in \mathcal{R}$ such that $v_i^*v_i=p_i$, $v_iv_i^*=q_i$ and $||\theta_i(x_i)-v_ix_iv_i^*||_2<1/i$. Observe that $(v_i)\in\ell^\infty(\mathcal{R})$ drops to a partial isometry $v\in \mathcal{R}^{\mathcal{U}}$ with support $p$ and range $q$. To show that $\Theta(X)=vXv^*$, fix $\varepsilon>0$ and consider the set
$$
S_\varepsilon=\{i\in\mathbb N : ||\theta_i(x_i)-v_ix_iv_i^*||_2<\varepsilon\}.
$$
This set contains the cofinite set $\{i\in\mathbb N : i\geq\frac{1}{\varepsilon}\}$ and therefore $S_\varepsilon\in\mathcal U$.
\end{proof}

\begin{exercise}
Prove Proposition \ref{prop:unitary} in the general case. (Hint: pick the $v_i$'s to obtain inequalities of the shape $||\theta_i(Y_i)-v_iY_iv_i^*||_2<1/i$ on a finite set of $Y_i$'s that corresponds to lifts of a finite subset of a generating set of $M$).
\end{exercise}

\begin{proof}[Proof of Theorem \ref{th:convexcombwelldefined}]
Let $\sigma_i:q_i\mathcal{R}^{\mathcal{U}} q_i\to \mathcal{R}^{\mathcal{U}}$ be standard isomorphisms, built by pairwise orthogonal projections $q_i$ of trace $t_i$ and assume $[\rho_i]=[\pi_i]$. By Proposition \ref{prop:unitary}, applied to the standard isomorphism $\sigma_i^{-1}\circ\theta_i:p_i\mathcal{R}^{\mathcal{U}} p_i\to q_i\mathcal{R}^{\mathcal{U}} q_i$, there are partial isometries $v_i\in \mathcal{R}^{\mathcal{U}}$ such that $v_i^*v_i=p_i$, $v_iv_i^*=q_i$ and 
$$
v_i\left(\theta_i^{-1}\circ\pi_i\right)(x)v_i^*=\left(\sigma_i^{-1}\circ\pi_i\right)(x),\qquad\qquad\text{for all }x\in M.
$$
Now since $[\pi_i]=[\rho_i]$, we can find unitaries $u_i$ such that $\rho_i=u_i\pi_iu_i^*$. The proof is then completed by the following exercise.
\begin{exercise}
Show that $u:=\sum\sigma_i^{-1}(u_i)v_i$ is a unitary conjugating $\sum\theta_i^{-1}\circ\pi_i$ over to $\sum\sigma_i^{-1}\circ\rho_i$.
\end{exercise}
\end{proof}

Having a notion of convex combinations on $\Hom(N,\mathcal{R}^{\mathcal{U}})$, it is natural to ask (i) whether this set can be embedded into a vector space; and, if so, (ii) what can be done with this vector space. Nate Brown himself proved in \cite{Br2} that his notion of convex combinations verifies the axioms of a so-called convex-like structure. Afterwards, Capraro and Fritz showed in \cite{Ca-Fr} that every convex-like structure is linearly and isometrically embeddable into a Banach space, but their proof is strongly based on Stone's representation theorem \cite{St}, which provides an abstract embedding into a vector space. Therefore, taken together, these two results provide only an abstract embedding of $\Hom(N,\mathcal{R}^{\mathcal{U}})$ into a Banach space. In the appendix, the interested reader can find a sketch of a concrete embedding, using a construction appeared in \cite{Br-Ca}.

\subsection{Extreme points of \texorpdfstring{$\Hom(M,\mathcal{R}^{\mathcal{U}})$}{Hom(M,RU)} and a problem of Popa}

In this section we present one application of the $\Hom$ space. Given a separable II$_1$ factor $M$ that embeds into $\mathcal{R}^{\mathcal{U}}$, Sorin Popa asked whether there is always another representation $\pi$ such that $\pi(M)'\cap \mathcal{R}^{\mathcal{U}}$ is a factor. The following theorem by Nate Brown\cite{Br2} shows that this problem is equivalent to a geometric problem on $\Hom(M,\mathcal{R}^{\mathcal{U}})$.

\begin{theorem}\label{th:extremepoints}
Let $\pi:M\to \mathcal{R}^{\mathcal{U}}$ be a representation. Then $\pi(M)'\cap \mathcal{R}^{\mathcal{U}}$ is a factor if and only if $[\pi]$ is an extreme point of $\Hom(M,\mathcal{R}^{\mathcal{U}})$.
\end{theorem}

In this section we prove only the ``only if'' part: if $[\pi]$ is an extreme point, then $\pi(M)'\cap \mathcal{R}^{\mathcal{U}}$ is a factor.

\begin{definition}\label{def:cutdown}
We define the \emph{cutdown}\index{representation!cutdown} of a representation $\pi:M\to \mathcal{R}^{\mathcal{U}}$ by a projection $p\in\pi(M)'\cap \mathcal{R}^{\mathcal{U}}$ to be the map $M\to \mathcal{R}^{\mathcal{U}}$ defined by $x\to\theta_p(p\pi(x))$, where $\theta_p:p \mathcal{R}^{\mathcal{U}} p\to  \mathcal{R}^{\mathcal{U}}$ is a standard isomorphism.
\end{definition}

Lemma 3.3.3 in \cite{Br2} shows that this definition is independent by the standard isomorphism, hence we can denote it by $[\pi_p]$.

\begin{lemma}\label{lem:cutdown}
Let $\pi:M\to \mathcal{R}^{\mathcal{U}}$ be a morphism $p,q\in\pi(M)'\to \mathcal{R}^{\mathcal{U}}$ be projections of the same trace. The following statement are equivalent:
\begin{enumerate}
\item $[\pi_p]=[\pi_q]$,
\item $p$ and $q$ are Murray-von Neumann equivalent in $\pi(M)'\cap \mathcal{R}^{\mathcal{U}}$, that is, there is a partial isometry $v\in\pi(M)'\cap \mathcal{R}^{\mathcal{U}}$ such that $v^*v=p$ and $vv^*=q$.
\item there exists $v\in \mathcal{R}^{\mathcal{U}}$ such that $v^*v=p$, $vv^*=q$ and $v\pi(x)v^*=q\pi(x)$, for all $x\in M$.
\end{enumerate} 
\end{lemma}

\begin{exercise}
Prove the equivalence between (2) and (3) in Lemma \ref{lem:cutdown}.
\end{exercise}

\begin{exercise}\label{ex:lift}
Given projections $p,q$ and a partial isometry $v$ such that $v^*v=p$ and $vv^*=q$, show that there exist lifts $(p_n),(q_n),(v_n)\in\ell^\infty(R)$ such that $p_n,q_n$ are projections of the same trace as $p$, and $v_n^*v_n=p_n$, $v_nv_n^*=q_n$, for all $n\in\mathbb N$.
\end{exercise}

\begin{proof}[Proof of $(3)\Rightarrow(1)$.]
Let $p_n,q_n,v_n$ as in Exercise \ref{ex:lift} and fix isomorphisms $\theta_n:p_n\mathcal{R}p_n\to \mathcal{R}$ and $\gamma_n:q_n\mathcal{R}q_n\to \mathcal{R}$ and use them to define standard isomorphisms $\theta:p \mathcal{R}^{\mathcal{U}} p\to \mathcal{R}^{\mathcal{U}}$ and $\gamma:q \mathcal{R}^{\mathcal{U}} q\to \mathcal{R}^{\mathcal{U}}$ and use them to define $\pi_p$ and $\pi_q$. The isomorphism on the right hand side of the following diagram\footnote{The notation $Ad u$ in the diagram stands for the conjugation by the unitary operator $u$.} is liftable by construction and so Proposition \ref{prop:unitary} can be applied to it, giving unitary equivalence between $\pi_p$ and $\pi_q$.

$$
\xymatrix{\ell^\infty(\mathcal{R})\ar[r]\ar[d]^{\oplus\theta_n^{-1}} & \mathcal{R}^{\mathcal{U}} \ar[d]^\theta^{-1}\\
\ell^\infty(p_n\mathcal{R}p_n)\ar[r]\ar[d]^{\oplus Ad v_n} & p\mathcal{R}^{\mathcal{U}} p\ar[d]^{Ad v}\\
\ell^\infty(q_n\mathcal{R}q_n)\ar[r]\ar[d]^{\oplus\gamma_n} & q \mathcal{R}^{\mathcal{U}} q\ar[d]^{\gamma}\\
\ell^\infty(\mathcal{R})\ar[r] &  \mathcal{R}^{\mathcal{U}} }
$$

\end{proof}

\begin{exercise}
Use a similar idea to prove the implication $(1)\Rightarrow(3)$.
\end{exercise}

We recall that a projection $p\in M$ is called \emph{minimal}\index{minimal projection} if $pMp=\mathbb C1$. A von Neumann algebra without minimal projections is called \emph{diffuse}\index{von Neumann algebra!diffuse}.

\begin{exercise}\label{ex:factor}
Let $M$ be a diffuse von Neumann algebra with the following property: every pair of projections with the same trace are Murray-von Neumann equivalent. Show that $M$ is factor. (Hint: show that every central projection is minimal). 
\end{exercise}

\begin{proof}[Proof of the ``only if'' of Theorem \ref{th:extremepoints}]
Let $[\pi]$ be an extreme point of $\Hom(M,\mathcal{R}^{\mathcal{U}})$ and $p\in\pi(M)'\cap \mathcal{R}^{\mathcal{U}}$. Since
$$
[\pi]=\tau(p)[\pi_p]+\tau(p^{\perp})[\pi_{p^\perp}],
$$
it follows that $[\pi_p]=[\pi]$, for all $p\in\pi(M)'\cap \mathcal{R}^{\mathcal{U}}$, $p\neq0$. By Lemma \ref{lem:cutdown}, it follows that two projections in $\pi(M)'\cap \mathcal{R}^{\mathcal{U}}$ are Murray-von Neumann equivalent into $\pi(M)'\cap \mathcal{R}^{\mathcal{U}}$ if and only if they have the same trace. Since $\pi(M)'\cap \mathcal{R}^{\mathcal{U}}$ is diffuse, Exercise \ref{ex:factor} completes the proof.
\end{proof}

From Theorem \ref{th:extremepoints} we obtain the following geometric reformulation of Popa's question.

\begin{problem}{\bf(Geometric reformulation of Popa's question)}
Does $\mathbb Hom(M,\mathcal{R}^{\mathcal{U}})$ have extreme points?
\end{problem}

This problem is still open. There are two obvious ways to try to attack it, leading to two related problems, whose positive solution would imply a positive solution of Popa's question.
\begin{enumerate}
\item Since $\Hom(M,\mathcal{R}^{\mathcal{U}})$ is a bounded, closed and convex subset of a Banach space\index{Banach space} one cannot apply Krein-Milman's theorem and conclude existence of extreme points. Nevertheless, one can ask the question whether the Banach space into which $\Hom(M,\mathcal{R}^{\mathcal{U}})$ embeds is actually a dual Banach space\index{Banach space!dual}. In this case, $\Hom(M,\mathcal{R}^{\mathcal{U}})$ would be compact in the weak*-topology and one could apply Krein-Milman's theorem.
\begin{problem}
Does $\Hom(M,\mathcal{R}^{\mathcal{U}})$ embed into a dual Banach space?
\end{problem}
Note that one way to try to attack this problem is by observing that $M$ is the dual Banach space\index{Banach space!dual} (every von Neumann algebra is a dual Banach space) of a unique Banach space, usually denoted by $M_*$. Consequently there might be some possibility to express representations $M\to \mathcal{R}^{\mathcal{U}}$ in terms of dual representations of $\mathcal{R}^{\mathcal{U}}_*$ into $M_*$.

We mention that Chirvasitu \cite{Chi} has shown that $\mathbb Hom(M,\mathcal{R}^{\mathcal{U}}\bar \otimes B(H))$ is Dedekind-complete with respect to the order induced by the cone $\mathbb Hom_+(M,\mathcal{R}^{\mathcal{U}}\bar \otimes B(H))$. Since Dedekind-completeness is a necessary condition for Banach spaces of the form $C_{\mathbb R}(K)$, with K compact Hausdorff, to have a predual, we can consider Chirvasitu's result as a small measure of evidence that $\mathbb Hom(M,\mathcal{R}^{\mathcal{U}}\bar \otimes B(H))$ is a dual Banach space. 

\item The second approach is through a simple observation about geometry of Banach spaces. Recall that a Banach space $B$ is called \emph{strictly convex}\index{Banach space!strictly convex} if $b_1\neq b_2$ and $||b_1||=||b_2||=1$ together imply that $||b_1+b_2||<2$.
\begin{exercise}\label{ex:strictlyconvex}
Let $B$ be a strictly convex Banach space and $C\subseteq B$ be a convex subset. Fix $c_0\in C$ and assume that there is $c\in C$ such that $d(c_0,c)$ is maximized in $C$. Show that $c$ is an extreme point of $C$.
\end{exercise}
\begin{exercise}\label{ex:maximum}
Let $M=W^*(X)$ be a singly generated II$_1$ factor which embeds into $\mathcal{R}^{\mathcal{U}}$. Fix $[\pi_0]\in\Hom(M,\mathcal{R}^{\mathcal{U}})$. Show that the function $\Hom(M,\mathcal{R}^{\mathcal{U}})\ni[\pi]\to d([\pi_0],[\pi])$ attains its maximum.
\end{exercise}

With a little bit more effort one can extend Exercise \ref{ex:maximum} to every separable II$_1$ factor. Consequently, Exercises \ref{ex:strictlyconvex} and \ref{ex:maximum} together imply that if the convex-like structure on $\Hom(M,\mathcal{R}^{\mathcal{U}})$ were strictly convex\index{Banach space!strictly convex}, then Popa's question would have affirmative answer. 

\begin{problem}
Does $\mathbb Hom(M,\mathcal{R}^{\mathcal{U}})$ embed into a strictly convex Banach space\index{Banach space!strictly convex}?
\end{problem}

Observe that Theorem \ref{th:propertyT} suggests that $\mathbb Hom(M,\mathcal{R}^{\mathcal{U}})$ itself should not be strictly convex.

\end{enumerate}

The study of the extreme points of $\Hom(M,\mathcal{R}^{\mathcal{U}})$ is not only interesting in light of Popa's question, but also because it provides a method to distinguish II$_1$ factors. For instance, Brown proved in \cite{Br2}, Corollary 5.4, that rigidity of an $\mathcal{R}^{\mathcal{U}}$-embeddable II$_1$ factor $M$ with property (T) reflects on the rigidity of the set of the extreme points of its $\Hom(M,\mathcal{R}^{\mathcal{U}})$, that turns out to be discrete. Property (T) for von Neumann algebras is a form of rigidity introduced by Connes and Jones in \cite{Co-Jo} and inspired to Kazhdan's property (T) for groups \cite{Ka}. A simple way to define property (T) for von Neumann algebras is through the following definition.

\begin{definition}
A II$_1$ factor $M$ with trace $\tau$ has \emph{property (T)}\index{property (T)} if for all $\varepsilon>0$, there exist $\delta>0$ and a finite subset $F$ of $M$ such that for all $\tau$-preserving ucp maps $\phi:M\to M$, one has
$$
\sup_{x\in F}||\phi(x)-x||_2\leq\delta \Rightarrow \sup_{Ball(M)}||\phi(x)-x||_2\leq\varepsilon.
$$ 
\end{definition}

The interpretation of property (T) as a form of rigidity should be clear: if a trace-preserving ucp map is closed to the identity on a finite set, then it is actually close to the identity on the whole von Neumann algebra.

Classical examples of factors with property (T) are the ones associated to $SL(n,\mathbb Z)$, with $n\geq3$. The following result was proved by Nate Brown in\cite{Br2}, Corollary 5.4.

\begin{theorem}\label{th:propertyT}
If $M$ has property (T), then the set of extreme points of $\Hom(M,\mathcal{R}^{\mathcal{U}})$ is discrete.
\end{theorem}

\begin{proof}
Popa proved in \cite{Po2}, Section 4.5, that for every $\varepsilon>0$, there is a $\delta>0$ such that if $[\pi], [\rho]\in\Hom(M,\mathcal{R}^{\mathcal{U}})$ are at distance $\leq\delta$, then there are projections $p\in\pi(M)'\cap  \mathcal{R}^{\mathcal{U}}$ and $q\in\rho(M)\cap \mathcal{R}^{\mathcal{U}}$ and a partial isometry $v$ such that $v^*v=p$, $vv^*=q$, $\tau(p)>1-\varepsilon$, and $v\pi(x)v^*=q\rho(x)$, for all $x\in M$. This implies that $p\pi$ and $q\rho$ are approximately unitarily equivalent and consequently, by countable saturation, $[\pi_p]=[\rho_q]$.

Now fix $\varepsilon>0$ assume that $[\pi]$ and $[\rho]$ are $\delta$-close extreme points and take projections $p\in\pi(M)'\cap \mathcal{R}^{\mathcal{U}}$ and $q\in\rho(M)'\cap \mathcal{R}^{\mathcal{U}}$ such that $[\pi_p]=[\rho_q]$. Since $[\pi]$ and $[\rho]$ are extreme points, we can apply Theorem \ref{th:extremepoints}] and conclude that $[\pi]=[\pi_p]$ and $[\rho]=[\rho_q]$, that is, $[\pi]=[\rho]$.
\end{proof}

\begin{remark}
Observe that Theorem \ref{th:propertyT} tells that the set of extreme points is discrete but, as far as we know, it might be empty. Indeed the problem of proving that extreme points actually exist for $\mathcal{R}^{\mathcal{U}}$-embeddable II$_1$ factors is open even for factors with property (T). 
\end{remark}

We conclude mentioning that also examples of factors with a continuous non-empty set of extreme points are also known in (see \cite{Br2} Corollaries 6.10 and 6.11).

\renewcommand\leftmark{CONCLUSIONS}

\renewcommand\rightmark{}

%\chapter*{Conclusions}
\chapter*{Conclusions}
\addcontentsline{toc}{chapter}{\protect\numberline{}Conclusions}

The problem of whether every countable discrete group is sofic or hyperlinear is currently of paramount
importance. Along this monograph, we have shown that answering these problems would automatically settle a number of conjectures from different fields of pure and applied mathematics.

In the sofic case this problem seems to boil down to understanding which
relations (or, more generally, existential formulas in the language of
invarant length groups; see Section \ref{Section: logic invariant length
groups}) are \emph{approximately satisfiable} in the permutation groups
endowed with the Hamming distance. This connection has been made explicit in 
\cite{glebsky_almost_2009} and \cite{arzhantseva_almost_2014}. Similar arguments holds for hyperlinear groups.

Suppose that $w\left( \overline{x}\right) $ is a word, where $\overline{x}%
=\left( x_{1},\ldots ,x_{n}\right) $, and $G$ is an invariant length group.
A tuple $\overline{g}=\left( g_{1},\ldots ,g_{n}\right) $ in $G$ is a \emph{%
solution }of the equation $\ell \left( w\left( \overline{x}\right) \right)
=0 $ if $\ell _{G}\left( \overline{w}\left( \overline{g}\right) \right) =0$
or, equivalently, $w\left( \overline{g}\right) $ equals the identity of $G$.
A $\delta $-\emph{approximate solution} of $\ell \left( w\left( \overline{x}%
\right) \right) =0$ is a tuple $\left( \overline{g}\right) $ such that $\ell
^{G}\left( w\left( \overline{g}\right) \right) <\delta $.

Suppose now that $\mathcal{C}$ is a class of invariant length groups. The
formula $\ell \left( w\left( \overline{x}\right) \right) $ is \emph{stable }%
with respect to $\mathcal{C}$ if, roughly speaking, every approximated
solution of such an equation is close to an exact solution. More precisely
for any given $\varepsilon >0$ there is $\delta >0$ such that whenever $G$
is an element of the class $\mathcal{C}$ and $\overline{g}$ is $\delta $%
-approximate solution of the equation $\ell \left( w\left( \overline{x}%
\right) \right) =0$ in $G$, there is an exact solution $\overline{h}$ in $G$
such that $\max_{i}\ell _{G}\left( g_{i}h_{i}^{-1}\right) <\varepsilon $.

More generally if $\varphi \left( \overline{x}\right) $ is an arbitrary
formula in the language of invariant length groups, then one can define as
above the notion of solution and $\delta $-approximate solution of the
equation $\varphi \left( \overline{x}\right) =0$. The formula $\varphi
\left( \overline{x}\right) $ is \emph{stable }with respect to a class $%
\mathcal{C}$ if for any $\varepsilon >0$ there is $\delta >0$ such that
whenever $G$ is an element of $\mathcal{C}$ and $\overline{g}$ is a tuple in 
$G$ such that $\varphi \left( \overline{g}\right) <\delta $ there is a tuple 
$\overline{h}$ in $G$ such that $\varphi \left( \overline{h}\right) =0$ and $%
\max_{i}\ell _{G}\left( g_{i}h_{i}^{-1}\right) <\varepsilon $. Stability is
a key notion in model theory for metric structures, being tightly connected
to the concept of definability \cite%
{carlson_omitting_2014,farah_model_2014-1}. (This use of the word
\textquotedblleft stability\textquotedblright\ should not be confused with
the classical model-theoretic stability theory as in \cite%
{shelah_classification_1990}.) In the field of operator algebras, stability
and the related notion of (weak) semiprojectivity are of fundamental
importance; see \cite{Loring}.

Recall from Section \ref{Other metric approximations} that a countable
discrete group has the $\mathcal{C}$\emph{-approximation property} if it
admits a \emph{length-preserving} embedding into a length ultraproduct of
elements of $\mathcal{C}$, i.e. an embedding such that the image of every element other than the identity has length $1$. For a discrete group, soficity is
equivalent to the $\mathcal{C}$-approximation property, where $\mathcal{C}$
is the class of permutation groups endowed with the Hamming distance.
Similarly a discrete group is LEF (respectively LEA) if and only if it has
the $\mathcal{C}$-approximation property, where $\mathcal{C}$ is the class
of finite (resp. amenable) groups endowed with the trivial length function.

Stability of a formula with respect to a class $\mathcal{C}$ of invariant
length groups allows one to perturb a $\mathcal{C}$-approximation to obtain
a $\mathcal{C}_{disc}$-approximation, where $\mathcal{C}_{disc}$ is the
class of groups from $\mathcal{C}$ \emph{endowed with the trivial length
function}. The following proposition is a consequence of this observation.

\begin{proposition*}
\label{Proposition:Cdisc}Suppose that 
\begin{equation*}
\Gamma =\left\langle g_{1},\ldots ,g_{n}:w_{i}\left( \overline{g}\right) =1%
\text{ for }i\leq k\right\rangle 
\end{equation*}
is a finitely presented group. If $\Gamma $ has the $\mathcal{C}$%
-approximation property, and the formula 
\begin{equation*}
\max_{i\leq k}\ell \left( w_{i}\left( \overline{x}\right) \right)
+(1-\min_{j\leq n}\ell \left( x_{j}\right) )
\end{equation*}
is stable with respect to the class $\mathcal{C}$, then $\Gamma $ has the $%
\mathcal{C}_{disc}$-approximation property.
\end{proposition*}

In the particular case when $\mathcal{C}$ is the class of permutation groups
endowed with the Hamming length function one obtains the following
consequence; see also \cite[Proposition 3]{glebsky_almost_2009} and \cite[%
Theorem 7.3]{arzhantseva_almost_2014}. (Recall that a finitely presented LEF
group is residually finite.)

\begin{corollary*}[Glebsky-Rivera, Arzhantseva-Paunsecu]
\label{Corollary:LEF-nonsofic}Suppose that%
\begin{equation*}
\Gamma =\left\langle g_{1},\ldots ,g_{n}:w_{i}\left( \overline{g}\right) =1%
\text{ for }i\leq n\right\rangle 
\end{equation*}
is a finitely-presented group that is not residually finite. If the formula 
\begin{equation*}
\max_{i\leq n}\ell \left( w_{i}\left( \overline{x}\right) \right)
+1-\min_{j\leq n}\ell \left( x_{j}\right) \text{\label%
{Equation:formula-relation}}
\end{equation*}%
is stable with respect to the class of permutation groups endowed with the
Hamming length function, then $\Gamma $ is not sofic.
\end{corollary*}

This provides a possible line of attack to the soficity problem of some
groups, first suggested in \cite{glebsky_almost_2009} and \cite%
{arzhantseva_almost_2014}. There are in fact many finitely-presented groups
that are known to be not residually finite. Among these Higman's group which
we have considered in Section \ref{Other metric approximations}. Recall that
this is the group $H$ with presentation%
\begin{equation*}
\left\langle h_{1},h_{2},h_{3},h_{4}:h_{i+1}h_{i}h_{i+1}^{-1}h_{i}^{-2}=1%
\text{ for }i\leq 4\right\rangle 
\end{equation*}%
where the sum $i+1$ is evaluated modulo $4$. Thus the corollary above
provides the following sufficient conditions for $H$ being not sofic: the
fomula%
\begin{equation*}
\max_{i\leq k}\ell \left( x_{i+1}x_{i}x_{i+1}^{-1}x_{i}^{-2}\right) +1-\min
\left\{ \ell \left( x_{1}\right) ,\ell \left( x_{2}\right) ,\ell \left(
x_{3}\right) ,\ell \left( x_{4}\right) \right\} 
\end{equation*}%
is stable with respect to the class of permutation groups endowed with the
Hamming length function. This means that for every $\varepsilon >0$ there is 
$\delta >0$ such that whenever $n\in \mathbb{N}$ and $\sigma _{1},\sigma
_{2},\sigma _{3},\sigma _{4}\in S_{n}$ are such that $\ell _{S_{n}}\left(
\sigma _{i}\right) >1-\delta $ and%
\begin{equation*}
\ell _{S_{n}}\left( \sigma _{i+1}\sigma _{i}\sigma _{i+1}^{-1}\sigma
_{i}^{-2}\right) <\delta 
\end{equation*}%
for $i\leq 4$, then there are $\tau _{1},\tau _{2},\tau _{3},\tau _{4}\in
S_{n}$ such that $\tau _{i+1}\tau _{i}\tau _{i+1}^{-1}=\tau _{i}^{2}$ and $%
\ell _{S_{n}}\left( \tau _{i}\sigma _{i}^{-1}\right) <\varepsilon $ for $%
i\leq 4$.

Another candidate for this line of approach is Thompson's group $F$. This is
a very famous groups with several equivalent descriptions. For our purposes
it can be regarded as the finitely presented groups with presentation%
\begin{equation*}
\left\langle a,b:\left[ ab^{-1},a^{-1}ba\right] =\left[ ab^{-1},a^{-2}ba^{2}%
\right] =1\right\rangle \text{.}
\end{equation*}%
While it is known that $F$ is not elementarily amenable, it is a famous open
problem whether $F$ is amenable. If fact it even seems to be unkown whether $%
F$ is sofic. In view of the corollary above one can conclude that a
sufficient condition to refute soficity of $F$ is the stability of the
formula%
\begin{equation*}
\max \left\{ \ell \left( \left[ xy^{-1},x^{-1}yx\right] \right) ,\ell \left( %
\left[ xy^{-1},x^{-2}yx^{2}\right] \right) \right\} +1-\min \left\{ \ell
\left( x\right) ,\ell \left( y\right) \right\} 
\end{equation*}%
with respect to the class of permutation groups endowed with the Hamming
length function. Concretely this means that for every $\varepsilon >0$ there
is $\delta >0$ such that whenever $n\in \mathbb{N}$ and $\sigma ,\rho \in
S_{n}$ are such that $\ell _{S_{n}}\left( \sigma \right) >1-\delta $, $\ell
_{S_{n}}\left( \rho \right) >1-\delta $, $\ell _{S_{n}}\left( \left[ \sigma
\rho ^{-1},\sigma ^{-1}\rho \sigma \right] \right) <\delta $, and $\ell
_{S_{n}}\left( \left[ \sigma \rho ^{-1},\sigma ^{-2}\rho \sigma ^{2}\right]
\right) <\delta $, there are $\tau ,\lambda \in S_{n}$ such that $\tau
\lambda ^{-1}$ commutes with $\tau ^{-1}\lambda \tau $ and $\tau
^{-2}\lambda \tau ^{2}$, $\ell _{S_{n}}\left( \sigma \tau ^{-1}\right)
<\varepsilon $, and $\ell _{S_{n}}\left( \rho \lambda ^{-1}\right)
<\varepsilon $.

These reformulations show the importance of determining stability of
formulas in permutation groups. This problem seems to be currently not well
understood. Among the few papers dedicated to this subject we can mention 
\cite{glebsky_almost_2009, moreno_blocks_2013, arzhantseva_almost_2014}. In
particular it is shown in \cite{arzhantseva_almost_2014} that, remarkably,
the commutator relation $\ell \left( xyx^{-1}y^{-1}\right) =0$ is stable in
permutation groups. This seems to be the first natural step towards
determining whether the nonsoficity criterion above applies to Higman's and
Thompson's group.

%\newpage

\appendix

\renewcommand\leftmark{APPENDICES}

\chapter{Tensor product of C*-algebras} \label{suse:tensor product}

%\addcontentsline{toc}{chapter}{A. Tensor product of C*-algebras}

\renewcommand\rightmark{A. TENSOR PRODUCT OF C*-ALGEBRAS}

A \emph{normed *-algebra}\index{normed *-algebra} is an algebra $A$ (over $\mathbb C$) equipped with:

\begin{enumerate}
\item an involution $*$ such that

\begin{itemize}
\item $(x+y)^*=x^*+y^*$,
\item $(xy)^*=x^*y^*$,
\item $(\lambda x)^*=\overline\lambda x^*$.
\end{itemize}
\item a norm $||\cdot||$ such that $||xy||\leq||x||||y||$.
\end{enumerate}

A Banach *-algebra is a normed *-algebra that is complete.

\begin{appdefinition}
A \emph{C*-algebra}\index{C*-algebra} is a Banach *-algebra verifying the following additional property, called \emph{C*-identity}\index{C*-identity}:
$$
||x^*x||=||x||^2, \qquad\forall a\in A.
$$
\end{appdefinition}

The C*-identity is a relation between algebraic and topological properties. It indeed implies that sometimes algebraic properties imply topological properties. A classical example of this interplay is the following fact.

\textbf{Fact.} A *-homomorphism from a normed *-algebra to a C*-algebra is always a contraction.

Example of C*-algebras certainly include the (commutative) algebra of complex valued functions on a compact space equipped with the sup norm and $B(H)$ itself. But, exactly as in case of von Neumann algebras, one can start from a group and construct a C*-algebra in a natural way.

\begin{appexample}\label{ex:fullCstaralgebra}
Let $G$ be a locally compact group. Fix a left-Haar measure $\mu$ and construct the convolution *-algebra $L^1(G)$ as follows: the elements of $L^1(G)$ are $\mu$-integrable complex-valued functions on $G$. The convolution is defined by $(f*g)(x)=\int_Gf(y)g(y^{-1}x)d\mu(y)$ and the involution is defined by $f^*(x)=\overline{f(x^{-1})}\Delta(x^{-1})$, where $\Delta$ is the modular function, i.e.\ the (unique) function $\Delta:G\rightarrow[0,\infty)$ such that for all Borel subsets $A$ of $G$ one has $\mu(Ax^{-1})=\Delta(x)\mu(A)$. The \emph{full C*-algebra}\index{C*-algebra!full} of $G$, denoted by $C^*(G)$ is the enveloping C*-algebra of $L^1(G)$, i.e.\ the completion of
$L^1(G)$ with respect to the norm $||f||=sup_{\pi}||\pi(f)||$, where
$\pi$ runs over all non-degenerate *-representations of $L^1(G)$ in a
Hilbert space\footnote{A representation $\pi:L^1(G)\rightarrow B(H)$ is said to be non-degenerate if the set $\{\pi(f)\xi : f\in L^1(G),\xi\in H\}$ is dense in $H$.}.

\begin{appexercise}
Show that in fact $||\cdot||$ is a norm on $L^1(G)$.
\end{appexercise}

In particular, we can make this construction for
the free group on countably many generators, usually denoted by $\mathbb F_\infty$. Observe that this group is countable and so its Haar measure is the counting measure which is bi-invariant. Consequently, the modular function is constantly equal to 1.

\begin{appexercise}\label{exer:universalunitaries}
Let $\delta_g:G\to\mathbb R$ be the characteristic function of the point $g$. Show that $g\to\delta_g$ is an embedding $\mathbb F_\infty\hookrightarrow U(C^*(\mathbb F_\infty))$.
\end{appexercise}

\begin{appexercise}\label{exer:universalunitariestotal}
Prove that the unitaries in Exercise \ref{exer:universalunitaries} form a norm total sequence in $C^*(\mathbb F_\infty)$.
\end{appexercise} 
\end{appexample}
Given two C*-algebras $A_1$ and $A_2$, their algebraic tensor product is a *-algebra in
a natural way, by setting
$$
(x_1\otimes x_2)(y_1\otimes y_2)=x_1x_2\otimes y_1y_2,
$$
$$
(x_1\otimes x_2)^*=x_1^*\otimes x_2^*.
$$
Nevertheless it is not clear how to define a norm to obtain a
C*-algebra.

\begin{appdefinition}
Let $A_1,A_2$ be two C*-algebras and $A_1\odot A_2$ their
algebraic tensor product. A norm $||\cdot||_{\beta}$ on $A_1\otimes
A_2$ is called \emph{C*-norm}\index{norm!C*-} if the following properties are satisfied:
\begin{enumerate}
\item $||xy||_{\beta}\leq||x||_{\beta}||y||_{\beta}$, for all
$x,y\in A_1\odot A_2$;
\item $||x^*x||_{\beta}=||x||_{\beta}^2$, for all $x\in A_1\odot
A_2$.
\end{enumerate}
If $||\cdot||_{\beta}$ is a C*-norm on $A_1\odot A_2$, then
$A_1\otimes_{\beta}A_2$ denotes the completion of $A_1\odot
A_2$ with respect to $||\cdot||_{\beta}$. It is a C*-algebra.
\end{appdefinition}

\begin{appexercise}
Prove that every C*-norm $\beta$ is multiplicative on elementary tensors, i.e.\ $||x_1\otimes x_2||_\beta=||x_1||_{A_1}||x_2||_{A_2}$.
\end{appexercise}

The interesting thing is that one can define at least two different C*-norms, a minimal one and a maximal one, and they are indeed different in general. To define these norms, let us first recall that a representation of a *-algebra $A$ is a *preserving algebra-morphism from $A$ to some $B(H)$. We denote $\text{Rep}(A)$ the set of representation of $A$.

\begin{appdefinition}
\begin{align}
||x||_{\max}=\sup\left\{||\pi(x)|| :
\pi\in\text{Rep}(A_1\odot A_2)\right\}
\end{align}
\end{appdefinition}

\begin{appexercise}
Prove that $||\cdot||_{\max}$ is indeed a C*-norm. (Hint: to prove that $||x||_{\max}<\infty$, for all $x$, take inspiration from Lemma 11.3.3(iii) in \cite{Kadison-RingroseII}).
\end{appexercise}

The norm $||\cdot||_{\max}$ is named \emph{maximal norm}\index{norm!maximal}, or projective norm, or Turumaru's norm, being first introduced in \cite{Tu}. The completion of
$A_1\odot A_2$ with respect to it is denoted by
$A_1\otimes_{\max}A_2$.

Let $A$ be a C*-algebra and $S\subseteq A$. Denote by $C^*(S)$ the C*-subalgebra of $A$ generated by $S$. The maximal norm has the following universal property (see
\cite{Ta1}, IV.4.7).

\begin{appproposition}\label{universal}
Given C*-algebras $A_1,A_2,B$. Assume $\pi_i:A_i\rightarrow B$ are
*-homomorphisms with commuting ranges, that is for all $x\in\pi_1(A_1)$ and $y\in\pi_2(A_2)$, one has $xy=yx$. Then there exists a unique *-homomorphism $\pi:A_1\otimes_{\max}A_2\rightarrow B$ such that
\begin{enumerate}
\item $\pi(x_1\otimes x_2)=\pi_1(x_1)\pi_2(x_2)$
\item $\pi(A_1\otimes_{\max}A_2)=C^*(\pi_1(A_1),\pi_2(A_2))$
\end{enumerate}
\end{appproposition}

\begin{appexercise}
Prove Proposition \ref{universal}.
\end{appexercise}

We now turn to the definition of the minimal C*-norm. The idea behind its definition is very simple. Instead of considering all representations of the algebraic tensor product, one takes only representations which split into the tensor product of representations of the factors.

\begin{appdefinition}
\begin{align}
||x||_{\min}=\sup\left\{||(\pi_1\otimes\pi_2)(x)|| :
\pi_i\in\text{Rep}(A_i)\right\}
\end{align}
\end{appdefinition}

\begin{appexercise}
Prove that $||\cdot||_{\min}$ is a C*-norm on $A_1\odot A_2$.
\end{appexercise}

This norm is named \emph{minimal norm}\index{norm!minimal}, or injective norm, or Guichardet's norm, being first introduced in \cite{Gu}. The completion of
$A_1\odot A_2$ with respect to it is denoted by
$A_1\otimes_{\min}A_2$.

\begin{appremark}
Clearly $||\cdot||_{\min}\leq||\cdot||_{\max}$, since representations
of the form $\pi_1\otimes\pi_2$ are particular *-representation of
the algebraic tensor product $A_1\odot A_2$. These norms are
different, in general, as Takesaki showed in \cite{Ta2}. Notation
$||\cdot||_{\max}$ reflects the fact that there are no
C*-norms greater than the maximal norm, that follows straightforwardly from the GNS construction. Notation $||\cdot||_{\min}$ has the
same justification, but it is much harder to prove:
\begin{apptheorem}{\bf (Takesaki, \cite{Ta2})}
$||\cdot||_{\min}$ is the smallest C*-norm on
$A_1\odot A_2$.
\end{apptheorem}

\end{appremark}

\titleformat{\section}%
  [hang]% <shape>
  {\normalfont\bfseries\Large}% <format>
  {}% <label>
  {0pt}% <sep>
  {}% <before code>
	
	\titleformat{\subsection}%
  [runin]% <shape>
  {\normalfont\bfseries}% <format>
  {\thesubsection}% <label>
  {12pt}% <sep>
  {}% <before code>
	
	\titleformat{\subsubsection}%
  [runin]% <shape>
  {\normalfont\bfseries}% <format>
  {\thesubsubsection}% <label>
  {12pt}% <sep>
  {}% <before code>

\renewcommand{\thesection}{}% Remove section references
\renewcommand{\thesubsection}{B.\arabic{subsection}}

\newcommand{\hiddensubsection}[1]{
    \stepcounter{subsection}
    \subsection*{B.\arabic{subsection}\hspace{1em}{#1}}
}

\makeatletter
\newcommand{\hiddenlabel}[1]{\def\@currentlabel{B.\arabic{subsection}}\label{#1}}
\makeatother
 
%\titleformat{\subsection}[leftmargin]{}

\chapter[Ultrafilters and ultralimits (by V. Pestov)]{Ultrafilters and ultralimits} \label{app:ultrafilters}
\setcounter{tocdepth}{1}

\centerline{\bf by Vladimir G. Pestov \footnote{ {\em Departamento de Matem\'atica,
Universidade Federal de Santa Catarina,
Campus Universit\'ario Trindade,
  CEP 88.040-900 Florian\'opolis-SC, Brasil} (Pesquisador Visitante Especial do CAPES, processo 085/2012) and
  {\em Department of Mathematics and Statistics, University of Ottawa,
585 King Edward Avenue, Ottawa, Ontario, K1N6N5 Canada} (Permanent address).}}

\setcounter{subsection}{0}

\bigskip

\hiddensubsection{} Let us recall that the symbol $\ell^\infty$ denotes the linear space consisting of all bounded sequences of real numbers. This linear space is sometimes viewed as an algebra, meaning that it supports a natural multiplication (the product of two bounded sequences is bounded). Besides, $\ell^\infty$ is equipped with the supremum norm,
\[\norm x = \sup_{n=1}^\infty \abs{x_n}\]
and the corresponding metric.
\hiddensubsection{\hiddenlabel{vp1}}

The usual notion of the limit of a convergent sequence of real numbers can be interpreted as a mapping from a subset of $\ell^\infty$ consisting of all convergent sequences (this subset is usually denoted $c$, it is a closed normed subspace of $\ell^\infty$) to $\R$. This map,
\[c\ni x =(x_n)_{n=1}^\infty\mapsto \lim_{n\to\infty}x_n\in\R,\]
has the following well-known properties (refer to first-year Calculus for their proofs):
\begin{enumerate}
\item $\lim_{n\to\infty}(x_n+y_n)=\lim_{n\to\infty}x_n+\lim_{n\to\infty} y_n$.
\item $\lim_{n\to\infty}(x_ny_n)=\lim_{n\to\infty}x_n\cdot \lim_{n\to\infty} y_n$.
\item $\lim_{n\to\infty}(\lambda x_n)=\lambda \lim_{n\to\infty}x_n$.
\item If $x_n\leq y_n$, then $\lim_{n\to\infty}x_n\leq \lim_{n\to\infty} y_n$.
\item \label{four}
  The limit of a constant sequence of ones is $1$:
  \[\lim_{n\to\infty}\bar 1 = 1,\]
  where $\bar 1 = (1,1,1,\ldots)$.
\end{enumerate}

\renewcommand\rightmark{B. ULTRAFILTERS AND ULTRALIMITS}

\hiddensubsection{\hiddenlabel{vp2}} In a sense, the main shortcoming of the notion of a limit is that it is not defined for {\em every} bounded sequence. For instance, the following sequence has no limit in the usual sense:
\begin{equation}
  \label{eq:binary}
  0,1,0,1,0,1,0,\ldots,0,1,0,\ldots
  \end{equation}
As we all remember from the student years, there are very subtle cases where proving or disproving the convergence of a particular sequence can be tricky! 

Can one avoid such difficulties and devise the notion of a ``limit'' which would have all the same properties as above, and yet with regard to which {\em every sequence} will have a limit, even the one in formula \eqref{eq:binary}?

\hiddensubsection{\hiddenlabel{vp3}} The response to the question in (\ref{vp2}) turns out to be very simple and disappointing. It is enough, for example, to define a ``limit'' of a sequence $(x_n)$ by selecting the first term:
\[\phi(x)=x_1.\]
This clearly satisfies all the properties listed in (\ref{vp1}). But this is not what we want: such a ``limit'' does not reflect the asymptotic behaviour of a sequence at the infinity. And of course this unsatisfying response is due to a poorly-formulated question.
We have, therefore, to reformulate the question as follows.

\hiddensubsection{\hiddenlabel{vp4}} Does there exist a notion of a ``limit'' of a sequence which 
\begin{itemize}
\item is defined for all bounded sequences,
\item
  has all the same properties listed in (\ref{vp1}), and 
\item coincides with the classical limit in the case of convergent sequences?
\end{itemize}

\hiddensubsection{\hiddenlabel{vp5}} Answering this question is the main goal of the present Appendix. Meanwhile, the sequence in Eq.\ \eqref{eq:binary} will serve as a guiding star, or rather as a guinea pig. Indeed, assuming such a wonderful specimen of a limit does exist, it will in particular assign a limit to this concrete sequence, so perhaps it would be a good idea, to begin by trying to guess what it will be? Say, will $1/2$ be a reasonable suggestion?

Let us analyse this question in the context of all {\em binary} sequences, that is, sequences taking values $0$ and $1$ only. The family of all binary sequences is $\{0,1\}^{\N_+}$, the Cantor space (though the metric and topology on this space will not play any role for the moment). Every binary sequence is just a map
\[x\colon \N_+\to \{0,1\},\]
and so can be identified with the characteristic function of a suitable subset $A\subseteq\N_+$ of natural numbers, namely the set of all $n$ where $x_n$ takes value one:
\[\chi_A(n) =\begin{cases}1,&\mbox{ if }n\in A,\\
0,&\mbox{ otherwise.}\end{cases}\]
For instance, the sequence in Eq.\ \eqref{eq:binary} is the characteristic (indicator) function of the set of all even natural numbers.

\hiddensubsection{\hiddenlabel{vp6}} Suppose the desired limit exists. Let us denote it, for the time being, simply by $\lim$. What are the properties of this limit on the set of binary sequences?

First of all, it must send the constant sequence of ones (the indicator function of $\N_+$) to one, this was one of the rules required:

\subsubsection{\label{111}} $\lim \chi_{\N_+}=\lim (1,1,1,1,\ldots)=1$.
\smallskip

This implies that the limit of the zero sequence (the indicator function of the empty set) is zero: indeed, the sum of the two sequences is the sequence of ones, and the additive property of the limit implies
\[1+ \lim\chi_{\emptyset}=\lim\chi_{\N_+}+\lim\chi_{\emptyset} = \lim\chi_{\N_+} =1.\]
We summarize:

\subsubsection{} $\lim \chi_{\emptyset}=\lim (0,0,0,0,\ldots)=0$.
\smallskip

The next observation is that the only possible values of our limit on binary sequences are $0$ or $1$. Indeed, for a binary sequence $(\ve_n)$, where $\ve_n\in\{0,1\}$, one has $\ve_n^2=\ve_n$, and so
\[\lim(\ve_n)=\lim(\ve_n^2) = (\lim\ve_n)^2,\]
but there are only two real numbers satisfying the property $\ve^2=\ve$.

\subsubsection{} For every subset $A\subseteq\N_+$, one has $\lim\chi_A\in\{0,1\}$.
\smallskip

This excludes a ``natural'' possibility to assign the limit $1/2$ to the sequence in Eq.\ \eqref{eq:binary}. 

Now let $\chi_A$ be any binary sequence. We can ``flip'' the sequence and replace all ones with zeros, and vice versa. This corresponds to the indicator function of the complement $A^c=\N_+\setminus A$. What about the limit of this sequence? It turns out the limit flips as well.

\subsubsection{} $\lim (\chi_{A^c})=1-\lim\chi_A$.
\smallskip

Indeed, $\chi_A+\chi_{A^c}=\chi_{\N_+}$ is the identical sequence of ones whose limit is one, and now one has
\[\lim\chi_{A^c} = \lim\chi_{\N_+}-\lim\chi_A = 1-\lim\chi_A.\]
Since the limit of {\em every} binary sequence must exist by assumption, this property implies the next one:

\subsubsection{} Let $A$ be any subset of the positive natural numbers. Then either $\lim\chi_A=1$, or $\lim\chi_{A^c}=1$, but not both at the same time.
\smallskip

Let us address the following situation. Suppose we have a sequence $\chi_A$ whose limit we know to be one. Now we take a subset $B$ of the natural numbers which contains $A$:
\[B\supseteq A,\]
and consider the sequence $\chi_B$. Thus, some zeros in $\chi_A$ have been (possibly) replaced with ones, and all the ones in $\chi_A$ keep their values. What about the limit of the new sequence, $\chi_B$?

\subsubsection{\label{subset}} If $A\subseteq B$ and $\lim\chi_A=1$, then $\lim\chi_B=1$.
\smallskip

Indeed, $\chi_A\leq\chi_B$, and according to the property (\ref{four}) in  Subs. \ref{vp1}, we conclude: 
\[\lim\chi_A\leq\lim\chi_B.\]
However, $\lim\chi_A=1$, and $\lim\chi_B$ cannot be strictly greater than one. We conclude.

The next situation to consider is as follows. Suppose $A$ and $B$ are two subsets of the natural numbers. Their intersection corresponds to the product of the indicator functions, as is well known and easy to see:
\[\chi_A\cdot \chi_B =\chi_{A\cap B}.\]
It follows that
\[\lim\chi_A\cdot\lim\chi_B =\lim\chi_{A\cap B}.\]
In particular:

\subsubsection{\label{acapb}} If $\lim\chi_A=\lim\chi_B=1$, then $\lim\chi_{A\cap B}=1$.

\hiddensubsection{\hiddenlabel{ultralimit}} Now it is time to put together some of the properties proved. Notice that instead of binary functions, one can just as well talk of subsets of the set $\N_+$. To some subsets $A$ there is associated the value one (when $\lim\chi_A=1$), to others, the value zero (if $\lim\chi_A=0$). All subsets of $\N_+$ are therefore grouped in two classes. Let us denote the class of all subsets $A$ to which we associate the limit one by $\mathcal U$:
\[{\mathcal U} =\{A\subseteq\N_+\colon\lim\chi_A=1\}.\]
The properties established above immediately translate into the following.

\begin{enumerate}
\item $\emptyset\notin{\mathcal U}$.
\item For every subset $A\subseteq\N_+$, either $A\in {\mathcal U}$, or $A^c\in{\mathcal U}$.
\item If $A,B\in {\mathcal U}$, then $A\cap B\in {\mathcal U}$.
\end{enumerate}

\subsubsection{} {\em Ultrafilters.} A collection of subsets of natural numbers (or, in fact, of any fixed non-empty set) satisfying the above three axioms is called an {\em ultrafilter} (on this set).

There is no need to include more axioms in the definition of an ultrafilter.

\subsubsection{} {\em Exercise.} Convert the rest of the properties that we have established in (\ref{111})--(\ref{acapb}) into a set-theoretic form and deduce them from the three axioms above. 

\subsubsection{\label{asubb}} {\em Example.} The property (\ref{subset}) becomes: if $A\subseteq B$ and $A\in {\mathcal U}$, then $B\in {\mathcal U}$. 
\smallskip

$\triangleleft$
This follows from the axioms of an ultrafilter: assuming the contrary, one must have $\N\setminus B\in {\mathcal U}$ (axiom 2), meaning $\emptyset = A\cap (X\setminus B)\in {\mathcal U}$, a contradiction. \hfill $\triangleright$

\hiddensubsection{} Thus, each time we have the concept of a limit with the properties that we have specified in Subs. \ref{vp1}, we get an ultrafilter on the set of natural numbers.

\subsubsection{} {\em Principal ultrafilters.} In particular, to the disappointing example of a ``limit'' described in Subs. \ref{vp3}, there corresponds the following ultrafilter:
\[{\mathcal U}=\{A\subseteq\N\colon 1\in A\}.\]
In other words, a binary sequence has limit one if and only if the first term of this sequence is one. This happens exactly when $A\ni 1$, where our sequence is $\chi_A$. 

More generally, one can repeat the construction given any chosen element $n\in\N_+$. The resulting ultrafilter is denoted $(n)$,
\[(n) =\{A\subseteq\N_+\colon n\in A\},\]
and called a {\em trivial} (or: {\em principal}) ultrafilter.

\subsubsection{} {\em Non-principal ultrafilters.} This is, in a sense, a non-interesting situation. What is an interesting one? An ultrafilter $\mathcal U$ is called {\em free,} or {\em non-principal,} if its elements have no points in common:
\[\bigcap{\mathcal U}=\emptyset.\]

\subsubsection{} {\em Every ultrafilter is either trivial or free.}
\smallskip

$\triangleleft$
Suppose $\mathcal U$ is an ultrafilter which is not trivial. This means: for every $n\in\N_+$, $\mathcal U\neq (n)$. This can mean two things: either there is $A\ni n$ which is not in the ultrafilter (in which case $A^c$ is), or else there is $A\in {\mathcal U}$ so that $A\not\ni n$. In both cases, we can find an element $A$ of $\mathcal U$ not containing $n$. Since this holds for all $n$, it is clear that the intersection of all members of $\mathcal U$ is an empty set, and the ultrafilter is non-principal. 
\hfill $\triangleright$

\hiddensubsection{} How to establish the existence of free ultrafilters? We need the following notion. A system $\mathcal C$ of subsets of a certain non-empty set is called a {\em centred system} if the intersection of every finite collection of elements of $\mathcal C$ is non-empty:
\[\forall n,~~\forall A_1,A_2,\ldots,A_n\in {\mathcal C}, ~~
A_1\cap A_2\cap\ldots\cap A_n\neq\emptyset.\]

\subsubsection{} {\em Every ultrafilter is a centred system.}
\smallskip

This follows immediately from one of the axioms of an ultrafilter.

% \subsubsection{} Here is a source of more examples of centred systems.
% A non-empty collection $\mathcal F$ of subsets of some set is called a {\em filter} if it satisfies the following properties.
% 
% \begin{enumerate}
% \item $\emptyset\notin {\mathcal F}$.
% \item If $A,B\in {\mathcal F}$, then $A\cap B\in {\mathcal F}$.
% \item If $A\in {\mathcal F}$ and $A\subseteq B$, then $B\in {\mathcal F}$.
% \end{enumerate}
% 
% \subsubsection{}
% {\em Every ultrafilter $\mathcal U$ is a filter.} 
% \smallskip
% 
% $\triangleleft$ Only the third property needs verifying, and it is Example \ref{asubb}. \hfill $\triangleright$
% 
% \subsubsection{Example} The {\em Fr\'echet filter} consists of all cofinite subsets of $\N_+$ (that is, subsets $A$ with $\N_+\setminus A$ finite). It is clearly not an ultrafilter, and thus there are many more filters than ultrafilters.

\subsubsection{} {\em Every ultrafilter is a maximal centred system. In other words, if $\mathcal U$ is an ultrafilter and $A\notin {\mathcal U}$, then the system ${\mathcal U}\cup\{A\}$ is not centred.}

$\triangleleft$
Indeed, since $A\notin {\mathcal U}$, we must have $A^c\in {\mathcal U}$, and since $A$ and $A^c$ are both members of the system ${\mathcal U}\cup\{A\}$ and their intersection is empty, the system is not centred.
\hfill
$\triangleright$

\subsubsection{} {\em Ultrafilters are exactly maximal centred systems.}

More precisely, let $\mathscr C$ be a collection of subsets of a non-empty set (in our case, we are interested in the set of natural numbers, $\N_+$). Then $\mathscr C$ is an ultrafilter if and only if $\mathscr C$ is a maximal centred system. 

We have already established $\Rightarrow$ in the previous paragraph, now let us prove the other implication. 

$\triangleleft$ Suppose $\mathscr C$ is a maximal centred system. ``Maximal'' means that if we add to $\mathscr C$ any subset $A$ of $\N_+$ which is not in $\mathscr C$, then the resulting system of sets
\[{\mathscr C}\cup\{A\}\]
is no longer centred. Let us verify all three properties of an ultrafilter.

(1) Clearly, $\emptyset\notin {\mathscr C}$, because otherwise $\mathscr C$ would not be a centred system to begin with: for example,
\[\emptyset\cap\emptyset=\emptyset.\]

(2) Let $A,B\in {\mathscr C}$. We need to prove $A\cap B\in {\mathscr C}$. Consider the system
\[{\mathscr C}\cup\{A\cap B\}.\]
This system is centred: indeed, if $A_1,\ldots,A_n$ are elements of this system, then, if $A\cap B$ is not among them, their intersection is clearly non-empty (as $\mathscr C$ is centred). If we add $A\cap B$, then
\[A_1\cap\ldots\cap A_n\cap (A\cap B) = A_1\cap\ldots\cap A_n\cap A\cap B\neq\emptyset,\]
because $\mathscr C$ is centred and all sets $A_1,\ldots,A_n,A,B$ belong to $\mathscr C$.

We conclude: the system ${\mathscr C}\cup\{A\cap B\}$
cannot be strictly larger than $\mathscr C$, because of maximality of the latter. But this means that
\[A\cap B\in {\mathscr C}.\]

(3) Let $A\subseteq\N_+$ be any subset. We want to show that either $A$ or its complement $A^c=\N_+\setminus A$ belong to $\mathscr C$. Suppose neither holds, towards a contradiction. The assumption that $A\notin{\mathscr C}$ implies that the system ${\mathscr C}\cap \{A\}$ is not centred, and so there exist $A_1,A_2,\ldots,A_m\in{\mathscr C}$ with
\[A_1\cap A_2\cap\ldots\cap A_m\cap A=\emptyset.\]
Note that this observation can be rewritten as
\begin{equation}
  \label{eq:a}
  A_1\cap A_2\cap\ldots\cap A_m\subseteq A^c.\end{equation}
Likewise, the assumption $A^c\notin{\mathscr C}$ implies the existence of elements $B_1,B_2,\ldots,B_k\in {\mathscr C}$ satisfying
\[B_1\cap B_2\cap\ldots\cap B_k\cap A^c=\emptyset,\]
or, in an equivalent form,
\begin{equation}
\label{eq:b}
B_1\cap B_2\cap\ldots\cap B_k\subseteq A.\end{equation}
Together, the equations \eqref{eq:a} and \eqref{eq:b} imply
\[A_1\cap A_2\cap\ldots\cap A_m\cap B_1\cap B_2\cap\ldots\cap B_k\subseteq A\cap A^c =\emptyset,\]
which contradicts the fact that $\mathscr C$ is a centred system. We are done.
\hfill $\triangleright$

\hiddensubsection{} {\bf Existence of free ultrafilters.}

\subsubsection{} {\em Zorn's lemma.} {\em Let $({\mathfrak X},\leq)$ be a partially ordered set. Suppose that $X$ has the following property: every totally ordered subset $\mathfrak C$ of $\mathfrak X$ has an upper bound. (One says that $\mathfrak X$ is {\em inductive}.) Then $\mathfrak X$ has a maximal element.}
\vskip 0.3cm

An element $x$ is {\em maximal} if there is no $y\in {\mathfrak X}$ which is strictly larger than $x$. It does not mean that $x$ is necessarily the largest element of $\mathfrak X$.

The statement of Zorn's lemma is an equivalent form of the Axiom of Choice, which is a part of the standard system of axioms of set theory ZFC.

\subsubsection{\label{exe:chain}} {\em Exercise.}
  Let $\mathfrak C$ be a family of centred systems on a set $\N_+$ which is totally ordered by inclusion. Prove that the union of this family, $\cup{\mathfrak C}$, is again a centred system.
  
\subsubsection{}
{\em Every centred system is contained in a maximal centred system.}
\vskip .2cm

$\triangleleft$
Let $\mathscr C$ be a centred system of subsets of a certain non-empty set (let us again assume that this set is $\N_+$, this does not affect the argument in any way.) Denote $\mathfrak X$ the family of all centred systems of subsets of $\N_+$ which contain $\mathscr C$. This family is ordered by inclusion. Due to the previous exercise, it is inductive. 
Since clearly $\cup{\mathfrak C}$ contains $\mathscr C$, it follows that $\cup{\mathfrak C}$ belongs to $\mathfrak X$ and it forms an upper bound for the chain $\cup{\mathfrak C}$.

Zorn's lemma implies the existence of a maximal element, $\mathscr D$, in $\mathfrak X$. This $\mathscr D$ is a centred system which contains $\mathscr C$ and which is not contained in any strictly larger centred system containing $\mathscr C$. Since every centred system containing $\mathscr D$ will automatically contain $\mathscr C$ as well, the latter statement can be cut down to: $\mathscr D$ is a centred system which contains $\mathscr C$ and which is not contained in any strictly larger centred system. In other words, $\mathscr D$ is a maximal centred system containing $\mathscr C$, as required.
\hfill
$\triangleright$

Let us reformulate this result as follows.

\subsubsection{} {\em Every centred system is contained in an ultrafilter.}

\subsubsection{} Consider the following collection of subsets of $\N_+$:
\[{\mathscr C}=\{\N_+\setminus \{n\}\colon n\in\N_+\}.\]
In other words, $\mathscr C$ consists of all complements to singletons.
This system is clearly centred, and its intersection is empty. According to the previous result, there is an ultrafilter $\mathcal U$ containing the system $\mathscr C$. One has:
\[\cap{\mathcal U}\subseteq \cap {\mathscr C}=\emptyset.\]
Thus, $\mathcal U$ is a non-principal (free) ultrafilter.

In fact, one can show that $\N_+$ supports $2^{\mathfrak c}$ pairwise different free ultrafilters. They are extremely numerous.

\subsubsection{}{\em Solution to Exercise \ref{exe:chain}.}
Let $n\in\N_+$ and
  \[A_1,A_2,\ldots,A_n\in \bigcup{\mathfrak C}\]
  be arbitrary elements of $\mathfrak C$. For every $i=1,2,\ldots,n$ fix a centred system $\mathcal C_i$ which belongs to the family $\mathfrak C$ and contains $A_i$:
  \[A_i\in {\mathcal C}_i\in {\mathfrak C}.\]
  Elements of the collection ${\mathcal C}_i$, $i=1,2,3,\ldots,n$ of centred systems are pairwise comparable by inclusion between themselves because $\mathfrak C$ is totally ordered by inclusion. Every finite totally ordered set has the largest element: for some $i_0=1,2,\ldots,n$ one has
  \[\forall i=1,2,\ldots,n,~~{\mathcal C}_i\subseteq {\mathcal C}_{i_0}.\]
  This means in particular that
  \[\forall i=1,2,\ldots,n,~~A_i\in {\mathcal C}_{i_0}.\]
  Since ${\mathcal C}_{i_0}$ is a centred system, one concludes
  \[A_1\cap A_2\cap\ldots\cap A_n\neq\emptyset,\]
  as required.

\hiddensubsection{}{\bf Ultralimits.} Now let us re-examine the existence of our conjunctural limits making sense for every bounded sequence. 
Notice first that if $(\ve_n)$ is a binary sequence, then, according to the initial way we defined ultrafilters in \ref{ultralimit}, 
\begin{equation}
  \lim \ve_n =1\iff \{n\colon \ve_n=1\}\in {\mathcal U}.
  \label{eq:ve}
\end{equation}
This definition depends on the choice of an ultrafilter, and for this reason, the corresponding limit is called the {\em ultralimit along the ultrafilter $\mathcal U$.} It is denoted
\[\lim_{n\to{\mathcal U}}\ve_n.\]
(This notation actually makes good sense, as we will see later.)

This definition in the above form will clearly not work in a general case (simply because it is possible that all the terms of a convergent sequence are different from the limit). However, it extends readily to sequences taking {\em finitely many} distinct values ({\em Exercise.}).
How to extend this definition to an arbitrary bounded sequence? In fact, such an extension is quite natural and moreover unique.

\subsubsection{\label{exe:approx}} {\em Exercise.} Let $\mathcal U$ be an ultrafilter on natural numbers.
By approximating a given bounded sequence $(x_n)$ with sequences taking finitely many different values in the $\ell^\infty$ norm, extend the definition of the ultralimit to all bounded sequences in such a way that it respects all of our required axioms. Show that this extension is well defined (does not depend on an approximation). 

\subsubsection{} {\em Exercise.} Conclude that, given an ultrafilter $\mathcal U$ on the natural numbers, there is a unique way to define the ultralimit along $\mathcal U$ on all bounded sequences which satisfies the properties listed in (\ref{vp1}) as well as the property in Eq.\ \eqref{eq:ve}.

\subsubsection{} {\em Exercise.} Conclude that there is a natural correspondence between the ultrafilters on the set of natural numbers and the limits on bounded sequences satisfying the properties stated in (\ref{vp1}).

At the same time, the definition of an ultralimit based on approximations is not very convenient. We are going to reformulate it now. 

\hiddensubsection{} {\bf A workable definition of an ultralimit.} The right definition in a usable form is obtained by adjusting in an obvious way the classical concept of a convergent sequence $x_n\to x$:
\[\forall\varepsilon >0,~~\exists N,~~\forall n\geq N,~~\abs{x-x_n}<\varepsilon .\]
This means that, given an $\varepsilon >0$, for ``most'' values of $n$ the point $x_n$ is within $\varepsilon $ from the limit. In this context, ``most'' means that the set of such $n$ is cofinite.

If $\mathcal U$ is an ultrafilter, ``most'' values of $n$ means that the set of $n$ with this property belongs to $\mathcal U$. So we say that a sequence $(x_n)$ of real numbers converges to a real number $x$ along the ultrafilter $\mathcal U$,
\[x=\lim_{n\to{\mathcal U}}x_n,\]
if
\[\forall \varepsilon >0,~~\exists A\in {\mathcal U},~~\forall n\in A,~~\abs{x-x_n}<\varepsilon .\]
Since every superset of such an $A$ also belongs to $\mathcal U$, an equivalent statement is:
\[\forall \varepsilon >0,~~\{n\in\N_+\colon \abs{x-x_n}<\varepsilon \}\in {\mathcal U}.\]

  \subsubsection{}{\em Exercise.} Show that the definition of the ultralimit given in this subsection is equivalent to the definition in the language of approximations in Exercise \ref{exe:approx}.

\subsubsection{}{\em Example.}
  If ${\mathcal U}=(n_0)$ is a trivial ultrafilter generated by the natural number $n_0$, then the statement
  \[\lim_{n\to{\mathcal U}}x_n=x\]
  means that for every $\varepsilon >0$,
  \[\{n\in\N_+\colon \abs{x-x_n}<\varepsilon \}\ni n_0,\]
  that is, simply 
  \[\forall \varepsilon >0,~~\abs{x-x_{n_0}}<\varepsilon ,\]
  which means
  \[x=x_{n_0}.\]
  This is an ``uninteresting'' case of an ultralimit. The interesting cases correspond to free ultrafilters. 

  \subsubsection{\label{exe:1}}{\em Exercise.}
  Show that if $\lim_{n\to\infty}x_n=x$ (in a usual sense), then for every free ultrafilter $\mathcal U$ on $\N_+$, 
  \[\lim_{n\to{\mathcal U}}x_n=x.\]

  \subsubsection{\label{exe:2}}{\em Exercise.} Show that, if one has
  \[\lim_{n\to{\mathcal U}}x_n=x\]
  for {\em every} free ultrafilter $\mathcal U$ on $\N_+$, then 
  \[\lim_{n\to\infty}x_n=x.\]

  Thus, the familiar symbol $\infty$ essentially means all the free ultrafilters lumped together, and the existence of a limit in the classical sense signifies that all the free ultrafilters on $\N_+$ agree between themselves on what the value of this limit should be. If there is no such agreement, then the limit in the classical sense does not exist, and the sequence $(x_n)$ is divergent. However, every free ultrafilter still gives its own interpretation of what the limit of the sequence is.
  
  \subsubsection{}{\em Example.} Recall the alternating sequence 
  \[\ve_n=(-1)^{n+1}\]
  of zeros and ones as in Eq.\ \eqref{eq:binary} Every ultrafilter $\mathcal U$ on $\N$ either contains the set of odd numbers, or the set of even numbers. In the former case, the ultralimit of our sequence is $0$, in the latter case, $1$. Because of this, the ultrafilters ``disagree'' between themselves on what the limit of the sequence should be. Consequently, $\lim_{n\to\infty}\ve_n$ does not exist and the sequence is divergent in the classical sense.

  \subsubsection{}{\bf Theorem.}
  {\em
  Every bounded sequence of real numbers has an ultralimit along every ultrafilter on $\N_+$.}
  \smallskip

\noindent $\triangleleft$
  Let $(x_n)$ be a bounded sequence of real numbers, which we, without loss in generality, will assume to belong to the interval $[0,1]$. Let $\mathcal U$ be an ultrafilter on $\N_+$. The proof closely resembles the proof of the Heine--Borel theorem about compactness of the closed unit interval, but is actually simpler. 
  
  Subdivide the interval $I_0=[0,1]$ into two subintervals of length half, $[0,1/2]$ and $[1/2,1]$. Set
  \[A=\{n\in\N_+\colon x_n\in [0,1/2]\}.\]
  Either $A$ belongs to $\mathcal U$, or else its complement, $A^c$, does. In the latter case, since
  \[\{n\colon x_n\in [1/2,1]\}\supseteq A^c,\]
  the set $\{n\colon x_n\in [1/2,1]\}$ is in $\mathcal U$ as well. We conclude: at least one of the two intervals of length $1/2$, denote it $I_1$, has the property
  \[\{n\in\N_+\colon x_n\in I_1\}\in {\mathcal U}.\]
  Continue dividing the intervals in two and selecting one of them. At the end, we will have selected a nested sequence of closed intervals $I_k$ of length $2^{-k}$, $k=0,1,2,3,\ldots$, with the property that for all $k$,
  \[\{n\colon x_n\in I_k\}\in{\mathcal U}.\]
  According to the Cantor Intersection Theorem, there is $c\in [0,1]$ so that
  \[\bigcap_{k=0}^\infty I_k=\{c\}.\]
  We claim this $c$ is the ultralimit of $(x_n)$ along $\mathcal U$. Indeed, let $\varepsilon >0$ be arbitrary. For a sufficiently large $k$, the interval $I_k$ is contained in the interval $(c-\varepsilon ,c+\varepsilon )$ (indeed, enough to take $k=-\log_2\varepsilon  + 1$). Now one has
  \[\{n\colon c-\varepsilon <x_n<x+\varepsilon \}\supseteq \{n\colon x_n\in I_k\}\ni {\mathcal U},\]
  from where one concludes that the former set is also in $\mathcal U$.
\hfill $\triangleright$

\subsubsection{}{\em Solution to Exercise \ref{exe:1}.}
We will verify the definition of an ultralimit. Let $\varepsilon >0$. For some $N$ and every $n\geq N$, one has $\abs{x-x_n}<\varepsilon $. Let us write
  \[\N=\{1\}\cup \{2\}\cup\ldots\cup \{N-1\}\cup \{N,N+1,N+2,\ldots\}.\]
  One of these sets must belong to $\mathcal U$. If it were one of the singletons, $\{i\}$, then every other element $A\in {\mathcal U}$ must meet $\{i\}$, therefore contain $i$, so $\mathcal U$ would be a principal ultrafilter generated by $i$. One concludes: 
  \[ \{N,N+1,N+2,\ldots\}\in{\mathcal U},\]
  and so
  \[\{n\colon \abs{x-x_n}<\varepsilon \}\in {\mathcal U},\]
  because the set on the left hand side is a superset of $\{N,N+1,N+2,\ldots\}$.
  
  \subsubsection{}{\em Solution to Exercise \ref{exe:2}.}
  We have just seen that if $x_n\to x$ in the classical sense, then we have convergence $x_n\overset{\mathcal U}\to x$ along every free ultrafilter $\mathcal U$. This means that the only case to eliminate is where $x_n\overset{\mathcal U}\to x$ along every free ultrafilter $\mathcal U$, yet at the same time $x_n\not\to x$ in the classical sense. 
  Assume that $x_n$ does not converge to $x$. Then for some $\varepsilon >0$ and every $N$ there is $n=n(N)\geq N$ with $d(x,x_{n})\geq \varepsilon $. The set $I=\{n(N)\colon N\in\N\}$ is infinite. There is a free ultrafilter $\mathcal U$ on $\N_+$ containing $I$, namely an ultrafilter containing the centred system $\{I\setminus\{n\}\colon n=1,2,3,\ldots\}$. Since $I\in{\mathcal U}$, 
  \[\{n\colon d(x,x_n)<\varepsilon \}\notin {\mathcal U},\]
  because the set above is disjoint from $I$. We conclude: $x_n\not\overset{\mathcal U}\to x$.

  \hiddensubsection{}{\bf Ultralimits in metric spaces.} So far, we have only considered ultralimits of sequences of real numbers.
  Of course the definition extends readily to an arbitrary metric space, as follows. A sequence $(x_n)$ of elements of a metric space $(X,d)$ converges to a point $x$ along an ultrafilter $\mathcal U$ on the set of positive natural numbers if for every $\varepsilon >0$ the set
  \[\{n\in\N_+\colon d(x,x_n)<\varepsilon \}\]
  belongs to $\mathcal U$.
  
  \subsubsection{\label{exe:uniqueness}}{\em Exercise.}
  Prove that a sequence $(x_n)$ of elements of a metric space $X$ can have at most one ultralimit along a given ultrafilter. That is, let $\mathcal U$ be an ultrafilter on $\N_+$. Prove that the ultralimit $\lim_{n\to{\mathcal U}}x_n$, if it exists, is unique.
  
  \subsubsection{\label{not}}
  At the same time, let us note that in an arbitrary metric space, not every bounded sequence of elements needs to have an ultralimit. For instance, if $X$ is a metric space with a $0$-$1$ metric and $(x_n)$ is a sequence of pairwise different elements of $X$, then for every free ultrafilter $\mathcal U$ on $\N_+$, the ultralimit $\lim_{n\to{\mathcal U}}x_n$ does not exist.
  
  Indeed, let $x\in X$ be arbitrary. The set
  \[\{n\in\N_+\colon d(x,x_n)<1\}\]
  contains at most one element (in the case where $x=x_n$ for some $n$) and so does not belong to $\mathcal U$. We conclude: $x_n$ does not converge to $x$ along $\mathcal U$. Since this argument applies to every point $x\in X$, the sequence $(x_n)$ does not have an ultralimit.
  
  \subsubsection{}  Since not every bounded sequence in a metric space $X$ needs to have an ultralimit, a natural question to ask is, can be enlarge $X$ in such a way that every sequence has an ultralimit, much in the same way as we can form a completion of $X$ so that every Cauchy sequence is assigned a limit in it? 
  
  The following result provides a negative answer to this question.
  
  \subsubsection{\label{th:compact}}{\bf Theorem.}
  {\em For a metric space $X$, the following conditions are equivalent.
    \begin{enumerate}
  \item $X$ is compact.
  \item Every sequence of elements in $X$ has an ultralimit along every free ultrafilter $\mathcal U$ on $\N_+$.
  \item There exists a free ultrafilter $\mathcal U$ on $\N_+$ with the property that every sequence of elements in $X$ has an ultralimit along $\mathcal U$.
  \end{enumerate}
  }
  
  \noindent $\triangleleft$
     (a) $\Rightarrow$ (b): very much the same proof as for the interval. For every $\varepsilon >0$, choose a finite cover of $X$ with open $\varepsilon $-balls. One of these balls, say $B$, has the property
    \[\{n\colon x_n\in B\}\in {\mathcal U}.\]
    Proceeding recursively, we obtain a centred sequence of open balls $B_i$ of radius converging to zero satisfying
    \[A_i=\{n\colon x_n\in B_i\}\in {\mathcal U}.\]
    The ball centers form a Cauchy sequence and so converge to some limit, $x$. Every neighbourhood $V$ of $x$ contains some $B_i$ for $i$ large enough, and so the set
    \[\{n\colon x_n\in V\}\]
    contains $A_i$ and so belongs to $\mathcal U$.
    
    (b)$\Rightarrow$(c): trivial, given that free ultrafilters on $\N$ exist. Now just pick any one of them.
    
    (c)$\Rightarrow$(a): by contraposition. If $X$ is not compact, then either it is not totally bounded (in which case there is a sequence of points at  pairwise distances $\geq\varepsilon _0>0$ from each other, which has no ultralimit due to the argument in \ref{not}), or non-complete. In the latter case, select a Cauchy sequence $(x_n)$ without a limit. We will show that it does not admit an ultralimit along $\mathcal U$ either. Assume $x$ is such an ultralimit. For every $\varepsilon >0$, the set $\{n\colon d(x_n,x)<\varepsilon \}$ is in $\mathcal U$, so non-empty, and using this, one can recursively select a subsequence of $x_n$ converging to $x$ in the usual sense. But $(x_n)$ is a Cauchy sequence, so this would mean $x_n\to x$.
  \hfill $\triangleright$
  
  \subsubsection{\label{tb}}{\em Exercise.} Deduce that a metric space $X$ metrically embeds into a space $\hat X$ in which every bounded sequence has an ultralimit if and only if every ball in $X$ is totally bounded. In this case, $\hat X$ is the metric completion of $X$, and every closed ball in $X$ is compact.
  
  This is of course a rather restrictive condition, which, for instance, is failed by any infinite-dimensional normed space.

 \subsubsection{}{\em Solution to Exercise \ref{exe:uniqueness}.}
 Suppose $x,y\in X$ and
  \[\lim_{n\to{\mathcal U}}x_n=x\mbox{ and }\lim_{n\to{\mathcal U}}x_n=y.\]
  For every $\varepsilon >0$, both sets 
  \[\{n\in\N_+\colon d(x,x_n)<\varepsilon \}\mbox{ and }\{n\in\N_+\colon d(y,x_n)<\varepsilon \}\]
  belong to the ultrafilter $\mathcal U$, and so does their intersection. As a consequence, this intersection is non-empty, and one can find $n\in\N$ with
  \[d(x,x_n)<\varepsilon \mbox{ and }d(y,x_n)<\varepsilon .\]
  By the triangle inequality,
  \[d(x,y)<2\varepsilon ,\]
  and since this holds for every $\varepsilon >0$, we conclude: $x=y$.
  
  \hiddensubsection{}{\bf Ultraproducts.} In view of (\ref{tb}), we cannot hope to attach a ``virtual'' ultralimit to every bounded sequence of points of a metric space, $X$. If the space is not totally bounded, some sequences are destined to diverge along some ultrafilters in every {\em metric} extension of $X$.
  Still, we can assign to every such sequence an ideal new point, which will be the limit (hence, ultralimit) in case of a Cauchy sequence, and otherwise will register the asymptotic behaviour of the sequence with regard to the given ultrafilter. Now we will briefly examine this important construction.
  
  \subsubsection{}{\em Space of bounded sequences in a metric space.} Let $X=(X,d_X)$ be a metric space. Fix an ultrafilter $\mathcal U$ on the integers. Since we are going to assign an ``ideal'' point to every bounded sequence, the right place to start will be the set $\ell^\infty(\N_+;X)$ of all bounded sequences with elements in $X$. (For example, the space $\ell^\infty$ is, in this notation, $\ell^{\infty}(\N_+;\R)$.)
  We would formally identify the space of ``ideal'' points with $\ell^\infty(\N_+;X)$, or rather with its quotient space under a suitable equivalence relation --- just like we do when we construct the completion of a metric space beginning with the set of all Cauchy sequences. 
  
  Observe that the set $\ell^\infty(\N_+;X)$
  admits a structure of a metric space with regard to the $\ell^\infty$-distance:
  \[d_{\infty}(x,y) = \sup_{n\in\N_+}d_X(x_n,y_n).\]
  
  \subsubsection{\label{exe:complete}}{\em Exercise.} Show that the metric space $\ell^\infty(\N_+;X)$ is complete if and only if $X$ is complete.
  
  % Recall that the {\em density} of a metric space $X$ is the size of the smallest everywhere dense subset in $X$.
  % 
  % \subsubsection{Exercise} Show that the metric space $\ell^\infty(\N_+;X)$ has density $d(X)^{\aleph_0}$.
  % 
  \subsubsection{}
  It is quite natural to make two bounded sequences $x$ and $y$ share the same ``ideal point'' if the distance between the corresponding terms of those sequences converges to zero {\em along the ultrafilter} $\mathcal U$:
  \[x\overset{\mathcal U}\sim y \iff \lim_{n\to{\mathcal U}}d_X(x_n,y_n)=0.\]
  The resulting equivalence relation $\overset{\mathcal U}\sim$ on $\ell^\infty(\N_+;X)$ agrees with the metric in the sense that the quotient metric on the quotient set is well defined.
  For an $x\in \ell^\infty(\N_+;X)$ denote 
  \[[x]_{\mathcal U}=\{y\in X\colon x\overset{\mathcal U}\sim y\}\]
  the corresponding equivalence class.
  
  \subsubsection{}{\em Exercise.}
  Show that the rule
  \[d([x],[y]) =\inf\{d_{\infty}(x^\prime,y^\prime)\colon x^\prime\in [x]_{\mathcal U},y^\prime\in[y]_{\mathcal U}\}\]
  defines a metric on the quotient set $X/\overset{\mathcal U}\sim$, and that this quotient metric can be alternatively described by
  \[d([x],[y]) =\lim_{n\to{\mathcal U}} d_X(x_n,y_n).\]
   
  \subsubsection{}{\em Metric ultrapowers.}
  The quotient metric space $\ell^\infty(\N_+;X)/\overset{\mathcal U}\sim$ is called the {\em metric ultrapower} of $X$ with regard to the ultrafilter $\mathcal U$, and denoted $X^{\N}_{\mathcal U}$.
  
  \subsubsection{}{\em Exercise.}
  Show that the original metric space $X$ canonically isometrically embeds inside of its metric ultrapower under the {\em diagonal embedding}
  \[X\ni x\mapsto [(x,x,x,\ldots)]_{\mathcal U}\in X^{\N}_{\mathcal U},\]
  no matter what the ultrafilter $\mathcal U$ is. In particular, the metric ultrapower of a non-trivial space is itself non-trivial.
  
  \subsubsection{}{\em Exercise.} Show that, if ${\mathcal U}=(n)$ is a principal ultrafilter, the metric ultrapower $X^{\N}_{\mathcal U}$ is canonically isometric to $X$.
  
  \subsubsection{}{\em Exercise.} Let $\mathcal U$ be a non-principal ultrafilter, and suppose that $X$ is a {\em separable} metric space. Show that the metric ultrapower $X^{\N}_{\mathcal U}$ is isometric to $X$ (not necessarily in a canonical way!) if and only if $X$ is compact. 
  
  Before establishing a couple of other properties of the ultrapower, let us extend the definition to the case where terms of sequences come from {\em possibly different} metric spaces: for every $n$, 
  \[x_n\in X_n,\]
  where $(X_n)$, $n=1,2,\ldots$, are metric spaces which may or may not differ from each other. 
  
  \subsubsection{}{\em Pointed metric spaces.}
  In this situation, how do we define a bounded sequence $x=(x_n)$? The notion of boundedness becomes purely relative: one can say that a sequence $y$ is bounded with regard to another sequence, $x$, if the values $d_{X_n}(x_n,y_n)$, $n\in\N$, form a bounded set. However, there is no absolute notion of boundedness. For this reason, it is necessary to select a sequence $x^\ast=(x^\ast_n)$, $x^\ast_n\in X_n$, thus fixing a class of sequences bounded with regard to $x^\ast$. In other words, we are dealing with a family of {\em pointed metric spaces,} $(X_n,x^\ast_n)$. 
  
  For example, in cases of considerable interest, where $X_n$ are normed spaces or metric groups, the selected points are usually zero and the neutral element, respectively. If all spaces $X_n=X$ are the same, in order to obtain the usual notion of (absolute) boundedness of a sequence,
  one chooses any constant sequence $(x^\ast,x^\ast,\ldots)$, all of them giving the same result. If the spaces $X_n$ have a uniformly bounded diameter, the choice of a distinguished sequence does not matter. 
  
  \subsubsection{}{\em Ultraproduct of pointed metric spaces.}
  Now it is clear how to reformulate the definitions. The metric space consisting of all relatively bounded sequences with regard to $x^\ast$ is the $\ell^\infty$-type sum of pointed spaces $(X_n,x^\ast_n)$, denoted \[\oplus_n^{\ell^\infty}(X_n,x^\ast_n).\]
  In cases like the above where the choice of distinguished points is clear or does not matter, these points are suppressed. 
  
  The equivalence relation and the distance on the quotient metric space are defined in exactly the same fashion as above. The resulting metric space 
  \[\oplus_n^{\ell^\infty}(X_n,x^\ast_n)/\overset{\mathcal U}\sim\] is called the {\em metric ultraproduct} of (pointed) metric spaces $(x_n,x^\ast_n)$ modulo the ultrafilter $\mathcal U$, and denoted
  \[\left(\prod_n(X_n,x^\ast_n)\right)_{\mathcal U}.\]
  
  \subsubsection{}{\em Exercise.}
  Assuming that $\mathcal U$ is a non-principal ultrafilter,
  show that there is a canonical isometry between the metric ultraproducts
  \[\left(\prod_n(X_n,x^\ast_n)\right)_{\mathcal U}\mbox{ and } \left(\prod_n(\widehat{X_n},x^\ast_n)\right)_{\mathcal U},\]
  where $\widehat{X_n}$ denotes the completion of the metric space $X_n$.
  \smallskip
  
  $\triangleleft$ {\em Hint:} observe that the ultraproduct on the left admits a canonical isometric embedding into the one on the right. 
  \hfill $\triangleright$
  
  \subsubsection{}{\em Exercise.} Given a Cauchy sequence in the metric ultraproduct of spaces $X_n$, show that it is the image under the quotient map of some Cauchy sequence in the $\ell^\infty$-type sum of the spaces $X_n$.
  \smallskip
  
  $\triangleleft$ {\em Hint:} use the following equivalent definition: a sequence $(a_n)$ is Cauchy if and only if for every $\varepsilon >0$ there is $N$ such that the open ball $B_\varepsilon (x_N)$ contains all $x_n$ with $n\geq N$. 
  \hfill $\triangleright$
  
  The following stands in contrast to Exercise \ref{exe:complete}.
  
  \subsubsection{}{\em Exercise.} Deduce from the previous two exercises that the metric ultraproduct of a family of arbitrary metric spaces, formed with regard to a non-principal ultrafilter, is always a complete metric space.
  \smallskip
  
  Recall that the covering number $N(X,\varepsilon )$, where $\varepsilon >0$, of a metric space $X$ is the smallest size of an $\varepsilon $-net for $X$. 
  
  \subsubsection{}{\em Exercise.} ($\ast$)
  Prove that if the metric ultraproduct $X$ of a family $(X_n)$ of separable metric spaces is separable, then it is compact. Moreover, in this case for every $\varepsilon >0$ the sequence of covering numbers $N(X_n,\varepsilon )$ is essentially bounded, that is, uniformly bounded from above for all $n$ belonging to some element $A=A(\varepsilon )$ of the ultrafilter. 
  
  Otherwise, $X$ has density character continuum.
  
  \hiddensubsection{}{\em The space of ultrafilters.} In conclusion, we want to give a clear literal meaning to the symbol
  \[\lim_{n\to{\mathcal U}} x_n,\]
  used for an ultralimit along an ultrafilter $\mathcal U$.
  In fact, an ultralimit of a sequence can be interpreted as a limit in the usual sense. 

  \subsubsection{}
Denote by $\beta\N$ the set of all ultrafilters on the set
$\N$ of natural numbers. Identify $\N$ with a subset of $\beta\N$ by assigning to each natural number $n$ the corresponding principal ultrafilter, $(n)=\{A\subseteq\N\colon A\ni n\}$:
\[\N\hookrightarrow \beta\N,~~\N\ni n\mapsto (n)\in\beta\N.\]
This mapping is clearly one-to-one, and in this sense we will often refer to
$\N$ as a subset of $\beta\N$.

\subsubsection{}
For every $A\subseteq\N$, denote by $\tilde A$ the family of
all ultrafilters $\xi\in\beta\N$ containing $A$ as an element:
\[\tilde A =\{{\mathcal U}\in\beta\N\colon A\in {\mathcal U}\}.\]

\subsubsection{\label{exe:base}}{\em Exercise.}
Prove that the collection
\[\{\tilde A\colon A\subseteq\N\}\]
forms a base for a topology on $\beta\N$.

\subsubsection{\label{exe:discrete}}{\em Exercise.} Prove that the above described topology on $\beta\N$ induces the discrete topology on $\N$ as on a topological subspace.

\subsubsection{\label{exe:hausdorff}}{\em Exercise.}
Prove that the topology as above on $\beta\N$
is Hausdorff.

\subsubsection{\label{exe:betacomp}}{\em Exercise.}
Prove that $\beta\N$ with the above topology is compact.

\subsubsection{\label{exe:dense}}{\em Exercise.}
Prove that $\N$ forms an {\it everywhere dense subset}
in $\beta\N$, that is, the closure of $\N$ in $\beta\N$ is
the entire space $\beta\N$.

\subsubsection{\label{exe:extends}}{\em Exercise.}
Let $X$ be an arbitrary compact space, and let $f\colon\N\to X$
be an arbitrary mapping. Prove that $f$ extends
to a continuous mapping $\tilde f\colon \beta X\to X$.

\subsubsection{\label{exe:funique}}{\em Exercise.}
Prove that a continuous mapping $\tilde f$ as above is unique
for any given $f$, assuming the space $X$ is compact Hausdorff. In other words, if $g\colon \beta\N\to X$ is
a continuous mapping with $g\vert_\N=f$, then $g=\tilde f$.

\subsubsection{}
The compact space $\beta\N$ as above is called the {\it Stone--\u Cech
compactification} (or else {\it universal compactification}) 
of the discrete space $\N$.

\subsubsection{}{\em Exercise.}
Let ${\mathcal U}\in\beta\N$ be an ultrafilter, and let $x=(x_n)$ be a sequence of points in a compact metric space $X$, in other words, a function $x\colon\N\to X$. Prove that the ultralimit of $(x_n)$ along $\mathcal U$ is exactly the value of the unique continuous extension $\tilde x$ at $\mathcal U$, that is, the usual classical limit of the function $x$ as $n\to {\mathcal U}$:
\[\lim_{n\to{\mathcal U}}x_n = \tilde x(\mathcal U) = \lim_{n\to{\mathcal U}}x.\]

\subsubsection{}{\em Exercise.} Generalize the above as follows: suppose $\mathcal U$ is an ultrafilter on $\N$, and let $(x_n)$ be a sequence of points in a metric space $X$. Prove that the following are equivalent:
\begin{enumerate}
\item The sequence $(x_n)$ converges along the ultrafilter $\mathcal U$ to some point $x\in X$.
\item The function $n\mapsto x_n$ has a classical limit $x\in X$ as $n$ approaches the point $\mathcal U$ in the topological space $\beta\N$.
\end{enumerate}
Moreover, if the two limits in (1) and (2) exist, they are equal.

\subsubsection{}{\em Solution to Exercise \ref{exe:base}.}
Since every ultrafilter ${\mathcal U}$
 on $\N$ contains $\N$ and thus ${\mathcal U}\in\tilde\N$,
one concludes that $\tilde\N=\beta\N$ and therefore
\[\cup_{A\subseteq\N}\tilde A=\beta\N.\]
Now let $A,B\subseteq\N$, and let ${\mathcal U}\in \tilde A\cap\tilde B$.
Then ${\mathcal U}\ni A$ and ${\mathcal U}\ni B$, therefore ${\mathcal U}\ni A\cap B$.
It means that ${\mathcal U}\in\widetilde{A\cap B}$. On the other hand,
clearly $\widetilde{A\cap B}\subseteq\tilde A\cap\tilde B$ because,
more generally, if $C\subseteq D\subseteq\N$, then
$\tilde C\subseteq\tilde D$ (un ultrafilter containing $C$
necessarily contains $D$ as well).
We have proved that $\widetilde{A\cap B}=\tilde A\cap\tilde B$, from
where the second axiom of a base follows.

\subsubsection{}{\em Solution to Exercise \ref{exe:discrete}.}
Here one should notice that $\N$ is naturally identified with a
subset of $\beta\N$ through identifying each element $n\in\N$
with the principal (trivial) ultrafilter, $(n)$, generated by
$n$:
\[\N\ni n\mapsto (n)=\{A\subseteq\N\colon A\ni n\}\in\beta\N.\]
This mapping is clearly one-to-one, and so we will often think of
$\N$ as a subset of $\beta\N$.

To establish the claim, it is enough to show that every singleton
in $\N$ is open in the induced topology. Let $n\in\N$.
Then it is easy to see that
\[\widetilde{\{n\}}=\{(n)\},\]
because if an ultrafilter ${\mathcal U}$ contains the set $\{n\}$,
it must coincide with the trivial ultrafilter $(n)$.
Since the singleton $\{(n)\}$ is open in $\beta\N$,
it is also open in $\N$.
$\triangleright$

\subsubsection{}{\em Solution to Exercise \ref{exe:hausdorff}.}
Let ${\mathcal U},\zeta\in\beta\N$, and let ${\mathcal U}\neq\zeta$.
What is means, is this: there is an $A\subseteq\N$ such that
$A$ belongs to one ultrafilter (say ${\mathcal U}$) and not the other
(that is, $A\notin\zeta$). A major property of ultrafilters implies
then that $\N\setminus A\in\zeta$. The sets
$\tilde A$ and $\widetilde{\N\setminus A}$ are both open in
$\beta\N$, contain ${\mathcal U}$ and $\zeta$ respectively, and are
disjoint: indeed, assuming $\kappa\in\tilde A\cap 
\widetilde{\N\setminus A}$ would mean that the ultrafilter $\kappa$
contains both $A$ and $\N\setminus A$, which is impossible.

\subsubsection{}{\em Solution to Exercise \ref{exe:betacomp}.}
Let $\gamma$ be an arbitrary open cover of $\beta\N$.
Then
\[\beta:=\{\tilde A\colon 
\mbox{ for some $V\in\gamma$, $\tilde A\subseteq
V$}\}\]
is also an open cover of $\beta\N$: indeed, if ${\mathcal U}\in\beta\N$,
then for some $V\in\gamma$ one has
${\mathcal U}\in V$, and by the very definition of the topology on $\beta\N$,
there is an $A\subseteq\N$ with ${\mathcal U}\in\tilde A\subseteq V$.
Notice that it is now enough to select a finite subcover,
say $\beta_1$, of $\beta$:
this done, we will then replace each $\tilde A\in\beta_1$
with an arbitrary $V\in\gamma$ such that $V\supseteq\tilde A$,
and thus we will get a finite subcover of $\gamma$. From now on, we
can forget of $\gamma$ and concentrate on $\beta$ alone, and the
cover $\beta$ consists of basic open sets.

Consider the system of subsets of $\N$,
\[\delta:=\{A\subseteq \N\colon \tilde A\in\beta\}.\]
This is clearly a cover of $\N$: $\cup\delta=\N$.
If we assume that $\delta$ contains no finite subcover, it is
the same as to assume that the system of complements,
\[\{\N\setminus A\colon A\in\delta\},\]
is centred. Being centred, it is contained in an ultrafilter,
say ${\mathcal U}$. One has
\[\forall A\in\delta, ~~ \N\setminus A\in{\mathcal U},\]
and consequently,
\[\forall A\in\delta, ~~ A\notin{\mathcal U},\]
which can be rewritten as
\[\forall A\in\delta, ~~ {\mathcal U}\notin\tilde A,\]
that is,
\[{\mathcal U}\notin\cup\beta,\]
a contradiction. One concludes: there are finitely many elements
$A_1,A_2,\dots,A_n\in\delta$ with
\[A_1\cup A_2\cup\cdots\cup A_n=\N.\]
By force of a familiar property of ultrafilters, if ${\mathcal U}$ is an
arbitrary
ultrafilter on $\N$, it must contain at least one of the sets
$A_i$, $i=1,\dots,n$, or, equivalently, ${\mathcal U}\in\widetilde{A_i}$
for some $i=1,2,\dots,n$.
It follows that
\[\widetilde{A_1}\cup \widetilde{A_2}\cup\cdots\cup 
\widetilde{A_n}=\beta\N,\]
which of course finishes the proof by supplying the
desired finite subcover $\beta_1$ of $\beta$.

\subsubsection{}{\em Solution to Exercise \ref{exe:dense}.}
It is enough to find an element of $\N$ in an arbitrary 
non-empty open subset of $\beta\N$, say $U$.
Since $U$ is the union of basic open subsets, for some
$A\subseteq\N$, $A\neq\emptyset$, one has $\tilde A\subseteq U$.
Let $a\in A$ be arbitrary. Then $A\in (a)$, that is,
$(a)\in\tilde A\subseteq U$, as required.

\subsubsection{}{\em Solution to Exercise \ref{exe:extends}.}
Let ${\mathcal U}\in\beta\N$ be arbitrary. Since the space $X$ is compact, there is a limit, $\lim_{n\to{\mathcal U}}f(x)\in X$. Denote this limit $\tilde f({\mathcal U})$. 
% 
% Denote by $f_\ast({\mathcal U})$,
% as before, the direct image ultrafilter of ${\mathcal U}$ under $f$.
% (Recall that $f_\ast({\mathcal U})$ is an ultrafilter on $X$ defined by
% the condition: $X\supseteq
% A\in f_\ast({\mathcal U})\Leftrightarrow f^{-1}(A)\in{\mathcal U}$.)
% Since the space $X$ is compact, there is a limit,
% $y=\lim f_\ast({\mathcal U})\in X$. (It means that $f_\ast({\mathcal U})$
% contains the neighbourhood system of $y$.) Set
% \[\tilde f({\mathcal U}):=\lim f_\ast({\mathcal U}).\]
It remains to prove the continuity of $\tilde f$. 
Let $U\subseteq X$ be open, and let $\tilde f({\mathcal U})\in U$
for some ${\mathcal U}\in\beta\N$. We want to find a neighbourhood 
$V$ of ${\mathcal U}$ in $\beta\N$ such that $\tilde f(V)\subseteq U$.

Every compact space ($T_1$ or not) is regular: this is
easily proved using an argument involving finite
subcovers of covers of closed sets with open neighbourhoods.
Therefore, one can find an open set, $W$, in $X$ with
\[\tilde f({\mathcal U})\in W\subseteq {\mathrm{cl}}\,W\subseteq U.\]
The condition $\tilde f({\mathcal U})\in W$, that is,
$\lim_{n\to{\mathcal U}}f(x)\in W$, implies that $f^{-1}(W)\in {\mathcal U}$.
Set $A:=f^{-1}(W)$. It is a non-empty
subset of $\N$, and therefore $V:=\tilde A$ forms an open neighbourhood
of ${\mathcal U}$ in $\beta\N$. If now $\zeta\in V$,
that is, $\zeta\ni A=f^{-1}(W)$, then every neighbourhood of the limit, $\tilde f(\zeta)$, in $X$ meets $W$, meaning that
$\tilde f(\zeta)\in {\mathrm{cl}}\,W$ and consequently
$\tilde f(\zeta)\in U$, as required. We conclude:
$\tilde f(V)\subseteq W$. The continuity of
the mapping
\[\tilde f\colon\beta\N\to X\]
is thus established.

\subsubsection{}{\em Solution to Exercise \ref{exe:funique}.}
This follows at once
from a much more general assertion, making no use of compactness
whatsoever: if $f,g\colon X\to Y$ are
continuous mappings between two topological spaces, where $Y\in T_2$,
and if $Z\subseteq X$ is everywhere dense in $X$, and
if $f\vert_Z=g\vert_Z$, then $f=g$. 

Indeed, assume that for some $x\in X$ one has
\[f(x)\neq g(x).\]
Find disjoint open neighbourhoods $V$ and $U$ in $Y$ of $f(x)$
and $g(x)$, respectively. Their preimages, $f^{-1}(V)$
and $g^{-1}(U)$, form open neighbourhoods of $x$ in $X$,
and so does their intersection.
Since $Z$ is everywhere dense in $X$, there is a $z\in Z$
such that $z\in f^{-1}(V)\cap g^{-1}(U)$.
Now one has
\[V\ni f(z)=g(z)\in U,\]
a contradiction since $V\cap U=\emptyset$.

Without the assumption that $X$ be Hausdorff,
the statement is no longer true. A simple example is this:
let $X=\{0,1\}$ be a two-element set with indiscrete topology
(that is, only $\emptyset$ and all of $X$ are open).
Clearly, such an $X$ is compact. Define a map
$f\colon \N\to X$ as the constant map, sending each $n\in\N$ to $0$.
Since an arbitrary map from any topological space to an indiscrete
space is always continuous, the map $f$ admits more than one
extension to a map to $X$, and all of them are continuous.
(E.g., one can send all elements of
the remainder $\beta\N\setminus\N$ to
$1$, or else to $0$.)

Thanks to Tullio G. Ceccherini-Silberstein and Michel Coornaert for a number of remarks on this Appendix.

\renewcommand\leftmark{BIBLIOGRAPHY}
\renewcommand\rightmark{BIBLIOGRAPHY}

\bibliographystyle{plain}
\bibliography{bibSoficLocal}

\renewcommand\leftmark{INDEX}
\renewcommand\rightmark{INDEX}

\printindex

\end{document}